\documentclass{amsart}

\usepackage{amssymb}
\usepackage{bbm}
\usepackage[backend=biber,style=alphabetic]{biblatex}
\usepackage{hyperref}
\usepackage{ifthen}
\usepackage{subcaption}
\usepackage{tikz}
\usepackage{tikz-cd}
\usepackage{xfrac}

\usetikzlibrary{arrows.meta,calc,decorations.markings,decorations.pathmorphing,decorations.pathreplacing,shapes}
\newcommand{\ArrTip}{Latex[length=4,width=4]}
\tikzset{->-/.style={decoration={markings,mark=at position .5 with {\arrow{\ArrTip}}}, postaction={decorate}}}
\tikzset{->>-/.style={decoration={markings,mark=at position .45 with{\arrow{\ArrTip}},mark=at position .6 with{\arrow{\ArrTip}}},postaction={decorate}}}
\tikzset{->>>-/.style={decoration={markings,mark=at position .4 with{\arrow{\ArrTip}},mark=at position .5 with{\arrow{\ArrTip}},mark=at position .6 with{\arrow{\ArrTip}}},postaction={decorate}}}

\newtheorem{conjecture}{Conjecture}[section]
\newtheorem{corollary}[conjecture]{Corollary}
\newtheorem{definition}[conjecture]{Definition}
\newtheorem{example}[conjecture]{Example}
\newtheorem{lemma}[conjecture]{Lemma}
\newtheorem{proposition}[conjecture]{Proposition}
\newtheorem{remark}[conjecture]{Remark}
\newtheorem{theorem}[conjecture]{Theorem}

\newcommand{\cho}{\chi^{\textnormal{orb}}}
\newcommand{\cht}{\chi^{\textnormal{top}}}
\newcommand{\corner}[1][]{
 \underset{f#1}
 {
  \overset{c#1}{\curvearrowleft}
 }
}
\newcommand{\covers}{\mathbb{X}_n}
\newcommand{\defeq}{\stackrel{\tn{def}}{=}}
\newcommand{\dft}{\textnormal{dft}}
\newcommand{\Fspx}{F_{\textnormal{spx}}}
\newcommand{\Forb}[1][m]{F^{(#1)}_{\textnormal{orb}}}
\newcommand{\Hom}[3][]{\tn{Hom}_{#1}\left(#2\to #3\right)}
\newcommand{\md}{\textnormal{md}}
\newcommand{\iterorb}{{\tiny\begin{matrix}m\in\mathbb{N}\\f\!\in\!\Forb(\Sigma)\end{matrix}}}
\newcommand{\pio}{\pi_1^{\textnormal{orb}}}
\newcommand{\tn}{\textnormal}
\newcommand{\tr}{\operatorname{tr}}
\newcommand{\whS}{{\widehat{\Sigma}}}

\author{Adam Klukowski}
\address{\newline Mathematical Institute\newline University of Oxford}
\email{klukowski@maths.ox.ac.uk/klukowski@mpim-bonn.mpg.de}

\begin{document}

\title[Counting tangles]{Counting tangles in coverings of closed 2-orbifolds via ribbon categories}

\begin{abstract}
We describe what a typical covering of a closed 2-orbifold looks like on a small scale. Specifically, we prove that a uniformly random covering of degree $n$ contains in expectation $\Theta(n^0)$ short closed geodesics, potentially $\Theta(n^0)$ pairs of nearby order-2 cone points, in general $\Theta(n^{\frac{1}{m}})$ cone points of order $m$, and with high probability no small regions with more complicated topology. This generalises many results of Magee and Puder (2023) from orientable surfaces to possibly non-orientable orbifolds, and answers some of the questions raised by Puder and Zimhoni (2024).
\end{abstract}

\maketitle

\section{Introduction}\label{sec:intro}

Non-Euclidean crystallographic (NEC) groups are discrete subgroups of $\tn{Isom}(\mathbb{H}^2)$, and generalise Fuchsian groups by allowing orientation-reversing isometries. For a NEC group $\Gamma$, denote by
\begin{equation*}
\covers\defeq\Hom{\Gamma}{S_n}
\end{equation*}
the set of homomorphisms from it to the symmetric group. It is naturally identified with the set of degree-$n$ coverings of the 2-orbifold $\Sigma=\sfrac{\mathbb{H}^2}{\Gamma}$, which we also denote $\covers$ by abuse of notation. Motivated by applications to spectral geometry, random matrix theory, and strong convergence, \cite{asymptotic_stats_random_covering_surfaces_MageePuder23} studied the small-scale properties of a uniformly random $\phi\in\covers$ in the special case when $\Sigma$ is a non-singular orientable surface. Concrete examples of such properties are: in the dynamical language, the expected number of fixed points of $\phi(\gamma)$ for a given $\gamma\in\Gamma$, or in geometric language, the prevalence of tangles (small subsurfaces with complicated topology) in the covering of $\Sigma$ corresponding to $\phi$. It was shown e.g. that bounded-diameter regions are with high probability either topological disks or annuli. Unfortunately, this knowledge does not immediately transfer to more general NEC groups, since their subgroups are typically far from orientable surfaces. When $\Sigma$ is a non-orientable surface, then no odd-degree covering can be orientable, and Mednykh's formula (Lemma \ref{lmm:Mednykh}) reveals that among the even-degree coverings the proportion of orientable ones is still vanishingly small. In a similar vein, typical coverings of singular 2-orbifolds contain a lot of cone points. This is quantified by Theorem \ref{thm:cone_pt_count} below, proved in Section \ref{sec:main_combi}. In the statement, closed means without boundary (this property corresponds to $\Gamma$ containing no reflections, i.e. non-trivial isometries of $\mathbb{H}^2$ that fix a line pointwise).
\begin{theorem}\label{thm:cone_pt_count}
Fix $m\in\mathbb{N}_+$, and a compact closed 2-orbifold $\Sigma$ with $\cho(\Sigma)<0$. Denote by $N(\whS)$ the number of cone points of order $m$ in a covering $\whS$ of $\Sigma$, and by $m_1,\dots,m_k$ the orders of cone points of $\Sigma$ divisible by $m$. Considered as a random variable over $\covers$ equipped with the uniform measure, we have that
\begin{equation*}
n^{-\frac{1}{m}}N\quad\to\quad\sum_{i=1}^k\frac{m}{m_i}
\end{equation*}
in probability as $n\to\infty$.
\end{theorem}

The goal of this article is to close this gap and characterise the statistics of small-scale phenomena in random coverings of 2-orbifolds. Our main result, Theorem~\ref{thm:alg_tf}, generalises \cite[Theorem 1.3]{asymptotic_stats_random_covering_surfaces_MageePuder23} from surfaces to arbitrary cocompact NEC groups without reflections, and confirms some of the conjectures in \cite[\S7]{local_stats_random_perms_free_prods_PuderZimhoni24} about the scope of the described behaviour. In its statement below, $\tn{fix}_A(\phi)$ denotes the subset of $[n]\defeq\{1,\dots,n\}$ fixed by every element of $A\subseteq\Gamma$ under the action $\phi\colon\Gamma\curvearrowright[n]$, and $\mathbb{E}_n[\bullet]\defeq\mathbb{E}[\bullet\mid\whS\in\covers]$ is a shorthand for the expected value over $\covers$ equipped with the counting (=uniform) measure.
\begin{theorem}\label{thm:alg_tf}
Let $\Gamma$ be a cocompact NEC group without reflections. For a f.g. (finitely generated) subgroup $\Lambda\leq\Gamma$, we have
\begin{equation*}
\mathbb{E}_n[\#\tn{fix}_\Lambda(\phi)]=\#\tn{mog}(\Lambda)\cdot n^{\chi_\tn{max}(\Lambda)}\cdot(1+o(1)),
\end{equation*}
where
\begin{equation*}
\chi_\tn{max}(\Lambda)=\max\{\cho(J)\mid\Lambda\le J\le\Gamma\}
\end{equation*}
is the largest orbifold Euler characteristic of an overgroup of $\Lambda$, and
\begin{equation*}
\tn{mog}(\Lambda)=\{J\le\Gamma\mid J\ge\Lambda,\ \cho(J)=\chi_\tn{max}(\Lambda)\}
\end{equation*}
is the set of ``maximal'' overgroups realising it.
\end{theorem}
Theorem \ref{thm:alg_tf} is proved in Section \ref{ssec:res_main_pf} (together with the finiteness of $\tn{mog}(\Lambda)$ in Remark~\ref{rmk:mog_finite}). Applying it to cyclic subgroups generalises \cite[Theorem 1.2]{asymptotic_stats_random_covering_surfaces_MageePuder23}. For a single element $\gamma\in\Gamma$ we abbreviate $\tn{fix}_\gamma=\tn{fix}_{\{\gamma\}}$, and denote the indicator function by $\mathbbm{1}$.
\begin{corollary}\label{cor:E_fixpt}
For $\gamma\in\Gamma$ of order $m<\infty$ we have
\begin{equation*}
\mathbb{E}_n[\#\tn{fix}_\gamma(\phi)]=n^{\frac{1}{m}}\cdot(1+o(1)).
\end{equation*}
If instead $\gamma\in\Gamma$ is neither torsion nor proper power, then for any $m\in\mathbb{N}_+$ we have
\begin{equation*}
\mathbb{E}_n[\#\tn{fix}_{\gamma^m}(\phi)]=\sigma_0(m)+\mathbbm{1}_{\{\gamma=xy\text{ with }x^2=y^2=1\}}\sigma_1(m)+o(1),
\end{equation*}
where $\sigma_0$ and $\sigma_1$ denote respectively the number and the sum of positive divisors.
\end{corollary}
The term $\sigma_0(m)$ counts cyclic subgroups of the form $\langle\gamma^d\rangle$ with $d|m$, and $\sigma_1(m)$ counts dihedral overgroups that appear for example when $\gamma$ bounds a disk with two cone points of order 2.

\begin{remark}
Our methods also show that for any fixed $\Lambda$ and $I_0$ there is an expansion
\begin{equation*}
\mathbb{E}_n[\#\tn{fix}_\Lambda(\phi)]=\sum_{\tiny\begin{matrix}-1\leqslant i<I_0\\i\in\alpha\cdot\mathbb{Z}\end{matrix}}a_i(\Lambda)n^{-i}+O\left(n^{-I_0}\right)
\end{equation*}
for some coefficients $a_i(\Lambda)\in\mathbb{R}$, where $\alpha$ is the reciprocal of the least common multiple of orders of torsion elements in $\Gamma$.
\end{remark}

The equivalence between homomorphisms to $S_n$ and degree-$n$ coverings translates Theorem~\ref{thm:alg_tf} into more geometrically flavoured Corollary~\ref{cor:geo_tf}. In this language, a fixed point of a f.g. subgroup $\Lambda$ in an action on $[n]$ is equivalent to the corresponding covering containing an immersed ``core'' sub-2-orbifold (possibly with boundary) whose fundamental group is $\Lambda$. When $\Lambda$ is finite cyclic, $\mathbb{Z}$, or infinite dihedral, then the subsurface is, respectively, a disk containing a cone point, an annulus surrounding an essential loop, or a disk with a pair of nearby order-2 cone points. More complicated sub-orbifolds are called tangles.
\begin{definition}
Let $\Sigma$ be a 2-orbifold equipped with a hyperbolic metric, and $L>0$ a fixed length scale. A \emph{geometric $L$-tangle} in $\Sigma$ is a metric $L$-ball whose orbifold fundamental group is virtually \{free non-cyclic\}.
\end{definition}
This allows us to state Corollary~\ref{cor:geo_tf} of Theorem~\ref{thm:alg_tf}. As in Theorem~\ref{thm:alg_tf}, $\covers$ is taken with the uniform measure.
\begin{corollary}\label{cor:geo_tf}
Let $\Sigma$ be a closed compact 2-orbifold equipped with a hyperbolic metric. For any fixed $L>0$ the following hold.
\begin{enumerate}
\item A typical covering of $\Sigma$ is geometrically tangle-free, in the sense that
\begin{equation*}
\mathbb{P}\left(\whS\text{ contains a geometric }L\text{-tangle}\ \middle|\ \whS\in\covers\right)\to0
\end{equation*}
as $n\to\infty$.
\item The expected number of closed geodesics of length at most $L$ (not homotopic to a loop around a cone point) in a uniformly random covering from $\covers$ remains bounded as $n\to\infty$.
\end{enumerate}
\end{corollary}

\begin{remark}\label{rmk:fibre_label}
The choice of $\covers=\Hom{\pio\Sigma}{S_n}$ as our random model leads to a technical subtlety, in that all coverings come equipped with a fibre labelling, i.e. a bijection between the preimages of a basepoint and the set $[n]$. Coverings without this auxiliary data would correspond to homomorphisms $\pio\Sigma\to S_n$ up to conjugation. The subset of $\covers$ consisting of connected coverings (i.e. homomorphisms with transitive image) is in natural bijection with subgroups of $\Gamma$ of index $n$ equipped with a labelling of cosets; index-$n$ subgroups with no coset labelling would correspond to homomorphisms $\pio\Sigma\to S_n$ with transitive image up to conjugation by $S_{n-1}$.
\end{remark}

\subsection{Motivation and related work}

This article builds on \cite{asymptotic_stats_random_covering_surfaces_MageePuder23}. Our Theorem~\ref{thm:alg_tf} is a direct generalisation of one of their main results, \cite[Theorem 1.3]{asymptotic_stats_random_covering_surfaces_MageePuder23}, from (non-singular) orientable surfaces to general closed 2-orbifolds. Independently of this work, another proof (also for orientable surfaces) was given in \cite{word_maps_surface_rels_symm_gps_Cassidy26}. Let us now review a slightly wider context and further motivations.

\subsubsection{Spectral geometry}

Historically, the notion of tangles came from and was closely related to spectral geometry and spectral gaps. It originates in \cite[\S4]{pf_alon_2nd_eval_conj_related_problems_Friedman08}, in the course of the proof of Alon's conjecture \cite[Conjecture 5.1]{evals_expanders_Alon86}. Tangles often appear as a difficulty in applying trace methods (see discussion in \cites[\S2]{pf_alon_2nd_eval_conj_related_problems_Friedman08}[\S2.5]{spectral_gap_random_hyp_surfaces_AnantharamanMonk24}). Consequently, numerous works on spectral gaps of random objects in various settings incorporated a step establishing tangle-freeness in some quantitative sense. Example random models in which research followed this pattern include random $d$-regular graphs \cites{pf_alon_2nd_eval_conj_related_problems_Friedman08}{Ramanujan_prop_edge_universality_random_regular_graphs_HuangMcKenzieYau25}, random coverings of a fixed graph \cite{eigenvals_random_lifts_polys_random_permutation_matrices_BordenaveCollins19}, Weil--Petersson random hyperbolic surfaces \cite{FriedmanRamanujan_fns_random_hyp_geo_application_spectral_gaps_2_AnantharamanMonk25} building on \cite{Moebius_inversion_formula_discard_tangled_hyp_surfaces_AnantharamanMonk24}, and random coverings of a fixed hyperbolic surface \cite{random_cover_of_cpct_hyp_surface_has_relative_spectral_gap_316-eps_MageeNaudPuder22} relying on \cite{asymptotic_stats_random_covering_surfaces_MageePuder23}. We hope that the work presented here can be used in the future to extend the spectral results to random branched coverings of surfaces.

\subsubsection{Ribbon categories}

The first point of divergence of this work from \cite{asymptotic_stats_random_covering_surfaces_MageePuder23} is our approach replacing the standard one-relator presentation of a surface group with more flexible presentations induced by triangulations. As a result, instead of relying on domain-specific tricks, our proof highlights the links with modular functors and lattice TQFTs. Overall, our proof strategy is similar to the computation of lattice topological quantum field theories in \cite{mednykh_formula_lattice_tqfts_Snyder17}. The derivation of Mednykh's formula given there roughly parallels our proof of Theorem~\ref{thm:count_tangles}, specialised to the case of non-singular surface $\Sigma$ and empty immersed complex $T$.

\subsubsection{Residual finiteness, subgroup growth, Benjamini--Schramm convergence}

Our main result, Theorem~\ref{thm:alg_tf}, can be viewed as a refinement of residual finiteness or subgroup growth of NEC groups, or Benjamini--Schramm convergence of branched coverings.

Residual finiteness of a group $\Gamma$ is equivalent to existence for every $x\in\Gamma\setminus\{1\}$ of a homomorphism $\phi\colon\Gamma\to S_n$ that does not vanish on $x$. Corollary~\ref{cor:E_fixpt} implies a quantitative version of this property, namely that for a NEC $\Gamma$ without reflections and $1\neq x\in\Gamma$ we have
\begin{equation*}
\frac{\#\{\phi\colon\Gamma\to S_n\mid\phi(x)\neq1\}}{\#\Hom{\Gamma}{S_n}}\geqslant1-O\left(n^{-\frac{1}{2}}\right)\qquad\text{as}\qquad n\to\infty.
\end{equation*}
In other words, not only some but actually most homomorphisms to $S_n$ do not kill a given $x$ (compare \cite[Corollary 1.5]{asymptotic_stats_random_covering_surfaces_MageePuder23}).

Subgroup growth is the theory of the asymptotic growth rate of the number of index-$n$ subgroups with $n$, and related questions. For an NEC group, this growth rate was determined in \cites[Theorem A]{char_theory_symm_gps_subgp_growth_fuchsian_gps_random_walks_MuellerSchlagePuchta07}[Theorem 1.12(i)]{fuchsian_gps_coverings_Riemann_surfaces_subgp_growth_random_quots_walks_LiebeckShalev04} (the difference in their statements is related to Remark~\ref{rmk:fibre_label}). Theorem~\ref{thm:alg_tf} answers a refinement of this question, namely what is the asymptotic behaviour of the number of index-$n$ subgroups containing a specified f.g. $\Lambda\leq\Gamma$.

The notion of Benjamini--Schramm convergence was introduced in \cite{recurrence_distributions_limits_finite_planar_graphs_BenjaminiSchramm01} for graphs. By close analogy with their definition, a sequence of coverings $(\widehat{\mathcal{M}}_k)_k$ of a Riemannian manifold $\mathcal{M}$ is said to Benjamini--Schramm converge to the universal covering $\widetilde{\mathcal{M}}$ when for arbitrary length scale $L\in\mathbb{R}_+$ we have
\begin{equation*}
\frac{\tn{Vol}\left(\left\{x\in\widehat{\mathcal{M}}_k\mid\text{injective radius at }x\text{ is at least }L\right\}\right)}{\tn{Vol}\left(\widehat{\mathcal{M}}_k\right)}\to1\qquad\text{as}\qquad k\to\infty,
\end{equation*}
where Vol is the Riemannian volume. In our case, Corollary~\ref{cor:geo_tf} and Theorem~\ref{thm:cone_pt_count} imply that large-degree random coverings of hyperbolic orbifolds Benjamini--Schramm converge to the universal covering $\mathbb{H}^2$.

\subsubsection{Geometric group theory of NEC groups}

One of the by-products of this work is an algorithm that, given a word $w$ in the generators of a NEC group $\Gamma$ (cocompact without reflections), computes the exponent in the asymptotics of the expected number of fixed points of $w$ under a random action $\Gamma\curvearrowright[n]$. As demonstrated by Corollary~\ref{cor:E_fixpt}, this exponent is different for the identity and non-trivial group elements, so such a procedure entails solving the word problem in $\Gamma$. This and similar results of Section~\ref{sec:resolutions} mirror Dehn's algorithm for word problem in surface groups \cite{transformationen_kurven_zweiseitigen_flachen_Dehn12} and the theory of combinatorial compact cores in \cite{core_surfaces_MageePuder22}, except again we work with ``simplicial'' instead of standard one-relator presentation.

\subsection{Overview and organisation}

The space $\covers$ is a finite set. In order to perform counting in it, we need to reformulate Theorem~\ref{thm:alg_tf} into a combinatorial statement. We do this in Section~\ref{sec:simplicial}. NEC group $\Gamma$ is represented by a branched $\Delta$-complex (a certain generalisation of simplicial complexes) $\Sigma$, finite-index subgroups by coverings $\whS\to\Sigma$, and the f.g. subgroup $\Lambda\leq\Gamma$ by a cellular immersion $T\looparrowright\Sigma$. We also develop a combinatorial version of the notion of boundary of a surface.

In Section~\ref{sec:main_combi} we state Theorem~\ref{thm:E_combi_emb}, the master result controlling the expected number of embeddings of a given complex in a random covering. This requires the notion of boundary defect of an immersed complex, which we define and then demonstrate that it behaves like a kind of boundary curvature. We introduce resolutions, which partition the set of immersions of the given complex $T$ into sets of embeddings, and state Proposition~\ref{prop:resolution} on existence of resolutions with controlled boundary. Finally we use resolutions to bridge immersions and embeddings and deduce Theorem~\ref{thm:alg_tf} from Theorem~\ref{thm:E_combi_emb}.

The core of this article will rely on large linear algebra. Section~\ref{sec:ribbon_graph} lays out the background on ribbon categories and ribbon graphs, tools that form the bridge between combinatorics and linear algebra, and thanks to which the latter becomes tractable.

In Section~\ref{sec:extensibility} we prove Proposition~\ref{prop:Rg_trace}. Given a group element $g_e\in G$ for every edge $e\in E(\Sigma)$, Proposition~\ref{prop:Rg_trace} gives a criterion for the extendability of the tuple $\vec{g}$ to a homomorphism $\pio\Sigma\to G$, phrased in terms of ribbon graphs from Section~\ref{sec:ribbon_graph}. We apply it to give a proof of Mednykh's formula, by summing the indicator function of the subset $\Hom{\pio\Sigma}{G}$ over the ambient space $\Hom{\pi_1\Sigma^{(1)}}{G}$ of homomorphisms out of the (free) fundamental group of the 1-skeleton. This is a simplified special case of this article's main technical proof in Section~\ref{sec:counting}.

Section~\ref{sec:rep_Sn} describes the background on Okounkov--Vershik approach to representation theory of symmetric groups. It is geared towards applications for calculations in Sections~\ref{sec:counting} and \ref{sec:asymptotics}.

Section~\ref{sec:counting} is the technical core of this article. We prove Theorem~\ref{thm:count_tangles}, a formula for the number of embeddings of a given immersed complex $T$ in all coverings in $\covers$. Taking empty $T=\emptyset$ in Theorem~\ref{thm:count_tangles} recovers Mednykh's formula as a special case. The proof strategy is similar to the approach from Section~\ref{sec:extensibility}, except that the indicator function is summed only over the subset of those $\varphi\in\Hom{\pi_1\Sigma^{(1)}}{S_n}$ for which there is an embedding $T^{(1)}\hookrightarrow\whS^{(1)}_\varphi$ between 1-skeletons of $T$ and the covering $\whS_\varphi$ induced by $\varphi$. This subset admits a simple description as a Cartesian product, and the indicator is accessible from Proposition~\ref{prop:Rg_trace}. Theorem~\ref{thm:count_tangles} follows after simplifications using the ribbon graph calculus from Section~\ref{sec:ribbon_graph}.

The master Theorem~\ref{thm:E_combi_emb} on frequency of embeddings is proved in Section~\ref{sec:asymptotics}, by analysing the asymptotics of the exact formula in Theorem~\ref{thm:count_tangles}. This relies on the understanding of representation theory of $S_n$ from Section~\ref{sec:rep_Sn}.

Section~\ref{sec:resolutions} has a geometric group theory flavour. We describe the algorithm for constructing resolutions with controlled boundary for a given immersed complex, thus proving Proposition~\ref{prop:resolution} used in Section~\ref{sec:main_combi}. The algorithm necessarily entails solving the word problem in NEC groups, and proceeds by descent on an invariant that combines two ingredients. The first one is a measure of the boundary length, which can be shortened whenever the boundary defect from Section~\ref{sec:main_combi} thinks some boundary component looks non-convex. The second one is a combinatorial analogue of accumulating negative curvature.

\subsection*{Acknowledgements} I am grateful to Vladimir Markovi\'c, Dawid Kielak, Ewan Cassidy, Marc Lackenby and Michael Magee for helpful comments.

\subsection{Notation}

The table below collects the notation used.

\smallskip

\noindent
\begin{tabular}{rl}
f.g.&finitely generated\\
$[n]$, $[m,n]$&the set $\{1,\dots,n\}$ and $\{m,m+1,\dots,n\}$ respectively\\
$\mathbbm{1}_X$&indicator function of the set or event $X$\\
$n^{\underline{k}}$&falling factorial $n(n-1)\dots(n-k+1)$\\
$\chi_V,\chi_V(1)$&the character and dimension of representation $V$\\
$\tau_m(G),\chi^{m\tn{-tors}}_V$&number and average character of $m$-torsion elements (Sec~\ref{sssec:sum_overreps})\\
$\covers$&set of (fibre-labelled) degree-$n$ coverings of $\Sigma$, or $\Hom{\!\pio\Sigma\!}{\!S_n\!}$\\
$\mathbb{E}_n$&expectation over $\covers$ with the uniform measure (Section~\ref{sec:intro})\\
$p\colon T\looparrowright\Sigma$&immersed complex (Def~\ref{def:imm_cplx})\\
$\partial T$&boundary of an immersed complex $T$ (Def~\ref{def:bdry})\\
$V,E,F$&vertices, edges, and faces of a branched $\Delta$-complex ($\Sigma$ or $T$)\\
$C$&corners of a branched $\Delta$-complex (Sec~\ref{sec:simplicial})\\
$\Fspx,\Forb$&simplicial, $m$-orbigon faces of a branched $\Delta$-complex (Sec~\ref{sec:simplicial})\\
$E_e,F_f$&set of edges, faces of $T$ mapping to $e\in E(\Sigma),f\in F(\Sigma)$ (Sec~\ref{ssec:imm_cplx})\\
$\mathfrak{v},\mathfrak{e}_e$&cardinality of $V(T),E_e$ (Sec~\ref{ssec:imm_cplx})\\
$\mathfrak{f}_f$&degree of restricted immersion $\bigcup F_f\to f$ (Sec~\ref{ssec:imm_cplx})\\
$H_\iota$&hanging half-edges of $T$ above vertex-edge inclusion $v\!\stackrel{\iota}{\hookrightarrow}\!e$ (Def~\ref{def:bdry})\\
$S_\iota,O_\iota$&exposed, covered sides lifting edge-face inclusion $e\stackrel{\iota}{\hookrightarrow}f$ (Def~\ref{def:bdry})\\
$e_+\corner e_-$&corner $c$ of face $f$, oriented away from edge $e_-$ towards $e_+$ (Sec~\ref{sec:counting})\\
$\vec\gamma$&tuple $(\gamma_i)_{i\in I}$ indexed by implicit set $I$ (often $C(\Sigma),E(\Sigma),F(\Sigma)$)\\
$\lambda\vdash n$, $|\lambda|$&$\lambda$ is a partition of $n=|\lambda|$ (Def~\ref{def:rep_Sn})\\
$\lambda^\vee$&transpose partition or Young diagram (Sec~\ref{sec:rep_Sn})\\
$\tn{Tab}(\lambda)$&set of standard Young (skew-)tableaux of shape $\lambda$ (Def~\ref{def:rep_Sn})\\
$\dft(\mathcal{P})$&defect of a boundary piece $\mathcal{P}$ (Def~\ref{def:dft})\\
$\md(T)$&maximal defect of a complex $T$ (Def~\ref{def:dft})
\end{tabular}
Most of the results in this article are about a finite immersed complex $p\colon T\looparrowright\Sigma$, so for brevity we leave it implicit in the statements. We are working in two dimensions, so we will omit the prefix 2- and write just ``orbifold''.

\section{Setup: NEC groups and 2-dimensional orbifolds}\label{sec:simplicial}

A 2-dimensional orbifold $\Sigma$ is a topological surface $\Sigma^\tn{top}$ with distinguished points $x_i\in\Sigma^\tn{top}$, called \emph{cone points}, each decorated with a positive integer $m_i\in\mathbb{N}_+$ called its \emph{order}. A neighbourhood of a cone point $x_i$ is modelled after the quotient $\sfrac{D^2}{\frac{\mathbb{Z}}{(m_i)}}$ of the disk modulo rotations by multiples of $\tfrac{2\pi}{m_i}$. The \emph{orbifold fundamental group} of $\Sigma$ is
\begin{equation*}
\pio\Sigma\defeq\frac{\pi_1\left(\Sigma^\tn{top}\setminus\{x_i\}_i\right)}{\langle\!\langle \gamma_i^{m_i}\rangle\!\rangle_i},
\end{equation*}
where $\gamma_i$ is the loop winding once around the respective $x_i$. The genus and orientability of $\Sigma$ are the same as those of the underlying non-singular surface $\Sigma^\tn{top}$, and the orbifold Euler characteristic is
\begin{equation*}
\cho(\Sigma)=\chi(\Sigma^\tn{top})-\sum_i\left(1-\tfrac{1}{m_i}\right),
\end{equation*}
where $\chi$ is the usual topological Euler characteristic. A covering map between orbifolds is allowed to be branched at cone points, as long as the order of the image cone point gets multiplied by the branching degree. This way the Galois correspondence between coverings of orbifolds and subgroups of the orbifold fundamental group holds in the branched setting.

There is a correspondence between NEC groups (cocompact without reflections) and (compact closed) hyperbolic orbifolds. An NEC $\Gamma$ gives rise to the quotient orbifold $\sfrac{\mathbb{H}^2}{\Gamma}$ and a hyperbolic metric on it. Conversely, an orbifold $\Sigma$ equipped with a hyperbolic metric has universal covering isometric to the hyperbolic plane, and the deck transformations $\pio\Sigma\curvearrowright\widetilde\Sigma=\mathbb{H}^2$ make the fundamental group into a NEC group. Consequently, it makes sense to define the Euler characteristic of a group as the $\cho$ of its corresponding orbifold (or $-\infty$ if it is not f.g.). If $\Gamma$ contains a reflection, then the image of its axis is a boundary component of $\Sigma$, and all of the boundary arises this way.

\medskip

We base our combinatorial model for orbifolds on $\Delta$-complexes from \cite[\S 2.1]{alg_top_Hatcher02}. They are a mild generalisation of simplicial complexes, where a simplex can be included in another one as a face in multiple distinct ways; for example, one-vertex triangulations are allowed. To model branching we introduce a special 2-cell which we call \emph{orbigon}. An $m$-orbigon is a polygon with $\mathfrak{s}$ sides containing in the middle a cone point annotated with an \emph{order} $o\in\mathbb{N}_+$, such that $\mathfrak{s}o=m$. The allowed coverings of an orbigon are topologically conjugate to the function $z\mapsto z^d$ over the complex unit disk. Explicitly, an $(\mathfrak{s}o)$-orbigon with $\mathfrak{s}$ sides and cone point of order $o$ can be covered by an $(\mathfrak{s}o)$-orbigon with $\mathfrak{s}d$ sides and cone point of order $\tfrac{o}{d}$, for any branching degree $d|o$. See Figure \ref{fig:orbigon_coverings} for the lattice of covering maps between 6-orbigons.

\begin{figure}[h]
\centering
\begin{tikzpicture}
\begin{scope}[circle,inner sep=2]\node[draw]at(0,0){$6$};\node[draw]at(-4,2){$3$};\node[draw]at(-4,-2){$2$};\node[draw]at(-8,0){$1$};\end{scope}
\draw(0,-.5)to[out=0,in=180](1,0)to[out=180,in=0](0,.5)arc(90:270:.5);\filldraw(1,0)circle(1pt);
\draw(-5,2)to[bend left=45](-3,2)to[bend left=45](-5,2);\filldraw(-5,2)circle(1pt)(-3,2)circle(1pt);
\draw(-3,-2)--(-4.5,-1.3)--(-4.5,-2.7)--(-3,-2);\filldraw(-3,-2)circle(1pt)(-4.5,-1.3)circle(1pt)(-4.5,-2.7)circle(1pt);
\draw(-7,0)--(-7.5,.7)--(-8.5,.7)--(-9,0)--(-8.5,-.7)--(-7.5,-.7)--(-7,0);\filldraw(-7,0)circle(1pt)(-7.5,.7)circle(1pt)(-8.5,.7)circle(1pt)(-9,0)circle(1pt)(-8.5,-.7)circle(1pt)(-7.5,-.7)circle(1pt);
\draw[->](-2.5,1.5)--node[fill=white,pos=.5]{$2$}(-1,.8);\draw[->](-2.5,-1.5)--node[fill=white,pos=.5]{$3$}(-1,-.8);\draw[->](-6.7,.8)--node[fill=white,pos=.5]{$3$}(-5.5,1.5);\draw[->](-6.7,-.8)--node[fill=white,pos=.5]{$2$}(-5.5,-1.5);\draw[->](-6.5,0)--node[fill=white,pos=.5]{$6$}(-1.5,0);
\end{tikzpicture}
\caption{The lattice of covering maps between 6-orbigons. The numbers on arrows indicate the covering degrees (deck groups are cyclic and act by rotations).}
\label{fig:orbigon_coverings}
\end{figure}
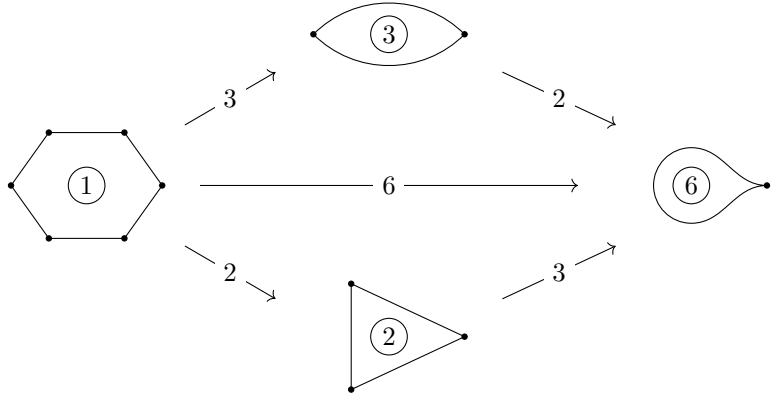

We will refer to the combinatorial structures built from simplices and orbigons as \emph{branched $\Delta$-complexes}. We call the 0-dimensional cells of a branched $\Delta$-complex $\Sigma$ \emph{vertices} and denote them $V(\Sigma)$, 1-cells are \emph{edges} $E(\Sigma)$, and the standard 2-cells are called \emph{simplices} $\Fspx(\Sigma)$ for short. We denote the set of $m$-orbigons of $\Sigma$ by $\Forb(\Sigma)$, and call orbigons and simplices collectively \emph{faces} $F(\Sigma)=\Fspx(\Sigma)\cup\bigcup_m\Forb(\Sigma)$. All faces have \emph{corners}; formally, they are inclusions of a vertex into a face. We denote the set of corners of a face $f$ by $C(f)$, and the collection of all corners in $\Sigma$ by $C(\Sigma)$. Morphisms between branched $\Delta$-complexes are similar to combinatorial maps between $\Delta$-complexes, except that we allow branching at cone points of orbigons, subject to multiplicativity of the cone point order by the branching degree. In particular, the covering maps between orbigons we just described are valid morphisms of branched $\Delta$-complexes. A morphism is called an \emph{embedding} when it is injective on cells. Some examples of orbifolds triangulated into branched $\Delta$-complexes are shown in Figure \ref{fig:cplx_ex}. The orbifold Euler characteristic can be read off from a branched $\Delta$-complex structure as
\begin{equation}\label{eq:cho_combi}
\cho(\Sigma)=\#V(\Sigma)-\#E(\Sigma)+\#\Fspx(\Sigma)+\sum_\iterorb\frac{\mathfrak{s}_f}{m},
\end{equation}
where $\mathfrak{s}_f$ is the number of sides of $f$.

\begin{figure}[h]
\begin{subfigure}{.5\textwidth}
\centering
\begin{tikzpicture}
\draw[->-](-.5,1)--(.5,1);\draw[->>-](.5,1)--(1,.5);\draw[->-](1,-.5)--(1,.5);\draw[->>-](.5,-1)--(1,-.5);
\begin{scope}[decoration={markings,mark=at position .55 with {\arrow{Latex[fill=white,length=4,width=4]}}}]\draw[postaction={decorate}](.5,-1)--(-.5,-1);\draw[postaction={decorate}](-1,.5)--(-1,-.5);\end{scope}
\begin{scope}[decoration={markings,mark=at position .55 with {\arrow{>}},mark=at position .45 with {\arrow{Latex[fill=white,length=4,width=4]}}}]\draw[postaction={decorate}](-.5,-1)--(-1,-.5);\draw[postaction={decorate}](-.5,1)--(-1,.5);\end{scope}
\draw(-1,-.5)--(.5,-1)--(.5,1)(.5,1)--(1,-.5)(-1,.5)--(.5,-1)(.5,-1)--(-.5,1);
\end{tikzpicture}
\caption{Example one-vertex triangulation of $\Sigma_2$}
\end{subfigure}
\hfill
\begin{subfigure}{.47\textwidth}
\centering
\begin{tikzpicture}
\draw[->-](0,1)--(1,1);\draw[->-](1,0)--(0,0);
\draw[->>-](1,0)--(1,1);\draw[->>-](0,1)--(0,0);
\filldraw(0,1)circle(1pt)(1,0)circle(1pt);\draw[fill=white](0,0)circle(1pt)(1,1)circle(1pt);
\draw(1,0)--(0,1);
\end{tikzpicture}
\caption{A triangulation of $\mathbb{RP}^2$ (with two vertices)}
\label{sfig:rpp}
\end{subfigure}
\begin{subfigure}{\textwidth}
\begin{equation*}
\begin{tikzpicture}[baseline=0]
\draw[red](0,0)to[out=0,in=-135](1.1,.3)(.5,.9)to[out=-135,in=60](0,0);\draw[red,->-](1.1,.3)arc(-45:135:.424);\node[red]at(1.4,.2){$e_1$};
\draw[teal](0,0)to[out=120,in=-45](-.5,.9)(-1.1,.3)to[out=-45,in=180](0,0);\node[teal]at(-1.4,.2){$e_2$};
\draw[teal,decoration={markings,mark=at position .7 with {\arrow{Latex[length=4,width=4]}},mark=at position .5 with {\arrow{Latex[length=4,width=4]}}},postaction={decorate}](-.5,.9)arc(45:225:.424);
\draw[blue,-{Latex[length=4,width=4]}](0,0)to[out=-120,in=90](-.5,-.7)arc(180:270:.4);\node[blue]at(-.7,-.9){$e_3$};
\draw[blue,{Latex[length=4,width=4,reversed]}-](.1,-1.1)arc(-90:0:.4)to[out=90,in=-60](0,0);
\draw[blue,decoration={markings,mark=at position .85 with {\arrow{Latex[length=4,width=4]}}},postaction={decorate}](-.12,-1.1)--(.1,-1.1);
\node[circle,draw,font=\tiny,inner sep=0]at(.8,.6){$m_1$};\node[circle,draw,font=\tiny,inner sep=0]at(-.8,.6){$m_2$};\node[circle,draw,font=\tiny,inner sep=0]at(0,-.7){$m_3$};
\filldraw(0,0)circle(1pt);
\end{tikzpicture}
\qquad\bigcup_\partial\qquad
\begin{tikzpicture}[baseline=0]
\draw[red,->-](1,-.5)--(0,1);\draw[teal,->>-](0,1)--(-1,-.5);\draw[blue,->>>-](-1,-.5)--(1,-.5);
\node[red]at(.8,.3){$e_1$};\node[teal]at(-.8,.3){$e_2$};\node[blue]at(0,-.8){$e_3$};
\filldraw(0,1)circle(1pt)(-1,-.5)circle(1pt)(1,-.5)circle(1pt);
\end{tikzpicture}
\end{equation*}
\caption{A sphere with 3 cone points of orders $m_1,m_2,m_3$, obtained by attaching a simplex to three orbigons along their boundary edges.}
\label{sfig:orbi_S2}
\end{subfigure}
\begin{subfigure}{\textwidth}
\begin{equation*}
\begin{tikzpicture}[baseline=0]
\draw(0,-.8)--(0,.8);\draw[red,->-](1,0)--(0,.8);\draw[teal,->>-](1,0)--(0,-.8);\draw[blue,->>>-](0,.8)--(-1,0);\draw[blue,->>>-](-1,0)--(0,-.8);\node[red]at(.6,.6){$e_1$};\node[teal]at(.6,-.6){$e_2$};\node[blue]at(-.6,.7){$e_3$};\node[blue]at(-.6,-.6){$e_3$};\node at(.2,0){$e_4$};
\end{tikzpicture}
\quad\bigcup_{{\color{red}e_1},{\color{teal}e_2}}\quad
\begin{tikzpicture}[baseline=0]
\begin{scope}[red]\draw(-.5,.7)to[out=-90,in=135](0,0)to[out=45,in=-90](.5,.7);\draw[->-](.5,.7)arc(0:180:.5);\node at(-.7,.45){$e_1$};\end{scope}
\begin{scope}[teal]\draw(-.5,-.7)to[out=90,in=-135](0,0)to[out=-45,in=90](.5,-.7);\draw[->>-](.5,-.7)arc(0:-180:.5);\node at(-.7,-.45){$e_2$};\end{scope}
\node[circle,draw,font=\tiny,inner sep=1pt]at(0,.7){$m_1$};\node[circle,draw,font=\tiny,inner sep=1pt]at(0,-.7){$m_2$};
\end{tikzpicture}
\qquad=\qquad
\begin{tikzpicture}[baseline=0]
\draw(-.05,-.25)to[out=225,in=0](-.3,-.4)arc(270:90:.6)--(.3,.8)arc(90:-90:.6)to[out=180,in=0](-.3,0)arc(270:90:.2)--(.3,.4)arc(90:-90:.2)to[out=180,in=45](.05,-.15);\node at(-.9,-.4){$e_4$};
\draw[blue,->>>-](.3,.8)to[out=200,in=90](-.7,.2);\draw[blue](-.7,.2)arc(180:270:.4)--(-.1,-.2)(.1,-.2)--(.3,-.2)arc(-90:0:.4)to[out=90,in=-50](.3,.8);\filldraw(.3,.8)circle(1pt);
\end{tikzpicture}
\quad\bigcup_{e_4}\quad
\begin{tikzpicture}[baseline=0]
\draw[red,->-](0,0)to[out=10,in=-150](.65,.15)arc(-60:160:.35)to[out=-110,in=70](0,0);\node[circle,draw,font=\tiny,inner sep=0pt]at(.5,.5){$m_1$};
\draw[teal,->>-](0,0)to[out=-10,in=150](.65,-.15)arc(60:-160:.35)to[out=110,in=-70](0,0);\node[circle,draw,font=\tiny,inner sep=0pt]at(.5,-.5){$m_2$};
\draw(-.2,0)ellipse(.2 and 1.3);\filldraw(0,0)circle(1pt);\draw(-.2,-1.3)arc(-90:90:1.3);\node at(-.6,-1.1){$e_4$};
\end{tikzpicture}
\end{equation*}
\caption{A triangulation of $\mathbb{RP}^2$ with two cone points. The two orbigons together with one of the simplices (the one incident to {\color{red}$e_1$} and {\color{teal}$e_2$}) form a topological disk with two cone points, while the other simplex becomes a M\"obius band after identifying the two copies of {\color{blue}$e_3$}.}
\label{sfig:orbi_rp2}
\end{subfigure}
\begin{subfigure}{\textwidth}
\begin{equation*}
\begin{tikzpicture}[baseline=0]
\begin{scope}[inner sep=2pt]
\node(ml)[draw,rectangle]at(-1,0){};\node(bm)[draw,rectangle]at(0,-1){};
\node(tm)[draw,diamond]at(0,1){};\node(mr)[draw,diamond]at(1,0){};
\node(tl)[draw,circle]at(-1,1){};\node(bl)[draw,circle]at(-1,-1){};
\end{scope}
\begin{scope}[inner sep=1pt]\node(tr)[draw,circle split]at(1,1){};\node(br)[draw,circle split]at(1,-1){};\end{scope}
\node[circle,draw,inner sep=1pt,font=\tiny]at(0,0){$1$};
\draw[->-](tl)--(tm);\draw[->>-](tr)--(tm);
\draw[decoration={markings,mark=at position .6 with {\arrow{Latex[fill=white,length=4,width=4]}},mark=at position .45 with {\arrow{Latex[fill=white,length=4,width=4]}}},postaction={decorate}](bl)--(bm);
\draw[decoration={markings,mark=at position .5 with {\arrow{Latex[fill=white,length=4,width=4]}}},postaction={decorate}](br)--(bm);
\draw[decoration={markings,mark=at position .6 with {\arrow{>}},mark=at position .45 with {\arrow{Stealth[fill=white,length=4,width=4]}}},postaction={decorate}](ml)--(tl);
\draw[decoration={markings,mark=at position .5 with {\arrow{Stealth[fill=white,length=4,width=4]}}},postaction={decorate}](ml)--(bl);
\draw[decoration={markings,mark=at position .6 with {\arrow{>}},mark=at position .45 with {\arrow{Stealth[length=4,width=4]}}},postaction={decorate}](mr)--(br);
\draw[decoration={markings,mark=at position .5 with {\arrow{Stealth[length=4,width=4]}}},postaction={decorate}](mr)--(tr);
\draw(ml)--(tm)--(mr)--(bm)--(ml)to[in=150,out=30](mr)to[in=-30,out=-150](ml);
\begin{scope}[font=\small]\node at(-.7,.7){$B$};\node at(.7,.7){$A$};\node at(-.7,-.7){$A$};\node at(.7,-.7){$B$};\end{scope}
\end{tikzpicture}
\quad\bigcup_\partial\quad
\begin{tikzpicture}[baseline=0]
\begin{scope}[inner sep=2pt]
\node(ml)[draw,circle]at(-1,0){};\node(bm)[draw,circle]at(0,-1){};
\node(tr)[draw,diamond]at(1,1){};\node(br)[draw,diamond]at(1,-1){};
\node(tl)[draw,rectangle]at(-1,1){};\node(bl)[draw,rectangle]at(-1,-1){};
\end{scope}
\begin{scope}[inner sep=1pt]\node(tm)[draw,circle split]at(0,1){};\node(mr)[draw,circle split]at(1,0){};\end{scope}
\node[circle,draw,inner sep=1pt,font=\tiny]at(0,0){$1$};
\draw[->-](bm)--(br);\draw[->>-](tm)--(tr);
\draw[decoration={markings,mark=at position .6 with {\arrow{Latex[fill=white,length=4,width=4]}},mark=at position .45 with {\arrow{Latex[fill=white,length=4,width=4]}}},postaction={decorate}](bm)--(bl);
\draw[decoration={markings,mark=at position .5 with {\arrow{Latex[fill=white,length=4,width=4]}}},postaction={decorate}](tm)--(tl);
\draw[decoration={markings,mark=at position .6 with {\arrow{>}},mark=at position .45 with {\arrow{Stealth[fill=white,length=4,width=4]}}},postaction={decorate}](tl)--(ml);
\draw[decoration={markings,mark=at position .5 with {\arrow{Stealth[fill=white,length=4,width=4]}}},postaction={decorate}](bl)--(ml);
\draw[decoration={markings,mark=at position .6 with {\arrow{>}},mark=at position .45 with {\arrow{Stealth[length=4,width=4]}}},postaction={decorate}](br)--(mr);
\draw[decoration={markings,mark=at position .5 with {\arrow{Stealth[length=4,width=4]}}},postaction={decorate}](tr)--(mr);
\draw(tm)--(mr)--(bm)--(ml)--(tm)to[in=60,out=-60](bm)to[in=-120,out=120](tm);
\begin{scope}[font=\small]\node at(-.7,.7){$A$};\node at(.7,.7){$B$};\node at(-.7,-.7){$B$};\node at(.7,-.7){$A$};\end{scope}
\end{tikzpicture}
\quad\longrightarrow\quad
\begin{tikzpicture}[baseline=0]
\node[circle,draw,font=\tiny,inner sep=1pt]at(0,0){$2$};
\draw (0,1)to[out=-100,in=180](0,-.3)to[out=0,in=-80](0,1)to[out=-135,in=135](0,-1)to[out=45,in=-45](0,1);
\draw[->-](0,1)--(-1.2,0);\draw[->-](1.2,0)--(0,-1);\draw[->>-](-1.2,0)--(0,-1);\draw[->>-](0,1)--(1.2,0);
\begin{scope}[font=\small]\node at(-.7,0){$A$};\node at(.7,0){$B$};\end{scope}
\end{tikzpicture}
\end{equation*}
\caption{Quadruple covering of a torus with a cone point of order 2 by a genus 2 surface. Degree-2 branching happens at the designated cone points of orbigons. The mapping between edges induced by the covering preserves the number of arrow tips and their orientation.}
\label{sfig:orbi_covering}
\end{subfigure}
\caption{Example branched $\Delta$-complex structures on orbifolds.}
\label{fig:cplx_ex}
\end{figure}
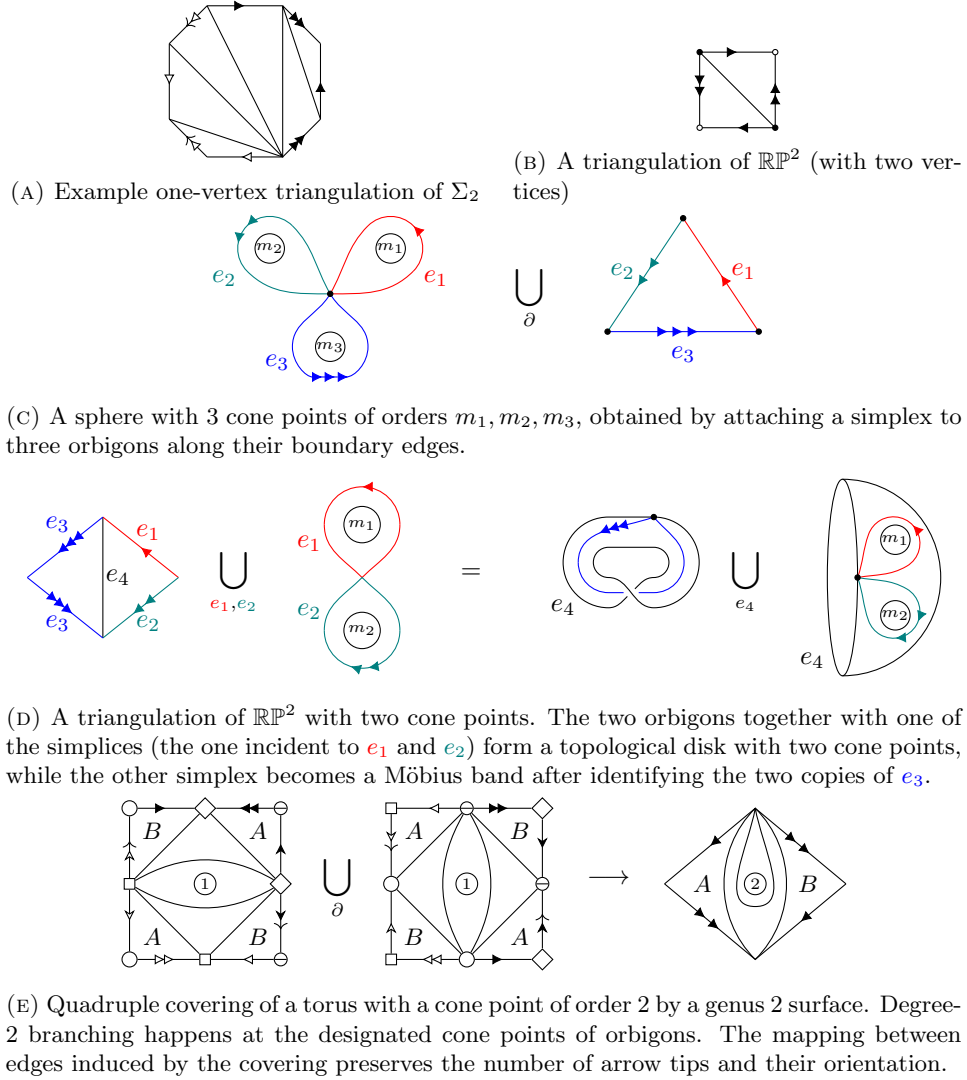

\medskip

In the rest of the article we adopt the following convenience convention: the orbifolds serving as the base space of coverings (denoted $\Sigma$) will be triangulated with a single vertex and such that every $m$-orbigon has one edge (i.e. is a monogon with a cone point of order $m$). Such a branched $\Delta$-complex structure can be chosen on any orbifold with $\cho<0$. All complexes shown in Figure~\ref{fig:cplx_ex} are of this standard form, except $\mathbb{RP}^2$ in \ref{sfig:rpp} that has two vertices, and the genus two covering space in \ref{sfig:orbi_covering} that has four vertices and two two-edged orbigons. A standard triangulation of an orientable orbifold of genus $g$ with $k$ cone points has $6g-3+2k$ edges and $4g-2+2k$ faces, $4g-2+k$ of which are simplices and $k$ are orbigons. A standard triangulation of a non-orientable orbifold of genus $g$ with $k$ cone points has $3g+2k-3$ edges and $2g+k-2$ simplices (non-orientable genus $g$ surface is a connected sum of $g=2-\chi$ cross-caps, so e.g. the Klein bottle has $g=2$). A branched $\Delta$-complex structure of such a standard form induces a cellular presentation for the fundamental group
\begin{equation}\label{eq:simplicial_pres}
\pio\Sigma=\left\langle e\in E(\Sigma)\ \middle|\ \begin{matrix}W_f(\vec{e})\text{ for simplicial }f\in \Fspx(\Sigma)\\e^m\text{ for }e\text{ bounding }m\text{-orbigon}\end{matrix}\right\rangle,
\end{equation}
where $W_f$ is the three-letter word in edges and their inverses spelled by traversing the perimeter of simplex $f$. Furthermore, with $\#V(\Sigma)=1$ we can treat the labelling of the fibre of $\whS\in\covers$ from Remark~\ref{rmk:fibre_label} as a bijection $\mathcal{L}_\whS\colon V(\whS)\to[n]$.

\subsection{Immersed complexes}\label{ssec:imm_cplx}

We will represent tangles combinatorially by immersed complexes. They extend the idea of Stallings core graphs \cite{topology_finite_graphs_Stallings83} by allowing 2-cells. All immersed complexes in this article will be finite.
\begin{definition}\label{def:imm_cplx}
For a branched $\Delta$-complex structure on an orbifold $\Sigma$, an \emph{immersed complex} is a branched $\Delta$-complex $T$ equipped with an immersion $p\colon T\looparrowright\Sigma$, i.e. a morphism such that the induced mapping on links of vertices is injective.

A \emph{morphism of immersed complexes} is a morphism $q\colon T_1\to T_2$ (of branched $\Delta$-complexes) such that the diagram
\begin{equation*}\begin{tikzcd}
T_1\arrow[rr,"q"]\arrow[rd,"p_1",{Glyph[glyph math command=looparrowleft, swap]}->]&&T_2\arrow[ld,"p_2"',{Glyph[glyph math command=looparrowleft]}->]\\
&\Sigma&
\end{tikzcd}\end{equation*}
commutes.

A morphism of immersed complexes is called an \emph{embedding} when it is injective on vertices (and hence all cells), and a \emph{quotient} if it is surjective on cells.
\end{definition}
An immersed complex can be equivalently regarded as a directed branched $\Delta$-complex labelled with cells of $\Sigma$, such that locally the labels look like in a subcomplex of $\Sigma$. A similar concept appears under the name of tiled surfaces in \cite[Definition 2.1]{asymptotic_stats_random_covering_surfaces_MageePuder23} and sub-cover in \cite[Definition 2.1]{local_stats_random_perms_free_prods_PuderZimhoni24}. The subtle difference is that we do not require the immersed complex to embed in some covering of $\Sigma$, so Definition~\ref{def:imm_cplx} is completely local; for example, a circle spelling $w^2$ for a nullhomotopic but non-backtracking edge path $w$ in $\Sigma$ is a valid immersed complex. Later in Lemma~\ref{lmm:no_dft_subgp} we will give sufficient conditions to ensure that the immersion does come from an embedding into a covering. An example immersed complex is shown in Figure~\ref{sfig:imm_cplx}. There we denote
\begin{itemize}
\item the number of vertices by $\mathfrak{v}=\#V(T)$
\item the set of edges of $T$ above $e\in E(\Sigma)$ by $E_e\defeq p^{-1}(e)\subseteq E(T)$, and its cardinality by $\mathfrak{e}_e\defeq\#E_e$
\item the preimage of $f\in F(\Sigma)$ by $F_f\defeq p^{-1}(f)\subseteq F(T)$, and the degree of the restriction $p|_{F_f}\to f$ of $p$ to it by $\mathfrak{f}_f$. Explicitly,
\begin{equation*}
\mathfrak{f}_f=\begin{cases}\#F_f&\text{if }f\in \Fspx(\Sigma)\\\sum_{f'\in F_f}\mathfrak{s}_{f'}&\text{if }f\in\Forb(\Sigma).\end{cases}
\end{equation*}
\end{itemize}

Algebraically, a (based connected) immersed complex $p\colon T\looparrowright\Sigma$ determines a subgroup $p_*(\pio T)\le\pio\Sigma$. Conversely, for every subgroup $\Lambda\le\pio\Sigma$ there exists a (highly non-unique) (based connected) immersed complex $(p,T)$ such that $p_*(\pio T)=\Lambda$ (and a finite one can be chosen whenever $\Lambda$ is f.g.). A morphism of immersed complexes $q\colon T_1\to T_2$ translates algebraically to the inclusion $p_{1*}(\pio T_1)\le p_{2*}(\pio T_2)$ of the respective fundamental groups. Note that the homomorphism $p_*\colon\pio T\to\pio\Sigma$ is not always injective; again, later in Lemma~\ref{lmm:no_dft_subgp} we will develop some conditions to guarantee this based on the boundary.

\medskip

The immersion $p$ provides each cell of $T$ with its ``model neighbourhood'' in $\widetilde\Sigma$. The cell's actual neighbourhood in $T$ is a subset of it, and we refer to the complement as consisting of ``phantom'' cells. Concretely, every edge of $T$ has two sides, which can be covered by faces or exposed (i.e. facing phantom cells). Furthermore, from each vertex of $T$ emanates a pencil of actual or phantom edges that is a copy of the neighbourhood of the vertex in $\Sigma^{(1)}$. We think of the phantom ones as (initial or terminal) halves of edges hanging at the vertex.
\begin{definition}\label{def:bdry}
Immersion $p$ induces the following structures on $T$.
\begin{itemize}
\item A side of an edge of $T$ is called \emph{covered} if it touches an actual face of $T$. Formally, a covered side is an inclusion of an edge in a face of $T$. We denote the set of covered sides lifting an inclusion $\iota\colon e\hookrightarrow f$ from $\Sigma$ by $O_\iota$.
\item A side that is not covered is called \emph{exposed}. Formally, an exposed side is an edge $e\in E(T)$ together with a choice of inclusion $\iota$ of $p(e)$ in a face of $\Sigma$, such that $\iota$ does not lift from $\Sigma$ to $T$ at $e$. The set of exposed sides above this $\iota$ is denoted $S_\iota$.
\item A (half of) phantom edge is called a \emph{hanging half-edge}. Formally, it is a pair $(v,\iota)$, where $v\in V(T)$ and $\iota\colon p(v)\hookrightarrow e$ for some $e\in E(\Sigma)$, such that $\iota$ does not lift to $T$ at $v$. Set of hanging half-edges with this $\iota$ is denoted $H_\iota$.
\end{itemize}
The \emph{boundary} $\partial T=\bigcup_\iota S_\iota\cup\bigcup_\iota H_\iota$ consists of all exposed sides and hanging half-edges of $T$. Adjacent exposed sides or hanging half-edges (i.e. separated only by a single phantom corner) are said to belong to the same \emph{boundary component}.
\end{definition}
See Figure~\ref{sfig:bdry} for an example of a boundary. Our notion of exposed sides and hanging half-edges is analogous to the thick version of a tiled surface from \cites[Definitions 2.2, 2.3]{asymptotic_stats_random_covering_surfaces_MageePuder23}[\S 3.2]{core_surfaces_MageePuder22}, and we reuse their terminology for consistency. Note that the boundary vertices and edges, considered as a subcomplex of $T$, form a graph that may contain vertices of valence larger than 2 (see Figure~\ref{sfig:oa_slot_example} for an example). However, we think of the boundary as ``pushed slightly off the complex'', and this way all boundary components are topological circles.
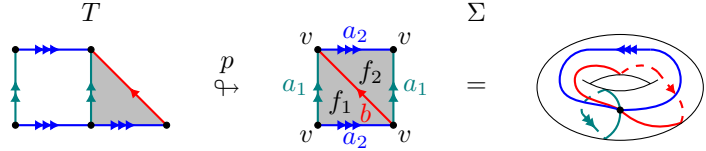
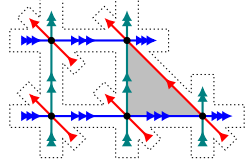
\begin{figure}[h]
\begin{subfigure}{\textwidth}
\centering
\begin{tikzpicture}
\fill[lightgray](-3,0)--(-2,0)--(-3,1);\draw[thick,red,->-](-2,0)--(-3,1);\draw[thick,teal,->>-](-3,0)--(-3,1);\draw[thick,blue,->>>-](-3,0)--(-2,0);
\draw[thick,teal,->>-](-4,0)--(-4,1);\draw[thick,blue,->>>-](-4,0)--(-3,0);\draw[thick,blue,->>>-](-4,1)--(-3,1);
\filldraw(-4,0)circle(.04)(-4,1)circle(.04)(-3,0)circle(.04)(-3,1)circle(.04)(-2,0)circle(.04);
\node at(-3,1.5){$T$};
\node at(-1.2,.5){$\looparrowright$};\node at(-1.2,.8){$p$};
\fill[lightgray](0,0)--(0,1)--(1,1)--(1,0);
\draw[thick,red,->-](1,0)--(0,1);\draw[thick,teal,->>-](0,0)--(0,1);\draw[thick,teal,->>-](1,0)--(1,1);\draw[thick,blue,->>>-](0,0)--(1,0);\draw[thick,blue,->>>-](0,1)--(1,1);
\filldraw(0,0)circle(.04)(0,1)circle(.04)(1,0)circle(.04)(1,1)circle(.04);
\node at(.3,.3){$f_1$};\node at(.7,.7){$f_2$};\node at(.5,-.2){\color{blue}$a_2$};\node at (.5,1.2){\color{blue}$a_2$};\node at(-.3,.5){\color{teal}$a_1$};\node at(1.3,.5){\color{teal}$a_1$};\node at(.63,.17){\color{red}$b$};\node at(-.15,-.15){$v$};\node at(1.15,-.15){$v$};\node at(-.15,1.15){$v$};\node at(1.15,1.15){$v$};
\node at(2.1,1.5){$\Sigma$};
\node at(2.1,.5){$=$};
\begin{scope}[thick,teal]\draw(3.8,-.18)to[bend right](4,.2)to[bend right](3.8,.47);
\draw[dashed](3.8,.47)to[out=-160,in=110](3.5,.15);\draw[dashed,->>-](3.5,.15)--(3.8,-.18);\end{scope}
\begin{scope}[thick,red]\draw(4,.2)to[bend right](4.8,.03);\draw(4,.67)to[out=135,in=90](3.4,.6)to[out=-90,in=150](4,.2);
\draw[dashed](4.8,.02)to[bend right](4.74,.4)(4.5,.7)to[out=135,in=45](4,.67);\draw[dashed,-{\ArrTip}](4.5,.7)--(4.74,.4);\end{scope}
\begin{scope}[thick,blue]\draw(4,.2)to[out=0,in=-90](4.8,.6)to[out=90,in=0](4.4,1)(3.6,1)to[out=180,in=90](3.2,.6)to[out=-90,in=160](4,.2);\draw[->>>-](4.4,1)--(3.6,1);\end{scope}
\filldraw(4,.2)circle(.04);
\draw(4,.5)ellipse(1.1 and .7);
\draw(3.5,.6)to[bend right](4.5,.6)(3.6,.55)to[bend left](4.4,.55);
\end{tikzpicture}
\caption{Immersion of $\Delta$-complexes $p\colon T\looparrowright\Sigma$, following Definition~\ref{def:imm_cplx}. The complex $T$ consists of $\mathfrak{v}=5$ vertices, 6 edges (with ${\color{teal}\mathfrak{e}_{a_1}}=2$, ${\color{blue}\mathfrak{e}_{a_2}}=3$, and ${\color{red}\mathfrak{e}_b}=1$), and 1 face (in $F_{f_1}$, so $\mathfrak{f}_{f_1}=1,\mathfrak{f}_{f_2}=0$). The surface $\Sigma$ is a triangulation of a torus with 1 vertex, 3 edges and 2 faces, obtained by adding the diagonal {\color{red}$b$} to the standard square with side identifications.}
\label{sfig:imm_cplx}
\end{subfigure}
\begin{subfigure}{\textwidth}
\centering
\begin{tikzpicture}
\fill[lightgray](0,0)--(1,0)--(0,1);\draw[thick,red,->-](1,0)--(0,1);\draw[thick,teal,->>-](0,0)--(0,1);\draw[thick,blue,->>>-](0,0)--(1,0);
\draw[thick,teal,->>-](-1,0)--(-1,1);\draw[thick,blue,->>>-](-1,0)--(0,0);\draw[thick,blue,->>>-](-1,1)--(0,1);
\begin{scope}[>={Latex[length=4,width=4]},thick,red]\draw[>->](-.7,.7)--(-1.3,1.3);\draw[>->](.3,-.3)--(-.3,.3);\draw[>->](-.7,-.3)--(-1.3,.3);\draw[->](0,1)--+(-.3,.3);\draw[>-](1.3,-.3)--(1,0);\end{scope}
\begin{scope}[>={Latex[length=4,width=4,sep=-2pt]Latex[length=4,width=4]},thick,teal]\draw[>-](-1,-.4)--(-1,0);\draw[->](-1,1)--+(0,.4);\draw[>-](0,-.4)--(0,0);\draw[->](0,1)--+(0,.4);\draw[>->](1,-.4)--(1,.4);\end{scope}
\begin{scope}[>={Latex[length=4,width=4,sep=-2pt]Latex[length=4,width=4,sep=-2pt]Latex[length=4,width=4]},thick,blue]\draw[>-](-1.4,0)--(-1,0);\draw[>-](-1.4,1)--(-1,1);\draw[->](0,1)--+(.4,0);\draw[->](1,0)--+(.4,0);\end{scope}
\filldraw(-1,0)circle(.04)(-1,1)circle(.04)(0,0)circle(.04)(0,1)circle(.04)(1,0)circle(.04);
\draw[dash pattern={on .5pt off 1pt}](-.15,.3)--(-.15,.85)--(-.7,.85)--++(.15,-.15)--++(-.15,-.15)--++(-.15,.15)--(-.85,.15)--(-.3,.15)--++(-.15,.15)--++(.15,.15)--+(.15,-.15);
\draw[dash pattern={on .5pt off 1pt}](-1.15,-.15)--
++(-.4,0)--++(0,.3)--++(.25,0)--++(-.15,.15)--++(.15,.15)--++(.15,-.15)--(-1.15,.85)--
++(-.4,0)--++(0,.3)--++(.25,0)--++(-.15,.15)--++(.15,.15)--++(.15,-.15)--++(0,.25)--++(.3,0)--++(0,-.4)--(-.3,1.15)--
++(-.15,.15)--++(.15,.15)--++(.15,-.15)--++(0,.25)--++(.3,0)--++(0,-.4)--++(.4,0)--++(0,-.3)--++(-.25,0)--(.85,.3)--
++(0,.25)--++(.3,0)--++(0,-.4)--++(.4,0)--++(0,-.3)--++(-.25,0)--++(.15,-.15)--++(-.15,-.15)--++(-.15,.15)--++(0,-.25)--++(-.3,0)--++(0,.4)--(.3,-.15)--
++(.15,-.15)--++(-.15,-.15)--++(-.15,.15)--++(0,-.25)--++(-.3,0)--++(0,.4)--(-.7,-.15)--
++(.15,-.15)--++(-.15,-.15)--++(-.15,.15)--++(0,-.25)--++(-.3,0)--++(0,.4);
\end{tikzpicture}
\caption{Boundary structures of the immersed complex $(p,T)$ following Definition \ref{def:bdry}. Each of the edges adjacent to the face of $T$ has 1 covered and 1 exposed side (so there is 1 covered side in each of $O_{{\color{teal}a_1}\hookrightarrow f_1}$, $O_{{\color{blue}a_2}\hookrightarrow f_1}$, and $O_{{\color{red}b}\hookrightarrow f_1}$). There are two boundary components, marked with dotted lines. The shorter one consists of 4 exposed sides of edges (1 above each of ${\color{teal}a_1}\hookrightarrow f_1$, ${\color{teal}a_1}\hookrightarrow f_2$, ${\color{blue}a_2}\hookrightarrow f_1$, and ${\color{blue}a_2}\hookrightarrow f_2$), and 2 hanging half-edges (1 for the initial and terminal inclusions $v\hookrightarrow{\color{red}b}$). The second boundary component (outer in the picture) consists of 5 exposed sides of edges (2 in $S_{{\color{blue}a_2}\hookrightarrow f_2}$ and 1 in each of $S_{{\color{teal}a_1}\hookrightarrow f_2}$, $S_{{\color{blue}a_2}\hookrightarrow f_1}$, and $S_{{\color{red}b}\hookrightarrow f_2}$), and 16 hanging half-edges (2 in both sets $H_{v\hookrightarrow{\color{blue}a_2}}$, and 3 in each $H_{v\hookrightarrow{\color{teal}a_1}}$ and $H_{v\hookrightarrow{\color{red}b}}$ for the initial and terminal inclusions).}
\label{sfig:bdry}
\end{subfigure}
\caption{Example of a $\Delta$-complex immersion into a surface and the corresponding boundary.}
\label{fig:immersion}
\end{figure}

\section{Master counting theorem}\label{sec:main_combi}

In this section we state our master statistical result, Theorem~\ref{thm:E_combi_emb}, and use it to deduce Theorem~\ref{thm:alg_tf}. Theorem~\ref{thm:E_combi_emb} characterises the expected number of embeddings $T\hookrightarrow\whS$ (in the category of immersed complexes from Definition~\ref{def:imm_cplx}) of a fixed $T$ in a random $\whS\in\covers$. More directly, these are vertex-injective morphisms of (unbased) branched $\Delta$-complexes $T\to\whS$ lifting the immersion $T\looparrowright\Sigma$ into a covering $\whS\to\Sigma$. The proof of Theorem~\ref{thm:E_combi_emb} is completed in Section \ref{ssec:pf_E_combi_emb}, by extracting the asymptotics of the exact count in Theorem~\ref{thm:count_tangles}. The statement of \ref{thm:E_combi_emb} involves maximal defect $\md(T)$, a certain number depending on $\partial T$ in an explicit way; since the full formula is a bit lengthy, we give it after a brief motivating discussion.
\begin{theorem}\label{thm:E_combi_emb}
Let $\Sigma$ be a combinatorial orbifold with $\cho(\Sigma)<0$, and $T\looparrowright\Sigma$ an immersed complex. Then for any $\varepsilon>0$ we have
\begin{equation*}
\mathbb{E}_n[\#\{T\hookrightarrow\whS\}]=O\left(n^{\cho(T)+\max\{0,\md(T)\}+\varepsilon}\right),
\end{equation*}
where $\md(T)$ is the maximal defect from Definition~\ref{def:dft}. Furthermore, if $\md(T)<0$ then
\begin{equation*}
\mathbb{E}_n[\#\{T\hookrightarrow\whS\}]=n^{\cho(T)}(1+o(1)).
\end{equation*}
\end{theorem}
Before formally defining $\md(T)$, let us explain why some kind of boundary contribution is necessary (see also \cite[discussion of (1.7) in \S 1.6]{asymptotic_stats_random_covering_surfaces_MageePuder23}). An analogous problem makes sense for graphs. In their setting, if $H\looparrowright G$ is an immersion of graphs with $\chi(G)<0$, then the expected number of embeddings of $H$ in a random degree-$n$ covering of $G$ is an explicit rational function of $n$, and one can see directly that it is asymptotic to $n^{\chi(H)}$ (this appears e.g. in \cites[Lemma 6.4]{meas_preserving_words_primitive_PuderParzanchevski15}[Lemma 3.1]{local_stats_random_perms_free_prods_PuderZimhoni24}[Lemma 6.4]{tanglefree_perms_PW_property_random_covers_KlukowskiMarkovic24}, phrased in various ways). Let us return and try to apply this intuition to surfaces. If for example $T\looparrowright\Sigma$ is a simplicial immersed circle, then $\chi(T)=0$, so naively we would expect $O(1)$ embeddings of $T$ in a random $\whS\in\covers$. However, if $T$ is nullhomotopic then the immersion has $n$ lifts into any $\whS\in\covers$, and it seems plausible that $\Theta(n)$ of them should be embeddings, contrary to our guess. This means that statistics of embeddings detect triviality in the fundamental group. This is exactly the role of defect in Theorem \ref{thm:E_combi_emb}. We are now ready to give its precise definition.
\begin{definition}\label{def:dft}
A \emph{boundary piece} of an immersed complex $T\looparrowright\Sigma$ is a subset of $\partial T$ (considered as the set of exposed sides and hanging half-edges).

Suppose that a boundary piece $\mathcal{P}$ contains
\begin{itemize}
\item $\mathfrak{e}_\tn{spx}(\mathcal{P})$ exposed sides facing phantom simplices
\item $\mathfrak{h}_\tn{spx}(\mathcal{P})$ hanging half-edges that separate two phantom simplices
\item $\mathfrak{e}^{(m)}_\tn{orb}(\mathcal{P})$ exposed sides that face an $m$-orbigon
\item $\mathfrak{h}^{(m)}_\tn{orb}(\mathcal{P})$ hanging half-edges having a phantom simplex on one side and $m$-orbigon on the other
\item and $2\chi(\mathcal{P})$ adjacencies between sides or half-edges in $\mathcal{P}$ and $\partial T\setminus\mathcal{P}$.
\end{itemize}
Then the \emph{defect} of $\mathcal{P}$ is
\begin{equation*}
\dft(\mathcal{P})\defeq\frac{\mathfrak{e}_\tn{spx}(\mathcal{P})}{3}-\frac{\mathfrak{h}_\tn{spx}(\mathcal{P})}{6}+\sum_{m\geqslant 2}\left(\frac{\mathfrak{e}^{(m)}_\tn{orb}(\mathcal{P})}{m}-\left(\tfrac{1}{3}-\tfrac{1}{2m}\right)\mathfrak{h}^{(m)}_\tn{orb}(\mathcal{P})\right)-\chi(\mathcal{P}).
\end{equation*}

The \emph{maximal defect} $\md(T)\defeq\max_{\emptyset\neq\mathcal{P}\subseteq\partial T}\dft(\mathcal{P})$ of $T$ is the largest defect of a nonempty boundary piece (or $-\infty$ if $\partial T=\emptyset$).
\end{definition}
Joining the adjacent sides and half-edges of a boundary piece into contiguous subsets produces a union of segment and circular components. A segment gives rise to two ends while a circle has none, so $\chi(\mathcal{P})$ is equal to the number of segments in the result, or equivalently, to the topological Euler characteristic of $\mathcal{P}$. A notion analogous to boundary piece from Definition~\ref{def:dft} appears under the name of ``piece collection'' in \cite[\S 5.6]{asymptotic_stats_random_covering_surfaces_MageePuder23}.
\begin{example}
All exposed sides of the complex in Figure~\ref{sfig:bdry} are facing phantom simplices. The shorter boundary component consists of 4 exposed sides and 2 hanging half-edges, so (considered as a circular piece) its defect is 1. The longer one has 5 exposed sides and 16 hanging half-edges, so its defect is $\tfrac{5}{3}-\tfrac{16}{6}=-1$. All proper subpieces of the latter (which are topologically disjoint unions of segments) have negative defect, so $\md(T)=1$.
\end{example}

Negative maximal defect can be thought of as a kind of convexity condition. In the special case of non-singular $\Delta$-complexes, positive defect of a piece comes from pairs of adjacent exposed sides separated by zero or one hanging half-edges (i.e. one or two phantom corners respectively). The absence of these two boundary configurations (and hence a condition strengthening $\md\leqslant0$) appears as ``local 3-convexity'' in \cite[Definition 5.1]{combi_nonpositive_curvature_weak_systolicity_Osajda13}, and is proved to imply geodesic convexity of the universal covering $\widetilde{T}$ in $\widetilde\Sigma$ \cite[Lemma 5.2]{combi_nonpositive_curvature_weak_systolicity_Osajda13}. In our setting, convexity-like properties for complexes with $\md<0$ will be guaranteed by Lemma~\ref{lmm:no_dft_subgp}, and such complexes will play a role similar to strongly boundary reduced tiled surfaces from \cite[Definition 2.5]{asymptotic_stats_random_covering_surfaces_MageePuder23}. For another interpretation of the boundary defect, via $\cho$, see Corollary~\ref{cor:cho_V_dft} and Remark~\ref{rmk:inner_dft}.

\medskip

We conclude this section by demonstrating how the master Theorem \ref{thm:E_combi_emb} implies Theorem \ref{thm:cone_pt_count} on frequency of cone points in random coverings.
\begin{proof}[Proof of Theorem \ref{thm:cone_pt_count}]
Let $x_1,\dots,x_k$ be the cone points of orders $m_1,\dots,m_k$ divisible by $m$, and $f_1,\dots,f_k$ the corresponding orbigons of $\Sigma$ (with a single side edge, according to our convenience convention from Section~\ref{sec:simplicial}). Let $T_i$ be an $m_i$-orbigon with $\mathfrak{s}_i\defeq\tfrac{m_i}{m}$ sides and a cone point of order $m$. Make it into an immersed complex by considering it as an $\mathfrak{s}_i$-fold branched covering of $f_i$. A cone point of order $m$ in a covering $\whS$ and mapping to $x_i$ is surrounded by a polygon mapping to $f_i$ and having $\mathfrak{s}_i$ edges. Such a polygon is a copy of $T_i$, so gives rise to $\mathfrak{s}_i$ embeddings of it (that differ by rotation). Therefore
\begin{equation*}
N(\whS)=\sum_{i=1}^k\frac{N_i(\whS)}{\mathfrak{s_i}}\qquad\text{where}\qquad N_i(\whS)\defeq\#\{T_i\hookrightarrow\whS\}.
\end{equation*}

The complex $T_i$ has $\cho(T_i)=\tfrac{1}{m}$, and by inspecting the boundary we can check that $\md(T_i)<0$. Theorem~\ref{thm:E_combi_emb} then tells us that
\begin{equation*}
\mathbb{E}_n[N_i]\sim n^{\frac{1}{m}}.
\end{equation*}
Applying Theorem~\ref{thm:E_combi_emb} to the disjoint union of two copies of $T_i$ counts pairs of embeddings with disjoint images of vertices, so
\begin{equation*}
\mathbb{E}_n[N_i(N_i-\mathfrak{s}_i)]=\mathbb{E}_n[\#\{T_i\sqcup T_i\hookrightarrow\whS\}]\sim n^{\frac{2}{m}}.
\end{equation*}

This means that the expectation of random variables $n^{-\frac{1}{m}}N_i$ converges to 1 and their variance converges to 0 as $n\to\infty$. By Chebyshev's inequality each converges to 1 in probability, so $n^{-\frac{1}{m}}N$ converges to $\sum_{i=1}^k\tfrac{1}{\mathfrak{s}_i}$.
\end{proof}

\subsection{Resolutions and proof of Theorem \ref{thm:alg_tf}}\label{ssec:res_main_pf}

To prove Theorem~\ref{thm:alg_tf} from Section~\ref{sec:intro} we translate it into a combinatorial statement about the frequency of immersions of a given subcomplex in a random covering. Theorem~\ref{thm:E_combi_emb} answers the same question about embeddings. To bridge this difference we introduce resolutions. Every immersion is an embedding of a quotient, and intuitively, resolutions enumerate all the quotients and possible neighbourhoods of their images. They already appear in \cite[Definition 2.8]{asymptotic_stats_random_covering_surfaces_MageePuder23}; in the setting of free groups a similar object is the left derivation $L^X$ of the expected number of fixed points in \cite[Lemma 6.3]{meas_preserving_words_primitive_PuderParzanchevski15} or the finite set of minimal overgroups of a f.g. subgroup in \cite[Theorem 2]{chain_conditions_free_gps_Takahasi51}.
\begin{definition}\label{def:resolution}
A collection $\mathcal{R}=\{q\colon T\looparrowright W\}$ of morphisms of immersed complexes is called a \emph{resolution} of $T$ if every immersion $p\colon T\looparrowright\whS$ into a covering of $\Sigma$ factors uniquely as
\begin{equation*}\begin{tikzcd}
T\arrow[r,{Glyph[glyph math command=looparrowleft, swap]}->,"q"]\arrow[rr,to path={--++(0,-.5)-|(\tikztotarget)\tikztonodes},below,pos=.25,"p"]&W\arrow[r,hook,"\iota"]&\whS
\end{tikzcd}\end{equation*}
with $(q,W)\in\mathcal{R}$ and $\iota$ being an embedding.
\end{definition}
Existence of useful resolutions is guaranteed by the following proposition. Its proof is quite technical, so we defer it to Section~\ref{ssec:construct_resolution}.
\begin{proposition}\label{prop:resolution}
For any constant $K$, there exists a finite resolution $\mathcal{R}$ of $T$, such that all $(q,W)\in\mathcal{R}$ satisfy $\md(W)<0$ or $\cho(W)+\max\{0,\md(W)\}<K$. Moreover, $\mathcal{R}$ can be found algorithmically.
\end{proposition}

To interpret the elements of a resolution spat out by Proposition~\ref{prop:resolution} in terms of overgroups we will use two consequences of bounded-above defect. The first one can be thought of as a convexity statement, and ensures that the immersion is $\pi_1$-injective and arises from an embedding into some covering (this promotes an immersed complex to a ``sub-cover'' in the sense of \cite[Definition 2.1]{local_stats_random_perms_free_prods_PuderZimhoni24}).
\begin{lemma}\label{lmm:no_dft_subgp}
If $p\colon\!T\!\looparrowright\!\Sigma$ is a (based connected) immersed complex with $\md(T)<0$, then the following hold.
\begin{enumerate}
\item The complex $T$ embeds (lifting $p$) in $\sfrac{\widetilde\Sigma}{p_*(\pio T)}$ as a deformation retract.
\item The fundamental group $\pio T$ is identified with a subgroup of $\pio\Sigma$, in the sense that the homomorphism $p_*\colon\pio T\to\pio\Sigma$ is injective.
\item Orbifold Euler characteristics of the complex $T$ and group $p_*(\pio T)$ agree.
\end{enumerate}
\end{lemma}
Methods of Section \ref{sec:resolutions} allow to prove a result in the other direction: for any f.g. subgroup $\Lambda\le\pio\Sigma$ one can find a finite immersed complex $p\colon T\looparrowright\Sigma$ with $\md(T)\!<\!0$ and $p_*(\pio T)=\Lambda$. An analogue of Lemma~\ref{lmm:no_dft_subgp} for the standard one-relator cell structure on a surface appears as \cite[Prop 4.3 and Corollary 4.11]{core_surfaces_MageePuder22}.

The second result says that the defect controls how much the topology can be simplified by capping the boundaries. As a basic example, a negative-defect boundary component cannot be closed off with a disk.
\begin{proposition}\label{prop:no_cap}
Let $T\hookrightarrow\whS$ be an embedding of an immersed complex $T$ into a (not necessarily finite) covering $\whS$ of $\Sigma$. Then the following hold.
\begin{enumerate}
\item\label{itm:nocap_general} $\cho(\whS)\leqslant\cho(T)+\max\{0,\md(T)\}$
\item\label{itm:nocap_nodft} If $\md(T)<0$ and $\cho(\whS)=\cho(T)$, then $\whS$ deformation retracts onto $T$.
\end{enumerate}
\end{proposition}

Proofs of Lemma~\ref{lmm:no_dft_subgp} and Proposition~\ref{prop:no_cap} are deferred to Section~\ref{ssec:pfs_aux_combi}.

\medskip

Now we are ready to deduce Theorem~\ref{thm:alg_tf} about fixed points under random homomorphisms from the master Theorem~\ref{thm:E_combi_emb} about frequency of embeddings in a random covering.
\begin{proof}[Proof of Theorem~\ref{thm:alg_tf}]
We begin by translating groups into $\Delta$-complexes. Pick a triangulation of $\Sigma=\sfrac{\mathbb{H}^2}{\Gamma}$. Recall that there is a natural correspondence between homomorphisms $\Gamma\to S_n$ and degree-$n$ coverings of $\Sigma$. Let $T$ be the (based) result of Stallings-folding a simplicial rose on a finite generating set for $\Lambda$. It comes with an immersion $p\colon T\looparrowright\Sigma$ such that $p_*(\pio T)=\Lambda$. Crucially, an action of $\Lambda$ fixes a number $i\in[n]$ precisely when $p$ lifts to the corresponding covering at the $i$-th point of the fibre. Therefore
\begin{equation*}
\mathbb{E}_n[\#\tn{fix}_\Lambda(\phi)]=\mathbb{E}_n[\#\{T\looparrowright\whS\}].
\end{equation*}
Now we have to understand how often $T$ immerses in a random covering.

Let $\mathcal{R}=\mathcal{R}_\tn{nd}\sqcup\mathcal{R}_\tn{err}$ be a finite resolution for $(p,T)$, where every complex in $\mathcal{R}_\tn{nd}$ has negative maximal defect and every element of $\mathcal{R}_\tn{err}$ satisfies $\cho+\max\{0,\md\}<\cho(\Gamma)-1$. Existence of $\mathcal{R}$ is guaranteed by Proposition~\ref{prop:resolution}. Every immersion of $T$ factors through an embedding of a unique complex from $\mathcal{R}$, so
\begin{equation}\label{eq:emb_partition_immersion}
\mathbb{E}_n[\#\{T\looparrowright\whS\}]=\sum_{W\in\mathcal{R}}\mathbb{E}_n[\#\{W\hookrightarrow\whS\}].
\end{equation}

We need to identify the leading terms in (\ref{eq:emb_partition_immersion}). Applying the master Theorem~\ref{thm:E_combi_emb} to the orbifold $\Sigma$ itself tells us that in expectation there are $\Theta(n^{\cho(\Gamma)})$ ways to write $\whS\in\covers$ as a disjoint union of a copy of $\Sigma$ and a degree-$(n-1)$ covering. Trivially $T$ immerses in the former (dynamically this corresponds to a global fixed point of $\Gamma\curvearrowright[n]$), so the RHS of equation (\ref{eq:emb_partition_immersion}) is at least $n^{\cho(\Gamma)}(1+o(1))$. Theorem~\ref{thm:E_combi_emb} tells us that each summand coming from $\mathcal{R}_\tn{err}$ is $O(n^{\cho(\Gamma)-1+\varepsilon})$, so its contribution is negligible. Hence we see that
\begin{equation*}
\mathbb{E}_n[\#\{T\looparrowright\whS\}]=A\cdot n^B\cdot(1+o(1)),
\end{equation*}
where $B$ is the maximal orbifold Euler characteristic of a complex in $\mathcal{R}_\tn{nd}$ (as we just argued $B\geqslant\cho(\Gamma)$) and $A$ is the number of maximisers.

It remains to identify the constants $A,B$. We begin with the exponent $B$. Consider the function
\begin{alignat*}{4}
f\colon&\qquad\mathcal{R}_\tn{nd}&&\to\{J\le\Gamma\}\\
f\colon&\left(W\stackrel{r}{\looparrowright}\Sigma\right)&&\mapsto r_*(\pio W)
\end{alignat*}
that sends a complex to the image of its fundamental group in $\Gamma$ (which is an overgroup of $\Lambda$). The image of the basepoint of $T$ distinguishes a vertex in each complex $W$, so talking about $\pio W$ and not just its conjugacy class is justified. Lemma~\ref{lmm:no_dft_subgp} tells us that $f$ preserves the orbifold Euler characteristic, so $B\leqslant\chi_\tn{max}(\Lambda)$. For the opposite inequality, take $J\ge\Lambda$ with $\cho(J)=\chi_\tn{max}(\Lambda)$ maximal. Combinatorially $J$ being an overgroup of $\Lambda$ means that $T$ immerses in $\sfrac{\widetilde\Sigma}{J}$. By the definition of resolution, there exists an embedding $W\hookrightarrow\sfrac{\widetilde\Sigma}{J}$ for some $W\in\mathcal{R}$. If we had $W\in\mathcal{R}_\tn{err}$, then item~\ref*{itm:nocap_general} of Proposition~\ref{prop:no_cap} would imply
\begin{equation*}
\chi_\tn{max}(\Lambda)=\cho\left(\sfrac{\widetilde\Sigma}{J}\right)\leqslant\cho(W)+\max\{0,\md(W)\}<\cho(\Gamma)-1,
\end{equation*}
which is impossible. Therefore $W\in\mathcal{R}_\tn{nd}$, so $\md(W)<0$. Again item~\ref*{itm:nocap_general} of Proposition~\ref{prop:no_cap} gives us $\cho(J)\leqslant\cho(W)$. But $\cho(W)\leqslant B$ by the definition of $B$. Putting these inequalities together tells us that $\chi_\tn{max}(\Lambda)\leqslant B$. We proved inequalities both ways, so $B=\chi_\tn{max}(\Lambda)$.

We move on to $A$. To relate it to the number of maximal overgroups, we will prove that $f$ induces a bijection between maximal-$\cho$ complexes and maximal-$\cho$ overgroups. For injectivity, suppose $f(W_1)=f(W_2)=J$. Then both $W_1,W_2$ embed into $\sfrac{\widetilde\Sigma}{J}$ (the same image of the basepoint), so by the uniqueness condition from Definition~\ref{def:resolution} of resolution we must have $W_1=W_2$. For surjectivity, consider an overgroup $J\ge\Lambda$ with $\cho(J)=\chi_\tn{max}(\Lambda)$. Then $T$ immerses in $\sfrac{\widetilde\Sigma}{J}$, so there is $W\in\mathcal{R}$ that embeds in it. Using statement~\ref*{itm:nocap_general} of Proposition~\ref{prop:no_cap} in the same way as in the previous paragraph, we obtain that $W\in\mathcal{R}_\tn{nd}$ and that $\cho(W)=\chi_\tn{max}(\Lambda)=\cho(\sfrac{\widetilde\Sigma}{J})$. Statement~\ref*{itm:nocap_nodft} of Proposition~\ref{prop:no_cap} then implies that $W$ is a deformation retract of $\sfrac{\widetilde\Sigma}{J}$, and hence $f(W)=\pio W=J$. This concludes the proof of bijectivity of $f$ between top-$\cho$ elements, so $A=\#\tn{mog}(\Lambda)$.
\end{proof}
\begin{remark}\label{rmk:mog_finite}
In the proof of Theorem~\ref{thm:alg_tf} we showed that the set $\tn{mog}(\Lambda)$ is finite, because it is in bijection with a subset of the finite set $\mathcal{R}_{\tn{nd}}$.
\end{remark}

\section{Background: ribbon graph calculus}\label{sec:ribbon_graph}

In this section we describe the calculus of ribbon graphs, which is a tool to keep track of large linear algebra. The basic idea is simple: whenever we evaluate a multilinear map by contracting some indices on a bunch of tensors, we write the tensors on ``coupons'' and indicate the contractions by joining them with ``wires''. However, to fit the existing literature we phrase everything in the language of ribbon categories. For a comprehensive treatment of this topic consult \cites{graphical_language_monoidal_categories_Selinger11}{quantum_invariants_knots_3mflds_Turaev16}.

\begin{definition}[ribbon graph]
The building blocks of ribbon graphs are coloured coupons and wires.
\begin{itemize}
\item A \emph{wire} is an oriented segment, equipped with a vector space called \emph{colour}.
\item A \emph{coupon} is a rectangle with distinguished top and bottom base, and some number of wires attached to bases.
\item Given a coupon $C$ with wires coloured $U_1,\dots,U_m$ and $V_1,\dots,V_n$ attached to its bottom and top base respectively, the \emph{colour} of $C$ can be any linear map $U_1^{\varepsilon_1}\otimes\dots\otimes U_m^{\varepsilon_m}\to V_1^{\eta_1}\otimes\dots\otimes V_n^{\eta_n}$, where $\varepsilon_i,\eta_j$ are $+1$ if the corresponding wire is directed up (from bottom to top base of $C$) and $-1$ if it is directed down, and $V^{-1}$ is a convenience notation for the dual $V^*$.
\end{itemize}
A \emph{ribbon graph} is a graph consisting of coloured wires and coupons, such that each end of a wire is either attached to a coupon, to the other end of the same wire, or free and designated as incoming or outgoing.
\end{definition}
Note that incoming/outgoing designation of free ends is not constrained by the ribbon's orientation. Some examples of ribbon graphs are presented in Figure \ref{fig:ribbon_examples}. We mark the bottom base of a coupon with a double edge.
\begin{figure}
\begin{subfigure}{.8\textwidth}
\caption{A wire coloured with a vector space $V$ and oriented from a free incoming end (bottom) to a free outgoing end (top). It represents the operator $\tn{id}_V\colon V\to V$.}
\end{subfigure}
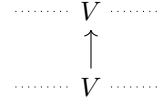
\begin{subfigure}{.19\textwidth}
\centering
\begin{tikzpicture}
\draw[dotted](-1,0)--(1,0)(-1,1)--(1,1);
\node[fill=white](i)at(0,0){$V$};\node[fill=white](o)at(0,1){$V$};\draw[->](i)--(o);
\end{tikzpicture}
\end{subfigure}
\begin{subfigure}{.8\textwidth}
\caption{A wire coloured with a vector space $V$ and oriented from a free outgoing end (top) to a free incoming end (bottom). It represents the operator $\tn{id}_{V^*}\colon V^*\to V^*$.}
\label{sfig:id_dual}
\end{subfigure}
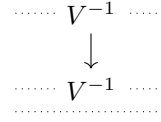
\begin{subfigure}{.19\textwidth}
\centering
\begin{tikzpicture}
\draw[dotted](-1,0)--(1,0)(-1,1)--(1,1);
\node[fill=white](o)at(0,0){$V^{-1}$};\node[fill=white](i)at(0,1){$V^{-1}$};\draw[->](i)--(o);
\end{tikzpicture}
\end{subfigure}
\begin{subfigure}{.8\textwidth}
\caption{A wire coloured $V$ with two incoming and no outgoing ends. It represents the evaluation map $V^*\otimes V\to\mathbb{C},u^*\otimes v\mapsto u^*(v)$.}
\label{sfig:eval}
\end{subfigure}
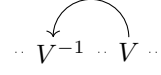
\begin{subfigure}{.19\textwidth}
\centering
\begin{tikzpicture}
\draw[dotted](-.5,0)--(1.5,0)(-.5,1)--(1.5,1);
\node[fill=white](d)at(.1,0){$V^{-1}$};\node[fill=white](v)at(1,0){$V$};\draw[->](1,.2)arc(0:180:.5);
\end{tikzpicture}
\end{subfigure}
\begin{subfigure}{.8\textwidth}
\caption{A wire coloured $V$ with no incoming and two outgoing ends. It represents the coevaluation map $\mathbb{C}\to V^*\otimes V$ given by scaling the vector $\sum_ie_i^*\otimes e_i$ for any pair of dual bases $\{e_i\}_i,\{e_i^*\}_i$.}
\label{sfig:coeval}
\end{subfigure}
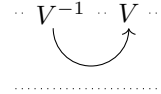
\begin{subfigure}{.19\textwidth}
\centering
\begin{tikzpicture}
\draw[dotted](-.5,0)--(1.5,0)(-.5,-1)--(1.5,-1);
\node[fill=white](d)at(.1,0){$V^{-1}$};\node[fill=white](v)at(1,0){$V$};\draw[->](0,-.2)arc(180:360:.5);
\end{tikzpicture}
\end{subfigure}
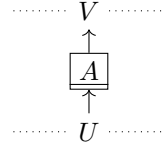
\begin{subfigure}[t]{.45\textwidth}
\centering
\begin{tikzpicture}
\draw[dotted](-1,-.8)--(1,-.8)(-1,.8)--(1,.8);
\node[draw](op)at(0,0){$A$};\draw($(op.south east)+(0,.07)$)--($(op.south west)+(0,.07)$);\node[fill=white](u)at(0,-.8){$U$};\node[fill=white](v)at(0,.8){$V$};\draw[->](u)--(op);\draw[->](op)--(v);
\end{tikzpicture}
\caption{A ribbon graph with two wires and a coupon coloured $A$. It represents an operator $A\colon U\to V$.}
\label{sfig:operator}
\end{subfigure}
\hfill
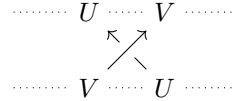
\begin{subfigure}[t]{.45\textwidth}
\centering
\begin{tikzpicture}
\draw[dotted](-1.5,-.5)--(1.5,-.5)(-1.5,.5)--(1.5,.5);
\node[fill=white](vi)at(-.5,-.5){$V$};\node[fill=white](vo)at(.5,.5){$V$};\node[fill=white](ui)at(.5,-.5){$U$};\node[fill=white](uo)at(-.5,.5){$U$};
\draw[->](vi)--(vo);\draw(ui)--(.1,-.1);\draw[->](-.1,.1)--(uo);
\end{tikzpicture}
\caption{A ribbon graph with two wires, each one with one incoming and outgoing free end. It represents the ``braiding'' operator $U\otimes V\to V\otimes U,u\otimes v\mapsto v\otimes u$.}
\label{sfig:braiding}
\end{subfigure}
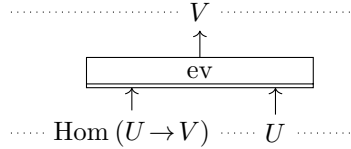
\begin{subfigure}[t]{.45\textwidth}
\centering
\begin{tikzpicture}
\draw[dotted](-2.5,-.8)--(2,-.8)(-2.5,.8)--(2,.8);
\node at(0,0){ev};\node[fill=white](u)at(1,-.8){$U$};\node[fill=white](hom)at(-.9,-.8){$\Hom{U\!}{\!V}$};\node[fill=white](v)at(0,.8){$V$};
\draw(-1.5,.2)--(1.5,.2)--(1.5,-.2)--(-1.5,-.2)--(-1.5,.2)(-1.5,-.15)--(1.5,-.15);
\draw[->](u)--(1,-.2);\draw[->](hom)--(-.9,-.2);\draw[->](0,.2)--(v);
\end{tikzpicture}
\caption{A coupon with three wires attached. It represents the evaluation operator $\tn{ev}\colon\Hom{U}{V}\otimes U\to V$.}
\end{subfigure}
\hfill
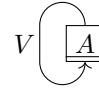
\begin{subfigure}[t]{.45\textwidth}
\centering
\begin{tikzpicture}
\node[draw](op)at(0,0){$A$};\draw($(op.south west)+(0,.07)$)--($(op.south east)+(0,.07)$);\draw[->](op.north)arc(0:180:.3)--($(op.south)-(.6,0)$)arc(-180:0:.3);\node at(-.8,0){$V$};
\end{tikzpicture}
\caption{A ribbon graph with one wire coloured $V$, a coupon coloured with $A\in\tn{End}(V)$, and no free ends. As we will see, it represents the scalar $\tr A$.}
\label{sfig:rib_tr}
\end{subfigure}
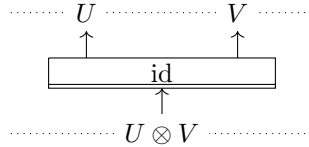
\begin{subfigure}{.45\textwidth}
\centering
\begin{tikzpicture}
\draw[dotted](-2,.8)--(2,.8)(-2,-.8)--(2,-.8);
\node at(0,0){id};\node[fill=white](u)at(-1,.8){$U$};\node[fill=white](v)at(1,.8){$V$};\node[fill=white](uv)at(0,-.8){$U\otimes V$};
\draw(-1.5,.2)--(1.5,.2)--(1.5,-.2)--(-1.5,-.2)--(-1.5,.2)(-1.5,-.15)--(1.5,-.15);
\draw[->](uv)--(0,-.2);\draw[->](-1,.2)--(u);\draw[->](1,.2)--(v);
\end{tikzpicture}
\caption{A ribbon graph representing the identity map $\tn{id}\colon U\otimes V\to U\otimes V$.}
\end{subfigure}
\caption{Examples of ribbon graphs, and the operators they represent via Proposition \ref{prop:ribbon_functor}. Here we place incoming free ends of wires on the bottom and outgoing on the top. Double edge of a coupon is the bottom base. Free ends marked with $V^{-1}$ end in an object $(V,-1)$ in the category $\mathbf{Rib}$ from Definition~\ref{def:rib_category}; all others end in an object of the form $(V,1)$.}
\label{fig:ribbon_examples}
\end{figure}

Ribbon graphs can be organised into a monoidal braided category, mirroring composition and tensor product of linear maps.
\begin{definition}\label{def:rib_category}
The category $\mathbf{Rib}$ has the following.
\begin{itemize}
\item The objects are tuples $((V_1,\epsilon_1),\dots,(V_n,\epsilon_n))$ for some finite-dimensional complex vector spaces $V_i$ and $\epsilon_i\in\{\pm1\}$.
\item The morphisms from $((U_i,\epsilon_i))_{i\in[m]}$ to $((V_j,\eta_j))_{j\in[n]}$ are ribbon graphs with $m$ incoming and $n$ outgoing free ends of wires. The wire attached to the $i$-th incoming end is oriented away from the free end and coloured $U_i$ if $\epsilon_i\!=\!1$, or oriented towards the end and coloured $U_i^*$ if $\epsilon_i\!=\!-1$. The wire attached to the $j$-th outgoing end is oriented towards the free end and coloured $V_j$ if $\eta_j\!=\!1$, or oriented away from it and coloured $V_j^*$ if $\eta_j\!=\!-1$.
\end{itemize}
The categorical operations are as follows.
\begin{itemize}
\item Composition of morphisms is by joining the appropriate outgoing and incoming ends of graphs.
\item The monoidal product of objects is concatenation of tuples, monoidal unit being the empty tuple.
\item The monoidal product of morphisms is disjoint union.
\item Braiding is given by the graph in Figure \ref{sfig:braiding}.
\end{itemize}
\end{definition}

The ribbon category is related to linear algebra by Proposition \ref{prop:ribbon_functor} below. It is essentially a special case of \cite[Theorem 2.5 in Chapter I]{quantum_invariants_knots_3mflds_Turaev16} applied to the category $\mathbf{FinVect}_\mathbb{C}$ (finite-dimensional complex vector spaces) with the monoidal structure given by tensor product, braiding $U\otimes V\to V\otimes U$ given by $u\otimes v\mapsto v\otimes u$, trivial twist, and the standard linear duality.
\begin{proposition}\label{prop:ribbon_functor}
There exists a unique monoidal functor $\mathcal{F}\colon\mathbf{Rib}\to\mathbf{FinVect}_\mathbb{C}$ which sends the object $(V,1)$ to $V$ and $(V,-1)$ to $V^*$, the graphs in Figures~\ref{sfig:eval},\ref{sfig:coeval} to evaluation and coevaluation respectively, and which sends every coupon to its colour.
\end{proposition}
Consequently, from now on when we write a ribbon graph in an equation, we mean the linear operator obtained by applying to it the monoidal functor $\mathcal{F}$ from Proposition \ref{prop:ribbon_functor}. Also, we will not keep track of which free ends are incoming or outgoing, since this does not change the validity of equations. In the ``primitive'' viewpoint mentioned at the beginning of this section, the functor $\mathcal{F}$ just contracts the coupon tensors according to pairs of indices from wires.
\begin{example}
The graph in Figure \ref{sfig:rib_tr} is a concatenation of those in \ref*{sfig:coeval}, \ref*{sfig:id_dual}$\sqcup$\ref*{sfig:operator}, and \ref*{sfig:eval}. This is a graphical way to phrase $\tr A=\sum_ie^*_i(A(e_i))$, where $\{e_i\}_i,\{e^*_i\}_i$ are dual bases for the space on which $A$ acts.
\end{example}
\begin{example}\label{ex:resolution_id}
Let $V$ be a vector space with a pair of dual bases $\left\{e_i\right\}_i,\left\{e_i^*\right\}_i$. We have the equality
\begin{equation*}
\begin{tikzpicture}[baseline=0]
\draw[dotted](0,0)circle(.5);\node[fill=white](i)at(0,-.5){$V$};\node[fill=white](o)at(0,.5){$V$};\draw[->](i)--(o);
\end{tikzpicture}
\ =\ \tn{id}_V\ =\ \sum_ie_i\otimes e_i^*\ =\ \sum_i
\begin{tikzpicture}[baseline=0]
\draw[dotted](0,0)circle(1.3);\node[fill=white](i)at(0,-1.3){$V$};\node[fill=white](o)at(0,1.3){$V$};
\node[draw](d)at(0,-.4){$e_i^*$};\draw($(d.south west)+(0,.07)$)--($(d.south east)+(0,.07)$);
\node[draw](b)at(0,.4){$e_i$};\draw($(b.south west)+(0,.07)$)--($(b.south east)+(0,.07)$);
\draw[->](i)--(d);\draw[->](b)--(o);
\end{tikzpicture}
\end{equation*}
where we identified $\Hom{V}{V}=V\otimes V^*$.
\end{example}
Example \ref{ex:resolution_id} formally shows that the operator represented by a ribbon graph is the contraction of the tensors that are colours of the coupons, according to the pairing of indices provided by wires. Moreover, in this description the monoidal functor $\mathcal{F}$ from Proposition~\ref{prop:ribbon_functor} is manifestly multilinear in the colours of coupons.
\begin{example}\label{ex:ribbon_composition}
The composition of linear operators $A_1\colon U\to V$ and $A_2\colon V\to W$ is expressed in ribbon graph calculus as
\begin{equation*}
\begin{tikzpicture}[baseline=0]
\draw[dotted](0,0)ellipse(1.7 and .8);
\node[fill=white](u)at(-1.7,0){$U$};\node[fill=white](w)at(1.7,0){$W$};\node at(0,.2){$V$};
\node[draw](s)at(-.7,0){$A_1$};\draw($(s.south west)+(.07,0)$)--($(s.north west)+(.07,0)$);
\node[draw](t)at(.7,0){$A_2$};\draw($(t.south west)+(.07,0)$)--($(t.north west)+(.07,0)$);
\draw[->](u)--(s);\draw[->](s)--(t);\draw[->](t)--(w);
\end{tikzpicture}
\qquad=\qquad
\begin{tikzpicture}[baseline=0]
\draw[dotted](0,0)ellipse(1.3 and .8);
\node[fill=white](u)at(-1.3,0){$U$};\node[fill=white](w)at(1.3,0){$W$};
\node[draw](ts)at(0,0){$A_2\circ A_1$};\draw($(ts.south west)+(.07,0)$)--($(ts.north west)+(.07,0)$);
\draw[->](u)--(ts);\draw[->](ts)--(w);
\end{tikzpicture}
\end{equation*}
\end{example}

\begin{remark}
There are two main differences between our convention and the notation in \cite{quantum_invariants_knots_3mflds_Turaev16}.

The first one is that \cite{quantum_invariants_knots_3mflds_Turaev16} uses rectangular bands instead of wires. This is needed e.g. when working with framed links or categories with non-trivial twist, but for our needs it suffices to only keep track of cores of bands.

The second difference is that our convention on direction of wires matches the one used in \cite{graphical_language_monoidal_categories_Selinger11} and is opposite to \cite{quantum_invariants_knots_3mflds_Turaev16}.
\end{remark}

In the rest of this article we will mainly be concerned with graphs that have no free ends. Such a graph is an endomorphism of the empty tuple in the category $\mathbf{Rib}$, so Proposition \ref{prop:ribbon_functor} assigns to it an endomorphism of the base field $\mathbb{C}$, which is just a multiplication by a number we call trace.
\begin{definition}[trace]\label{def:ribbon_trace}
Let $X$ be a ribbon graph with no free ends of wires. The trace $\tr X$ of $X$ is the unique scalar such that $\mathcal{F}(X)=(z\mapsto(\tr X)\cdot z)$ as endomorphisms of $\mathcal{F}(\emptyset)=\mathbb{C}$.
\end{definition}
Note that the ribbon-graph trace of the graph shown in Figure \ref{sfig:rib_tr} in the sense of Definition \ref{def:ribbon_trace} is the same as the trace of the operator $A$ in the sense of linear algebra.

\subsection{Summation over subgroups}\label{ssec:sum_subgp}

In this section we state and prove Lemma \ref{lmm:swap}, which gives a formula for the average action of a subgroup on a tensor product of representations. It will be the key tool to compute sums over two-sided translates of subgroups in Lemma \ref{lmm:Mednykh} and Proposition \ref{prop:count_E}. Recall that when we write a ribbon graph in an equation, we mean the linear operator represented by it via the monoidal functor $\mathcal{F}$ from Proposition \ref{prop:ribbon_functor}. Let us begin with the following rephrasing of Schur's Lemma.
\begin{lemma}\label{lmm:swap_irr}
Let $V_1,V_2$ be two irreducible representations of a finite group $G$. Then, as operators on $V_1\otimes V_2^*$, we have the equality
\begin{equation*}
\frac{1}{\#G}\sum_{g\in G}g.(v_1\otimes v_2^*)=\begin{cases}\chi_V(1)^{-1}\cdot v_2^*(v_1)\cdot\sum_ie_i\otimes e_i^*&\text{ if }V_1\cong V_2\\0&\text{ if }V_1\ncong V_2\end{cases}
\end{equation*}
where in the first case, $\{e_i\}_i$ and $\{e_i^*\}_i$ are a pair of dual bases for $V_1=V_2$.
\end{lemma}
By $\chi_V$ we mean the character of $V$, so that $\chi_V(1)=\dim V$. Note that in the first case we implicitly use an intertwiner $V_1\cong V_2$, but the final value is independent of the choice of it. Lemma \ref{lmm:swap_irr} can be conveniently expressed using the ribbon graph calculus (compare Figures \ref{sfig:eval} and \ref{sfig:coeval}).
\begin{equation*}
\begin{tikzpicture}[baseline=.9]
\draw[dotted](0,0)circle(1.6);
\node[fill=white](li)at(-1,-1.2){$V_1$};\node[fill=white](lo)at(-1,1.2){$V_1$};\node[fill=white](ri)at(1,1.2){$V_2$};\node[fill=white](ro)at(1,-1.2){$V_2$};\node at(0,0){$\displaystyle\frac{1}{\#G}\sum_{g\in G}g$};
\draw(-1.4,.6)--(1.4,.6)--(1.4,-.6)--(-1.4,-.6)--(-1.4,.6)(-1.4,-.5)--(1.4,-.5);
\draw[->](li)--(-1,-.6);\draw[->](-1,.6)--(lo);\draw[->](ri)--(1,.6);\draw[->](1,-.6)--(ro);
\end{tikzpicture}
\qquad=\begin{cases}
\chi_V(1)^{-1}\times
\begin{tikzpicture}[baseline=.5]
\draw[dotted](0,0)circle(.86);
\node[fill=white](li)at(-.5,-.7){$V$};\node[fill=white](lo)at(-.5,.7){$V$};\node[fill=white](ri)at(.5,.7){$V$};\node[fill=white](ro)at(.5,-.7){$V$};
\draw[->](li)to[out=90,in=90](ro);\draw[->](ri)to[out=-90,in=-90](lo);
\end{tikzpicture}
&\text{if }V_1=V_2\\\ \\\hfill 0\hfill&\text{otherwise}
\end{cases}
\end{equation*}
\begin{proof}[Proof of Lemma \ref{lmm:swap_irr}]
We can identify $V_1\otimes V_2^*=\Hom[\mathbb{C}]{V_2}{V_1}$, with action $g.\varphi=\rho_1(g)\circ\varphi\circ\rho_2(g^{-1})$. By a direct verification, for any $\mathbb{C}$-linear map $\varphi\colon V_2\to V_1$, the resulting $A\defeq\tfrac{1}{\#G}\sum_gg.\varphi$ intertwines the actions $G\curvearrowright V_2,V_1$, or in other words is $\mathbb{C}[G]$-linear. By Schur's lemma, $A$ must be zero if $V_1\ncong V_2$, and a multiple of identity $\sum_ie_i\otimes e_i^*$ otherwise. In the latter case we determine the scaling factor by comparing traces, which are $v_2^*(v_1)$ for the LHS and $\chi_V(1)$ for the identity.
\end{proof}

Lemma \ref{lmm:swap_irr} allows us to understand averaging over the whole group $G$, and this is already enough to count $\Hom{\pio\Sigma}{G}$. We accomplish this and prove Mednykh's formula (Lemma \ref{lmm:Mednykh}) in Section \ref{ssec:Mednykh}. However, in order to keep track of immersed complexes we will need a more refined statement, that of Lemma \ref{lmm:swap}. Recall that for a representation $V$ of a finite group $G$ and a subgroup $H\leq G$, the restricted representation admits a direct sum decomposition into $H$-isotypic components
\begin{equation}\label{eq:isotypic}
\tn{Res}^G_HV=\bigoplus_{U\in\tn{Irr}(H)}U\otimes\Hom[H]{U}{V}.
\end{equation}
Denote the isotypic projections by $P^V_U\colon V\to U\otimes\Hom[H]{U}{V}$ and the isotypic inclusions by $I^V_U\colon U\otimes\Hom[H]{U}{V}\to V$ (the latter is just evaluation).
\begin{lemma}\label{lmm:swap}
Consider complex representations $V_1,V_2$ of a finite group $G$ and let $H\leq G$ be a subgroup. Then we have the following equality of endomorphisms of $V_1\otimes V_2^*$:
\begin{equation*}
\frac{1}{\#H}\sum_{h\in H}h=\sum_{U\in\tn{Irr}(H)}\chi_U(1)^{-1}\cdot(I^{V_1}_U\otimes (P^{V_2}_U)^*)\circ S_U\circ(P^{V_1}_U\otimes (I^{V_2}_U)^*)
\end{equation*}
where
\begin{align*}
S_U\in&\tn{End}\big(U\otimes\Hom{U}{V_1}\otimes U^*\otimes\Hom{U}{V_2}^*\big)\\
S_U\colon&u_1\otimes\varphi_1\otimes u^*_2\otimes\varphi^*_2\mapsto u^*_2(u_1)\cdot\sum_ie_i\otimes\varphi_1\otimes e^*_i\otimes\varphi^*_2
\end{align*}
for any pair of dual bases $\{e_i\}_i,\{e^*_i\}_i$ of $U,U^*$.

In the notation of ribbon graphs, we have
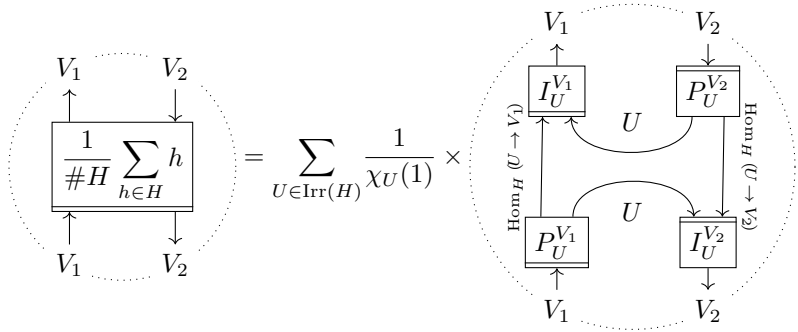
\begin{equation*}
\begin{tikzpicture}[baseline=1]
\draw[dotted](0,0)circle(1.5);
\node[draw](avg)at(0,0){$\displaystyle\frac{1}{\#H}\sum_{h\in H}h$};\draw($(avg.south east)+(0,.07)$)--($(avg.south west)+(0,.07)$);
\node[fill=white](u1o)at(-.7,1.3){$V_1$};\node[fill=white](u1i)at(-.7,-1.3){$V_1$};\node[fill=white](u2i)at(.7,1.3){$V_2$};\node[fill=white](u2o)at(.7,-1.3){$V_2$};
\draw[->](u1i)--($(avg.south)-(.7,0)$);\draw[->]($(avg.north)-(.7,0)$)--(u1o);\draw[->]($(avg.south)+(.7,0)$)--(u2o);\draw[->](u2i)--($(avg.north)+(.7,0)$);
\end{tikzpicture}
=\sum_{U\in\tn{Irr}(H)}\!\frac{1}{\chi_U(1)}\times
\begin{tikzpicture}[baseline=1]
\draw[dotted](0,0)circle(2.15);
\node[fill=white](v1i)at(-1,-1.9){$V_1$};\node[fill=white](v1o)at(-1,1.9){$V_1$};\node[fill=white](v2o)at(1,-1.9){$V_2$};\node[fill=white](v2i)at(1,1.9){$V_2$};
\node[draw](p1)at(-1,-1){$P^{V_1}_U$};\draw($(p1.south west)+(0,.07)$)--($(p1.south east)+(0,.07)$);
\node[draw](i1)at(-1,1){$I^{V_1}_U$};\draw($(i1.south west)+(0,.07)$)--($(i1.south east)+(0,.07)$);
\node[draw](p2)at(1,1){$P^{V_2}_U$};\draw($(p2.north west)-(0,.07)$)--($(p2.north east)-(0,.07)$);
\node[draw](i2)at(1,-1){$I^{V_2}_U$};\draw($(i2.north west)-(0,.07)$)--($(i2.north east)-(0,.07)$);
\draw[->](v1i)--(p1);\draw[->](i1)--(v1o);\draw[->](i2)--(v2o);\draw[->](v2i)--(p2);
\path(p1.north west)--(p1.north east)coordinate[pos=.25](p1ho)coordinate[pos=.75](p1ro);
\path(i1.south west)--(i1.south east)coordinate[pos=.25](i1hi)coordinate[pos=.75](i1ri);
\path(i2.north west)--(i2.north east)coordinate[pos=.25](i2ro)coordinate[pos=.75](i2ho);
\path(p2.south west)--(p2.south east)coordinate[pos=.25](p2ri)coordinate[pos=.75](p2hi);
\draw[->](p1ho)--(i1hi);\draw[->](p2hi)--(i2ho);
\draw[->](p1ro)to[out=90,in=90](i2ro);\draw[->](p2ri)to[out=-90,in=-90](i1ri);
\node at(0,.6){$U$};\node at(0,-.6){$U$};\node[font=\tiny]at(-1.55,0)[rotate=90]{$\Hom[H]{\!U\!}{\!V_1\!}$};\node[font=\tiny]at(1.55,0)[rotate=-90]{$\Hom[H]{\!U\!}{\!V_2\!}$};
\end{tikzpicture}
\end{equation*}
\end{lemma}
\begin{proof}
We apply the isotypic decomposition, equation (\ref{eq:isotypic}), to $V_1,V_2$ on the LHS. This yields
\begin{equation*}
\begin{tikzpicture}[baseline=1]
\draw[dotted](0,0)circle(1.5);
\node[draw](avg)at(0,0){$\displaystyle\frac{1}{\#H}\sum_{h\in H}h$};\draw($(avg.south east)+(0,.07)$)--($(avg.south west)+(0,.07)$);
\node[fill=white](u1o)at(-.7,1.3){$V_1$};\node[fill=white](u1i)at(-.7,-1.3){$V_1$};\node[fill=white](u2i)at(.7,1.3){$V_2$};\node[fill=white](u2o)at(.7,-1.3){$V_2$};
\draw[->](u1i)--($(avg.south)-(.7,0)$);\draw[->]($(avg.north)-(.7,0)$)--(u1o);\draw[->]($(avg.south)+(.7,0)$)--(u2o);\draw[->](u2i)--($(avg.north)+(.7,0)$);
\end{tikzpicture}
=\sum_{U_1,U_2\in\tn{Irr}(H)}
\begin{tikzpicture}[baseline=1]
\draw[dotted](0,0)circle(2.42);
\node[draw](avg)at(0,0){$\displaystyle\frac{1}{\#H}\sum_{h\in H}h$};\draw($(avg.south east)+(0,.07)$)--($(avg.south west)+(0,.07)$);
\node[draw](u1o)at(-1,1.3){$P_{U_1}^{V_1}$};\draw($(u1o.south west)+(0,.07)$)--($(u1o.south east)+(0,.07)$);
\node[draw](u1i)at(-1,-1.3){$I_{U_1}^{V_1}$};\draw($(u1i.south west)+(0,.07)$)--($(u1i.south east)+(0,.07)$);
\node[draw](u2i)at(1,1.3){$P_{U_2}^{V_2}$};\draw($(u2i.north west)-(0,.07)$)--($(u2i.north east)-(0,.07)$);
\node[draw](u2o)at(1,-1.3){$I_{U_2}^{V_2}$};\draw($(u2o.north west)-(0,.07)$)--($(u2o.north east)-(0,.07)$);
\draw[->]($(u1i.north)+(.3,0)$)--($(avg.south)-(.7,0)$);\draw[->]($(avg.north)-(.7,0)$)--($(u1o.south)+(.3,0)$);\draw[->]($(avg.south)+(.7,0)$)--($(u2o.north)-(.3,0)$);\draw[->]($(u2i.south)-(.3,0)$)--($(avg.north)+(.7,0)$);
\draw[->]($(u1i.north)-(.3,0)$)--($(u1o.south)-(.3,0)$);\draw[->]($(u2i.south)+(.3,0)$)--($(u2o.north)+(.3,0)$);
\node[font=\tiny]at(-.5,.75){$U_1$};\node[font=\tiny]at(.5,.8){$U_2$};\node[font=\tiny]at(-.5,-.8){$U_1$};\node[font=\tiny]at(.5,-.75){$U_2$};\node[font=\tiny,rotate=90]at(-1.55,0){$\Hom[H]{U_1}{V_1}$};\node[font=\tiny,rotate=-90]at(1.5,0){$\Hom[H]{U_2}{V_2}$};
\node[fill=white](v1o)at(-1,2.2){$V_1$};\node[fill=white](v1i)at(-1,-2.2){$V_1$};\node[fill=white](v2o)at(1,-2.2){$V_2$};\node[fill=white](v2i)at(1,2.2){$V_2$};
\draw[->](v1i)--(u1i);\draw[->](u1o)--(v1o);\draw[->](v2i)--(u2i);\draw[->](u2o)--(v2o);
\end{tikzpicture}
\end{equation*}
The middle coupon on the RHS acts on $U_1\otimes U_2^*$. Applying Lemma \ref{lmm:swap_irr} to it gives the desired equality.
\end{proof}

\subsubsection{Clashing orientations, duality, and Frobenius--Schur indicator}\label{sssec:FrobSch}

In this subsection we make preparations to deal with an extra complication specific to non-orientable orbifolds. Let $V_1,V_2$ be irreducible $G$-representations and $B\colon V_1\to V_2^*$ a $G$-equivariant isomorphism. Equivalently, $B(\bullet)(\bullet)$ is a non-degenerate $G$-invariant bilinear map $V_1\times V_2\to\mathbb{C}$. Following almost the same reasoning as in Lemma~\ref{lmm:swap_irr} we can prove the identity below.
\begin{equation}\label{eq:clashing_resolution}
\begin{tikzpicture}[baseline=.9]
\draw[dotted](0,0)circle(1.6);
\node[fill=white](li)at(-1,-1.2){$V_1$};\node[fill=white](lo)at(-1,1.2){$V_1$};\node[fill=white](ri)at(1,1.2){$V_2$};\node[fill=white](ro)at(1,-1.2){$V_2$};\node at(0,0){$\displaystyle\frac{1}{\#G}\sum_{g\in G}g$};
\draw(-1.4,.6)--(1.4,.6)--(1.4,-.6)--(-1.4,-.6)--(-1.4,.6)(-1.4,-.5)--(1.4,-.5);
\draw[->](li)--(-1,-.6);\draw[->](-1,.6)--(lo);\draw[->](1,.6)--(ri);\draw[->](ro)--(1,-.6);
\end{tikzpicture}
\qquad=\chi_V(1)^{-1}\times
\begin{tikzpicture}[baseline=.5]
\draw[dotted](0,0)circle(1.42);
\node[fill=white](li)at(-.9,-1.1){$V_1$};\node[fill=white](lo)at(-.9,1.1){$V_1$};\node[fill=white](ro)at(.9,1.1){$V_2$};\node[fill=white](ri)at(.9,-1.1){$V_2$};
\node[draw](bi)at(0,.5){$B^{-1}$};\draw($(bi.south east)-(.07,0)$)--($(bi.north east)-(.07,0)$);
\node[draw](b)at(0,-.5){$B$};\draw($(b.south west)+(.07,0)$)--($(b.north west)+(.07,0)$);
\draw[->](bi)to[out=180,in=-80](lo);\draw[->](bi)to[out=0,in=-100](ro);\draw[->](li)to[out=80,in=180](b);\draw[->](ri)to[out=100,in=0](b);
\end{tikzpicture}
\end{equation}
It will cause some coupons coloured $B$ to appear in the non-orientable case, and simplifying them will give rise to the Frobenius--Schur indicator defined below.
\begin{definition}\label{def:Frob_Sch}
The \emph{Frobenius--Schur indicator} of $V\in\tn{Irr}(G)$ is defined as
\begin{equation*}
\nu(V)=\begin{cases}\phantom{-}0&\text{if }V\ncong V^*\\\phantom{-}1&\text{if }V\text{ admits a }G\text{-invariant symmetric bilinear form}\\-1&\text{if }V\text{ admits a }G\text{-invariant alternating bilinear form}.\end{cases}
\end{equation*}
\end{definition}
Definition \ref{def:Frob_Sch} for a self-dual $V$ can be summarised in ribbon graph calculus as
\begin{equation}\label{eq:FS_Rad1}
\begin{tikzpicture}[baseline=-5]
\draw[dotted](0,-.2)circle(1);
\node[draw](b)at(0,0){$B$};\draw($(b.south west)+(.07,0)$)--($(b.north west)+(.07,0)$);
\node[fill=white](l)at(-.8,-.8){$V$};\node[fill=white](r)at(.8,-.8){$V$};
\draw[->](l)to[out=80,in=180](b);\draw[->](r)to[out=100,in=0](b);
\end{tikzpicture}
\quad=\quad
\begin{tikzpicture}[baseline=5]
\draw[dotted](0,.1)circle(1.2);
\node[draw](b)at(0,.8){$B$};\draw($(b.south east)-(.07,0)$)--($(b.north east)-(.07,0)$);
\node[fill=white](l)at(-.8,-.8){$V$};\node[fill=white](r)at(.8,-.8){$V$};
\draw(l)to[out=80,in=-160](0,0)to[out=20,in=-90](.6,.5);\draw[->](.6,.5)to[out=90,in=0](b);
\draw(r)to[out=100,in=-20](.25,-.1);\draw(-.25,.1)to[out=160,in=-90](-.6,.5);\draw[->](-.6,.5)to[out=90,in=180](b);
\end{tikzpicture}
\quad=\quad\nu(V)\times
\begin{tikzpicture}[baseline=-5]
\draw[dotted](0,-.2)circle(1);
\node[draw](b)at(0,0){$B$};\draw($(b.south east)-(.07,0)$)--($(b.north east)-(.07,0)$);
\node[fill=white](l)at(-.8,-.8){$V$};\node[fill=white](r)at(.8,-.8){$V$};
\draw[->](l)to[out=80,in=180](b);\draw[->](r)to[out=100,in=0](b);
\end{tikzpicture}
\end{equation}

\subsubsection{Summation over representations of an overgroup}\label{sssec:sum_overreps}

In the upcoming reasoning we will need the following lemma. Its proof is deferred to Section \ref{ssec:pf_change_sum_tors}. For a finite group $G$, denote by
\begin{align*}
\tau_m(G)\defeq&\#\{x\in G\mid x^m=1\}\\
\chi^{m\tn{-tors}}\defeq&\frac{1}{\tau_m(G)}\sum_{x^m=1}\chi(x)
\end{align*}
the number of elements of order dividing $m$ and the average value of a character $\chi$ on them respectively. Then
\begin{lemma}\label{lmm:change_sum_tors}
Let $G$ be a finite group, $U$ an irreducible representation of a subgroup $H\leq G$, and let $x\in G$ commute with $H$. Then for every $m\in\mathbb{N}$ we have
\begin{equation*}
\tau_m(G)\sum_{V\in\tn{Irr}(G)}\chi^{m\tn{-tors}}_V\cdot\tr\left(x\curvearrowright\Hom[H]{U}{V}\right)=[G:H]\sum_{\tiny\begin{matrix}y\in H\\y^m=x^m\end{matrix}}\chi_U(y).
\end{equation*}
\end{lemma}

\section{Extensibility of coverings from the 1-skeleton}\label{sec:extensibility}

In this section we investigate which coverings of the 1-skeleton $\Sigma^{(1)}$ extend across the faces to $\Sigma$, or equivalently which group homomorphisms $\pi_1\Sigma^{(1)}\to G$ from the free group on generators descend to $\pio\Sigma$. Concretely, we prove Proposition~\ref{prop:Rg_trace}, which can detect when this happens using the framework of ribbon categories and representation theory. As an application we give a proof of Mednykh's formula in Lemma~\ref{lmm:Mednykh}. This section does not yet depend on immersed subcomplexes.

The key objects, $\mathfrak{R}$-graphs, translate the combinatorics of $\Sigma$ into ribbon categories. Traces of $\mathfrak{R}$-graphs (Definition~\ref{def:ribbon_trace}) will appear in subsequent counting formulas, leading to the expression for the number of embeddings in Theorem~\ref{thm:count_tangles}.
\begin{definition}[$\mathfrak{R}$-graph]\label{def:R_graph}
Take the orbifold $\Sigma$ with a branched $\Delta$-complex structure. If it is not orientable, then fix some auxiliary orientation of each face. Suppose that each edge $e\in E(\Sigma)$ is decorated with $x_e\in\mathbb{C}[G]$, and each face $f\in F(\Sigma)$ with a $G$-representation $\lambda_f$. The \emph{$\mathfrak{R}$-graph} $\mathfrak{R}\left(\vec{\lambda},\vec{x}\right)$ is a ribbon graph obtained from $\Sigma$ by the following procedure.
\begin{enumerate}
\item\label{sdef:rg_corner} At each corner $c$ of a simplicial or orbigon face $f$, add a wire traversing $c$. Direct it according to the orientation of $f$ and colour with $\lambda_f$.
\item\label{sdef:rg_edge} At an edge $e\in E(\Sigma)$, add a coupon with the bottom base facing the initial end of $e$ and the top base facing the terminal end. Colour the coupon with $x_e$ (considered as an endomorphism of $\lambda_{f_1}\otimes\lambda_{f_2}$ or their duals, where $f_1,f_2$ are the faces adjacent to $e$).
\item\label{sdef:rg_spx_conn} For each edge separating two simplices, connect the four adjacent ends of corner wires (added in step~\ref*{sdef:rg_corner}) to the closest bases of the edge's coupon (from step~\ref*{sdef:rg_edge}).
\item\label{sdef:rg_orbi_conn} For each edge $e$ bordering a simplex $f_\tn{spx}$ and an orbigon $f_\tn{orb}$, connect the ends of $f_\tn{orb}$'s wire to the closest bases of $e$'s coupon, making a loop. Next, connect the bases of $e$'s coupon to the closest ends of the wires in the corners of $f_\tn{spx}$ adjacent to $e$.
\end{enumerate}
\end{definition}

The local picture from Definition \ref{def:R_graph} at edges is shown in Figure \ref{sfig:Rg_edge}. The part coming from a simplex and orbigon sharing an edge, described in step (\ref*{sdef:rg_orbi_conn}), is shown in Figure~\ref{sfig:Rg_orb}. Each simplex of $\Sigma$ gives rise to a path in the $\mathfrak{R}$-graph crossing three wires and three coupons, depicted in Figure \ref{sfig:Rg_spx}.

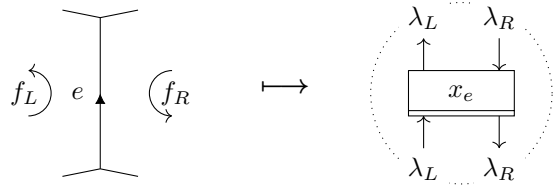
\begin{figure}[h]
\begin{subfigure}{\textwidth}
\begin{equation*}
\begin{tikzpicture}[baseline=0]
\draw[->-](0,-1)--(0,1);
\draw(-.5,1.1)--(0,1)--++(.5,.1);\draw(-.5,-1.1)--(0,-1)--++(.5,-.1);
\node at(-1,0){$f_L$};\node at(1,0){$f_R$};\node at(-.3,0){$e$};
\draw[->](-.95,-.3)arc(-90:90:.3);\draw[->](.95,.3)arc(90:270:.3);
\end{tikzpicture}
\qquad\boldsymbol{\longmapsto}\qquad
\begin{tikzpicture}[baseline=0]
\draw[dotted](0,0)circle(1.2);
\node(perm)at(0,0){$x_e$};\node[fill=white](lo)at(-.5,1){$\lambda_L$};\node[fill=white](li)at(-.5,-1){$\lambda_L$};\node[fill=white](ri)at(.5,1){$\lambda_R$};\node[fill=white](ro)at(.5,-1){$\lambda_R$};
\draw(-.7,.3)--(.7,.3)--(.7,-.3)--(-.7,-.3)--(-.7,.3)(-.7,-.23)--(.7,-.23);
\draw[->](li)--(-.5,-.3);\draw[->](-.5,.3)--(lo);\draw[->](ri)--(.5,.3);\draw[->](.5,-.3)--(ro);
\end{tikzpicture}
\end{equation*}
\caption{A part of an $\mathfrak{R}$-graph coming from $e\in E(\Sigma)$ incident to faces $f_L,f_R$. Orientation of $e$ matches the orientation of $f_L$ and is opposite to $f_R$. The coupon is coloured with $x_e$, considered as an endomorphism of $\lambda_L\otimes\lambda^*_R$. This is the result of steps (\ref*{sdef:rg_edge},\ref*{sdef:rg_spx_conn}) of Definition \ref{def:R_graph}. If $\Sigma$ is non-orientable, then at some edges the orientation of $f_L,f_R$ will both agree or disagree with that of $e$; then the wires $\lambda_L,\lambda_R$ on the RHS will have the same direction, and $x_e$ will act on $\lambda_L\otimes\lambda_R$ or $\lambda_L^*\otimes\lambda_R^*$.}
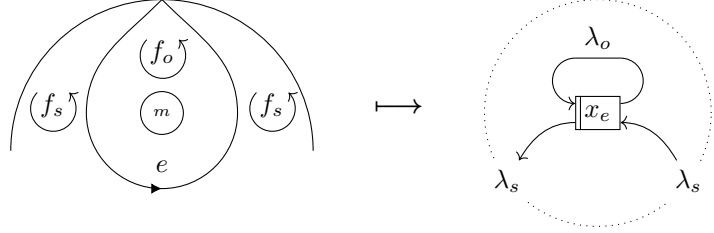
\label{sfig:Rg_edge}
\begin{subfigure}{\textwidth}
\begin{equation*}
\begin{tikzpicture}[baseline=0]
\node[circle,draw,font=\tiny]at(0,0){$m$};\node at(0,-.7){$e$};
\node at(0,.8){$f_o$};\draw[->](-.2,.97)arc(-225:45:.3);
\draw[->-](-1,0)arc(-180:0:1);\draw(0,1.5)to[in=90,out=-135](-1,0);\draw(0,1.5)to[in=90,out=-45](1,0);
\draw(-2,-.5)arc(180:0:2);
\node at(-1.45,.1){$f_s$};\draw[->](-1.65,.27)arc(-225:45:.3);
\node at(1.45,.1){$f_s$};\draw[->](1.25,.27)arc(-225:45:.3);
\end{tikzpicture}
\qquad\boldsymbol{\longmapsto}\qquad
\begin{tikzpicture}[baseline=0]
\draw[dotted](0,0)circle(1.5);
\node[draw](e)at(0,0){$x_e$};\draw($(e.south west)+(.07,0)$)--($(e.north west)+(.07,0)$);
\draw($(e.north east)-(0,.1)$)arc(-90:90:.3);\draw($(e.north west)+(0,.5)$)--($(e.north east)+(0,.5)$);\draw[->]($(e.north west)+(0,.5)$)arc(90:270:.3);
\node at(0,1){$\lambda_o$};\node[fill=white](o)at(-1.2,-.9){$\lambda_s$};\node[fill=white](i)at(1.2,-.9){$\lambda_s$};
\draw[->]($(e.south west)+(0,.1)$)to[out=180,in=60](o);\draw[->](i)to[out=120,in=0]($(e.south east)+(0,.1)$);
\end{tikzpicture}
\end{equation*}
\caption{A part of an $\mathfrak{R}$-graph coming from $m$-orbigon $f_o$ of $\Sigma$ bordering simplex $f_s$ across its edge $e=\partial f_o$. This is the result of steps (\ref*{sdef:rg_edge},\ref*{sdef:rg_orbi_conn}) of Definition \ref{def:R_graph}.}
\label{sfig:Rg_orb}
\end{subfigure}
\end{subfigure}
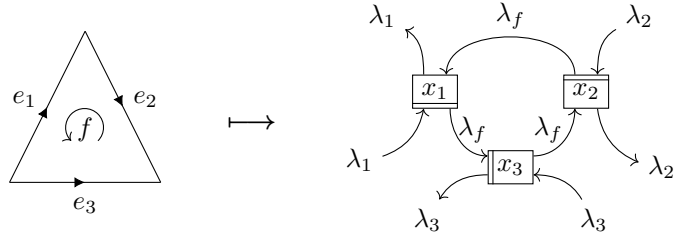
\begin{subfigure}{\textwidth}
\begin{equation*}
\begin{tikzpicture}[baseline=0]
\draw[->-](-1,-.7)--(0,1.3);\draw[->-](0,1.3)--(1,-.7);\draw[->-](-1,-.7)--(1,-.7);
\node at(0,0){$f$};\node at(-.8,.4){$e_1$};\node at(.8,.4){$e_2$};\node at(0,-1){$e_3$};
\draw[->](.17,-.15)arc(-45:225:.25);
\end{tikzpicture}
\qquad\boldsymbol{\longmapsto}\qquad
\begin{tikzpicture}[baseline=0]
\node[draw](x1)at(3,.5){$x_1$};\draw($(x1.south east)+(0,.07)$)--($(x1.south west)+(0,.07)$);
\node[draw](x2)at(5,.5){$x_2$};\draw($(x2.north east)-(0,.07)$)--($(x2.north west)-(0,.07)$);
\node[draw](x3)at(4,-.5){$x_3$};\draw($(x3.south west)+(.07,0)$)--($(x3.north west)+(.07,0)$);
\draw[->]($(x2.north)-(.15,0)$)to[out=90,in=90]($(x1.north)+(.15,0)$);\draw[->]($(x3.east)+(0,.15)$)to[out=0,in=-90]($(x2.south)-(.15,0)$);\draw[->]($(x1.south)+(.2,0)$)to[out=-90,in=180]($(x3.west)+(0,.15)$);
\node at(4,1.5){$\lambda_f$};\node at(4.5,0){$\lambda_f$};\node at(3.5,0){$\lambda_f$};
\node(f1o)at(2.3,1.5){$\lambda_1$};\node(f1i)at(2,-.4){$\lambda_1$};\draw[->]($(x1.north)-(.15,0)$)to[out=90,in=-30](f1o);\draw[->](f1i)to[bend right]($(x1.south)-(.15,0)$);
\node(f2i)at(5.7,1.5){$\lambda_2$};\node(f2o)at(6,-.5){$\lambda_2$};\draw[->]($(x2.south)+(.15,0)$)to[bend right](f2o);\draw[->](f2i)to[in=90,out=-150]($(x2.north)+(.15,0)$);
\node(f3o)at(2.9,-1.2){$\lambda_3$};\node(f3i)at(5.1,-1.2){$\lambda_3$};\draw[->](f3i)to[out=120,in=0]($(x3.east)-(0,.1)$);\draw[->]($(x3.west)-(0,.1)$)to[out=180,in=60](f3o);
\end{tikzpicture}
\end{equation*}
\caption{A part of an $\mathfrak{R}$-graph coming from a simplex $f$ of $\Sigma$. Across $e_1,e_2,e_3$ from $f$ lie faces $f_1,f_2,f_3$ respectively.}
\label{sfig:Rg_spx}
\end{subfigure}
\caption{Local pictures of an $\mathfrak{R}$-graph. We dropped a level of subscripts on $\lambda_{f_i}$ and $x_{e_i}$ for decluttering.}
\end{figure}

The $\mathfrak{R}$-graphs produced from example orbifolds are shown in Figure~\ref{fig:Rg_example}.
\begin{figure}[h]
\begin{subfigure}{\textwidth}
\centering
\begin{tikzpicture}
\node[draw,diamond,red](e1)at(1,1){$x_{e_1}$};\draw[red]($(e1.south)+(-.07,.07)$)--($(e1.east)+(-.07,.07)$);\draw[<-,red]($(e1.north)-(.1,.1)$)arc(225:45:.3)--($(e1.east)+(.32,.32)$)arc(45:-135:.3);\node at(1.9,1.9){\color{red}$\lambda_1$};
\node[draw,diamond,teal](e2)at(-1,1){$x_{e_2}$};\draw[teal]($(e2.east)-(.07,.07)$)--($(e2.north)-(.07,.07)$);\draw[<-,teal]($(e2.west)+(.1,-.1)$)arc(-45:-225:.3)--($(e2.north)+(-.32,.32)$)arc(135:-45:.3);\node[teal]at(-1.9,1.9){$\lambda_2$};
\node[draw,inner sep=6pt,blue](e3)at(0,-.5){$x_{e_3}$};\draw[blue]($(e3.south west)+(.07,0)$)--($(e3.north west)+(.07,0)$);\draw[<-,blue]($(e3.east)-(0,.15)$)arc(90:-90:.3)--($(e3.west)-(0,.75)$)arc(270:90:.3);\node[blue]at(-1,-1.1){$\lambda_3$};
\draw[->]($(e3.east)+(0,.15)$)to[out=0,in=-45]($(e1.south)+(.1,.1)$);\draw[->]($(e1.west)+(.1,.1)$)to[out=135,in=45]($(e2.east)+(-.1,.1)$);\draw[->]($(e2.south)+(-.1,.1)$)to[out=-135,in=180]($(e3.west)+(0,.15)$);
\node at(0,1.6){$\lambda_4$};\node at(-1.2,-.3){$\lambda_4$};\node at(1.3,-.3){$\lambda_4$};
\end{tikzpicture}
\caption{An $\mathfrak{R}$-graph associated to $S^2$ with three cone points from Figure \ref{sfig:orbi_S2}. We denoted the orbigons by ${\color{red}f_1},{\color{teal}f_2},{\color{blue}f_3}$ and the simplex by $f_4$. Note that the simplex $f_4$ is drawn in \ref{sfig:orbi_S2} ``from the back''.}
\end{subfigure}
\begin{subfigure}{\textwidth}
\centering
\begin{tikzpicture}
\node[draw,diamond,red](e1)at(2,.8){$x_{e_1}$};\draw[red]($(e1.south)+(-.07,.07)$)--($(e1.east)+(-.07,.07)$);\draw[->,red]($(e1.east)-(.1,.1)$)arc(-135:45:.3)--($(e1.north)+(.32,.32)$)arc(45:225:.3);\node[red]at(3.4,1){$\lambda_1$};
\node[draw,diamond,teal](e2)at(2,-.8){$x_{e_2}$};\draw[teal]($(e2.north)-(.07,.07)$)--($(e2.east)-(.07,.07)$);\draw[->,teal]($(e2.south)+(-.1,.1)$)arc(-225:-45:.3)--($(e2.east)+(.32,-.32)$)arc(-45:135:.3);\node[teal]at(3.3,-.8){$\lambda_2$};
\node[draw,diamond,blue](e3)at(-2,.8){$x_{e_3}$};\draw[blue]($(e3.east)-(.07,.07)$)--($(e3.north)-(.07,.07)$);\draw[blue]($(e3.south)+(-.1,.1)$)to[out=-135,in=0](-2.55,.1);\draw[->,blue](-2.75,.1)to[out=180,in=225]($(e3.north)+(-.25,.25)$)arc(135:-45:.25);\node[blue]at(-3,1.3){$\lambda_4$};
\node[draw](e4)at(0,0){$x_{e_4}$};
\draw[->]($(e1.west)+(.1,.1)$)to[out=135,in=90]($(e4.north)+(.2,0)$);\draw[->]($(e4.south)+(.2,0)$)to[out=-90,in=-135]($(e2.west)+(.1,-.1)$);\draw[->]($(e4.north)-(.2,0)$)to[out=90,in=45]($(e3.east)+(-.1,.1)$);\draw[->]($(e2.north)+(.1,-.1)$)to[out=45,in=-45]($(e1.south)+(.1,.1)$);
\draw[->]($(e3.west)+(.1,-.1)$)to[out=-135,in=135](-2.2,-.7)to[out=-45,in=-90]($(e4.south)-(.2,0)$);
\node at(-1.5,-.7){$\lambda_4$};\node at(-.5,1.3){$\lambda_4$};\node at(.5,1.3){$\lambda_3$};\node at(.5,-1.2){$\lambda_3$};\node at(2.5,0){$\lambda_3$};
\end{tikzpicture}
\caption{An $\mathfrak{R}$-graph obtained from the triangulation of $\mathbb{RP}^2$ with two cone points from Figure~\ref{sfig:orbi_rp2}. We denoted the two orbigons by ${\color{red}f_1},{\color{teal}f_2}$, the simplex sharing edges with them by $f_3$, and the other simplex (one forming the M\"obius band) by $f_4$. The orientation of $e_4$ was not specified, so we did not choose the bottom base of its coupon. The wire going from $x_{e_3}$ to itself comes from the corner of $f_4$ adjacent to ${\color{blue}e_3}$ on both sides. Note that both sides of ${\color{blue}e_3}$ are facing $f_4$ but with opposite orientations, so all wires at the coupon ${\color{blue}x_{e_3}}$ point in the same direction (as on the LHS of (\ref{eq:clashing_resolution})).}
\end{subfigure}
\begin{subfigure}{\textwidth}
\centering
\begin{tikzpicture}
\node[draw,diamond,red](b)at(0,0){$x_b$};\draw[red]($(b.south)+(-.07,.07)$)--($(b.east)+(-.07,.07)$);
\node[draw,teal](a1)at(1.5,.5){$x_{a_1}$};\draw[teal]($(a1.south west)+(0,.07)$)--($(a1.south east)+(0,.07)$);
\node[draw,blue](a2)at(.5,1.5){$x_{a_2}$};\draw[blue]($(a2.south west)+(.07,0)$)--($(a2.north west)+(.07,0)$);
\draw[->]($(b.east)-(.1,.1)$)to[out=-45,in=-90]($(a1.south)-(.2,0)$);\draw[->]($(a1.north)-(.2,0)$)to[out=90,in=0]($(a2.east)-(0,.15)$);\draw[->]($(a2.west)-(0,.15)$)to[out=180,in=135]($(b.north)-(.1,.1)$);
\node(xi)at(.4,-1.2){$X$};\node(xo)at(1.3,2.2){$X$};\draw[->](xi)to[out=90,in=-45]($(b.south)+(.1,.1)$);\draw[->]($(a2.east)+(0,.15)$)to[out=0,in=-90](xo);
\node(yi)at(2.5,1.2){$Y$};\node(yo)at(-1.5,.5){$Y$};\draw[->](yi)to[out=180,in=90]($(a1.north)+(.2,0)$);\draw[->]($(b.west)+(.1,.1)$)to[out=135,in=0](yo);
\node(zi)at(-.5,2.2){$Z$};\node(zo)at(2.5,-.2){$Z$};\draw[->](zi)to[out=-90,in=180]($(a2.west)+(0,.15)$);\draw[->]($(a1.south)+(.2,0)$)to[out=-90,in=180](zo);
\node at(.1,.9){$\lambda_2$};\node at(.8,.1){$\lambda_2$};\node at(1,1){$\lambda_2$};\node at(-1,.8){$\lambda_1$};\node at(.6,-.7){$\lambda_1$};\node at(2,-.4){$\lambda_1$};\node at(2,1.4){$\lambda_1$};\node at(-.5,1.6){$\lambda_1$};\node at(1.5,1.6){$\lambda_1$};
\end{tikzpicture}
\caption{An $\mathfrak{R}$-graph associated to the $\Delta$-complex structure on a torus from Figure~\ref{sfig:imm_cplx}. Pairs of ends labelled $X,Y,Z$ should be joined (so in the $\mathfrak{R}$-graph are three wires coloured with $\lambda_{f_1}$ and no free ends of wires).}
\end{subfigure}
\caption{Examples of $\mathfrak{R}$-graphs. Graph connectivity is determined by the branched $\Delta$-complex structure of $\Sigma$, but colours of wires and coupons depend on the extra data. For decluttering we abbreviated $\lambda_{f_i}$ by $\lambda_i$.}
\label{fig:Rg_example}
\end{figure}
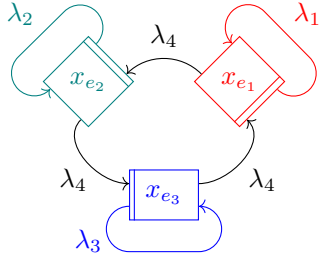
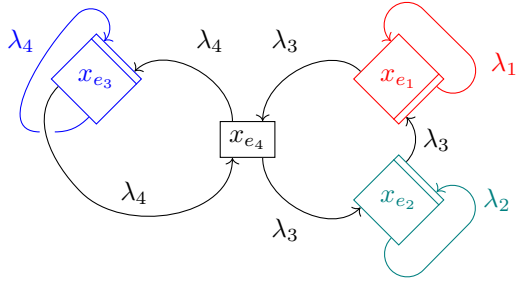
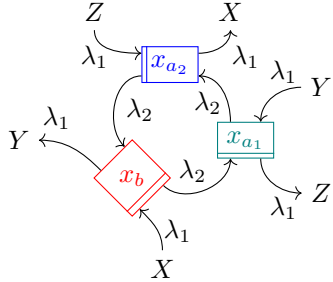

We are ready for the key result of this section.
\begin{proposition}\label{prop:Rg_trace}
Pick an element $x_e\in G$ of a finite group $G$ for each edge $e\in E(\Sigma)$, and consider them as living in $\mathbb{C}[G]$. Then
\begin{align*}
\sum_{\lambda_f\in\tn{Irr}(G)}&\prod_{f\in\Fspx(\Sigma)}\chi_{\lambda_f}(1)\cdot\prod_\iterorb\chi_{\lambda_f}^{m\tn{-tors}}\cdot\tr\mathfrak{R}\Big(\vec\lambda,\vec{x}\Big)=\\
&=\frac{\#G^{\#F(\Sigma)}}{\prod_{i=1}^k\tau_{m_i}(G)}\cdot\mathbbm{1}_{\{\vec{x}\text{ extends to a homomorphism }\pio\Sigma\to G\}},
\end{align*}
where $m_1,\dots,m_k$ are the orders of cone points of $\Sigma$, and $\tau_m,\chi^{m\tn{-tors}}$ are the number and average character of $m$-torsion elements from Section~\ref{sssec:sum_overreps}.
\end{proposition}
\begin{proof}
We begin by computing the trace of an $\mathfrak{R}$-graph. At each edge $e$ bordering faces $f_L,f_R$, where orientation of $e$ agrees with $f_L$ and is opposite to $f_R$, we can rewrite the coupons as
\begin{equation}\label{eq:prod_action_rg}
\begin{tikzpicture}[baseline=0]
\draw[dotted](0,0)circle(1.1);
\node at(0,0){$x_e$};\node[fill=white](lim)at(-.5,-1){$\lambda_L$};\node[fill=white](lom)at(-.5,1){$\lambda_L$};\node[fill=white](rom)at(.5,-1){$\lambda_R$};\node[fill=white](rim)at(.5,1){$\lambda_R$};
\draw(-.7,.3)--(.7,.3)--(.7,-.3)--(-.7,-.3)--(-.7,.3)(-.7,-.23)--(.7,-.23);
\draw[->](lim)--(-.5,-.3);\draw[->](-.5,.3)--(lom);\draw[->](rim)--(.5,.3);\draw[->](.5,-.3)--(rom);
\end{tikzpicture}
\qquad=\qquad
\begin{tikzpicture}[baseline=0]
\draw[dotted](.5,0)circle(1.1);
\node[fill=white](lis)at(0,-1){$\lambda_L$};\node[draw](opl)at(0,0){$x_e$};\draw($(opl.south west)+(0,.07)$)--($(opl.south east)+(0,.07)$);\node[fill=white](los)at(0,1){$\lambda_L$};
\node[fill=white](ris)at(1,1){$\lambda_R$};\node[draw](opr)at(1,0){$x_e$};\draw($(opr.south west)+(0,.07)$)--($(opr.south east)+(0,.07)$);\node[fill=white](ros)at(1,-1){$\lambda_R$};
\draw[->](lis)--(opl);\draw[->](opl)--(los);\draw[->](ris)--(opr);\draw[->](opr)--(ros);
\end{tikzpicture}
\qquad=\qquad
\begin{tikzpicture}[baseline=0]
\draw[dotted](.5,0)circle(1.1);
\node[fill=white](lis)at(0,-1){$\lambda_L$};\node[draw](opl)at(0,0){$x_e$};\draw($(opl.south west)+(0,.07)$)--($(opl.south east)+(0,.07)$);\node[fill=white](los)at(0,1){$\lambda_L$};
\node[fill=white](ris)at(1,1){$\lambda_R$};\node[draw](opr)at(1,0){$x^{-1}_e$};\draw($(opr.north west)-(0,.07)$)--($(opr.north east)-(0,.07)$);\node[fill=white](ros)at(1,-1){$\lambda_R$};
\draw[->](lis)--(opl);\draw[->](opl)--(los);\draw[->](ris)--(opr);\draw[->](opr)--(ros);
\end{tikzpicture}
\end{equation}
Equation (\ref{eq:prod_action_rg}) is nothing more than the definition of group action on a tensor product and dual representation, written in ribbon graphs. Analogous identities (with swapping top and bottom bases on either none or both coupons) apply with other combinations of orientations of $f_L,f_R$ at $e$.

We apply the identity (\ref{eq:prod_action_rg}) to $\tr\mathfrak{R}(\vec\lambda,\vec{x})$. This splits each edge coupon into two coupons, each with one incoming and one outgoing wire. Crucially, the resulting graph becomes a disjoint union of loops in 1-to-1 correspondence with faces of $\Sigma$. By Proposition \ref{prop:ribbon_functor}, trace of a disjoint union is the product of traces, so $\tr\mathfrak{R}(\vec\lambda,\vec{x})$ equals the product of traces of the resulting loops. A loop coming from a simplicial face $f$ has three coupons, which spell the relator $W_f$ from presentation (\ref{eq:simplicial_pres}). Therefore, the trace of this loop is $\chi_{\lambda_f}\left(W_f(\vec{x})\right)$. A loop coming from an $m$-orbigon $f$ has a single coupon coloured with $x_e$ coming from the incident edge $e$, so its trace is simply $\chi_{\lambda_f}(x_e)$. This gives us the key identity
\begin{equation}\label{eq:pure_Rgraph_trace}
\tr\mathfrak{R}\left(\vec\lambda,\vec{x}\right)=\prod{f\in\Fspx(\Sigma)}\chi_{\lambda_f}\left(W_f(\vec{x})\right)\cdot\prod_\iterorb\chi_{\lambda_f}(x_{\partial f}),
\end{equation}
where by $\partial f$ we mean the unique edge of orbigon $f$.

Summing equation (\ref{eq:pure_Rgraph_trace}) over $\vec\lambda$ with appropriate weight gives
\begin{align}\begin{split}\label{eq:summed_to_reg}
\sum_{\lambda_f\in\tn{Irr}(G)}&\prod_{f\in\Fspx(\Sigma)}\chi_{\lambda_f}(1)\cdot\hspace*{-10pt}\prod_\iterorb\chi_{\lambda_f}^{m\tn{-tors}}\cdot\tr\mathfrak{R}\left(\vec\lambda,\vec{x}\right)=\\
=&\prod_{f \in\Fspx(\Sigma)}\left(\sum_\lambda\chi_\lambda(1)\chi_\lambda\left(W_f(\vec{x})\right)\right)\cdot\hspace*{-10pt}\prod_\iterorb\left(\sum_\lambda\chi_\lambda^{m\tn{-tors}}\chi_\lambda(x_{\partial f})\right)
\end{split}\end{align}

We can recognise the sum involving a simplicial face $f$ as the trace of $W_f(\vec{x})$ in the regular representation of $G$. There, the identity of $G$ has trace $\#G$ and other elements are traceless, so
\begin{equation}\label{eq:reg_from_trRg}
\sum_\lambda\chi_\lambda(1)\chi_\lambda\left(W_f(\vec{x})\right)=\#G\cdot\mathbbm{1}_{[W_f(\vec{x})=1]}.
\end{equation}
The factors coming from orbigons can be computed using Lemma \ref{lmm:change_sum_tors}, by setting $H=\{1\}$ and $U$ to be the trivial representation, giving
\begin{equation}\label{eq:trRg_tors}
\sum_\lambda\chi_\lambda^{m\tn{-tors}}\chi_\lambda(x_e)=\frac{\#G}{\tau_m(G)}\cdot\mathbbm{1}_{[x_e^m=1]}.
\end{equation}

Equations (\ref{eq:reg_from_trRg}, \ref{eq:trRg_tors}) tell us that the RHS of equation (\ref{eq:summed_to_reg}) is $\tfrac{\#G^{\#F(\Sigma)}}{\prod_i\tau_{m_i}(G)}$, provided the words $W_f$ for every simplex $f$ evaluate to identity on $\vec{x}$, and $x_e^m=1$ for each edge $e$ incident to an $m$-orbigon. Otherwise, RHS of (\ref{eq:summed_to_reg}) evaluates to zero. But this condition is precisely saying that $\vec{x}$ satisfies all the relations of presentation (\ref{eq:simplicial_pres}), which is equivalent to $\vec{x}$ extending from a function between sets $E(\Sigma)\to G$ to a group homomorphism $\pio\Sigma\to G$.
\end{proof}

\subsection{Mednykh's formula and lattice TQFTs}\label{ssec:Mednykh}

Mednykh's formula \cite{determination_number_nonequiv_coverings_cpct_Riemann_surf_Mednykh78} is an expression for the number of homomorphisms out of a NEC group in terms of the representation theory of the target group. A quick proof of it follows by putting together our Proposition~\ref{prop:Rg_trace} and Lemma~\ref{lmm:swap_irr}. This strategy is similar to the approach in \cite{mednykh_formula_lattice_tqfts_Snyder17}, and as explained there, we are secretly computing the lattice topological quantum field theory invariant of $\Sigma$ attached to the group algebra $\mathbb{C}[G]$.
\begin{lemma}[Mednykh's formula]\label{lmm:Mednykh}
Let $\Sigma$ be an orbifold with $\cho<0$ and cone points of orders $m_1,\dots,m_k$, and let $G$ be a finite group. Then
\begin{equation*}
\#\Hom{\pio\Sigma}{G}=\#G^{1-\cht(\Sigma)}\prod_{i=1}^k\tau_{m_i}(G)\cdot\zeta^G,
\end{equation*}
where
\begin{equation*}
\zeta^G\defeq\sum_{V\in\tn{Irr}(G)}\chi_V(1)^{\cht(\Sigma)-k}\prod_{i=1}^k\chi_V^{m_i\tn{-tors}}\times\begin{cases}1&\text{if }\Sigma\text{ orientable}\\\nu(V)^g&\text{if }\Sigma\text{ non-orientable,}\end{cases}
\end{equation*}
$\nu(V)$ is the Frobenius--Schur indicator of $V$ from Definition~\ref{def:Frob_Sch}, and in the non-orientable case $g=2-\cht(\Sigma)$ is the genus of $\Sigma$.
\end{lemma}
Recall from Section~\ref{sec:simplicial} that the exponent $1-\cht(\Sigma)$ equals $2g-1$ if $\Sigma$ is orientable and $g-1$ if it is not. In the literature the role of the constant $\zeta^G$ from Lemma~\ref{lmm:Mednykh} is usually played by the Witten zeta function $s\mapsto\sum_{V\in\tn{Irr}(G)}\chi_V(1)^{-s}$.
\begin{proof}[Proof of Lemma \ref{lmm:Mednykh}]
Sum both sides of Proposition \ref{prop:Rg_trace} over all possible tuples $\vec{x}\in G^{E(\Sigma)}$. On the RHS we obtain
\begin{equation}\label{eq:pf_Mednykh_rhs}
\frac{\#G^{\#F(\Sigma)}}{\prod_{i=1}^k\tau_{m_i}(G)}\cdot\#\Hom{\pio\Sigma}{G}.
\end{equation}

The LHS becomes
\begin{align}\begin{split}\label{eq:pf_Mednykh_lhs}
\#G^{\#E(\Sigma)}\cdot\sum_{\lambda_f\in\tn{Irr}(G)}&\prod_{f\in\Fspx(\Sigma)}\chi_{\lambda_f}(1)\cdot\prod_\iterorb\chi_{\lambda_f}^{m\tn{-tors}}\times\\
&\times\tr\mathfrak{R}\left(\vec\lambda,\left(\tfrac{1}{\#G}\sum_{x\in G}x\right)_{e\in E(\Sigma)}\right),
\end{split}\end{align}
i.e. the coupon of an $\mathfrak{R}$-graph coming from each edge gets decorated with the average over the whole group. The prefactor $\#G^{\#E(\Sigma)}$ came from replacing sums with averages.

In the orientable case, Lemma \ref{lmm:swap_irr} allows us to compute the edge coupons. We see that whenever $\lambda_{f_1}\neq\lambda_{f_2}$ for any two neighbouring faces $f_1,f_2\in F(\Sigma)$, the corresponding coupon is coloured with the zero operator, so the $\mathfrak{R}$-graph vanishes. Otherwise, if $\lambda_f=\lambda$ is the same for all faces, then applying the Lemma~\ref{lmm:swap_irr} to $\mathfrak{R}$ will ``cut the coupons along the edges'', and the resulting graph will be a loop winding once around the vertex. Formally,
\begin{equation}\label{eq:pf_M_tr_samelambda}
\tr\mathfrak{R}\left((\lambda)_f,\left(\tfrac{1}{\#G}\sum_{x\in G}x\right)_{e\in E(\Sigma)}\right)=\chi_\lambda(1)^{-\#E(\Sigma)}\tr\left(\begin{tikzpicture}[baseline=0]\node[font=\tiny,inner sep=1pt](l)at(0,0){$\lambda$};\draw[->](l.south)arc(-180:0:.1)--($(l.north)+(.2,.05)$)arc(0:180:.1)--(l.north);\end{tikzpicture}\right)=\chi_\lambda(1)^{1-\#E(\Sigma)}.
\end{equation}

Substituting (\ref{eq:pf_M_tr_samelambda}) into the expression (\ref{eq:pf_Mednykh_lhs}), we see that it evaluates to
\begin{equation}\label{eq:pf_M_zeta}
\#G^{\#E(\Sigma)}\cdot\sum_{\lambda\in\tn{Irr}(G)}\chi_\lambda(1)^{\#\Fspx(\Sigma)+1-\#E(\Sigma)}\cdot\prod_\iterorb\chi_\lambda^{m\tn{-tors}}.
\end{equation}

Equating (\ref{eq:pf_Mednykh_rhs}) to (\ref{eq:pf_M_zeta}), and noting that $\#E(\Sigma)-\#F(\Sigma)=1-\cht(\Sigma)$ (recall $\#V(\Sigma)=1$ by the convenience convention from Section~\ref{sec:simplicial}) completes the proof in the orientable case.

It remains to finish the non-orientable case. We can still apply Lemma \ref{lmm:swap_irr} at edges where the neighbouring faces induce opposite orientations, but at edges where faces induce the same orientation we need to use equation (\ref{eq:clashing_resolution}) instead. After applying these two identities to the edge coupons, graphs coloured with non-self-dual $V$ vanish, and from the self-dual ones we again obtain a loop around the vertex. However, this time it contains copies of the configuration
\begin{equation*}
\begin{tikzpicture}
\node(i)at(-1,0){$V$};\node(o)at(3,0){$V$};\node at(.7,.2){$V$};
\node[draw](b)at(0,0){$B$};\draw($(b.south east)-(.07,0)$)--($(b.north east)-(.07,0)$);
\node[draw](bi)at(1.5,0){$B^{-1}$};\draw($(bi.south west)+(.07,0)$)--($(bi.north west)+(.07,0)$);
\draw[->](i)--(b);\draw[->](bi)--(b);\draw[->](bi)--(o);
\end{tikzpicture}
\end{equation*}
where $B\colon V\to V^*$ is the $G$-equivariant linear isomorphism. If we start with the triangulated $2g$-gon model for $\Sigma$, and orbigons inherit the orientation from their unique neighbouring simplices, then this configuration appears $g$ times. Equation (\ref{eq:FS_Rad1}) allows us to swap bottom and top bases in one of the coupons, accumulating a factor of $\nu(V)$, and then $B,B^{-1}$ cancel (as in Example~\ref{ex:ribbon_composition}). Formally,
\begin{equation*}
\begin{tikzpicture}[baseline=-10]
\node(b)[draw]at(0,0){$B$};\draw($(b.south east)-(.07,0)$)--($(b.north east)-(.07,0)$);
\node(i)[draw]at(1.1,0){$B^{-1}$};\draw($(i.south west)+(.07,0)$)--($(i.north west)+(.07,0)$);
\node at(.5,.2){$V$};\node at(2.1,.5){$g$};\node at(.5,-.8){$V$};
\draw[->](-.4,0)--(b);\draw[->](i)--(b);\draw[->](i)--(1.8,0);\draw[->](2,0)arc(90:-90:.3)--(-.6,-.6)arc(270:90:.3);
\draw[thick](-.4,.5)arc(90:180:.1)--(-.5,-.4)arc(180:270:.1);\draw[thick](1.8,.5)arc(90:0:.1)--(1.9,-.4)arc(0:-90:.1);
\end{tikzpicture}
\quad=\quad
\nu(V)^g\times
\begin{tikzpicture}[baseline=-10]
\node(b)[draw]at(0,0){$B$};\draw($(b.south west)+(.07,0)$)--($(b.north west)+(.07,0)$);
\node(i)[draw]at(1.1,0){$B^{-1}$};\draw($(i.south west)+(.07,0)$)--($(i.north west)+(.07,0)$);
\node at(.5,.2){$V$};\node at(2.1,.5){$g$};\node at(.5,-.8){$V$};
\draw[->](-.4,0)--(b);\draw[->](i)--(b);\draw[->](i)--(1.8,0);\draw[->](2,0)arc(90:-90:.3)--(-.6,-.6)arc(270:90:.3);
\draw[thick](-.4,.5)arc(90:180:.1)--(-.5,-.4)arc(180:270:.1);\draw[thick](1.8,.5)arc(90:0:.1)--(1.9,-.4)arc(0:-90:.1);
\end{tikzpicture}
\quad=\quad
\nu(V)^g\times
\begin{tikzpicture}[baseline=-10]
\node at(0,.2){$V$};\node at(.5,.5){$g$};\node at(0,-.8){$V$};
\draw[->](-.2,0)--(.2,0);\draw[->](.4,0)arc(90:-90:.3)--(-.4,-.6)arc(270:90:.3);
\draw[thick](-.2,.5)arc(90:180:.1)--(-.3,-.4)arc(180:270:.1);\draw[thick](.2,.5)arc(90:0:.1)--(.3,-.4)arc(0:-90:.1);
\end{tikzpicture}
\end{equation*}
This brings us back to having a single loop coloured $V$, as in equation (\ref{eq:pf_M_tr_samelambda}), and the rest of the proof proceeds as in the orientable case.
\end{proof}

\section{Background: representation theory of $S_n$}\label{sec:rep_Sn}

In the upcoming Sections~\ref{sec:counting} and \ref{sec:asymptotics} we will be performing certain calculations in irreducible representations of symmetric groups. Here we give a brief description of the necessary theory. We follow the approach to the representation theory of $S_n$ from \cite{new_approach_rep_theory_sym_gps_OkounkovVershik96}, where proofs and more details can be found. The basic objects are defined below.
\begin{definition}\label{def:rep_Sn}
A \emph{partition} of $n$ is a tuple $\lambda=(\lambda_1,\dots,\lambda_l)$ of positive integers satisfying $\lambda_1\geqslant\dots\geqslant\lambda_l$ and $\sum_{i=1}^l\lambda_i=n$. This is denoted by $\lambda\vdash n$.

For partitions $\mu=(\mu_1,\dots,\mu_k)\vdash m$ and $\lambda=(\lambda_1,\dots,\lambda_l)\vdash n$, we write $\mu\subseteq\lambda$ when $k\leqslant l$ and $\mu_i\leqslant\lambda_i$ for all $1\leqslant i\leqslant k$.

Partition $\lambda$ is pictorially represented by a \emph{Young diagram} $\tn{YD}_\lambda$. It consists of $l$ rows of boxes, $i$-th of which has length $\lambda_i$. Note that $\mu\subseteq\lambda$ if and only if $\tn{YD}_\mu\subseteq\tn{YD}_\lambda$, in the sense that every box of $\tn{YD}_\mu$ is contained in $\tn{YD}_\lambda$.

A \emph{Young skew-diagram} $\tn{YD}_{\lambda\setminus\mu}$ is the collection of those boxes of $\tn{YD}_\lambda$ which do not belong to $\tn{YD}_\mu$. Formally, it is a pair of partitions with $\mu\subseteq\lambda$.

A \emph{standard Young tableau} of shape $\tn{YD}_\lambda$ for $\lambda\vdash n$ is a bijection between the boxes of $\tn{YD}_\lambda$ and the set $[n]$, such that the numbers are strictly increasing in each row and down each column. A \emph{Young skew-tableau} of shape $\tn{YD}_{\lambda\setminus\mu}$ with $\mu\vdash m,\lambda\vdash n$ is a filling of boxes of $\tn{YD}_{\lambda\setminus\mu}$ with numbers from $[n]\setminus[m]$ (each appearing exactly once) that increases along rows and down columns. The set of Young tableaux of shape $\tn{YD}_\lambda$ (respectively, skew-tableaux of shape $\tn{YD}_{\lambda\setminus\mu}$) is denoted $\tn{Tab}(\lambda)$ (respectively $\tn{Tab}(\lambda\setminus\mu)$).

A skew-tableau $\tau_2\in\tn{Tab}(\lambda\setminus\mu)$ of shape $\tn{YD}_{\lambda\setminus\mu}$ can be stacked on a tableau $\tau_1\in\tn{Tab}(\mu)$ of shape $\tn{YD}_\mu$, giving a tableau of shape $\tn{YD}_\lambda$. We denote the result by $\tau_1\sqcup\tau_2\in\tn{Tab}(\lambda)$.

Given $i\in[n]$ and a Young tableau $\tau\in\tn{Tab}(\lambda)$ with $\lambda\vdash n$, the \emph{content} $c_\tau(i)$ is the difference between row and column coordinates of the box of $\tau$ containing $i$.
\end{definition}

These objects allow us to build an explicit model for the representation theory of $S_n$. The facts below are derived in \cite{new_approach_rep_theory_sym_gps_OkounkovVershik96}.
\begin{proposition}\label{prop:repSn_OV}
For a partition $\lambda\vdash n$, denote by $V_\lambda\defeq\mathbb{C}\cdot\tn{Tab}(\lambda)$ the linear span of standard Young tableaux of shape $\tn{YD}_\lambda$, and by $v_\tau\in V_\lambda$ the basis vector corresponding to $\tau\in\tn{Tab}(\lambda)$. Then $V_\lambda$ admits a linear action of $S_n$ that on transpositions of consecutive numbers is given by the formula
\begin{equation*}
(i,i+1).v_\tau=\begin{cases}\phantom{-}v_\tau&\text{if }i,i+1\text{ in the same row}\\-v_\tau&\text{if }i,i+1\text{ in the same column}\\rv_\tau+\sqrt{1-r^2}\:v_{\tau'}&\text{otherwise},\end{cases}
\end{equation*}
where $r=\frac{1}{c_\tau(i+1)-c_\tau(i)}$ and $\tau'$ is $\tau$ with $i,i+1$ swapped.

Furthermore, the assignment $\lambda\mapsto V_\lambda$ is a bijection $\{\lambda\vdash n\}\to\tn{Irr}(S_n)$. In other words, $\{V_\lambda\}_{\lambda\vdash n}$ forms a complete set of representatives for isomorphism classes of irreducible $S_n$-representations.

Similarly, let $V_{\lambda\setminus\mu}\defeq\mathbb{C}\cdot\tn{Tab}(\lambda\setminus\mu)=\langle v_\tau\mid\tau\in\tn{Tab}(\lambda\setminus\mu)\rangle$ be the linear span of Young skew-tableaux of shape $\tn{YD}_{\lambda\setminus\mu}$. Then, for $\mu\vdash m$ and $\lambda\vdash n$ we have
\begin{equation*}
\Hom[S_m]{V_\mu}{V_\lambda}=\begin{cases}V_{\lambda\setminus\mu}&\text{if }\mu\subseteq\lambda\\0&\text{if }\mu\nsubseteq\lambda.\end{cases}
\end{equation*}
In the case $\mu\subseteq\lambda$, evaluation of homomorphisms is by stacking Young skew-tableaux. Formally, we have $v_{\tau_2}(v_{\tau_1})=v_{\tau_1\sqcup\tau_2}\in V_\lambda$ for a basis vector $v_{\tau_1}\in V_\mu$ with $\tau_1\in\tn{Tab}(\mu)$ and a homomorphism $v_{\tau_2}\in V_{\lambda\setminus\mu}$ with $\tau_2\in\tn{Tab}(\lambda\setminus\mu)$.
\end{proposition}
The basis $\{v_\tau\mid\tau\in\tn{Tab}(\lambda)\}$ for $V_\lambda$ is called the \emph{Gelfand--Tsetlin basis}. From Proposition~\ref{prop:repSn_OV} we can see some further properties.
\begin{enumerate}
\item The diagrams consisting of a single row $\tn{YD}_{(n)}$ and a single column $\tn{YD}_{(1,\dots,1)}$ admit a unique standard Young tableau. Transpositions act trivially on the corresponding basis vector of $V_{(n)}$, so this is the trivial representation. Transpositions act on the vector spanning $V_{(1,\dots,1)}$ by $-1$, so this is the sign representation.
\item Let $\tn{YD}_{\lambda^\vee}$ be the result of swapping rows and columns of $\tn{YD}_\lambda$ (reflection in the diagonal), and $\lambda^\vee$ the corresponding \emph{transpose} partition. We claim that transposing the diagram corresponds to tensoring with the sign representation, that is $V_{\lambda^\vee}=\tn{sgn}\otimes V_\lambda=V_{(1,\dots,1)}\otimes V_\lambda$. Moreover, the intertwiner is given by
\begin{equation*}
\varphi\colon v_\tau\mapsto\tn{sgn}(\tau\circ\tau_0^{-1})\cdot v_{\tau^\vee}
\end{equation*}
for a reference $\tau_0\colon\{\text{boxes of }\tn{YD}_\lambda\}\to[n]$. Indeed, from the formula for action from Proposition~\ref{prop:repSn_OV} we see that conjugation by $\varphi$ multiplies the matrix of each $(i,i+1)$ by -1. This also applies to skew-diagrams, in the sense that $\tn{sgn}\otimes V_{\lambda\setminus\mu}\cong V_{(\lambda\setminus\mu)^\vee}$ considered as $S_{[m+1,n]}$-representations for $\mu\vdash m,\lambda\vdash n$.
\item Matrices of transpositions have real entries, so every representation of $S_n$ is defined over $\mathbb{R}$. Moreover, $V_\lambda$ admits a symmetric invariant bilinear form
\begin{minipage}{\textwidth}\begin{equation}\label{eq:repSn_bilinear_form}
B(v_{\tau_1},v_{\tau_2})=\begin{cases}1&\text{if }\tau_1=\tau_2\\0&\text{if }\tau_1\neq\tau_2\end{cases}
\end{equation}\end{minipage}
for $\tau_1,\tau_2\in\tn{Tab}(\lambda)$. These two statements are a way of saying that irreducible representations of $S_n$ have the Frobenius--Schur indicator (Definition~\ref{def:Frob_Sch}) equal 1.
\item If $S_m\times S_{[m+1,n]}$ is realised as subgroup of $S_n$ in the standard way as the setwise stabiliser of $[m]\subseteq[n]$, then restricted representations decompose as
\begin{minipage}{\textwidth}\begin{equation}\label{eq:branching_rule}
\tn{Res}_{S_m\times S_{[m+1,n]}}^{S_n}V_\lambda=\bigoplus_{\tiny\begin{matrix}\mu\vdash m\\\mu\subseteq\lambda\end{matrix}}V_\mu\otimes V_{\lambda\setminus\mu}.
\end{equation}\end{minipage}
The Cartesian product of groups and tensor product of representations are compatible, in the sense that $S_m$ affects the factor $V_\mu$ while $S_{[m+1,n]}$ moves $V_{\lambda\setminus\mu}$. In the notation from Section~\ref{ssec:sum_subgp}, for $\tau_1\in\tn{Tab}(\mu),\tau_2\in\tn{Tab}(\lambda\setminus\mu)$ we have $I_{V_\mu}^{V_\lambda}(v_{\tau_1}\otimes v_{\tau_2})=v_{\tau_1\sqcup\tau_2}$ and $P_{V_\mu}^{V_\lambda}(v_{\tau_1\sqcup\tau_2})=v_{\tau_1}\otimes v_{\tau_2}$; if the boxes of $\tau\in\tn{Tab}(\lambda)$ occupied by $[m]$ form a subdiagram of shape other than $\mu$ then $P_{V_\mu}^{V_\lambda}(v_\tau)=0$.
\end{enumerate}

The dimension of an irreducible representation $V_\lambda$ is given by the well-known hook length formula
\begin{equation}\label{eq:hooklen}
\chi_{V_\lambda}(1)=\#\tn{Tab}(\lambda)=\frac{n!}{\prod_{b\in\tn{YD}_\lambda}h(b)},
\end{equation}
where $h(b)$ is the number of boxes of $\tn{YD}_\lambda$ below (in the same column) or to the right (in the same row) of the box $b$ (inclusive).

\medskip

Let us walk through the theory in the simplest nontrivial example.
\begin{example}
The Young diagram of the partition $(n-1,1)\vdash n$ consists of a row of 1 box under a row of $n-1$ boxes. Pictorially
\begin{equation*}
\tn{YD}_{(n-1,1)}\qquad=\qquad
\begin{tikzpicture}[baseline=0]
\draw[decorate,decoration={brace},thick](-.8,.9)--(3.2,.9);
\draw(0,.8)--(0,-.8)--(-.8,-.8)--(-.8,.8)--(1,.8)(-.8,0)--(1,0)(.8,0)--(.8,.8)(2.2,.8)--(3.2,.8)--(3.2,0)--(2.2,0)(2.4,0)--(2.4,.8);
\node at(1.6,.4){$\dots$};\node[font=\tiny]at(1.2,1.2){$n-1$ boxes};
\end{tikzpicture}
\end{equation*}

The standard Young tableaux of shape $\tn{YD}_{(n-1,1)}$ are given by $\tn{Tab}((n-1,1))=\{\tau_i\mid2\leqslant i\leqslant n\}$, where
\begin{equation*}
\tau_i\qquad\defeq\qquad
\begin{tikzpicture}[baseline=0]
\draw[decorate,decoration={brace},thick](-.8,.9)--(1.56,.9);\node[font=\tiny]at(.38,1.2){$i-1$ boxes};
\draw[decorate,decoration={brace},thick](1.64,.9)--(3.2,.9);\node[font=\tiny]at(2.42,1.2){$n-i$ boxes};
\draw(0,.8)--(0,-.8)--(-.8,-.8)--(-.8,.8)--(3.2,.8)--(3.2,0)--(-.8,0)(.8,0)--(.8,.8)(1.6,0)--(1.6,.8)(2.4,0)--(2.4,.8);
\node at(-.4,.4){1};\node at(-.4,-.4){$i$};\node at(.4,.4){$\dots$};\node at(1.2,.4){$i\!-\!1$};\node at(2,.4){$i\!+\!1$};\node at(2.8,.4){$\dots$};
\end{tikzpicture}
\end{equation*}
The contents in $\tau_i$ are $c(1)=0$, $c(i)=-1$, $c(i-1)=i-2$, and $c(i+1)=i-1$.

The group $S_n$ acts on $\mathbb{C}^n$ by permuting the axes, i.e. $\sigma.e_i=e_{\sigma(i)}$ for $\sigma\in S_n$ and a basis vector $e_i\in\mathbb{C}^n$. The \emph{standard representation} of $S_n$ is the subspace $\tn{std}_n\defeq\{\vec{u}\in\mathbb{C}^n\mid\sum_iu_i=0\}\leq\mathbb{C}^n$. We claim that $V_{(n-1,1)}$ and $\tn{std}_n$ are isomorphic as $S_n$-representations. Indeed, one can check that
\begin{equation}\label{eq:std_rep_models}
v_{\tau_i}\quad\mapsto\quad\sqrt{\frac{i-1}{i}}\left(e_i-\frac{1}{i-1}\sum_{1\leqslant j<i}e_j\right)
\end{equation}
is equivariant under transpositions, so defines (nonzero) intertwiner $V_{(n\!-\!1,1)}\!\to\!\tn{std}_n$.

Let $\tau',\tau''$ be the skew-tableaux of shapes $\tn{YD}_{(n-1,1)\setminus(n-1)}$ and $\tn{YD}_{(n-1,1)\setminus(n-2,1)}$
\begin{equation*}
\tau'\ \defeq\
\begin{tikzpicture}[baseline=0]
\draw(0,0)--(0,-.8)--(-.8,-.8)--(-.8,0)--(0,0);\node at(-.4,-.4){$n$};
\draw[dashed](-.8,0)--(-.8,.8)--(1,.8)(0,.8)--(0,0)--(1,0)(.8,0)--(.8,.8)(2.2,.8)--(3.2,.8)--(3.2,0)--(2.2,0)(2.4,0)--(2.4,.8);\node at(1.6,.4){$\dots$};
\draw[decorate,decoration={brace},thick](-.8,.9)--(3.2,.9);\node[font=\tiny]at(1.2,1.2){$n-1$ phantom boxes};
\end{tikzpicture}
\qquad\qquad\tau''\ \defeq\
\begin{tikzpicture}[baseline=0]
\draw[dashed](.2,.8)--(-.8,.8)--(-.8,-.8)--(0,-.8)--(0,.8)(-.8,0)--(.2,0)(1.4,.8)--(2.4,.8)(1.4,0)--(2.4,.0)(1.6,0)--(1.6,.8);\node at(.8,.4){$\dots$};
\draw(2.4,0)--(3.2,0)--(3.2,.8)--(2.4,.8)--(2.4,0);\node at(2.8,.4){$n$};
\draw[decorate,decoration={brace},thick](-.8,.9)--(2.4,.9);\node[font=\tiny]at(.8,1.2){$n-2$ phantom boxes};
\end{tikzpicture}
\end{equation*}
They are unique, so $V_{(n-1,1)\setminus(n-1)}=\langle v_{\tau'}\rangle$ and $V_{(n-1,1)\setminus(n-2,1)}=\langle v_{\tau''}\rangle$ are 1-dimensional.

Standard Young tableaux of shape $\tn{YD}_{(n-1,1)}$ can be partitioned according to the position of $n$ (either in the sole box of the second row or at the end of first row) as
\begin{equation*}
\tn{Tab}\big((n-1,1)\big)=\left\{\tau\sqcup\tau'\ \middle|\ \tau\in\tn{Tab}\big((n-1)\big)\right\}\sqcup\left\{\tau\sqcup\tau''\ \middle|\ \tau\in\tn{Tab}\big((n-2,1)\big)\right\}.
\end{equation*}
Taking linear spans of both sides we obtain an $S_{n-1}$-invariant decomposition
\begin{equation}\label{eq:std_rep_decomp}
V_{(n-1,1)}=V_{(n-1)}\otimes V_{(n-1,1)\setminus(n-1)}\oplus V_{(n-2,1)}\otimes V_{(n-1,1)\setminus(n-2,1)}
\end{equation}
as linear spaces. This is the same as the decomposition of restricted representation
\begin{equation*}
\tn{Res}^{S_n}_{S_{n-1}}\tn{std}_n=\mathbb{C}\oplus\tn{std}_{n-1}
\end{equation*}
into irreducibles. Indeed, we said earlier that $V_{(n-1)}$ is the trivial representation, and the isomorphism (\ref{eq:std_rep_models}) sends $v_{\tau\sqcup\tau'}$ for the unique $\tau\in\tn{Tab}((n-1))$ to the $S_{n-1}$-invariant vector $\sqrt{\tfrac{n-1}{n}}\left(e_n-\tfrac{1}{n-1}\sum_{1\leqslant i<n}e_i\right)$. Similarly, we said already that $V_{(n-2,1)}$ is our model for $\tn{std}_{n-1}$. Accordingly, the isomorphism from equation (\ref{eq:std_rep_models}) sends the second summand $\left\langle v_{\tau\sqcup\tau''}\ \middle|\ \tau\in\tn{Tab}\big((n-2,1)\big)\right\rangle$ in decomposition (\ref{eq:std_rep_decomp}) to the subspace $\tn{std}_n\cap\langle e_1,\dots,e_{n-1}\rangle$, which is the copy of $\tn{std}_{n-1}$ in $\tn{std}_n$.
\end{example}

For decluttering, from now on a partition $\lambda$, Young diagram $\tn{YD}_\lambda$, and representation $V_\lambda$ will all be denoted $\lambda$ by abuse of notation. Similarly, a skew-diagram $\tn{YD}_{\lambda\setminus\mu}$ and the space of homomorphisms $V_{\lambda\setminus\mu}$ will be denoted simply by $\lambda\setminus\mu$. We will denote the Gelfand--Tsetlin basis by $\{v\in\tn{Tab}(\lambda)\}$, and write $u\sqcup v$ both for the stacked tableaux and its basis vector. We will sometimes denote the number $n$ by $|\lambda|$, if the emphasis is on the partition $\lambda\vdash n$.

\section{A formula for the number of embeddings}\label{sec:counting}

This section forms the core of this article, and is dedicated to proving Theorem~\ref{thm:count_tangles}, a formula for the number of embeddings of a fixed immersed complex into all coverings of $\Sigma$ of a given degree. The strategy is to combine Proposition~\ref{prop:Rg_trace} with Lemma~\ref{lmm:swap}, similarly to the proof of Mednykh's formula (Lemma~\ref{lmm:Mednykh}) in Section~\ref{ssec:Mednykh}. The main difference is that instead of summing over all elements of $S_n$ at each edge, this time the summation will be restricted based on the combinatorics of the immersion $T\looparrowright\Sigma$.

To state Theorem~\ref{thm:count_tangles} we need two ingredients. The first one, $\mathfrak{P}$-graphs, translate the combinatorial data of the triangulation and immersion into ribbon categories. They evolve from $\mathfrak{R}$-graphs (Definition~\ref{def:R_graph}) after performing the summation at edge coupons. Their traces enter the subsequent counting formulas, and there play a role analogous to $\Upsilon_n$ in \cite[Proposition 5.8]{asymptotic_stats_random_covering_surfaces_MageePuder23}. Below and in Figure~\ref{fig:Pg}, the operators $P,I$ are as discussed around equation (\ref{eq:isotypic}) in Section~\ref{ssec:sum_subgp}.
\begin{definition}[$\mathfrak{P}$-graph]\label{def:P_graph}
Take an orbifold $\Sigma$ with a standard branched $\Delta$-complex structure and with the following decorations.
\begin{itemize}
\item Each edge $e\in E(\Sigma)$ is equipped with a Young diagram or skew-diagram $\mu_e$.
\item Each face $f\in F(\Sigma)$ is equipped with a diagram or skew-diagram $\lambda_f$.
\item Each corner $c$ of a face $f$ has an assigned permutation $\gamma_c$, whose support is suitable to act on $\lambda_f$.
\end{itemize}
The \emph{$\mathfrak{P}$-graph} $\mathfrak{P}(\vec\lambda,\vec\mu,\vec\gamma)$ is the ribbon graph obtained from $\Sigma$ by the following procedure.
\begin{enumerate}
\item\label{sdef:Pg_corner} For each corner $c$ of a face $f$, add a coupon coloured $\gamma_c$ with one incoming and one outgoing wire coloured $\lambda_f$. Let the direction from bottom to top base of the coupon, as well as the orientation of the wires, agree with the orientation of $f$. This is the same as the configuration from Figure~\ref{sfig:operator} placed at $c$.
\item For each edge whose neighbouring faces induce opposite orientations on it, add four coupons and eight wires with orientations and colours as in Figure \ref{sfig:Pg_edge}. This is the same configuration as on the RHS of Lemma \ref{lmm:swap} expressed in the notation of ribbon graphs.
\item\label{sdef:Pg_e_no} If $\Sigma$ is not orientable, then at some edges the orientations induced by the neighbouring faces will agree. If $e$ is such an edge, then replace it with a configuration of six coupons and ten wires shown in Figure \ref{sfig:Pg_e_no}. This is the same as the RHS of Lemma \ref{lmm:swap} adjusted with identity (\ref{eq:clashing_resolution}).
\item Connect the free ends of wires in adjacent edge and corner pieces.
\end{enumerate}
\end{definition}
Each simplicial face $f$ of $\Sigma$ gives rise to a cycle in the $\mathfrak{P}$-graph that traverses six wires of colour $\lambda_f$, three coupons at the corners of $f$, three coupons of type $P^{\lambda_f}_\mu$ and three of type $I^{\lambda_f}_\mu$, and three wires of colours $\lambda_f\setminus\mu_e$. This configuration is shown in Figure \ref{sfig:Pg_spx}. An orbigon $f$ incident to edge $e$ gives rise to a loop traversing two wires coloured $\lambda_f$ and one coloured $\lambda_f\setminus\mu_e$, and three coupons coloured $\gamma_c,P_{\mu_e}^{\lambda_f},I_{\mu_e}^{\lambda_f}$, as shown in Figure~\ref{sfig:Pg_orb}.
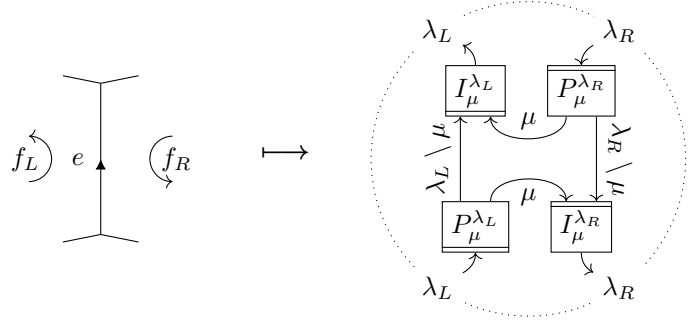
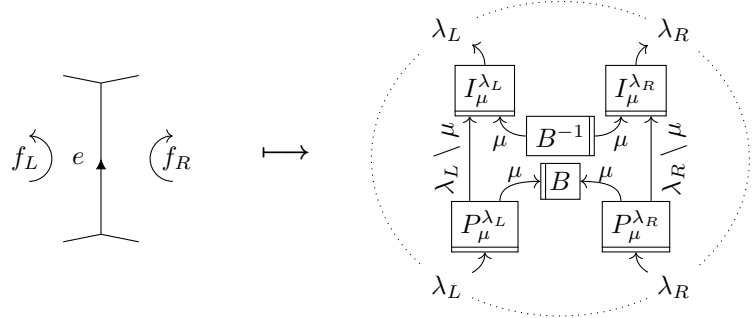
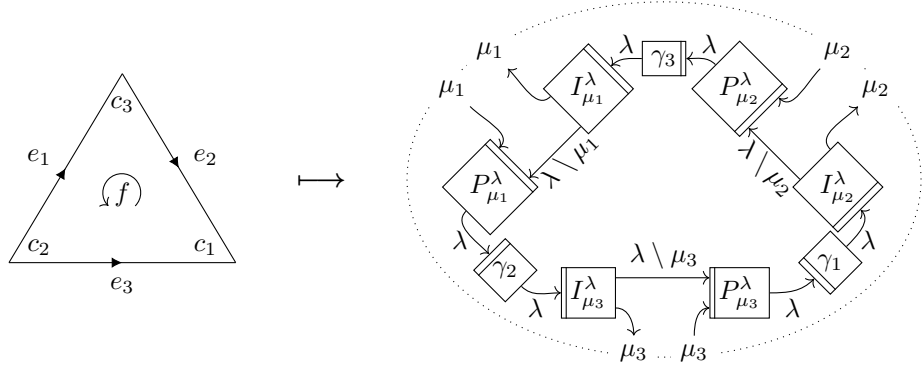
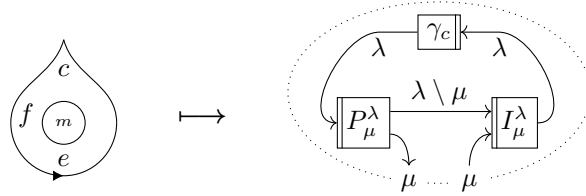
\begin{figure}
\begin{subfigure}{\textwidth}
\begin{equation*}
\begin{tikzpicture}[baseline=0]
\draw[->-](0,-1)--(0,1);
\draw(-.5,1.1)--(0,1)--++(.5,.1);\draw(-.5,-1.1)--(0,-1)--++(.5,-.1);
\node at(-1,0){$f_L$};\node at(1,0){$f_R$};\node at(-.3,0){$e$};
\draw[->](-.95,-.3)arc(-90:90:.3);\draw[->](.95,.3)arc(90:270:.3);
\end{tikzpicture}
\qquad\boldsymbol{\longmapsto}\qquad
\begin{tikzpicture}[baseline=0]
\draw[dotted](0,0)circle(2.08);
\node[draw](il)at(-.7,.9){$I_\mu^{\lambda_L}$};\draw($(il.south east)+(0,.07)$)--($(il.south west)+(0,.07)$);
\node[draw](pl)at(-.7,-.9){$P_\mu^{\lambda_L}$};\draw($(pl.south east)+(0,.07)$)--($(pl.south west)+(0,.07)$);
\node[draw](pr)at(.7,.9){$P_\mu^{\lambda_R}$};\draw($(pr.north east)-(0,.07)$)--($(pr.north west)-(0,.07)$);
\node[draw](ir)at(.7,-.9){$I_\mu^{\lambda_R}$};\draw($(ir.north east)-(0,.07)$)--($(ir.north west)-(0,.07)$);
\draw[->]($(pl.north)-(.2,0)$)--($(il.south)-(.2,0)$);\draw[->]($(pr.south)+(.2,0)$)--($(ir.north)+(.2,0)$);\draw[->]($(pr.south)-(.2,0)$)to[out=-90,in=-90]($(il.south)+(.2,0)$);\draw[->]($(pl.north)+(.2,0)$)to[out=90,in=90]($(ir.north)-(.2,0)$);
\node[rotate=90]at(-1.2,0){$\lambda_L\setminus\mu$};\node[rotate=-90]at(1.2,0){$\lambda_R\setminus\mu$};\node at(0,.5){$\mu$};\node at(0,-.5){$\mu$};
\node[fill=white](li)at(-1.2,-1.7){$\lambda_L$};\node[fill=white](lo)at(-1.2,1.7){$\lambda_L$};\node[fill=white](ri)at(1.2,1.7){$\lambda_R$};\node[fill=white](ro)at(1.2,-1.7){$\lambda_R$};
\draw[->](li)to[out=30,in=-90](pl);\draw[->](il)to[out=90,in=-30](lo);\draw[->](ri)to[out=210,in=90](pr);\draw[->](ir)to[out=-90,in=150](ro);
\end{tikzpicture}
\end{equation*}
\caption{A part of a $\mathfrak{P}$-graph coming from an edge $e$ of $\Sigma$. The four wires coloured $\lambda_L,\lambda_R$ connect to corner pieces added in step (\ref*{sdef:Pg_corner}) of Definition \ref{def:P_graph}.}
\label{sfig:Pg_edge}
\end{subfigure}
\begin{subfigure}{\textwidth}
\begin{equation*}
\begin{tikzpicture}[baseline=0]
\draw[->-](0,-1)--(0,1);
\draw(-.5,1.1)--(0,1)--++(.5,.1);\draw(-.5,-1.1)--(0,-1)--++(.5,-.1);
\node at(-1,0){$f_L$};\node at(1,0){$f_R$};\node at(-.3,0){$e$};
\draw[->](-.95,-.3)arc(-90:90:.3);\draw[->](.95,-.3)arc(270:90:.3);
\end{tikzpicture}
\qquad\boldsymbol{\longmapsto}\qquad
\begin{tikzpicture}[baseline=0]
\draw[dotted](0,0)ellipse(2.5 and 2.08);
\node[draw](il)at(-1,.9){$I_\mu^{\lambda_L}$};\draw($(il.south east)+(0,.07)$)--($(il.south west)+(0,.07)$);
\node[draw](pl)at(-1,-.9){$P_\mu^{\lambda_L}$};\draw($(pl.south east)+(0,.07)$)--($(pl.south west)+(0,.07)$);
\node[draw](pr)at(1,-.9){$P_\mu^{\lambda_R}$};\draw($(pr.south east)+(0,.07)$)--($(pr.south west)+(0,.07)$);
\node[draw](ir)at(1,.9){$I_\mu^{\lambda_R}$};\draw($(ir.south east)+(0,.07)$)--($(ir.south west)+(0,.07)$);
\node[draw](b)at(0,-.3){$B$};\draw($(b.south west)+(.07,0)$)--($(b.north west)+(.07,0)$);
\node[draw](bi)at(0,.3){$B^{-1}$};\draw($(bi.south east)-(.07,0)$)--($(bi.north east)-(.07,0)$);
\draw[->]($(pl.north)-(.2,0)$)--($(il.south)-(.2,0)$);\draw[->]($(pr.north)+(.2,0)$)--($(ir.south)+(.2,0)$);
\draw[->]($(pl.north)+(.2,0)$)to[out=90,in=180](b.west);\draw[->]($(pr.north)-(.2,0)$)to[out=90,in=0](b.east);\draw[->](bi.west)to[out=180,in=-90]($(il.south)+(.2,0)$);\draw[->](bi.east)to[out=0,in=-90]($(ir.south)-(.2,0)$);
\node[rotate=90]at(-1.5,0){$\lambda_L\setminus\mu$};\node[rotate=90]at(1.5,0){$\lambda_R\setminus\mu$};\node[font=\small]at(-.6,-.2){$\mu$};\node[font=\small]at(-.8,.2){$\mu$};\node[font=\small]at(.6,-.2){$\mu$};\node[font=\small]at(.8,.2){$\mu$};
\node[fill=white](li)at(-1.5,-1.7){$\lambda_L$};\node[fill=white](lo)at(-1.5,1.7){$\lambda_L$};\node[fill=white](ri)at(1.5,-1.7){$\lambda_R$};\node[fill=white](ro)at(1.5,1.7){$\lambda_R$};
\draw[->](li)to[out=30,in=-90](pl);\draw[->](il)to[out=90,in=-30](lo);\draw[->](ri)to[out=150,in=-90](pr);\draw[->](ir)to[out=90,in=210](ro);
\end{tikzpicture}
\end{equation*}
\caption{A part of a $\mathfrak{P}$-graph coming from an edge $e$ on which both neighbouring simplicial faces $f_L,f_R$ induce the same orientation. This situation occurs when $\Sigma$ is not orientable and step (\ref*{sdef:Pg_e_no}) of Definition \ref{def:P_graph} is executed. Here $B$ is the $S_{|\mu|}$-equivariant map $\mu\to\mu^*$ from Section~\ref{sssec:FrobSch}, or zero if $\mu$ is not self-dual.}
\label{sfig:Pg_e_no}
\end{subfigure}
\begin{subfigure}{\textwidth}
\begin{equation*}
\begin{tikzpicture}[baseline=0]
\draw[->-](-1.5,-1)--(0,1.5);\draw[->-](0,1.5)--(1.5,-1);\draw[->-](-1.5,-1)--(1.5,-1);
\node at(0,-.1){$f$};\node at(-1.1,.4){$e_1$};\node at(1.1,.4){$e_2$};\node at(0,-1.3){$e_3$};
\draw[->](.17,-.25)arc(-45:225:.25);
\node at(-1.1,-.8){$c_2$};\node at(1.1,-.8){$c_1$};\node at(0,1.1){$c_3$};
\end{tikzpicture}
\qquad\boldsymbol{\longmapsto}\qquad
\begin{tikzpicture}[baseline=0]
\draw[dotted](0,.1)ellipse(3.4 and 2.35);
\node[diamond,draw,inner sep=2pt](c2)at(-2.1,-1.1){$\gamma_2$};\draw($(c2.north)+(.07,-.07)$)--($(c2.west)+(.07,-.07)$);
\node[draw](c3)at(0,1.7){$\gamma_3$};\draw($(c3.north east)-(.07,0)$)--($(c3.south east)-(.07,0)$);
\node[diamond,draw,inner sep=2pt](c1)at(2.2,-1){$\gamma_1$};\draw($(c1.west)+(.07,.07)$)--($(c1.south)+(.07,.07)$);
\node[diamond,draw,inner sep=2pt](p1)at(-2.3,0){$P_{\mu_1}^\lambda$};\draw($(p1.north)-(.07,.07)$)--($(p1.east)-(.07,.07)$);
\node[diamond,draw,inner sep=2pt](i1)at(-1,1.3){$I_{\mu_1}^\lambda$};\draw($(i1.north)-(.07,.07)$)--($(i1.east)-(.07,.07)$);
\node[diamond,draw,inner sep=2pt](p2)at(1,1.3){$P_{\mu_2}^\lambda$};\draw($(p2.east)+(-.07,.07)$)--($(p2.south)+(-.07,.07)$);
\node[diamond,draw,inner sep=2pt](i2)at(2.3,0){$I_{\mu_2}^\lambda$};\draw($(i2.east)+(-.07,.07)$)--($(i2.south)+(-.07,.07)$);
\node[draw](i3)at(-1,-1.4){$I_{\mu_3}^\lambda$};\draw($(i3.south west)+(.07,0)$)--($(i3.north west)+(.07,0)$);
\node[draw](p3)at(1,-1.4){$P_{\mu_3}^\lambda$};\draw($(p3.south west)+(.07,0)$)--($(p3.north west)+(.07,0)$);
\draw[->](p1)to[out=-135,in=135](c2);\draw[->](c3)to[out=180,in=45](i1);\draw[->](p2)to[out=135,in=0](c3);\draw[->](c1)to[out=45,in=-45](i2);\draw[->](p3)to[out=0,in=-135](c1);\draw[->](c2)to[out=-45,in=180](i3);
\draw[->]($(i1.south west)+(.2,-.2)$)--($(p1.north east)+(.2,-.2)$);\draw[->]($(i2.north west)-(.2,.2)$)--($(p2.south east)-(.2,.2)$);\draw[->]($(i3.east)+(0,.2)$)--($(p3.west)+(0,.2)$);
\node[rotate=45]at(-1.25,.35){$\lambda\setminus\mu_1$};\node[rotate=-45]at(1.35,.25){$\lambda\setminus\mu_2$};\node at(0,-.9){$\lambda\setminus\mu_3$};
\node[fill=white](1i)at($(p1)+(-.5,1.3)$){$\mu_1$};\node[fill=white](1o)at($(i1)+(-1.3,.5)$){$\mu_1$};\node[fill=white](2i)at($(p2)+(1.3,.5)$){$\mu_2$};\node[fill=white](2o)at($(i2)+(.5,1.3)$){$\mu_2$};\node[fill=white](3o)at(-.4,-2.2){$\mu_3$};\node[fill=white](3i)at(.4,-2.2){$\mu_3$};
\draw[->](1i)to[out=-45,in=45]($(p1.north east)+(-.2,.2)$);\draw[->]($(i1.south west)+(-.2,.2)$)to[out=-135,in=-45](1o);\draw[->](2i)to[out=-135,in=-45]($(p2.south east)+(.2,.2)$);\draw[->]($(i2.north west)+(.2,.2)$)to[out=135,in=-135](2o);\draw[->](3i)to[out=90,in=180]($(p3.west)-(0,.2)$);\draw[->]($(i3.east)-(0,.2)$)to[out=0,in=90](3o);
\node at(-2.7,-.7){$\lambda$};\node at(-1.7,-1.6){$\lambda$};\node at(2.7,-.7){$\lambda$};\node at(1.7,-1.6){$\lambda$};\node at(-.5,1.9){$\lambda$};\node at(.6,1.9){$\lambda$};
\end{tikzpicture}
\end{equation*}
\caption{A part of a $\mathfrak{P}$-graph coming from a simplicial face $f$ of $\Sigma$ adjacent to edges $e_1,e_2,e_3$ and with corners $c_1,c_2,c_3$.}
\label{sfig:Pg_spx}
\end{subfigure}
\begin{subfigure}{\textwidth}
\begin{equation*}
\begin{tikzpicture}[baseline=0]
\node[circle,draw,font=\tiny]at(0,0){$m$};\node at(0,.7){$c$};\node at(0,-.5){$e$};\node at(-.5,.1){$f$};
\draw[->-](-.7,0)arc(-180:0:.7);\draw(-.7,0)to[out=90,in=-105](0,1.1);\draw(.7,0)to[out=90,in=-75](0,1.1);
\end{tikzpicture}
\qquad\boldsymbol{\longmapsto}\qquad
\begin{tikzpicture}[baseline=0]
\draw[dotted](0,.4)ellipse(2 and 1.2);
\node[draw](p)at(-1,0){$P^\lambda_\mu$};\draw($(p.south west)+(.07,0)$)--($(p.north west)+(.07,0)$);
\node[draw](i)at(1,0){$I^\lambda_\mu$};\draw($(i.south west)+(.07,0)$)--($(i.north west)+(.07,0)$);
\node[draw](c)at(0,1.2){$\gamma_c$};\draw($(c.south east)-(.07,0)$)--($(c.north east)-(.07,0)$);
\node[fill=white](eo)at(-.4,-.8){$\mu$};\node[fill=white](ei)at(.4,-.8){$\mu$};
\draw[->]($(p.east)+(0,.15)$)--($(i.west)+(0,.15)$);\draw[->]($(p.east)-(0,.15)$)to[out=0,in=90](eo);\draw[->](ei)to[out=90,in=180]($(i.west)-(0,.15)$);
\draw[->](i.east)to[out=0,in=0]($(c.east)+(.5,0)$)--(c.east);\draw[->](c.west)--($(c.west)-(.5,0)$)to[out=180,in=180](p.west);
\node at(0,.4){$\lambda\setminus\mu$};\node at(-.8,1){$\lambda$};\node at(.8,1){$\lambda$};
\end{tikzpicture}
\end{equation*}
\caption{A part of a $\mathfrak{P}$-graph coming from $f\in\Forb(\Sigma)$ with a corner $c$ and edge $e=\partial f$.}
\label{sfig:Pg_orb}
\end{subfigure}
\caption{Local pictures of a $\mathfrak{P}$-graph. We dropped a level of subscripts on $\lambda,\mu,\gamma$ for decluttering.}
\label{fig:Pg}
\end{figure}

The second ingredient for Theorem~\ref{thm:count_tangles} is the notion of $T$-matching. It plays a role similar to $\sigma_b^-(\sigma_a^+)^{-1}$ and the other permutations in \cite[Lemma 5.2.P4 in \S5.3]{asymptotic_stats_random_covering_surfaces_MageePuder23}. On a high level, under a chosen numbering of exposed sides and hanging half-edges, a $T$-matching tuple of corner permutations encodes the combinatorial adjacencies occurring in $\partial T$ in the numeric labels. To make the discussion more specific, denote by $e_+\corner e_-$ the configuration of an oriented corner $c$ in a face $f$ pointing away from (a side of) an edge $e_-$ towards $e_+$. Above every $e_+\corner e_-$ in $\Sigma$ lies a set of (actual or phantom) corners of $T$. Adjacency in $\partial T$ induces a pairing between exposed sides $S_{e_\pm\hookrightarrow f}$ or hanging half-edges $H_{e_\pm\hookrightarrow f}$ (with the notation from Section~\ref{ssec:imm_cplx}) above $e_+$ and $e_-$. It can be regarded as a ``corner-traversal'' function
\begin{equation}\label{eq:corner_traverse}
\mathcal{T}_c\colon H_{v\hookrightarrow e_-}\sqcup S_{e_-\hookrightarrow f}\to H_{v\hookrightarrow e_+}\sqcup S_{e_+\hookrightarrow f}.
\end{equation}
See Figure~\ref{sfig:match_funs} for such $\mathcal{T}_c$ from an example complex. Working with permutations will require choosing numberings $\vec{\mathcal{J}}\colon \vec{H},\vec{S}\to[n]$ of exposed sides and hanging half-edges. Then, the $T$-matching permutations are the conjugates of $\vec{\mathcal{T}}$ by the appropriate $\vec{\mathcal{J}}$. Equivalently, this is the tuple $(\gamma_c)_{c\in C(\Sigma)}$ that makes the following diagrams commute.
\begin{equation*}\begin{tikzcd}
H_{v\hookrightarrow e_-}\sqcup S_{e_-\hookrightarrow f}\arrow[r,"\mathcal{T}_c"]\arrow[d,"\mathcal{J}_{v\hookrightarrow e_-}\sqcup\mathcal{J}_{e_-\hookrightarrow f}"']&H_{v\hookrightarrow e_+}\sqcup S_{e_+\hookrightarrow f}\arrow[d,"\mathcal{J}_{v\hookrightarrow e_+}\sqcup\mathcal{J}_{e_+\hookrightarrow f}"]\\
{[n]}\arrow[r,"\gamma_c\sqcup\tn{id}_{\{\text{unused numbers}\}}"]&{[n]}
\end{tikzcd}\end{equation*}
We now give the formal definition. Recall the notation $\mathfrak{v}=\#V(T),\mathfrak{e}_e=\#E_e,\mathfrak{f}_f$ for the number of vertices, labelled edges, and labelled faces of $T$ (orbigons counted with multiplicity) from Section~\ref{ssec:imm_cplx} and Figure~\ref{fig:immersion}.
\begin{definition}\label{def:T_compat}
Suppose that for some $n\in\mathbb{N}$ we have chosen bijective numberings
\begin{align*}
\mathcal{J}_\iota\colon&H_\iota\to[n-\mathfrak{v}+1,n-\mathfrak{e}_e]\qquad\text{for each vertex-edge incidence }\iota\colon v\hookrightarrow e\\
\mathcal{J}_\iota\colon&S_\iota\to[n-\mathfrak{e}_e+1,n-\mathfrak{f}_f]\qquad\text{for each edge-face incidence }\iota\colon e\hookrightarrow f.
\end{align*}
We say a tuple of permutations $\left(\gamma_c\in S_{[n-\mathfrak{v}+1,n-\mathfrak{f}_{f(c)}]}\right)_{c\in C(\Sigma)}$ is \emph{$T$-matching $\vec{\mathcal{J}}$} if
\begin{equation*}
\gamma_c\colon\mathcal{J}_-(e_-')\mapsto\mathcal{J}_+(e_+')\quad\text{at every phantom }e'_+\corner[']e'_-\text{ of }T\text{ above }e_+\corner e_-\in C(\Sigma).
\end{equation*}
We say \emph{$T$-matching} for short when the numbering $\vec{\mathcal{J}}$ is implicit.

We say $\vec\gamma$ is \emph{bilaterally $T$-matching} if $\gamma_c$ are permutations of $[n-\mathfrak{v}+1,n]$, numberings of sides of edges $\mathcal{J}_{e\hookrightarrow f}$ is replaced with numbering of full edges themselves $\mathcal{J}_e\colon E_e\to[n-\mathfrak{e}_e+1,n]$, and an identical compatibility condition holds (including at actual corners of $T$).
\end{definition}
Bilateral $T$-matching will only appear in Propositions~\ref{prop:count_E},\ref{prop:count_F} that are intermediate steps in the proof of Theorem~\ref{thm:count_tangles}. The difference is that bilateral includes covered sides $\vec{O}$ and requires the same numeric label on both sides of an edge. Note that any numbering uniquely determines a tuple of permutations $T$-matching it. For an example see Figure~\ref{fig:T_match}.
\begin{figure}[h]
\begin{subfigure}{\textwidth}
\begin{equation*}
\begin{tikzpicture}[baseline=-15]
\fill[lightgray](0,0)--(0,-.4)arc(270:180:.4)--(0,0)(0,-1)--(0,-1.4)arc(270:180:.4)--(0,-1)(-1,0)--(-1,-.4)arc(270:180:.4)--(-1,0)(-2,0)--(-2,-.4)arc(270:180:.4)--(-2,0);
\draw[thick,teal,->>-](0,-1)--(0,0);\draw[thick,blue,->>>-](-1,0)--(0,0);\draw[thick,blue,->>>-](-2,0)--(-1,0);
\begin{scope}[>={Latex[length=4,width=4]},thick,red]\draw[>->](.3,-.3)--(-.3,.3);\draw[>->](.3,-1.3)--(-.3,-.7);\draw[>->](-.7,-.3)--(-1.3,.3);\draw[>->](-1.7,-.3)--(-2.3,.3);\end{scope}
\begin{scope}[>={Latex[length=4,width=4,sep=-2pt]Latex[length=4,width=4]},thick,teal,font=\tiny]\draw[>-](0,-1.45)--(0,-1);\draw[->](0,0)--(0,.4);\draw[>->](-1,-.45)--(-1,.45);\draw[>->](-2,-.45)--(-2,.45);
\node at(-2,-.6){$W_1$};\node at(-1,-.6){$W_2$};\node at(0,-1.6){$W_3$};\node at(-.15,-.5){$X$};\end{scope}
\begin{scope}[>={Latex[length=4,width=4,sep=-2pt]Latex[length=4,width=4,sep=-2pt]Latex[length=4,width=4]},thick,blue,font=\tiny]\draw[>-](-2.45,0)--(-2,0);\draw[->](0,0)--(.45,0);\draw[>->](-.45,-1)--(.45,-1);
\node at(-2.6,0){$Y_1$};\node at(-.6,-1){$Y_2$};\node at(-1.5,-.18){$Z_1$};\node at(-.5,-.18){$Z_2$};\end{scope}
\filldraw(0,0)circle(.04)(0,-1)circle(.04)(-1,0)circle(.04)(-2,0)circle(.04);
\end{tikzpicture}
\quad=T\looparrowright\Sigma=\quad
\begin{tikzpicture}[baseline=10]
\fill[lightgray](0,0)--(0,1)--(1,1)--(1,0);
\draw[thick,red,->-](1,0)to[bend left=10](0,1);\draw[thick,teal,->>-](0,0)--(0,1);\draw[thick,teal,->>-](1,0)--(1,1);\draw[thick,blue,->>>-](0,0)--(1,0);\draw[thick,blue,->>>-](0,1)--(1,1);
\filldraw(0,0)circle(.04)(0,1)circle(.04)(1,0)circle(.04)(1,1)circle(.04);
\node at(.85,.85){$c$};\draw[->](.87,.67)arc(270:180:.2);\node[font=\tiny]at(.6,.55){$f_2$};
\node at(.5,-.2){\color{blue}$a_2$};\node at (.5,1.2){\color{blue}$a_2$};\node at(-.3,.5){\color{teal}$a_1$};\node at(1.3,.5){\color{teal}$a_1$};
\end{tikzpicture}
\end{equation*}
\caption{A complex $T$ (with hanging half-edges drawn) immersed into torus $\Sigma$ with the $\Delta$-complex structure from Figure~\ref{sfig:imm_cplx}. The corner $c\in C(f_2)$ of the face $f_2\in F(\Sigma)$ lifts to 4 phantom corners in $T$ (shaded grey). They separate exposed sides $S_{{\color{teal}a_1}\hookrightarrow f_2}=\{X\}$ and hanging half-edges $H_{v\hookrightarrow{\color{teal}a_1}}=\{W_1,W_2,W_3\}$ from exposed sides $S_{{\color{blue}a_2}\hookrightarrow f_2}=\{Z_1,Z_2\}$ and hanging half-edges $H_{v\hookrightarrow{\color{blue}a_2}}=\{Y_1,Y_2\}$.}
\end{subfigure}
\begin{subfigure}{\textwidth}
\begin{equation*}
\mathcal{T}_c\colon\begin{cases}{\color{teal}W_1}&\mapsto{\color{blue}Y_1}\\{\color{teal}W_2}&\mapsto{\color{blue}Z_1}\\{\color{teal}W_3}&\mapsto{\color{blue}Y_2}\\{\color{teal}X}&\mapsto{\color{blue}Z_2}\end{cases}\qquad
\begin{matrix}\mathcal{J}_{v\hookrightarrow{\color{teal}a_1}}\colon\begin{cases}W_1\mapsto6\\W_2\mapsto7\\W_3\mapsto8\end{cases}\\\mathcal{J}_{{\color{teal}a_1}\hookrightarrow f_2}\colon X\mapsto9\end{matrix}\qquad
\begin{matrix}\mathcal{J}_{v\hookrightarrow{\color{blue}a_2}}\colon&\begin{cases}Y_1\mapsto6\\Y_2\mapsto7\end{cases}\\\mathcal{J}_{{\color{blue}a_2}\hookrightarrow f_2}\colon&\begin{cases}Z_1\mapsto8\\Z_2\mapsto9\end{cases}\end{matrix}
\end{equation*}
\begin{minipage}{.4\textwidth}
\begin{equation*}
\gamma_c\colon\begin{cases}{\color{teal}6}\mapsto{\color{blue}6}\\{\color{teal}7}\mapsto{\color{blue}8}\\{\color{teal}8}\mapsto{\color{blue}7}\\{\color{teal}9}\mapsto{\color{blue}9}\end{cases}\in S_{[6,9]}
\end{equation*}
\end{minipage}\begin{minipage}{.6\textwidth}
\centering
\begin{tikzpicture}
\fill[lightgray](0,0)--(0,-.4)arc(270:180:.4)--(0,0)(0,-1)--(0,-1.4)arc(270:180:.4)--(0,-1)(-1,0)--(-1,-.4)arc(270:180:.4)--(-1,0)(-2,0)--(-2,-.4)arc(270:180:.4)--(-2,0);
\draw[thick,teal,->>-](0,-1)--(0,0);\draw[thick,blue,->>>-](-1,0)--(0,0);\draw[thick,blue,->>>-](-2,0)--(-1,0);
\begin{scope}[>={Latex[length=4,width=4]},thick,red]\draw[>->](.3,-.3)--(-.3,.3);\draw[>->](.3,-1.3)--(-.3,-.7);\draw[>->](-.7,-.3)--(-1.3,.3);\draw[>->](-1.7,-.3)--(-2.3,.3);\end{scope}
\begin{scope}[>={Latex[length=4,width=4,sep=-2pt]Latex[length=4,width=4]},thick,teal,font=\tiny,inner sep=1pt]\draw[>-](0,-1.45)--(0,-1);\draw[->](0,0)--(0,.4);\draw[>->](-1,-.45)--(-1,.45);\draw[>->](-2,-.45)--(-2,.45);
\node(w1)at(-2,-.6){$6$};\node(w2)at(-1,-.6){$7$};\node(w3)at(0,-1.6){$8$};\node(x)at(-.15,-.5){$9$};\end{scope}
\begin{scope}[>={Latex[length=4,width=4,sep=-2pt]Latex[length=4,width=4,sep=-2pt]Latex[length=4,width=4]},thick,blue,font=\tiny,inner sep=1pt]\draw[>-](-2.45,0)--(-2,0);\draw[->](0,0)--(.45,0);\draw[>->](-.45,-1)--(.45,-1);
\node(y1)at(-2.6,0){$6$};\node(y2)at(-.6,-1){$7$};\node(z1)at(-1.5,-.18){$8$};\node(z2)at(-.5,-.18){$9$};\end{scope}
\filldraw(0,0)circle(.04)(0,-1)circle(.04)(-1,0)circle(.04)(-2,0)circle(.04);
\draw[->](w1)to[bend left](y1);\draw[->](w2)to[bend left](z1);\draw[->](w3)to[bend left](y2);\draw[->](x)to[bend left](z2);
\end{tikzpicture}
\end{minipage}
\caption{We take the face $f_2$ of $c$ to be oriented anticlockwise, so the induced orientation of $c$ points away from ${\color{teal}a_1}$ towards ${\color{blue}a_2}$. The resulting function $\mathcal{T}_c$ from equation (\ref{eq:corner_traverse}) is given above. We have chosen example numberings $\vec{\mathcal{J}}$ of exposed sides and hanging half-edges, in the format from Definition~\ref{def:T_compat} with $n=9$ (we have $\mathfrak{v}=4,{\color{teal}\mathfrak{e}_{a_1}}=1,{\color{blue}\mathfrak{e}_{a_2}}=2,\mathfrak{f}_{f_1}=\mathfrak{f}_{f_2}=0$). They conjugate $\mathcal{T}_c$ to the permutation $\gamma_c$ above that is $T$-matching $\vec{\mathcal{J}}$. The numeric labels placed next to their respective exposed sides and hanging half-edges are drawn in the picture above.}
\label{sfig:match_funs}
\end{subfigure}
\caption{A $T$-matching permutation for an example immersed complex $T$, corner $c$ of $\Sigma$, and numbering of exposed sides and hanging half-edges $\vec{\mathcal{J}}$, following Definition~\ref{def:T_compat}. The inclusions $v\hookrightarrow{\color{teal}a_1}$ and $v\hookrightarrow{\color{blue}a_2}$ in the discussion are the terminal ones.}
\label{fig:T_match}
\end{figure}

We are ready to state the main theorem of this section, and the key technical result of this work. Later in Section~\ref{sec:asymptotics}, the analysis of its asymptotics will yield the master Theorem~\ref{thm:E_combi_emb} on the expected number of embeddings.
\begin{theorem}\label{thm:count_tangles}
Let $n\in\mathbb{N}$ with $n\geqslant\mathfrak{v}$, and $\vec\gamma$ be a tuple of $T$-matching permutations. The number of all embeddings of $T$ in degree-$n$ coverings of $\Sigma$ is
\begin{equation*}
\#\left\{(\whS,\iota)\Bigg|\begin{matrix}\whS\in\covers\\\iota\colon T\hookrightarrow\whS\end{matrix}\right\}=Q(n)\times\sum_{(\eta,\vec\mu,\vec\lambda)\in\mathcal{Y}(n)}R(\eta,\vec\mu,\vec\lambda)\cdot\tr\mathfrak{P}\left(\vec\lambda\setminus\eta,\vec\mu\setminus\eta,\vec\gamma\right)\\
\end{equation*}
where
\begin{align*}
Q(n)&=(n!)^{1-\cht(\Sigma)}\frac{n^{\underline{\mathfrak{v}}}\cdot\prod_{f\in\Fspx(\Sigma)}n^{\underline{\mathfrak{f}_f}}\cdot\prod_\iterorb(n^{\underline{\mathfrak{f}_f}}\cdot\tau_m(S_{n-\mathfrak{f}_f}))}{\prod_{e\in E(\Sigma)}n^{\underline{\mathfrak{e}_e}}}\\
R(\eta,\vec\mu,\vec\lambda)&=\frac{\chi_\eta(1)\cdot\prod_{f\in\Fspx(\Sigma)}\chi_{\lambda_f}(1)\cdot\prod_\iterorb\chi_{\lambda_f}^{m\tn{-tors}}}{\prod_{e\in E(\Sigma)}\chi_{\mu_e}(1)}\\
\mathcal{Y}(n)&=\left\{\left(\eta,\vec\mu,\vec\lambda\right)\ \middle|\ \begin{matrix}\lambda_f\vdash n-\mathfrak{f}_f,\mu_e\vdash n-\mathfrak{e}_e,\eta\vdash n-\mathfrak{v}\\\forall_{e\hookrightarrow f}\ \eta\subseteq\mu_e\subseteq\lambda_f\end{matrix}\right\}.
\end{align*}
\end{theorem}
Note that Theorem~\ref{thm:count_tangles} specialises to Mednykh's formula in Lemma~\ref{lmm:Mednykh} when we substitute the empty immersed complex as $T$. In general, there is still the prefactor $Q$ that assumes the role of earlier $\#G^{1-\chi^\tn{top}(\Sigma)}$ and drives the asymptotics; an expression resembling the earlier zeta constant, except now the summation is over certain tuples $\mathcal{Y}$ instead of $\tn{Irr}(G)$; in it $R$ replacing the product of $\chi_V$'s. The new feature is the appearance of $\mathfrak{P}$-graphs, which remember the data of $(p,T)$.

The proof of Theorem \ref{thm:count_tangles} will be split into two smaller results, Propositions \ref{prop:count_E} and \ref{prop:count_F}. Morally, the former counts edges of $T$ and the latter counts faces. To state them we need one last bit of extra notation. Recall from Remark~\ref{rmk:fibre_label} that the coverings $\whS\in\covers$ come equipped with a fibre labelling $\mathcal{L}_{\whS}\colon V(\whS)\to[n]$. For a labelling $\mathcal{I}\colon V(T)\hookrightarrow[n]$ of vertices of $T$, denote by
\begin{equation}\label{eq:vertex_labelling_pullback}
\covers(T,\mathcal{I})\defeq\left\{(\widehat{\Sigma},\iota)\Bigg|\begin{matrix}\widehat{\Sigma}\in\covers,\iota\colon T\hookrightarrow\whS\text{ such that}\\\mathcal{L}_{\whS}\circ\iota=\mathcal{I}\text{ as functions }V(T)\to[n]\end{matrix}\right\}
\end{equation}
the set of embeddings of $T$ into coverings such that the pullback of $\mathcal{L}$ through the immersion agrees with $\mathcal{I}$.
\begin{proposition}\label{prop:count_E}
Fix $n\in\mathbb{N},n\geqslant\mathfrak{v}$ and a vertex labelling $\mathcal{I}\colon V(T)\to[n-\mathfrak{v}+1,n]$. Then for any bilaterally $T$-matching tuple of permutations $\vec\gamma'$ we have
\begin{equation*}
\#\covers(T,\mathcal{I})=\frac{\prod_{i=1}^k\tau_{m_i}(S_n)\prod_{e\in E(\Sigma)}(n-\mathfrak{e}_e)!}{(n!)^{\#F(\Sigma)}}\times\hspace*{-10pt}\sum_{\tiny\begin{matrix}\lambda'_f\vdash n\\\mu_e\vdash n-\mathfrak{e}_e\end{matrix}}R(\emptyset,\vec\mu,\vec\lambda')\tr\mathfrak{P}\left(\vec\lambda',\vec\mu,\vec\gamma'\right),
\end{equation*}
where $m_1,\dots,m_k$ are the orders of cone points of $\Sigma$.
\end{proposition}
Proposition \ref{prop:count_E} is proved in Section~\ref{ssec:pf_E}. It is a milestone towards Theorem \ref{thm:count_tangles}. However, it still differs from the target in three aspects: first, the count of $F(T)$ remains obscured; second, at each $f\in F(\Sigma)$ the summation is over partitions of size $n$ instead of $n-\mathfrak{f}_f$; third, $\vec\gamma'$ are bilaterally $T$-matching instead of $T$-matching. These issues are interrelated, and addressed by Proposition \ref{prop:count_F}.
\begin{proposition}\label{prop:count_F}
Let $\vec\gamma'$ be a bilaterally $T$-matching tuple of permutations, $\vec\gamma$ be $T$-matching, and suppose the witnessing numberings agree on hanging half-edges. Then for any tuple of representations $\mu_e\vdash n-\mathfrak{e}_e$ we have the equality
\begin{align*}
\sum_{\lambda'_f\vdash n}R(\emptyset,\vec\emptyset,\vec\lambda')\cdot\tr\mathfrak{P}(\vec\lambda',\vec\mu,\vec\gamma')=&\prod_{f\in\Fspx(\Sigma)}n^{\underline{\mathfrak{f}_f}}\prod_\iterorb\frac{n^{\underline{\mathfrak{f}_f}}\tau_m(S_{n-\mathfrak{f}_f})}{\tau_m(S_n)}\times\\
&\times\sum_{\tiny\begin{matrix}\eta\in M(\vec\mu)\\\lambda_f\in N_f(\vec\mu)\end{matrix}}R(\eta,\vec\emptyset,\vec\lambda)\cdot\tr\mathfrak{P}(\vec\lambda\setminus\eta,\vec\mu\setminus\eta,\vec\gamma).
\end{align*}
where
\begin{align*}
M(\vec\mu)&=\{\eta\vdash n-\mathfrak{v}\ |\ \forall_e\ \eta\subseteq\mu_e\}\\
N_f(\vec\mu)&=\{\lambda\vdash n-\mathfrak{f}_f\ |\ \forall_{e\hookrightarrow f}\ \lambda\supseteq\mu_e\}.
\end{align*}
\end{proposition}
Proposition \ref{prop:count_F} is proved in Section~\ref{ssec:pf_F}.

\medskip

Propositions \ref{prop:count_E} and \ref{prop:count_F} allow us to complete the proof of Theorem \ref{thm:count_tangles}.
\begin{proof}[Proof of Theorem \ref{thm:count_tangles}, assuming Propositions~\ref{prop:count_E},\ref{prop:count_F}]
We begin by partitioning the set of embeddings according to the induced vertex labelling of $T$. Each (covering, embedding) pair pulls back a unique vertex labelling, so
\begin{equation*}
\left\{(\whS,\iota)\ \middle|\ \whS\in\covers,\iota\colon T\hookrightarrow\whS\right\}=\bigsqcup_{\mathcal{I}\colon V(T)\hookrightarrow[n]}\covers(T,\mathcal{I}),
\end{equation*}
where $\covers(T,\mathcal{I})$ are the sets defined in equation (\ref{eq:vertex_labelling_pullback}). The action of $S_n$ by relabelling the fibre provides bijections between the sets $\covers(T,\mathcal{I})$ for different $\mathcal{I}$'s, so they all have equal sizes. Therefore, fixing $\mathcal{I}\colon V(T)\to[n-\mathfrak{v}+1,n]$ for the rest of the proof, we obtain
\begin{equation}\label{eq:summation_labellings}
\#\{(\whS,\iota)\mid\whS\in\covers,\iota\colon T\hookrightarrow\whS\}=n^{\underline{\mathfrak{v}}}\cdot\#\covers(T,\mathcal{I}),
\end{equation}
where the prefactor is the number of vertex labellings of $T$ with $[n]$.

The tuple $\vec\gamma$ is $T$-matching some numbering of exposed sides and hanging half-edges. Extend the numbering of hanging half-edges to whole edges, and denote by $\gamma'_c\in S_{[n-\mathfrak{v}+1,n]}$ the unique permutations bilaterally $T$-matching the result. Proposition \ref{prop:count_E} tells us that
\begin{align}\begin{split}\label{eq:labelled_embs}
\#\covers(T,\mathcal{I})=&\frac{(n!)^{\#E(\Sigma)-\#F(\Sigma)}\prod_{i=1}^k\tau_{m_i}(S_n)}{\prod_{e\in E(\Sigma)}n^{\underline{\mathfrak{e}_e}}}\times\\
&\times\sum_{\tiny\begin{matrix}\lambda'_f\vdash n\\\mu_e\vdash n-\mathfrak{e}_e\end{matrix}}R(\emptyset,\vec\mu,\vec\lambda')\tr\mathfrak{P}\left(\vec\lambda',\vec\mu,\vec\gamma'\right).
\end{split}\end{align}

Proposition \ref{prop:count_F} is equivalent, after dividing both sides by $\prod_e\chi_{\mu_e}(1)$, to
\begin{align}\begin{split}\label{eq:pull_F_factor}
\sum_{\lambda'_f\vdash n}R(\emptyset,\vec\mu,\vec\lambda')\cdot\tr\mathfrak{P}(\vec\lambda',\vec\mu,\vec\gamma')=&\hspace*{-5pt}\prod_{f\in\Fspx(\Sigma)}n^{\underline{\mathfrak{f}_f}}\prod_\iterorb\hspace*{-5pt}\frac{n^{\underline{\mathfrak{f}_f}}\tau_m(S_{n-\mathfrak{f}_f})}{\tau_m(S_n)}\times\\
&\times\hspace*{-10pt}\sum_{\tiny\begin{matrix}\eta\in M(\vec\mu)\\\lambda_f\in N_f(\vec\mu)\end{matrix}}R(\eta,\vec\mu,\vec\lambda)\cdot\tr\mathfrak{P}(\vec\lambda\setminus\eta,\vec\mu\setminus\eta,\vec\gamma).
\end{split}\end{align}

Putting equations (\ref{eq:summation_labellings}, \ref{eq:labelled_embs}, \ref{eq:pull_F_factor}) together gives us a formula for the size of $\{T\hookrightarrow\whS\}$. Remembering that $\#E(\Sigma)-\#F(\Sigma)=1-\chi^\tn{top}(\Sigma)$ and recognising that the summation is over $\mathcal{Y}(n)$ completes the proof of Theorem \ref{thm:count_tangles}.
\end{proof}

\subsection{Proof of Proposition \ref{prop:count_E} and the count of $E(T)$}\label{ssec:pf_E}

The proof of Proposition~\ref{prop:count_E} builds on the approach to Mednykh's formula~\ref{lmm:Mednykh} from Section~\ref{ssec:Mednykh}. Proposition~\ref{prop:Rg_trace} expresses the indicator function of $\Hom{\pio\Sigma}{S_n}$ inside the ambient $(S_n)^{E(\Sigma)}$ in terms of $\mathfrak{R}$-graphs, and we count embeddings by summing it over the appropriate tuples of edge permutations. The main difference now is that adding the immersed complex $T$ to the picture restricts the permutations at an edge $e$ to a certain subset $P_e$. Rewriting $P_e$ in terms of the standard subgroup gives rise to bilaterally $T$-matching corner permutations. Computation of the resulting ribbon graphs using Lemma~\ref{lmm:swap} and simplification then yields Proposition~\ref{prop:count_E}. For decluttering we drop primes from $\lambda$ and $\gamma$ in the proof below.
\begin{proof}[Proof of Proposition \ref{prop:count_E}]
First we need to put $\covers(T,\mathcal{I})$ into a suitable form. As mentioned in Section~\ref{sec:intro}, monodromy provides a correspondence between degree-$n$ coverings of $\Sigma$ (carrying a labelling $\mathcal{L}$ of the fibre with $[n]$ according to Remark~\ref{rmk:fibre_label}) and homomorphisms $\pio\Sigma\to S_n$. We claim that it induces a bijection
\begin{align}\begin{split}\label{eq:cover_hom_restricted}
\covers(T,\mathcal{I})\leftrightarrow&\{\varphi\colon\pio\Sigma\to S_n\mid\varphi(e)\in P_e\text{ for every }e\in E(\Sigma)\}\\
\text{where}\qquad P_e\defeq&\{\sigma\in S_n\mid\sigma\colon\mathcal{I}(i(e'))\mapsto\mathcal{I}(t(e'))\text{ for every }e'\in E_e\},
\end{split}\end{align}
and $i,t$ denote the initial and terminal vertices respectively. Indeed, consider $e'\in E_e(T)$ joining $v_-,v_+\in V(T)$. If $(\whS,\iota)\in\covers(T,\mathcal{I})$ has monodromy $\varphi$, then by definition (\ref{eq:vertex_labelling_pullback}) the embedding $\iota$ must send $v_\pm\mapsto\mathcal{L}_\whS^{-1}\circ\mathcal{I}(v_\pm)$, and since these image vertices are connected by the lift $\iota(e')$ of $e$, the monodromy $\varphi(e)\in S_n$ of $e$ sends the initial label $\mathcal{I}(v_-)$ to the terminal $\mathcal{I}(v_+)$. Conversely, if $\varphi\in\Hom{\pio\Sigma}{S_n}$ satisfies $\varphi(e)\colon\mathcal{I}(v_-)\mapsto\mathcal{I}(v_+)$, then the corresponding covering $\whS$ contains a lift of $e$ connecting the two vertices $\mathcal{L}_\whS^{-1}\circ\mathcal{I}(v_\pm)$, so the mapping $\mathcal{L}_\whS^{-1}\circ\mathcal{I}$ between vertices can be extended to $V(T)\cup\{e'\}\to\whS$. Taking the conjunction over all $E(T)$ proves that (\ref{eq:cover_hom_restricted}) is a bijection (note any embedding $T^{(1)}\hookrightarrow\whS$ of the 1-skeleton into a covering extends over the faces of $T$). As in Section~\ref{ssec:Mednykh}, we consider $\Hom{\pio\Sigma}{S_n}$ as a subset of the ambient $\Hom{\pi_1\Sigma^{(1)}}{S_n}$, which is identified with $(S_n)^{E(\Sigma)}$ since $E(\Sigma)$ freely generates the fundamental group of the 1-skeleton. In the ambient space the conditions $\varphi(e)\in P_e$ at different edges are independent and (\ref{eq:cover_hom_restricted}) becomes
\begin{equation*}
\covers(T,\mathcal{I})\quad\leftrightarrow\quad\prod_{e\in E(\Sigma)}P_e\cap\Hom{\pio\Sigma}{S_n}\quad\subseteq\quad(S_n)^{E(\Sigma)}.
\end{equation*}

The above intersection is suitable for our methods. It can be counted as
\begin{equation}\label{eq:double_counting}
\#\covers(T,\mathcal{I})=\sum_{\sigma_e\in P_e}\mathbbm{1}_{\{\vec\sigma\tn{ extends to homomorphism }\pio\Sigma\to S_n\}}.
\end{equation}
Using Proposition \ref{prop:Rg_trace}, we can rewrite the indicator functions from equation (\ref{eq:double_counting}) in terms of $\mathfrak{R}$-graphs. Then, remembering that the trace is multilinear in coupon colours, the RHS of (\ref{eq:double_counting}) becomes
\begin{align}\begin{split}\label{eq:sum_Rgraphs}
\sum_{\sigma_e\in P_e}&\mathbbm{1}_{\{\vec\sigma\tn{ extends to hom }\pio\Sigma\to S_n\}}=\frac{\prod_{i=1}^k\tau_{m_i}(S_n)}{(n!)^{\#F(\Sigma)}}\times\\
&\times\sum_{\lambda_f\vdash n}\prod_{f\in\Fspx(\Sigma)}\chi_{\lambda_f}(1)\hspace*{-5pt}\prod_\iterorb\hspace*{-10pt}\chi_{\lambda_f}^{m\tn{-tors}}\cdot\tr\mathfrak{R}\left(\vec\lambda,\left(\sum_{\sigma\in P_e}\sigma\right)_{e\in E(\Sigma)}\right),
\end{split}\end{align}
where $\mathfrak{R}$ is the $\mathfrak{R}$-graph constructed in Definition \ref{def:R_graph}, and we identified the set $\tn{Irr}(S_n)$ with $\{\lambda\vdash n\}$ according to Proposition~\ref{prop:repSn_OV} from Section~\ref{sec:rep_Sn}.

We will now reexpress the two-sided translate $P_e\subseteq S_n$ in terms of a standard subgroup. Let
\begin{alignat*}{4}
\mathcal{J}_e\colon&E_e\to[n-\mathfrak{e}_e+1,n]&&\text{for }e\in E(\Sigma)\\
\mathcal{J}_\iota\colon&H_\iota\to[n-\mathfrak{v}+1,n-\mathfrak{e}_e]\qquad&&\text{for incidence }\iota\colon v\hookrightarrow e\text{ in }\Sigma
\end{alignat*}
be the numbering of edges and half-edges that witnesses the bilateral $T$-matching of $\vec\gamma$ ($\mathcal{J}_e$ can be thought of as numbering midpoints of edges). Denote by $\alpha^{(i)}_e,\alpha^{(t)}_e\in S_{[n-\mathfrak{v}+1,n]}$ the permutations that send the $\mathcal{J}$-label of an edge or pair of half-edges to the $\mathcal{I}$-label of its initial and terminal vertex respectively. In symbols
\begin{align}\begin{split}\label{eq:alpha_beta}
\alpha^{(i)}_e\colon\mathcal{J}_e(e')\mapsto&\mathcal{I}(i(e'))\quad\text{for }e'\in E_e\\
\alpha^{(i)}_e\colon\mathcal{J}_\iota(h')\mapsto&\mathcal{I}(v')\hspace*{20pt}\begin{matrix}\text{for }h'\in H_\iota\text{ anchored at }v'\in V(T)\\\text{with }\iota\colon v\hookrightarrow e\text{ the initial incidence,}\end{matrix}
\end{split}\end{align}
and analogously for terminal ends $t$. Then it is straightforward to check that
\begin{equation}\label{eq:bicoset}
P_e=\alpha^{(t)}_e\cdot S_{n-\mathfrak{e}_e}\cdot(\alpha^{(i)}_e)^{-1}.
\end{equation}
For now, let us focus on the orientable case, so that all edge coupons of the $\mathfrak{R}$-graph look like in Figure \ref{sfig:Rg_edge}. Equation (\ref{eq:bicoset}) allows us to rewrite these edge coupons as
\begin{equation}\label{eq:standardise_summation}
\begin{tikzpicture}[baseline=0]
\draw[dotted](0,0)circle(1.8);
\node at(0,0){$\displaystyle\sum_{\sigma\in P_e}\sigma$};\node[fill=white](lli)at(-1,-1.5){$\lambda_L$};\node[fill=white](llo)at(-1,1.5){$\lambda_L$};\node[fill=white](lri)at(1,1.5){$\lambda_R$};\node[fill=white](lro)at(1,-1.5){$\lambda_R$};
\draw(-1.3,.6)--(1.3,.6)--(1.3,-.6)--(-1.3,-.6)--(-1.3,.6)(-1.3,-.5)--(1.3,-.5);
\draw[->](lli)--(-1,-.6);\draw[->](-1,.6)--(llo);\draw[->](lri)--(1,.6);\draw[->](1,-.6)--(lro);
\end{tikzpicture}
\qquad=\qquad
\begin{tikzpicture}[baseline=0]
\draw[dotted](0,0)circle(2.42);
\node at(0,0){$\displaystyle\sum_{\sigma\in S_{n-\mathfrak{e}_e}}\!\!\sigma$};\node[draw](bo)at(-1,1.3){$\alpha^{(t)}_e$};\node[draw](bi)at(1,1.3){$(\alpha^{(t)}_e)^{-1}$};\node[draw](ai)at(-1,-1.3){$(\alpha^{(i)}_e)^{-1}$};\node[draw](ao)at(1,-1.3){$\alpha^{(i)}_e$};\node[fill=white](li)at(-1,-2.2){$\lambda_L$};\node[fill=white](lo)at(-1,2.2){$\lambda_L$};\node[fill=white](ri)at(1,2.2){$\lambda_R$};\node[fill=white](ro)at(1,-2.2){$\lambda_R$};
\draw(-1.2,.6)--(1.2,.6)--(1.2,-.6)--(-1.2,-.6)--(-1.2,.6)(-1.2,-.5)--(1.2,-.5);
\draw[->](li)--(ai);\draw[->](ai)--(-1,-.6);\draw[->](-1,.6)--(bo);\draw[->](bo)--(lo);
\draw[->](ri)--(bi);\draw[->](bi)--(1,.6);\draw[->](1,-.6)--(ao);\draw[->](ao)--(ro);
\draw($(ai.south east)+(0,.07)$)--($(ai.south west)+(0,.07)$);\draw($(bo.south east)+(0,.07)$)--($(bo.south west)+(0,.07)$);\draw($(ao.north east)-(0,.07)$)--($(ao.north west)-(0,.07)$);\draw($(bi.north east)-(0,.07)$)--($(bi.north west)-(0,.07)$);
\node[font=\tiny]at(-.75,.8){$\lambda_L$};\node[font=\tiny]at(-.8,-.8){$\lambda_L$};\node[font=\tiny]at(1.2,.8){$\lambda_R$};\node[font=\tiny]at(1.2,-.8){$\lambda_R$};
\end{tikzpicture}\end{equation}
where on the right wires the inverse on $\alpha^{(t)}_e$ and the lack of it on $\alpha^{(i)}_e$ come from swapping the top and bottom bases of their coupons.

After this substitution two coupons appear on each corner wire, coming from the two edges adjacent to the corner. We join them into a single coupon, whose colour is the composition of the two permutations, as in Example~\ref{ex:ribbon_composition}. This operation is illustrated in detail in Figure \ref{fig:join_at_corner}.
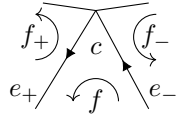
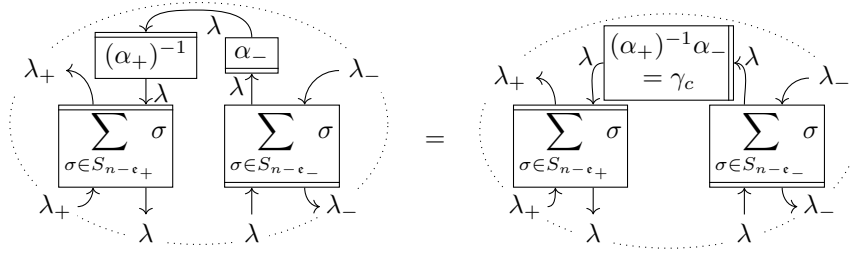
\begin{figure}[h]
\begin{subfigure}{\textwidth}
\centering
\begin{tikzpicture}
\draw[->-](.7,-.3)--(0,1);\draw[->-](0,1)--(-.8,-.3);
\node at(0,.5){$c$};\node at(-.95,-.1){$e_+$};\node at(.9,-.1){$e_-$};
\node at(0,-.2){$f$};\draw[->](.3,-.2)arc(0:180:.3);
\draw(-.7,1.1)--(0,1)--(.7,1.1);
\node at(-.8,.65){$f_+$};\draw[->](-.8,.35)arc(-90:90:.3);
\node at(.8,.65){$f_-$};\draw[->](.8,.95)arc(90:270:.3);
\end{tikzpicture}
\caption{Local picture of $e_+\corner e_-$ in $\Sigma$ where corner permutations arise. Edges $e_\pm$ meet at a corner $c$ of a face $f$, and separate $f$ from faces $f_\pm$ respectively. Note that this is one of four possibilities of orientations of $e_\pm$.}
\end{subfigure}
\begin{subfigure}{\textwidth}
\begin{equation*}
\begin{tikzpicture}[baseline=0]
\draw[dotted](0,.3)ellipse(2.5 and 1.6);
\node[draw,inner sep=2](s2)at(-1.1,0){$\displaystyle\sum_{\sigma\in S_{n-\mathfrak{e}_+}}\!\!\sigma$};\draw($(s2.north west)-(0,.07)$)--($(s2.north east)-(0,.07)$);
\node[draw](a2i)at(-.7,1.2){$(\alpha_+)^{-1}$};\draw($(a2i.north west)-(0,.07)$)--($(a2i.north east)-(0,.07)$);
\draw[->](a2i.south)--($(s2.north)+(.4,0)$);
\node[fill=white](a2o)at(-2.1,1){$\lambda_+$};\node[fill=white,inner sep=2](b2i)at(-1.9,-.8){$\lambda_+$};\node[fill=white](b2o)at(-.7,-1.15){$\lambda$};
\draw[->]($(s2.north)-(.3,0)$)to[out=90,in=0](a2o);\draw[->](b2i)to[out=0,in=-90]($(s2.south)-(.3,0)$);\draw[->]($(s2.south)+(.4,0)$)--(b2o);
\node[draw,inner sep=2](s1)at(1.1,0){$\displaystyle\sum_{\sigma\in S_{n-\mathfrak{e}_-}}\!\!\sigma$};\draw($(s1.south east)+(0,.07)$)--($(s1.south west)+(0,.07)$);
\node[draw](b2o)at(.7,1.2){$\alpha_-$};\draw($(b2o.south west)+(0,.07)$)--($(b2o.south east)+(0,.07)$);
\draw[->]($(s1.north)-(.4,0)$)--(b2o.south);
\node[fill=white](b1i)at(2.2,1){$\lambda_-$};\node[fill=white](a1i)at(.7,-1.15){$\lambda$};\node[fill=white,inner sep=2](a1o)at(1.9,-.8){$\lambda_-$};
\draw[->](b1i)to[out=180,in=90]($(s1.north)+(.3,0)$);\draw[->](a1i)--($(s1.south)-(.4,0)$);\draw[->]($(s1.south)+(.3,0)$)to[out=-90,in=180](a1o);
\draw[->](b2o.north)to[out=90,in=0](0,1.8)to[out=180,in=90](a2i.north);
\node at(-.5,.7){$\lambda$};\node at(.5,.8){$\lambda$};\node at(.2,1.6){$\lambda$};
\end{tikzpicture}
\quad=\quad
\begin{tikzpicture}[baseline=0]
\draw[dotted](0,.2)ellipse(2.5 and 1.55);
\node[draw,inner sep=2](s2)at(-1.3,0){$\displaystyle\sum_{\sigma\in S_{n-\mathfrak{e}_+}}\!\!\sigma$};\draw($(s2.north west)-(0,.07)$)--($(s2.north east)-(0,.07)$);
\node[fill=white](a2o)at(-2.1,.9){$\lambda_+$};\node[fill=white,inner sep=2](b2i)at(-1.9,-.8){$\lambda_+$};\node[fill=white](b2o)at(-1,-1.15){$\lambda$};
\draw[->]($(s2.north)-(.2,0)$)to[out=90,in=0](a2o);\draw[->](b2i)to[out=0,in=-90]($(s2.south)-(.2,0)$);\draw[->]($(s2.south)+(.3,0)$)--(b2o);
\node[draw,inner sep=2](s1)at(1.3,0){$\displaystyle\sum_{\sigma\in S_{n-\mathfrak{e}_-}}\!\!\sigma$};\draw($(s1.south east)+(0,.07)$)--($(s1.south west)+(0,.07)$);
\node[fill=white](b1i)at(2.2,.9){$\lambda_-$};\node[fill=white](a1i)at(1,-1.15){$\lambda$};\node[fill=white,inner sep=2](a1o)at(2,-.8){$\lambda_-$};
\draw[->](b1i)to[out=180,in=90]($(s1.north)+(.2,0)$);\draw[->](a1i)--($(s1.south)-(.3,0)$);\draw[->]($(s1.south)+(.2,0)$)to[out=-90,in=180](a1o);
\node[draw,inner sep=2](g)at(0,1.1){$\begin{matrix}(\alpha_+)^{-1}\alpha_-\\=\gamma_c\end{matrix}$};\draw($(g.south east)-(.07,0)$)--($(g.north east)-(.07,0)$);
\draw[->]($(s1.north)-(.3,0)$)to[out=90,in=0](g.east);\draw[->](g.west)to[out=180,in=90]($(s2.north)+(.3,0)$);
\node at(-1.1,1.2){$\lambda$};\node at(1.1,1.2){$\lambda$};
\end{tikzpicture}
\end{equation*}
\caption{The equality allowing to move the permutation coupons from edges to corners. It is an application of the identity from Example~\ref{ex:ribbon_composition}. For decluttering we dropped a level of subscripts in $\alpha,\mathfrak{e},\lambda$, and omitted the superscripts in $\alpha^{(i)}_{e_+},\alpha^{(t)}_{e_-}$.}
\end{subfigure}
\caption{Relocation of the permutation coupons appearing in equation (\ref{eq:standardise_summation}) from edge sides to corners. If edge orientations at the corner differ then $\alpha^{(i)}$ may have to be replaced with $\alpha^{(t)}$ or vice versa.}
\label{fig:join_at_corner}
\end{figure}
By construction (\ref{eq:alpha_beta}) of $\alpha_e^{(i)},\alpha_e^{(t)}$, the resulting tuple of corner permutations is bilaterally $T$-matching $\vec{\mathcal{J}}$. Hence it must equal $\vec\gamma$.

After the substitution (\ref{eq:standardise_summation}) the colour of the coupon on edge $e$ is a sum over the standard subgroup $S_{n-\mathfrak{e}_e}$. Lemma \ref{lmm:swap} tells us how to perform this summation. Applying its main equality to each edge coupon turns $\mathfrak{R}$-graphs into $\mathfrak{P}$-graphs
\begin{align}\begin{split}\label{eq:summation_edges}
\tr\mathfrak{R}\left(\vec\lambda,\left(\sum_{\sigma\in P_e}\sigma\right)_e\right)=&\prod_{e\in E(\Sigma)}(n-\mathfrak{e}_e)!\times\\
&\times\sum_{\mu_e\vdash n-\mathfrak{e}_e}\quad\prod_{e\in E(\Sigma)}\chi_{\mu_e}(1)^{-1}\cdot\tr\mathfrak{P}\left(\vec\lambda,\vec\mu,\vec\gamma\right),
\end{split}\end{align}
where the prefactor comes from replacing sums with averages.

Combining equations (\ref{eq:double_counting}, \ref{eq:sum_Rgraphs}, \ref{eq:summation_edges}) proves the main identity of Proposition \ref{prop:count_E} when $\Sigma$ is orientable.

The case of non-orientable $\Sigma$ requires some small adjustments. At some edges $e$ the orientations of neighbouring faces will induce the same direction of $e$, as drawn on the left of Figure~\ref{sfig:Pg_e_no}. There, at the stage of equation (\ref{eq:standardise_summation}) we will either not swap the top and bottom bases of $\alpha^{(i)}_e,\alpha^{(t)}_e$ (if the orientation of $e$ agrees with the one induced by the faces, as in \ref{sfig:Pg_e_no}), or swap them on all 4 coupons (otherwise). Furthermore, at those edges we cannot apply Lemma \ref{lmm:swap} as-is to obtain equation (\ref{eq:summation_edges}), but instead we have to modify it according to equation (\ref{eq:clashing_resolution}). This produces the extra coupons coloured with $B\colon\mu\to\mu^*$ from Section~\ref{sssec:FrobSch}, in the configuration shown on the right of Figure \ref{sfig:Pg_e_no}. These two are the only differences from the orientable case, and with the two corrections the proof carries over to the non-orientable $\Sigma$.
\end{proof}

\subsection{Proof of Proposition \ref{prop:count_F} and the count of $F(T)$}\label{ssec:pf_F}

Proof of Proposition \ref{prop:count_F} consists of four steps. In the first warmup step we notice that representations of $S_{n-\mathfrak{v}}$ can be factored out. Next steps concern faces. In the second we modify the corner permutations, so that in their domains the numeric labels of edges land in $[n-\mathfrak{v}+1,n-\mathfrak{e}_e]$ and of faces in $[n-\mathfrak{f}_f+1,n]$. This changes bilateral $T$-matching to $T$-matching. Then, as in the first step, certain linear operators factor as tensor products, and we use it in the third step to ``split off'' a wire coloured $\lambda'_f\setminus\lambda_f$ from each wire circling around a face $f$. The fourth step is a change of summation, and gives rise to a factor $\prod_{f\in\Fspx(T)}n^{\underline{\mathfrak{f}_f}}$ and another one for $\Forb$.
\begin{proof}[Proof of Proposition \ref{prop:count_F}]
In the first step we examine the restrictions of representations to setwise stabilisers of $[n-\mathfrak{v}]$. Corner permutations $\vec\gamma'$ are supported on $[n-\mathfrak{v}+1,n]$, so their action preserves the decomposition of face representations
\begin{equation*}
\tn{Res}^{S_n}_{S_{n-\mathfrak{v}}\times S_{[n-\mathfrak{v}+1,n]}}\lambda'_f=\bigoplus_{\eta\vdash n-\mathfrak{v}}\eta\otimes(\lambda'_f\setminus\eta).
\end{equation*}
Edge-crossing representations similarly decompose as
\begin{equation*}
\tn{Res}^{S_{n-\mathfrak{e}_e}}_{S_{n-\mathfrak{v}}\times S_{[n-\mathfrak{v}+1,n-\mathfrak{e}_e]}}\mu_e=\bigoplus_{\eta\vdash n-\mathfrak{v}}\eta\otimes(\mu_e\setminus\eta).
\end{equation*}
The operators $P_{\mu_e}^{\lambda'_f},I_{\mu_e}^{\lambda'_f}$ are block-diagonal under these decompositions, so
\begin{equation}\label{eq:separate_eta}
\tr\mathfrak{P}(\vec\lambda',\vec\mu,\vec\gamma')=\sum_{\eta\vdash n-\mathfrak{v}}\tr\mathfrak{P}\left(\eta\otimes(\vec\lambda'\setminus\eta),\eta\otimes(\vec\mu\setminus\eta),\vec\gamma'\right),
\end{equation}
where we use a shorthand $\eta\otimes(\vec\mu\setminus\eta)=\left(\eta\otimes(\mu_e\setminus\eta)\right)_{e\in E(\Sigma)}$, and similarly for faces. Both the actions of $\vec\gamma'$ and the blocks of operators $P_{\mu_e}^{\lambda'_f},I_{\mu_e}^{\lambda'_f}$ are of the form $\tn{id}_\eta\otimes(\text{some linear map between }\lambda'_f\setminus\eta\text{ or }\mu_e\setminus\eta)$ that ``does nothing on $\eta$''. Therefore, in the orientable case, from each $\mathfrak{P}$-graph we can ``split off'' a wire coloured $\eta$ and winding around the vertex, namely
\begin{equation}\label{eq:factor_eta}
\mathfrak{P}\left(\eta\otimes(\vec\lambda'\setminus\eta),\eta\otimes(\vec\mu\setminus\eta),\vec\gamma'\right)=\mathfrak{P}(\vec\lambda'\setminus\eta,\vec\mu\setminus\eta,\vec\gamma')\sqcup\begin{tikzpicture}[baseline=-2]\node[font=\tiny](l)at(0,0){$\eta$};\draw[->]($(l.north)-(0,.05)$)arc(180:0:.1)--($(l.south)+(.2,.15)$);\draw($(l.south)+(0,.05)$)arc(-180:0:.1)--++(0,.1);\end{tikzpicture}.
\end{equation}
Trace of the loop coloured $\eta$ with no coupons is $\dim\eta=\chi_\eta(1)$, so putting together (\ref{eq:separate_eta}) and the trace of (\ref{eq:factor_eta}) gives
\begin{align}\begin{split}\label{eq:split_v}
\tr\mathfrak{P}(\vec\lambda',\vec\mu,\vec\gamma')=\sum_{\eta\vdash n-\mathfrak{v}}\chi_\eta(1)\tr\mathfrak{P}(\vec\lambda'\setminus\eta,\vec\mu\setminus\eta,\vec\gamma').
\end{split}\end{align}
Notice that only the summands with $\eta\!\in\!M(\vec\mu)$ contribute, since if $\eta\!\nsubseteq\!\mu_e$ for some $e$ then the subspace $\mu_e\!\setminus\!\eta$ is zero-dimensional and the corresponding $\mathfrak{P}$-graph vanishes. Equation (\ref{eq:split_v}) parallels \cite[equation (5.17)]{asymptotic_stats_random_covering_surfaces_MageePuder23}, where the phrasing of this idea is that none of the linear operators moves numbers $[n\!-\!\mathfrak{v}]$ of a Young tableau.

In the non-orientable case, we again obtain a loop coloured $\eta$ winding around the vertex, but this time it contains some coupons coloured $B^{\pm 1}$ coming from the edges looking like in Figure \ref{sfig:Pg_e_no} (equation (\ref{eq:repSn_bilinear_form}) shows that the bilinear form $B$ factors compatibly with the tensor product). Since all representations of the symmetric group have Frobenius--Schur indicator $\nu(\eta)=1$ (see Section~\ref{sec:rep_Sn}), we can freely swap top and bottom bases and then cancel these coupons. After this extra operation we arrive back at equation (\ref{eq:split_v}).

Now we move on to faces. The second step is to relate the corner permutations $\vec\gamma,\vec\gamma'$ respectively $T$-matching the numbering $\vec{\mathcal{J}}$ and bilaterally $T$-matching $\vec{\mathcal{J}}'$. For each edge-face incidence $\iota\colon e\hookrightarrow f$ of $\Sigma$, pick an auxiliary labelling of covered sides $\mathcal{K}_\iota\colon O_\iota\to[n-\mathfrak{f}_f+1,n]$. Then $\mathcal{J}\sqcup\mathcal{K}$ is a labelling of all sides of whole edges (both exposed and covered, but possibly differing between the two sides of an edge). Set
\begin{equation*}
\delta_\iota\defeq(\mathcal{J}_\iota\sqcup\mathcal{K}_\iota)\circ(\mathcal{J}'_e)^{-1}\in S_{[n-\mathfrak{e}_e+1,n]}.
\end{equation*}
In words, $\delta_\iota$ converts $\mathcal{J}'$-labels of exposed and covered sides in $S_\iota\cup O_\iota$ to their $(\mathcal{J}\sqcup\mathcal{K})$-labels, and leaves invariant the labels of hanging half-edges. For a corner $e_+\corner e_-$ define the permutation
\begin{equation*}
\gamma''_c\defeq\delta_{e_+\hookrightarrow f}\circ\gamma'_c\circ\delta_{e_-\hookrightarrow f}^{-1}\in S_{[n-\mathfrak{v}+1,n]}.
\end{equation*}

\begin{figure}[h]
\begin{subfigure}{\textwidth}
\centering
\begin{tikzpicture}
\node at(.2,0){$f$};\node at(-1,0){$e$};\node at(-.5,.5){$c_-$};\node at(-.5,-.5){$c_+$};
\draw(.3,.7)--(-.8,.7)--(-.8,-.7)--(.3,-.7);
\draw[->](.2,.3)arc(90:270:.3);
\end{tikzpicture}
\caption{Incidence $\iota\colon e\hookrightarrow f$ and its two adjacent corners $c_\pm$.}
\label{sfig:e_f_c_incidence}
\end{subfigure}
\begin{subfigure}{\textwidth}
\begin{equation*}
\begin{tikzpicture}[baseline=-1]
\node[draw](p)at(0,.8){$P$};\draw($(p.north west)-(0,.07)$)--($(p.north east)-(0,.07)$);
\node[draw](i)at(0,-.8){$I$};\draw($(i.north east)-(0,.07)$)--($(i.north west)-(0,.07)$);
\node(mo)at(-1,.4){$\mu\!\setminus\!\eta$};\node(mi)at(-1,-.3){$\mu\!\setminus\!\eta$};
\node[draw](cm)at(1.2,1.5){$\delta\gamma_-$};\draw($(cm.north east)-(.07,0)$)--($(cm.south east)-(.07,0)$);
\node[draw](cp)at(1.2,-1.5){$\gamma_+\delta^{-1}$};\draw($(cp.south west)+(.07,0)$)--($(cp.north west)+(.07,0)$);
\node(li)at(2.5,1.5){$\lambda'\!\setminus\!\eta$};\node(lo)at(2.5,-1.5){$\lambda'\!\setminus\!\eta$};
\draw[->]($(p.south)+(.15,0)$)--($(i.north)+(.15,0)$);\draw[->]($(p.south)-(.15,0)$)to[out=-90,in=0](mo);\draw[->](mi)to[out=0,in=90]($(i.north)-(.15,0)$);\draw[->](cm)to[out=180,in=90](p);\draw[->](i)to[out=-90,in=180](cp);\draw[->](li)--(cm);\draw[->](cp)--(lo);
\node at(0,1.7){$\lambda'\!\setminus\!\eta$};\node at(0,-1.6){$\lambda'\!\setminus\!\eta$};\node[rotate=-90]at(.4,0){$\lambda'\!\setminus\!\mu$};
\end{tikzpicture}
\qquad=\qquad
\begin{tikzpicture}[baseline=-1]
\node[draw](dp)at(.15,.4){$\delta$};\draw($(dp.north west)-(0,.07)$)--($(dp.north east)-(0,.07)$);
\node[draw](dm)at(.15,-.4){$\delta^{-1}$};\draw($(dm.north west)-(0,.07)$)--($(dm.north east)-(0,.07)$);
\node[draw](p)at(0,1.2){$P$};\draw($(p.north west)-(0,.07)$)--($(p.north east)-(0,.07)$);
\node[draw](i)at(0,-1.2){$I$};\draw($(i.north east)-(0,.07)$)--($(i.north west)-(0,.07)$);
\node[draw](cm)at(1.2,1.7){$\gamma_-$};\draw($(cm.north east)-(.07,0)$)--($(cm.south east)-(.07,0)$);
\node[draw](cp)at(1.2,-1.7){$\gamma_+$};\draw($(cp.south west)+(.07,0)$)--($(cp.north west)+(.07,0)$);
\node(mo)at(-1,.8){$\mu\!\setminus\!\eta$};\node(mi)at(-1,-.8){$\mu\!\setminus\!\eta$};\node(li)at(2.5,1.7){$\lambda'\!\setminus\!\eta$};\node(lo)at(2.5,-1.7){$\lambda'\!\setminus\!\eta$};
\draw[->]($(p.south)+(.15,0)$)--(dp);\draw[->](dp)--(dm);\draw[->](dm)--($(i.north)+(.15,0)$);\draw[->]($(p.south)-(.15,0)$)to[out=-90,in=0](mo);\draw[->](mi)to[out=0,in=90]($(i.north)-(.15,0)$);\draw[->](cm)to[out=180,in=90](p);\draw[->](i)to[out=-90,in=180](cp);\draw[->](li)--(cm);\draw[->](cp)--(lo);
\node at(-.3,1.75){$\lambda'\!\setminus\!\eta$};\node at(-.3,-1.7){$\lambda'\!\setminus\!\eta$};\node[font=\tiny]at(.55,-.85){$\lambda'\!\setminus\!\mu$};\node[font=\tiny]at(.55,0){$\lambda'\!\setminus\!\mu$};\node[font=\tiny]at(.55,.8){$\lambda'\!\setminus\!\mu$};
\end{tikzpicture}
\end{equation*}
\caption{Relocating $\delta^{\pm 1}_\iota$. For decluttering we dropped a level of subscripts on $\delta,\gamma,\mu,\lambda',P,I$.}
\label{sfig:slide_delta}
\end{subfigure}
\caption{Equalities between ribbon graphs with $\vec\gamma''$ and $\vec\gamma'$ proving (\ref{eq:tr_not_affected}).}
\label{fig:tr_not_affected}
\end{figure}
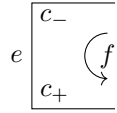
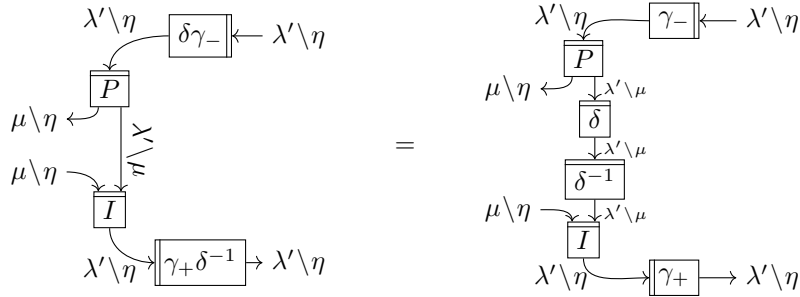

We claim that the tuples $\vec\gamma''$ and $\vec\gamma'$ give the same trace, i.e.
\begin{equation}\label{eq:tr_not_affected}
\tr\mathfrak{P}\left(\vec\lambda'\setminus\eta,\vec\mu\setminus\eta,\vec\gamma''\right)=\tr\mathfrak{P}\left(\vec\lambda'\setminus\eta,\vec\mu\setminus\eta,\vec\gamma'\right).
\end{equation}
To see this, consider an incidence $\iota\colon e\hookrightarrow f$ with $f\in\Fspx(\Sigma)$, and let $c_\pm$ be the two adjacent corners, with $f$ oriented away from $c_-$ towards $c_+$, as in Figure \ref{sfig:e_f_c_incidence}. The corresponding part of the $\mathfrak{P}$-graph is shown on the LHS of the equation in Figure \ref{sfig:slide_delta}. Since $\delta_\iota$ commutes with $S_{[n-\mathfrak{v}+1,n-\mathfrak{e}_e]}$, we can slide the coupons coloured $\delta^{\pm 1}_\iota$ past $P,I$ to the middle of the wire parallel to $e$, where they cancel. This equality is shown in Figure \ref{sfig:slide_delta}. Similar identity holds when $f$ is an orbigon, except that we relocate $\delta^{\pm 1}$ from both sides of the same corner coupon (see Figure \ref{sfig:Pg_orb} for the relevant part of the $\mathfrak{P}$-graph). Applying this equality at every edge-face incidence proves equation (\ref{eq:tr_not_affected}).

Let us examine the permutations $\vec\gamma''$ closer. By construction they are bilaterally $T$-matching $\mathcal{J}\sqcup\mathcal{K}$ (this numbering may differ on two sides of an edge, but the compatibility condition from Definition~\ref{def:T_compat} holds including covered sides). Another tuple satisfying this condition is $\vec\gamma$ extended by permutations ``$T$-matching the covered sides''. Thus by uniqueness we must have
\begin{equation}\label{eq:gamma_juxtaposition}
\gamma''_c=\gamma_c\sqcup\beta_c,
\end{equation}
where $\beta_c\in S_{[n-\mathfrak{f}_{f(c)}+1,n]}$ describes ``the next covered side across the corner'', i.e. is the function $\mathcal{T}_c$ from (\ref{eq:corner_traverse}) for covered sides $\vec{O}$ and conjugated by $\vec{\mathcal{K}}$.

In the third step we factor the graphs, similarly as in equation (\ref{eq:split_v}). We have decompositions
\begin{align*}
\tn{Res}_{S_{[n-\mathfrak{v}+1,n-\mathfrak{f}_f]}\times S_{[n-\mathfrak{f}_f+1,n]}}^{S_{[n-\mathfrak{v}+1,n]}}\lambda_f'\setminus\eta&=\bigoplus_{\lambda_f\vdash n-\mathfrak{f}_f}(\lambda_f\setminus\eta)\otimes(\lambda'_f\setminus\lambda_f)\\
\lambda_f'\setminus\mu_e&=\bigoplus_{\lambda_f\vdash n-\mathfrak{f}_f}(\lambda_f\setminus\mu_e)\otimes(\lambda'_f\setminus\lambda_f).
\end{align*}
Again, $\gamma_c\sqcup\beta_c$ act as $\gamma_c\otimes\beta_c$ under the first one, and operators $P,I$ are block-diagonal (blocks indexed by $\lambda_f$). Therefore we can split off wires coloured $\lambda'_f\setminus\lambda_f$ together with coupons coloured $\beta_c$ (we already saw splitting coupons in equation (\ref{eq:prod_action_rg})). Each simplex $f$ gives rise a loop with three coupons spelling the relator $W_f(\vec\beta)$ from presentation (\ref{eq:simplicial_pres}), and each orbigon $f$ with corner $c$ produces a loop with a single coupon coloured $\beta_c$. Therefore
\begin{align}\begin{split}\label{eq:splitoff_face}
\mathfrak{P}\left(\vec\lambda'\setminus\eta,\vec\mu\setminus\eta,\overrightarrow{\gamma\sqcup\beta}\right)=&\sum_{\lambda_f\vdash n-\mathfrak{f}_f}\mathfrak{P}\left(\vec\lambda\setminus\eta,\vec\mu\setminus\eta,\vec\gamma\right)\sqcup\\
&\sqcup\bigsqcup_{f\in\Fspx(\Sigma)}\begin{tikzpicture}[baseline=0]\node[font=\tiny,draw,inner sep=1pt](l)at(0,0){$W_f\!(\vec\beta)$};\draw(l.north)arc(180:0:.25)--($(l.south)+(.5,0)$);\draw[->]($(l.south)+(.5,0)$)arc(0:-180:.25);\node[font=\tiny]at(.3,.6){$\lambda'_f\!\!\setminus\!\!\lambda_f$};\end{tikzpicture}\sqcup\bigsqcup_\iterorb\begin{tikzpicture}[baseline=0]\node[font=\tiny,draw,inner sep=1pt](l)at(0,0){$\beta_{c(f)}$};\draw(l.north)arc(180:0:.25)--($(l.south)+(.5,0)$);\draw[->]($(l.south)+(.5,0)$)arc(0:-180:.25);\node[font=\tiny]at(.3,.6){$\lambda'_f\!\!\setminus\!\!\lambda_f$};\end{tikzpicture}.
\end{split}\end{align}
Since the permutations $\vec\beta$ represent ``the next numeric label across the corner'', applying the word $W_f(\vec\beta)$ for $f\in\Fspx(\Sigma)$ letter-by-letter just cycles the numeric labels of covered sides around the interior of each simplex in $F_f$, so
\begin{equation}\label{eq:W_covered_id}
W_f(\vec\beta)=1.
\end{equation}
Therefore the trace of a loop in (\ref{eq:splitoff_face}) coming from a simplex $f$ is the dimension of $\lambda'_f\setminus\lambda_f$, and the trace of (\ref{eq:splitoff_face}) equals
\begin{align*}
\tr\mathfrak{P}\left(\vec\lambda'\setminus\eta,\vec\mu\setminus\eta,\overrightarrow{\gamma\sqcup\beta}\right)=&\sum_{\lambda_f\vdash n-\mathfrak{f}_f}\tr\mathfrak{P}\left(\vec\lambda\setminus\eta,\vec\mu\setminus\eta,\vec\gamma\right)\times\\
&\times\prod_{f\in\Fspx(\Sigma)}\chi_{\lambda'_f\setminus\lambda_f}(1)\prod_\iterorb\chi_{\lambda'_f\setminus\lambda_f}(\beta_{c(f)}).
\end{align*}
Summing this over all $\lambda'_f$ with suitable weights gives
\begin{align}\begin{split}\label{eq:resummed_f}
\sum_{\lambda'_f\vdash n}\prod_{f\in\Fspx(\Sigma)}\chi_{\lambda'_f}(1)&\prod_\iterorb\chi_{\lambda'_f}^{m\tn{-tors}}\cdot\tr\mathfrak{P}\left(\vec\lambda'\setminus\eta,\vec\mu\setminus\eta,\overrightarrow{\gamma\sqcup\beta}\right)=\\
=&\sum_{\lambda_f\vdash n-\mathfrak{f}_f}\tr\mathfrak{P}(\vec\lambda\setminus\eta,\vec\mu\setminus\eta,\vec\gamma)\times\\
&\times\prod_{f\in\Fspx}\left(\sum_{\lambda'\vdash n}\chi_{\lambda'}(1)\cdot\tn{dim}(\lambda'\setminus\lambda_f)\right)\times\\
&\times\prod_\iterorb\left(\sum_{\lambda'\vdash n}\chi_{\lambda'_f}^{m\tn{-tors}}\chi_{\lambda'\setminus\lambda_f}(\beta_{c(f)})\right).
\end{split}\end{align}
Again, any summand where $\lambda_f\notin N_f(\vec\mu)$ at some $f$ vanishes, because the corresponding space $\lambda_f\setminus\mu_e$ is zero-dimensional. We can recognise that the product of $\chi_{\lambda'_f}(1)$ and $\chi_{\lambda'_f}^{m\tn{-tors}}$ on the LHS is $R(\emptyset,\vec\emptyset,\vec\lambda')$.

Summations in the last two lines can be performed with Lemma \ref{lmm:change_sum_tors}. Applying it with $G=S_n,H=S_{n-\mathfrak{f}_f},U=\lambda_f,x=1,m=1$ tells us
\begin{equation*}
\sum_{\lambda'\vdash n}\chi_{\lambda'}(1)\cdot\dim(\lambda'_f\setminus\lambda)=n^{\underline{\mathfrak{f}_f}}\cdot\chi_\lambda(1).
\end{equation*}
Similarly as we argued in equation (\ref{eq:W_covered_id}), the cycles of $\beta_{c(f)}$ for $f\in\Forb(\Sigma)$ reflect the circles of covered sides facing the interiors of faces in $F_f$, so $\beta_c^m=1$. Therefore, applying Lemma~\ref{lmm:change_sum_tors} with $G=S_n,H=S_{n-\mathfrak{f}_f},U=\lambda_f,x=1$ gives
\begin{equation*}
\tau_m(S_n)\sum_{\lambda'\vdash n}\chi_{\lambda'}^{m\tn{-tors}}\chi_{\lambda'\setminus\lambda}(\beta)=n^{\underline{\mathfrak{f}_f}}\cdot\tau_m(S_{n-\mathfrak{f}_f})\cdot\chi_\lambda^{m\tn{-tors}}.
\end{equation*}
These identities allow us to rewrite the equation (\ref{eq:resummed_f}) as
\begin{align*}
\sum_{\lambda'_f\vdash n}R(\emptyset,\vec\emptyset,\vec\lambda')\tr\mathfrak{P}\left(\vec\lambda'\!\setminus\!\eta,\vec\mu\!\setminus\!\eta,\overrightarrow{\gamma\!\sqcup\!\beta}\right)=&\hspace*{-5pt}\prod_{f\in\Fspx(\Sigma)}n^{\underline{\mathfrak{f}_f}}\cdot\hspace*{-5pt}\prod_\iterorb\frac{n^{\underline{\mathfrak{f}_f}}\tau_m(S_{n-\mathfrak{f}_f})}{\tau_m(S_n)}\times\\
&\times\hspace*{-5pt}\sum_{\lambda_f\in N_f(\vec\mu)}\hspace*{-5pt}R(\emptyset,\vec\emptyset,\vec\lambda)\tr\mathfrak{P}\left(\vec\lambda\setminus\eta,\vec\mu\setminus\eta,\vec\gamma\right).
\end{align*}
Putting this equality together with (\ref{eq:split_v}) and remembering (\ref{eq:tr_not_affected}, \ref{eq:gamma_juxtaposition}) and the definition of $R$ from Theorem~\ref{thm:count_tangles}, we obtain the desired equality of Proposition \ref{prop:count_F}.
\end{proof}

\section{Asymptotics}\label{sec:asymptotics}

\subsection{Auxiliary results}

Here we collect some facts and estimates that will be useful for the asymptotic analysis of Theorem~\ref{thm:count_tangles} and deducing the master Theorem~\ref{thm:E_combi_emb} in Section~\ref{ssec:pf_E_combi_emb}. The first two results concern torsion in symmetric groups. We write $\sim$ to mean the ratio converging to 1.
\begin{lemma}\label{lmm:Sn_tors_growth}
For every $d,m\in\mathbb{N}_+$ we have $\frac{\tau_m(S_{n+d})}{\tau_m(S_n)}\sim n^{d(1-m^{-1})}$ as $n\to\infty$.
\end{lemma}
\begin{lemma}\label{lmm:tors_avg_char}
Let $m\in\mathbb{N}_+$ and pick $\varepsilon>0$. For sufficiently large $n\in\mathbb{N}$, and $\lambda\vdash n$, we have $|\chi_\lambda^{m\tn{-tors}}|\leqslant\chi_\lambda(1)^{\frac{1}{m}+\varepsilon}$.
\end{lemma}
Lemma~\ref{lmm:Sn_tors_growth} is an immediate consequence of \cite[Theorem 6 in \S5.2(c)]{finite_gp_actions_asymptotic_expansion_ePz_Mueller97}, while Lemma \ref{lmm:tors_avg_char} follows readily from \cite[Theorem B.(ii)]{char_theory_symm_gps_subgp_growth_fuchsian_gps_random_walks_MuellerSchlagePuchta07}.

A distinction between stable and thick Young diagrams will be important in the upcoming Section~\ref{ssec:pf_E_combi_emb}. The stable diagrams are of the form
\begin{align*}
\lambda[n]\defeq&(n-m,\lambda_1,\dots,\lambda_k)\vdash n\quad\text{for}\quad\lambda=(\lambda_1,\dots,\lambda_k)\vdash m,n\geqslant m+\lambda_1,
\end{align*}
where $\lambda$ is considered small in a precise sense depending on the context. Pictorially, $\lambda[n]$ is obtained by stacking $\lambda$ under a single row, of length such that the total number of boxes is $n$ (we follow the convention in \cite{partition_alg_Kronecker_coeffs_BowmanVissherOrellana15} rather than \cite{asymptotic_stats_random_covering_surfaces_MageePuder23}). The set of thick diagrams is
\begin{equation*}
\Lambda(n,b)\defeq\left\{\lambda\vdash n\Bigg|\begin{matrix}\lambda\text{ has at least }b\text{ boxes outside the first row}\\\text{and at least }b\text{ boxes outside the first column}\end{matrix}\right\},
\end{equation*}
where $b$ is thought of as large. Note that, for $n\geqslant 2b$, the set of all diagrams can be partitioned into thick, stable, and transposed stable diagrams as
\begin{equation}\label{eq:stable_thick}
\{\lambda\vdash n\}=\Lambda(n,b)\sqcup\left\{\lambda[n]\mid\lambda\vdash k\text{ for }k<b\right\}\sqcup\left\{\lambda[n]^\vee\mid\lambda\vdash k\text{ for }k<b\right\}.
\end{equation}
With a fixed $b$ and growing parameter $n$, the stable part is a union of finitely many families of diagrams with growing first rows.

We will use the following properties of stable and thick diagrams. Directly from the hook length formula (\ref{eq:hooklen}) we obtain that
\begin{equation}\label{eq:growing_row_dim}
\dim\lambda[n]\sim\frac{\dim\lambda}{|\lambda|!}\cdot n^{|\lambda|}.
\end{equation}
The following proposition will control transposes of stable Young diagrams.
\begin{proposition}\label{prop:tall_tors_bd}
Fix a partition $\lambda=(\lambda_1,\dots,\lambda_k)$ and an even $m\in\mathbb{N}_+$. Consider partitions $(\lambda^\vee[n])^\vee=(1+\lambda_1,\dots,1+\lambda_k,1^{n-k-|\lambda|})$ of $n$ obtained by attaching $\lambda$ to a single column of boxes. There exists $c>0$ such that
\begin{equation*}
\chi_{(\lambda^\vee[n])^\vee}^{m\tn{-tors}}=O\left(\exp\left(-cn^{\frac{1}{m}}\right)\right)
\end{equation*}
as $n\to\infty$.
\end{proposition}
Proof of Proposition \ref{prop:tall_tors_bd} is deferred to Section \ref{ssec:pf_tall_tors_bd}. It will also reveal that $\chi^{m\tn{-tors}}_{\lambda[n]}$ has a series expansion in the powers of $n^{-1}$. The thick tuples will generally form error terms, and will be controlled using the following result of \cites[Proposition 2.5]{fuchsian_gps_coverings_Riemann_surfaces_subgp_growth_random_quots_walks_LiebeckShalev04}[Proposition 4.2]{Poisson_Dirichlet_distribution_rand_Belyi_surfaces_Gamburd06}.
\begin{lemma}\label{lmm:LSG}
For any fixed $b\in\mathbb{N},s>0$ we have
\begin{equation*}
\sum_{\lambda\in\Lambda(n,b)}\chi_\lambda(1)^{-s}=O_b(n^{-sb})\qquad\text{as}\qquad n\to\infty.
\end{equation*}
\end{lemma}
\begin{corollary}\label{corr:zeta_asymptotics}
Suppose $\Sigma$ is an orbifold with $\cho(\Sigma)<0$ and cone points of order $m_1,\dots,m_k$. Let $\zeta^{S_n}$ be the ``zeta constant'' from Lemma \ref{lmm:Mednykh}. Then
\begin{equation*}
\zeta^{S_n}\to1+\mathbbm{1}_{\{\text{all }m_i\text{ odd}\}}\qquad\text{as}\qquad n\to\infty.
\end{equation*}
\end{corollary}
\begin{proof}
Recall from Section~\ref{sec:rep_Sn} that the Frobenius--Schur indicator of irreducible $S_n$-representation is 1, so we can ignore $\nu(V)$ in the non-orientable formula. The set of partitions of $n$ consists of the subset $\Lambda(n,1)$ and two other partitions: single row $(n)$, which corresponds to the trivial representation of $S_n$, and single column $(1^n)$, which corresponds to the sign representation. Therefore
\begin{equation*}
\zeta^{S_n}=1+\prod_{i=1}^k\chi_\tn{sgn}^{m_i\tn{-tors}}+\sum_{\lambda\in\Lambda(n,1)}\chi_\lambda(1)^{\chi^\tn{top}(\Sigma)-k}\prod_{i=1}^k\chi_\lambda^{m_i\tn{-tors}}.
\end{equation*}
If all $m_i$ are odd, then $\chi_\tn{sgn}^{m_i\tn{-tors}}=1$, so the second summand is 1. Otherwise if $m_i$ is even, then Proposition \ref{prop:tall_tors_bd} (applied to the empty partition) tells us that $\chi_\tn{sgn}^{m_i\tn{-tors}}$ converges to 0, so the second summand vanishes in the limit. The remaining terms can be bounded with Lemma \ref{lmm:tors_avg_char} as
\begin{equation*}
\left|\sum_{\lambda\in\Lambda(n,1)}\chi_\lambda(1)^{\chi^\tn{top}(\Sigma)-k}\prod_{i=1}^k\chi_\lambda^{m_i\tn{-tors}}\right|\leqslant\sum_{\lambda\in\Lambda(n,1)}\chi_\lambda(1)^{\cho(\Sigma)+\varepsilon}
\end{equation*}
for some small $\varepsilon$ and sufficiently large $n$. By Lemma \ref{lmm:LSG}, this converges to 0.
\end{proof}

\smallskip

Finally, we will need to control the matrix coefficients of corner permutations. We achieve this using Lemma \ref{lmm:MP_skew_bd}, which is just an adaptation of \cite[Proposition 4.4]{asymptotic_stats_random_covering_surfaces_MageePuder23} to our notation.
\begin{lemma}\label{lmm:MP_skew_bd}
Suppose $\eta\subseteq\lambda$ are partitions, and $p,q\in\mathbb{N}$ with $q+|\eta|\geqslant p+|\lambda|$. Then for $n\geqslant |\lambda|+|\eta|+q+(p-q)^2$ and any $w,w'\in\tn{Tab}(\lambda[n-p]\setminus\eta[n-q])$ and $\gamma\in S_{[n-q+1,n-p]}$ we have
\begin{equation*}
|\langle w',\gamma w\rangle|\leqslant\left(\frac{(p-q)^2}{n-q-|\eta|-|\lambda|}\right)^{\frac{1}{2}\#(\tn{bot}(w')\vartriangle\gamma\cdot\tn{bot}(w))},
\end{equation*}
where $\tn{bot}$ is the set of numbers outside the top row of the skew-tableau, and $\vartriangle$ stands for the symmetric difference.
\end{lemma}
The Young skew-diagram $\lambda[n-p]\setminus\eta[n-q]$ contains a ``moving'' row of $q-p+|\eta|-|\lambda|$ boxes in the first row and the skew-diagram $\lambda\setminus\eta$ below it. By abuse of notation in Lemma~\ref{lmm:MP_skew_bd} we wrote $w,w'$ both for tableaux and their vectors, in accordance with Section~\ref{sec:rep_Sn}. We will also use a crude bound on traces of $\mathfrak{P}$-graphs.
\begin{proposition}\label{prop:univ_Pg_bd}
Let $\lambda_f\vdash n-\mathfrak{f}_f,\mu_e\vdash n-\mathfrak{e}_e,\eta\vdash n-\mathfrak{v}$. Then
\begin{equation*}
\left|\tr\mathfrak{P}(\vec\lambda\setminus\eta,\vec\mu\setminus\eta,\vec\gamma)\right|\leqslant\prod_{e\in E(\Sigma)}(\mathfrak{v}-\mathfrak{e}_e)!\cdot\prod_{e\hookrightarrow f}(\mathfrak{e}_e-\mathfrak{f}_f)!
\end{equation*}
for every choice of $\gamma_c\in S_{[n-\mathfrak{v}+1,n-\mathfrak{f}_f]}$.
\end{proposition}
\begin{proof}
Recall from Section \ref{sec:rep_Sn} that standard Young tableaux of a given shape form an orthonormal basis for the representation corresponding to the diagram. Apply the equation from the Example~\ref{ex:resolution_id} (with this basis as $\{e_i\}_i$) to all wires not adjacent to a corner coupon (i.e. the ones parallel to edges and coloured with $\lambda_f\setminus\mu_e$, and those crossing edges and coloured with $\mu_e\setminus\eta$). In the orientable case, doing this cuts the $\mathfrak{P}$-graph into a disjoint union of graphs of the form
\begin{equation}\label{eq:mat_coeff}
\begin{tikzpicture}[baseline=0]
\node[draw](c)at(0,0){$\gamma_c$};\draw($(c.south east)-(.07,0)$)--($(c.north east)-(.07,0)$);
\node(p)at(-1.5,0){$P^\lambda_{\mu_+}$};\draw(-1.8,-.7)--(-1.8,.7)--(-1.15,.7)--(-1.15,-.7)--(-1.8,-.7)(-1.22,-.7)--(-1.22,.7);
\node(i)at(1.5,0){$I^\lambda_{\mu_-}$};\draw(1.8,-.7)--(1.8,.7)--(1.15,.7)--(1.15,-.7)--(1.8,-.7)(1.73,-.7)--(1.73,.7);
\node[draw](b)at(3.8,-.5){$w_{e_-\hookrightarrow f}$};\draw($(b.south east)-(.07,0)$)--($(b.north east)-(.07,0)$);
\node[draw](d)at(-3.8,-.5){$w_{e_+\hookrightarrow f}^*$};\draw($(d.south east)-(.07,0)$)--($(d.north east)-(.07,0)$);
\node[draw](in)at(3.7,.5){$u_{v\hookrightarrow e_-}$};\draw($(in.south east)-(.07,0)$)--($(in.north east)-(.07,0)$);
\node[draw](out)at(-3.7,.5){$u_{v\hookrightarrow e_+}^*$};\draw($(out.south east)-(.07,0)$)--($(out.north east)-(.07,0)$);
\draw[->](-1.8,-.5)--(d);\draw[->](b)--(1.8,-.5);\draw[->](in)--(1.8,.5);\draw[->](-1.8,.5)--(out);\draw[->](1.15,0)--(c);\draw[->](c)--(-1.15,0);
\node at(-.7,.2){$\lambda\!\setminus\!\eta$};\node at(.7,.2){$\lambda\!\setminus\!\eta$};\node at(-2.5,-.25){$\lambda\!\setminus\!\mu_+$};\node at(2.5,-.25){$\lambda\!\setminus\!\mu_-$};\node at(-2.5,.7){$\mu_+\!\setminus\!\eta$};\node at(2.5,.7){$\mu_-\!\setminus\!\eta$};
\end{tikzpicture}
\end{equation}
for all possible choices of $w_{e\hookrightarrow f}\in\tn{Tab}(\lambda_f\setminus\mu_e)$ and $u_{v\hookrightarrow e}\in\tn{Tab}(\mu_e\setminus\eta)$. Action of $\gamma_c$ is unitary, so each matrix coefficient of the form \ref{eq:mat_coeff} is at most 1 in absolute value. Since the sizes of bases are bounded by $\#\tn{Tab}(\mu_e\setminus\eta)\leqslant(\mathfrak{v}-\mathfrak{e}_e)!$ and $\#\tn{Tab}(\lambda_f\setminus\mu_e)\leqslant(\mathfrak{e}_e-\mathfrak{f}_f)!$, the desired bound follows from the triangle inequality.

In the non-orientable case, cutting the $\mathfrak{P}$-graph at edges whose neighbourhood looks like on the LHS of Figure~\ref{sfig:Pg_e_no} will additionally produce graphs of the form
\begin{equation}\label{eq:cut_B_matcoef}
\begin{tikzpicture}[baseline=0]
\node[draw](l)at(-2,0){$u_{v\hookrightarrow e}^{(L)}$};\draw($(l.south west)+(.07,0)$)--($(l.north west)+(.07,0)$);
\node[draw](b)at(0,0){$B$};\draw($(b.south west)+(.07,0)$)--($(b.north west)+(.07,0)$);
\node[draw](r)at(2,0){$u_{v\hookrightarrow e}^{(R)}$};\draw($(r.south east)-(.07,0)$)--($(r.north east)-(.07,0)$);
\draw[->](l)--(b);\draw[->](r)--(b);
\node[font=\small]at(-1,.2){$\mu_e\!\setminus\!\eta$};\node[font=\small]at(1,.2){$\mu_e\!\setminus\!\eta$};
\end{tikzpicture}
\qquad=\qquad
\begin{cases}1&\text{if }u^{(L)}=u^{(R)}\\0&\text{if }u^{(L)}\neq u^{(R)}\end{cases}
\end{equation}
and ones with the central coupon coloured $B^{-1}$ and the direction of wires reversed. They evaluate to the model $S_{[n-\mathfrak{v}+1,n-\mathfrak{e}_e]}$-invariant bilinear form from equation (\ref{eq:repSn_bilinear_form}) in Section~\ref{sec:rep_Sn}. This just enforces the condition $u_{v\hookrightarrow e}^L=u_{v\hookrightarrow e}^R$, and the summation over bases becomes the same as in the orientable case.
\end{proof}

Finally, we will need a technical combinatorial result, whose proof we defer to Section~\ref{ssec:pf_prod_sgn_compat}.
\begin{proposition}\label{prop:prod_sgn_compat}
If all cone points of $\Sigma$ have odd order, and $\vec\gamma$ is a tuple of $T$-matching permutations, then $\prod_{c\in C(\Sigma)}\tn{sgn}(\gamma_c)=1$.
\end{proposition}

\subsection{Proof of Theorem \ref{thm:E_combi_emb}}\label{ssec:pf_E_combi_emb}

Here we finally prove the master Theorem~\ref{thm:E_combi_emb} on frequency of embeddings, by analysing the asymptotics of the exact formula from Theorem~\ref{thm:count_tangles}. The asymptotics is driven by tuples of stable Young diagrams. Controlling them is the main challenge, and will be achieved with the proposition below.
\begin{proposition}\label{prop:long_bd}
Fix some partitions $\eta^0$, $\mu^0_e$ for $e\in E(\Sigma)$, and $\lambda_f^0$ for $f\in F(\Sigma)$, such that $\eta^0\subseteq\mu^0_e$ and $\mu^0_e\subseteq\lambda^0_f$ whenever $e\hookrightarrow f$, and such that at least one is non-empty. For $n\in\mathbb{N},n\geqslant\mathfrak{v}+2\max_f|\lambda^0_f|$, define $\eta\defeq\eta^0[n-\mathfrak{v}],\mu_e\defeq\mu^0_e[n-\mathfrak{e}_e],\lambda_f\defeq\lambda^0[n-\mathfrak{f}_f]$, and let $\gamma_c\in S_{[n-\mathfrak{v}+1,n-\mathfrak{f}_{f(c)}]}$ be a tuple of $T$-matching permutations (we suppress dependence on $n$ for brevity). Then
\begin{equation*}
\left|R(\eta,\vec\mu,\vec\lambda)\cdot\tr\mathfrak{P}\left(\vec\lambda\setminus\eta,\vec\mu\setminus\eta,\vec\gamma\right)\right|=O\left(C\cdot n^{\max\{\md(T),\cho(\Sigma)\}+\varepsilon}\right)
\end{equation*}
for every $\varepsilon>0$.
\end{proposition}
With Proposition~\ref{prop:long_bd} and auxiliary results in hand we are ready to deduce the master Theorem~\ref{thm:E_combi_emb} from Theorem~\ref{thm:count_tangles}.
\begin{proof}[Proof of Theorem \ref{thm:E_combi_emb}]
By definition, we have
\begin{equation*}
\mathbb{E}_n\left[\#\{T\hookrightarrow\whS\}\right]=\frac{\#\{(\whS,\iota)|\iota\colon T\hookrightarrow\whS\in\covers\}}{\#\covers}.
\end{equation*}
We have calculated the numerator in Theorem \ref{thm:count_tangles} and the denominator in Lemma \ref{lmm:Mednykh}. Substituting these gives
\begin{align}\begin{split}\label{eq:ratio_formulae}
\mathbb{E}_n\left[\#\{T\hookrightarrow\whS\}\right]=&\frac{Q(n)}{(n!)^{1-\chi^\tn{top}(\Sigma)}\prod_i\tau_{m_i}(S_n)}\cdot\frac{1}{\zeta^{S_n}}\times\\
&\times\sum_{(\eta,\vec\mu,\vec\lambda)\in\mathcal{Y}(n)}R(\eta,\vec\mu,\vec\lambda)\tr\mathfrak{P}(\vec\lambda\setminus\eta,\vec\mu\setminus\eta,\vec\gamma)
\end{split}\end{align}
for a $T$-matching tuple of corner permutations $\vec\gamma$. The first fraction in front of the sum simplifies to
\begin{equation}\label{eq:asymptotic_drive}
\frac{n^{\underline{\mathfrak{v}}}\prod_{f\in\Fspx(\Sigma)}n^{\underline{\mathfrak{f}_f}}}{\prod_{e\in E(\Sigma)}n^{\underline{\mathfrak{e}_e}}}\cdot\prod_\iterorb\frac{n^{\underline{\mathfrak{f}_f}}\tau_m(S_{n-\mathfrak{f}_f})}{\tau_m(S_n)}
\end{equation}
Substituting the growth rate of the fraction of $\tau_m$'s from Lemma~\ref{lmm:Sn_tors_growth}, we see that the expression (\ref{eq:asymptotic_drive}) asymptotically behaves like $n^K$, where the exponent is
\begin{equation*}
K=\mathfrak{v}-\sum_e\mathfrak{e}_{e\in E(\Sigma)}+\sum_{f\in\Fspx(\Sigma)}\mathfrak{f}_f+\sum_\iterorb\frac{\mathfrak{f}_f}{m}=\cho(T),
\end{equation*}
the second equality being a repetition of equation (\ref{eq:cho_combi}). This is the leading asymptotics we were looking for. By Corollary \ref{corr:zeta_asymptotics}, $\zeta^{S_n}$ converges to $1+\mathbbm{1}_{\{\text{all }m_i\text{ odd}\}}$.

We now focus on the sum in the second line of equation (\ref{eq:ratio_formulae}). Let $b_0\in\mathbb{N}$ be a parameter whose exact value will be specified later. For $n\geqslant2b_0+\mathfrak{v}$ the summands in equation (\ref{eq:ratio_formulae}) can be partitioned, according to the length of the first row and first column of $\eta$ as in equation (\ref{eq:stable_thick}), into three groups
\begin{equation*}
\mathcal{Y}(n)=\mathcal{Y}^\tn{long}(n)\sqcup\mathcal{Y}^\tn{tall}(n)\sqcup\mathcal{Y}^\tn{thick}(n)
\end{equation*}
where
\begin{align*}
\mathcal{Y}^\tn{long}(n)&\defeq\{(\eta,\vec\mu,\vec\lambda)\ |\ \eta\text{ has}<b_0\text{ boxes outside the first row}\}\\
\mathcal{Y}^\tn{tall}(n)&\defeq\{(\eta,\vec\mu,\vec\lambda)\ |\ \eta\text{ has}<b_0\text{ boxes outside the first column}\}\\
\mathcal{Y}^\tn{thick}(n)&\defeq\{(\eta,\vec\mu,\vec\lambda)\ |\ \eta\in\Lambda(n-\mathfrak{v},b_0)\}.
\end{align*}

Let us first deal with the sum over $\mathcal{Y}^\tn{thick}$, which is the error term. Fix a small $\varepsilon>0$. Consider a tuple $(\eta,\vec\mu,\vec\lambda)\in\mathcal{Y}^\tn{thick}(n)$ of Young diagrams. We have $\chi_{\mu_e}(1)\geqslant\chi_\eta(1)$, since $\eta\subseteq\mu_e$. Applying this observation and Lemma \ref{lmm:tors_avg_char} implies that
\begin{equation*}
\left|R(\eta,\vec\mu,\vec\lambda)\right|\leqslant\chi_\eta(1)^{1-\#E(\Sigma)}\prod_{f\in\Fspx(\Sigma)}\chi_{\lambda_f}(1)\prod_\iterorb\chi_{\lambda_f}(1)^{\frac{1}{m}+\varepsilon}.
\end{equation*}
The Young diagram $\lambda_f$ is obtained by stacking $\mathfrak{v}-\mathfrak{f}_f$ boxes on top of $\eta$. By Proposition~\ref{prop:repSn_OV} and Frobenius reciprocity, the $S_{n-\mathfrak{f}_f}$-representation $\lambda_f$ is contained in the representation induced from $\eta$, so
\begin{equation*}
\chi_{\lambda_f}(1)\leqslant\dim\tn{Ind}^{S_{n-\mathfrak{f}_f}}_{S_{n-\mathfrak{v}}}\eta=[S_{n-\mathfrak{f}_f}:S_{n-\mathfrak{v}}]\cdot\dim\eta\leqslant n^{\mathfrak{v}-\mathfrak{f}_f}\chi_\eta(1).
\end{equation*}
Together these bounds imply, for a larger but still infinitesimal $\varepsilon$ and some $K_1$,
\begin{equation}\label{eq:bd_B}
\left|R(\eta,\vec\mu,\vec\lambda)\right|\leqslant n^{K_1}\cdot\chi_\eta(1)^{\cho(\Sigma)+\varepsilon}.
\end{equation}
Applying Proposition \ref{prop:univ_Pg_bd} to a ``thick'' term tells us that
\begin{equation}\label{eq:cut_Pg_tr_bd}
\tr\mathfrak{P}(\vec\lambda\setminus\eta,\vec\mu\setminus\eta,\vec\gamma)=O(1).
\end{equation}
Let us group all of the summands from $\mathcal{Y}^\tn{thick}(n)$ according to the value of $\eta$. Since there are at most $k+1$ ways to stack a new box on a diagram consisting of $k$ boxes, each partition $\eta$ belongs to at most $n^{K_2}$ tuples in $\mathcal{Y}^\tn{thick}(n)$, for some $K_2$ depending on $T$ but not $n$. Combining this observation with estimates (\ref{eq:bd_B}, \ref{eq:cut_Pg_tr_bd}) gives
\begin{equation}\label{eq:final_error_bd}
\sum_{(\eta,\vec\mu,\vec\lambda)\in\mathcal{Y}^\tn{thick}(n)}\hspace*{-15pt}R(\eta,\vec\mu,\vec\lambda)\tr\mathfrak{P}(\vec\lambda\setminus\eta,\vec\mu\setminus\eta,\vec\gamma)=O\left(n^{K_1+K_2}\hspace*{-15pt}\sum_{\eta\in\Lambda(n-\mathfrak{v},b_0)}\hspace*{-15pt}\chi_\eta(1)^{\cho(\Sigma)+\varepsilon}\right).
\end{equation}
We are ready to specify the value of $b_0$. We pick it so that $b_0(\cho(\Sigma)+\varepsilon)<-K_1-K_2$. Then bound (\ref{eq:final_error_bd}) and Lemma \ref{lmm:LSG} together guarantee that the contribution of ``thick'' terms becomes negligible as $n$ grows.

Second, let us consider the ``tall'' terms. Note that elements of $\mathcal{Y}^\tn{tall}(n)$ are just the transposes of the ``long'' tuples in $\mathcal{Y}^\tn{long}(n)$. On the level of representations, transposing corresponds to tensoring with the 1-dimensional sign representation (see Section~\ref{sec:rep_Sn}). This does not change the dimension, and if all $m_i$ are odd then $m_i$-torsion permutations are even, so
\begin{equation*}
R(\eta^\vee,\vec\mu^\vee,\vec\lambda^\vee)=R(\eta,\vec\mu,\vec\lambda).
\end{equation*}
The colour of each corner coupon (as a linear operator) is multiplied by the sign of its permutation, so if all $m_i$ are odd then
\begin{align*}
\tr\mathfrak{P}\left((\vec\lambda\setminus\eta)^\vee,(\vec\mu\setminus\eta)^\vee,\vec\gamma\right)&=\tr\mathfrak{P}(\vec\lambda\setminus\eta,\vec\mu\setminus\eta,\vec\gamma)\cdot\prod_{c\in C(\Sigma)}\tn{sgn}(\gamma_c)=\\
&=\tr\mathfrak{P}(\vec\lambda\setminus\eta,\vec\mu\setminus\eta,\vec\gamma),
\end{align*}
where the last equality is Proposition~\ref{prop:prod_sgn_compat}. Therefore, if all $m_i$ are odd then the contribution of $\mathcal{Y}^\tn{tall}(n)$ to equation (\ref{eq:ratio_formulae}) is the same as the contribution of $\mathcal{Y}^\tn{long}(n)$. Otherwise, if at least one $m_i$ is even, then Proposition \ref{prop:tall_tors_bd} tells us that the corresponding $\chi_{\lambda_f}^{m_i\tn{-tors}}=O(\exp(-cn^{\frac{1}{m_i}}))$, which dominates the at-most polynomial growth of other factors in $R$ (equation (\ref{eq:growing_row_dim})) and uniformly bounded $\tr\mathfrak{P}$ (Proposition \ref{prop:univ_Pg_bd}). Therefore, in this case all ``tall'' terms converge to 0. Note that this conditional doubling precisely cancels the effect of $\tfrac{1}{\zeta^{S_n}}$.

Finally, let us deal with the ``long'' diagrams. These are the ones contributing to the asymptotics. Elements of $\mathcal{Y}^\tn{long}(n)$ are
\begin{equation*}
\left(\eta^0[n-\mathfrak{v}],\left(\mu^0_e[n-\mathfrak{e}_e]\right)_{e\in E(\Sigma)},\left(\lambda^0_f[n-\mathfrak{f}_f]\right)_{f\in F(\Sigma)}\right)
\end{equation*}
where $(\eta^0,\vec{\mu^0},\vec{\lambda^0})$ range over some finite set independent of $n$. One of them has a simple description: if all $\eta^0=\mu^0_e=\lambda^0_f=\emptyset$ are empty, then every Young diagram in the corresponding tuple in $\mathcal{Y}^\tn{long}(n)$ is a single row of boxes, and by Section~\ref{sec:rep_Sn} all the representations colouring wires of the $\mathfrak{P}$-graph are trivial. Therefore for this tuple we have $R=\tr\mathfrak{P}=1$. The contribution of all other tuples in $\mathcal{Y}^\tn{long}(n)$ is bounded by Proposition~\ref{prop:long_bd} as $O(n^{\max\{\md(T),\cho(\Sigma)\}+\varepsilon})$. This tells us that
\begin{equation*}
\mathbb{E}_n\left[\#\{T\hookrightarrow\whS\}\right]=n^{\cho(T)}\left(1+o(1)\right)\left(1+O\left(n^{\max\{\md(T),\cho(\Sigma)\}+\varepsilon}\right)\right),
\end{equation*}
which implies both statements of Theorem~\ref{thm:E_combi_emb}.
\end{proof}

It remains to prove Proposition~\ref{prop:long_bd}. The strategy is similar to ``cutting'' the $\mathfrak{P}$-graph into corner graphs that we did in the proof of Proposition~\ref{prop:univ_Pg_bd}. However, this time we need more control on the resulting matrix coefficients, which we achieve with Lemma~\ref{lmm:MP_skew_bd}.
\begin{proof}[Proof of Proposition \ref{prop:long_bd}]
First, let us understand the growth of $R$. From equation (\ref{eq:growing_row_dim}) it follows that $\dim\eta$, $\dim\mu_e$, and $\dim\lambda_f$ asymptotically grow at the rates $n^{|\eta^0|}$, $n^{|\mu^0_e|}$, and $n^{|\lambda^0_f|}$ respectively. Together with the bound on $|\chi^{m\tn{-tors}}_{\lambda_f}|$ from Lemma~\ref{lmm:tors_avg_char}, this implies that
\begin{equation}\label{eq:B_powerlaw}
R(\eta,\vec\mu,\vec\lambda)=O(n^{K_R+\varepsilon}),
\end{equation}
where
\begin{equation}\label{eq:B_exponent}
K_R=|\eta^0|-\sum_{e\in E(\Sigma)}|\mu^0_e|+\sum_{f\in\Fspx(\Sigma)}|\lambda^0_f|+\sum_\iterorb\frac{|\lambda^0_f|}{m}.
\end{equation}

Second, let us bound the trace of the $\mathfrak{P}$-graph. As in the proof of Proposition \ref{prop:univ_Pg_bd}, use the main identity of Example~\ref{ex:resolution_id} to cut it into a product of ``corner graphs'' from equation (\ref{eq:mat_coeff}). Writing this in symbols gives us the formula
\begin{align}\begin{split}\label{eq:cut_into_corners}
\tr\mathfrak{P}(\vec\lambda\setminus\eta,\vec\mu\setminus\eta,\vec\gamma)&=\sum_{\tiny\begin{matrix}w_{e\hookrightarrow f}\!\in\!\tn{Tab}(\lambda_f\!\setminus\!\mu_e)\\u_{v\hookrightarrow e}\!\in\!\tn{Tab}(\mu_e\!\setminus\!\eta)\end{matrix}}\mathcal{M}\left(\vec{w},\vec{u}\right)\\
\text{where}\qquad\mathcal{M}\left(\vec{w},\vec{u}\right)&=\prod_{e_+\corner e_-}\left\langle u_{v\hookrightarrow e_+}\otimes w_{e_+\hookrightarrow f},\gamma_c\cdot(u_{v\hookrightarrow e_-}\otimes w_{e_-\hookrightarrow f})\right\rangle.
\end{split}\end{align}
In the non-orientable case of equation (\ref{eq:cut_into_corners}), we got rid of the coupons coloured with the invariant bilinear forms $B$, by evaluating them using equation (\ref{eq:cut_B_matcoef}) from the proof of Proposition \ref{prop:univ_Pg_bd}. Lemma \ref{lmm:MP_skew_bd} tells us that
\begin{equation}\label{eq:M_powerlaw}
\mathcal{M}(\vec{w},\vec{u})=O(n^{-K_\mathcal{M}(\vec{w},\vec{u})}),
\end{equation}
where
\begin{equation}\label{eq:corner_mat_coeff_decay}
K_\mathcal{M}(\vec{w},\vec{u})=\frac{1}{2}\hspace*{-5pt}\sum_{e_+\corner e_-}\hspace*{-5pt}\#\left(\tn{bot}\left(u_{v\hookrightarrow e_+}\sqcup w_{e_+\hookrightarrow f}\right)\vartriangle\gamma_c\cdot\tn{bot}\left(u_{v\hookrightarrow e_-}\sqcup w_{e_-\hookrightarrow f}\right)\right).
\end{equation}
Our $\mathcal{M},\tr\mathfrak{P}$ play the respective roles of $\mathcal{M},\Upsilon_n$ from \cite[Proposition 5.8]{asymptotic_stats_random_covering_surfaces_MageePuder23}.

Third, we relate $K_R$ and $K_\mathcal{M}$ to the combinatorics of $T$. Recall $\gamma_c$ are $T$-matching some edge labelling $\vec{\mathcal{J}}$. For each tuple $(\vec{w},\vec{u})$ of basis vectors let $\mathcal{P}(\vec{w},\vec{u})$ be the boundary piece consisting of those exposed sides and hanging half-edges whose $\vec{\mathcal{J}}$-label is outside the first row of the appropriate $\vec{w}$ or $\vec{u}$. Formally,
\begin{equation*}
\mathcal{P}\left(\vec{w},\vec{u}\right)\defeq\bigcup_{\iota\colon e\hookrightarrow f}\mathcal{J}_\iota^{-1}(\tn{bot}(w_\iota))\cup\bigcup_{\iota\colon v\hookrightarrow e}\mathcal{J}_\iota^{-1}(\tn{bot}(u_\iota))\subseteq\partial T.
\end{equation*}
Our construction of $\mathcal{P}$ parallels \cite[\S 5.6]{asymptotic_stats_random_covering_surfaces_MageePuder23}. The quantity $K_\mathcal{M}(\vec{w},\vec{u})$ counts half of the number of adjacencies between sides and half-edges in $\mathcal{P}(\vec{w},\vec{u})$ and outside of it, so equation (\ref{eq:corner_mat_coeff_decay}) can be expressed succinctly as
\begin{equation}\label{eq:corner_mc_decay_piece}
K_\mathcal{M}=\chi(\mathcal{P}).
\end{equation}

By construction, the piece $\mathcal{P}(\vec{w},\vec{u})$ has $\#\tn{bot}(\lambda_f\setminus\mu_e)=|\lambda^0_f\setminus\mu^0_e|$ exposed sides lifting an inclusion $e\hookrightarrow f$ in $\Sigma$, and $\#\tn{bot}(\mu_e\setminus\eta)=|\mu^0_e\setminus\eta^0|$ hanging half-edges above an inclusion $v\hookrightarrow e$. Therefore, with the notation from Definition~\ref{def:dft}, the counts of exposed sides and hanging half-edges in $\mathcal{P}(\vec{w},\vec{u})$ are
\begin{align*}
\mathfrak{e}_\tn{spx}(\mathcal{P})&=\sum_{\tiny\begin{matrix}e\hookrightarrow f\\f\in\Fspx(\Sigma)\end{matrix}}|\lambda^0_f|-|\mu^0_e|\\
\mathfrak{h}_\tn{spx}(\mathcal{P})&=2\sum_{e\tn{ touching 2 simplices}}|\mu^0_e|-|\eta^0|\\
\mathfrak{e}^{(m)}_\tn{orb}(\mathcal{P})&=\sum_{\tiny\begin{matrix}f\in\Forb(\Sigma)\\e=\partial f\end{matrix}}|\lambda^0_f|-|\mu^0_e|\\
\mathfrak{h}^{(m)}_\tn{orb}(\mathcal{P})&=2\sum_{e\tn{ bordering }m\tn{-orbigon}}|\mu^0_e|-|\eta^0|
\end{align*}
where the factors of 2 in the half-edges arise because every edge has two ends. Summing the above equations with appropriate weights and simplifying yields
\begin{align}\begin{split}\label{eq:almost_dft_P}
\frac{\mathfrak{e}_\tn{spx}(\mathcal{P})}{3}-&\frac{\mathfrak{h}_\tn{spx}(\mathcal{P})}{6}+\sum_m\left(\frac{\mathfrak{e}^{(m)}_\tn{orb}(\mathcal{P})}{m}-(\tfrac{1}{3}-\tfrac{1}{2m})\mathfrak{h}^{(m)}_\tn{orb}(\mathcal{P})\right)=\\
=&K_R+C\cdot|\eta^0|,\quad\text{where}\\
C=&\left(\!\frac{\#\{e\!\in\!E(\Sigma)\mid e\text{ touches }2\text{ simplices}\}}{3}-1+\hspace*{-10pt}\sum_\iterorb\left(\frac{2}{3}-\frac{1}{m}\right)\!\right),
\end{split}\end{align}
and $K_R$ is the quantity from equation (\ref{eq:B_exponent}). Let us simplify the last term. From the convention from Section~\ref{sec:simplicial}, $\#\{e\in E\mid e\text{ touches an orbigon}\}=\sum_m\#\Forb$ and $\#V=1$. Every edge has 2 sides, so
\begin{equation*}
2\#\{e\in E(\Sigma)\mid e\text{ touches 2 simplices}\}=3\#\Fspx(\Sigma)-\sum_m\#\Forb(\Sigma).
\end{equation*}
Putting these observations together with equation (\ref{eq:cho_combi}) we get $C=-\cho(\Sigma)$ in equation (\ref{eq:almost_dft_P}). Hence, combining equations (\ref{eq:corner_mc_decay_piece}, \ref{eq:almost_dft_P}) we obtain
\begin{equation*}
K_R-K_\mathcal{M}\left(\vec{w},\vec{u}\right)=\dft\left(\mathcal{P}\left(\vec{w},\vec{u}\right)\right)+\cho(\Sigma)|\eta^0|.
\end{equation*}

If some $\lambda_f^0\setminus\mu^0_e$ or $\mu^0_e\setminus\eta^0$ is nonempty, then $\mathcal{P}(\vec{w},\vec{u})$ is a nonempty piece, so $K_R-K_\mathcal{M}\leqslant\dft(\mathcal{P})\leqslant\md(T)$. Otherwise, $\lambda^0_f=\mu^0_e=\eta^0$ and it contains at least one box (we assumed not all are empty), so $K_R-K_\mathcal{M}\leqslant\max\{0,\md(T)\}+\cho(\Sigma)$. In either case we have
\begin{equation}\label{eq:total_exp_bd}
K_R-K_\mathcal{M}(\vec{w},\vec{u})\leqslant\max\{\md(T),\cho(\Sigma)\}.
\end{equation}
Together equations (\ref{eq:B_powerlaw}), (\ref{eq:M_powerlaw}), and (\ref{eq:total_exp_bd}) complete the proof of Proposition~\ref{prop:long_bd}.
\end{proof}

\section{Resolutions}\label{sec:resolutions}

The main goal of this section is to prove Proposition~\ref{prop:resolution}, which asserts the existence of resolutions with controlled boundary defects. Its proof is completed in Section~\ref{ssec:construct_resolution}, and rests on a measure of boundary length from Section~\ref{ssec:bdry_len} and combinatorial negative curvature from Section~\ref{ssec:combi_negcurv}. Afterwards in Section~\ref{ssec:pfs_aux_combi} we reuse these tools to complete the proofs of Proposition~\ref{prop:no_cap} and Lemma~\ref{lmm:no_dft_subgp}, two results stated in Section~\ref{ssec:res_main_pf} that guarantee predictable behaviour of complexes whose boundary defect is small.

Let us sketch the strategy for constructing resolutions. It follows the same basic idea as the growing process in \cite[Definition 2.10]{asymptotic_stats_random_covering_surfaces_MageePuder23}. An immersion $T\looparrowright\whS$ of a complex $T$ in a covering $\whS$ is an embedding of some quotient of $T$. Hence, the first kind of operation we use is the enumeration of quotients (note the collection of all quotients already satisfies Definition~\ref{def:resolution}, but with no control on boundaries). The problem is that quotienting a complex by an equivalence relation on vertices may create new boundary pieces with positive defect (or $\partial T$ may contain such to begin with). When this happens, our master probabilistic Theorem~\ref{thm:E_combi_emb} will give a suboptimal asymptotics. The remedy is to cover positive-defect pieces by attaching to them the adjacent phantom faces. This is our second kind of operation, and is illustrated in Figure~\ref{fig:spackling}. A resolution is obtained by exhaustively repeating these two kinds of moves.

The key problem is the halting of this algorithm. In Section~\ref{ssec:bdry_len} we come up with a measure of boundary length that, usually, shortens under both operations. This ensures termination and suffices to prove Proposition~\ref{prop:resolution} for many orbifolds. To complete the proof in full generality, in Section~\ref{ssec:combi_negcurv} we investigate combinatorial negative curvature in the spirit of \cite[\S 2.1]{sectional_curvature_cpct_cores_local_quasiconvexity_Wise04}. Combining these two ideas in Section~\ref{ssec:construct_resolution} produces an improved invariant, which upper-bounds $\cho+\md$ and can be reduced whenever the complex has boundary pieces with nonnegative defect.

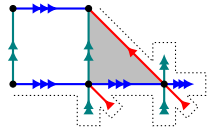
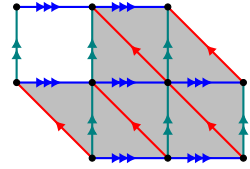
\begin{figure}[h]
\begin{subfigure}{.45\textwidth}
\centering
\begin{tikzpicture}
\fill[lightgray](0,0)--(1,0)--(0,1);\draw[thick,red,->-](1,0)--(0,1);\draw[thick,teal,->>-](0,0)--(0,1);\draw[thick,blue,->>>-](0,0)--(1,0);
\draw[thick,teal,->>-](-1,0)--(-1,1);\draw[thick,blue,->>>-](-1,0)--(0,0);\draw[thick,blue,->>>-](-1,1)--(0,1);
\begin{scope}[>={Latex[length=4,width=4]},thick,red]\draw[>-](1.3,-.3)--(1,0);\draw[>-](.3,-.3)--(0,0);\end{scope}
\begin{scope}[>={Latex[length=4,width=4,sep=-2pt]Latex[length=4,width=4]},thick,teal]\draw[>->](1,-.4)--(1,.4);\draw[>-](0,-.4)--(0,0);\end{scope}
\draw[>={Latex[length=4,width=4,sep=-2pt]Latex[length=4,width=4,sep=-2pt]Latex[length=4,width=4]},thick,blue,->](1,0)--+(.4,0);
\filldraw(-1,0)circle(.04)(-1,1)circle(.04)(0,0)circle(.04)(0,1)circle(.04)(1,0)circle(.04);
\draw[dash pattern={on .5pt off 1pt}](.15,1)--(.85,.3)--
++(0,.25)--++(.3,0)--++(0,-.4)--++(.4,0)--++(0,-.3)--++(-.25,0)--++(.15,-.15)--++(-.15,-.15)--++(-.15,.15)--++(0,-.25)--++(-.3,0)--++(0,.4)--(.3,-.15)--
++(.15,-.15)--++(-.15,-.15)--++(-.15,.15)--++(0,-.25)--++(-.3,0)--++(0,.4)--(-1,-.15);
\end{tikzpicture}
\caption{A boundary piece $\mathcal{P}$ of $T$, topologically a segment, containing 3 exposed sides and 6 hanging half-edges. We have $\dft(\mathcal{P})=-1$.}
\end{subfigure}
\hfill
\begin{subfigure}{.45\textwidth}
\centering
\begin{tikzpicture}
\fill[lightgray](0,0)--(1,0)--(0,1);
\fill[lightgray](-1,0)--(1,0)--(0,1)--(1,1)--(2,0)--(2,-1)--(0,-1);
\begin{scope}[thick,red]\draw[->-](1,0)--(0,1);\draw[->-](2,0)--(1,1);\draw[->-](0,-1)--(-1,0);\draw[->-](1,-1)--(0,0);\draw[->-](2,-1)--(1,0);\end{scope}
\begin{scope}[thick,teal]\draw[->>-](-1,0)--(-1,1);\draw[->>-](0,0)--(0,1);\draw[->>-](1,0)--(1,1);\draw[->>-](0,-1)--(0,0);\draw[->>-](1,-1)--(1,0);\draw[->>-](2,-1)--(2,0);\end{scope}
\begin{scope}[thick,blue]\draw[->>>-](-1,1)--(0,1);\draw[->>>-](0,1)--(1,1);\draw[->>>-](-1,0)--(0,0);\draw[->>>-](0,0)--(1,0);\draw[->>>-](1,0)--(2,0);\draw[->>>-](0,-1)--(1,-1);\draw[->>>-](1,-1)--(2,-1);\end{scope}
\filldraw(-1,0)circle(.04)(-1,1)circle(.04)(0,-1)circle(.04)(0,0)circle(.04)(0,1)circle(.04)(1,-1)circle(.04)(1,0)circle(.04)(1,1)circle(.04)(2,-1)circle(.04)(2,0)circle(.04);
\end{tikzpicture}
\caption{Addition of the phantom cells touching $P$. There are 5 new vertices, 12 edges, and 7 faces.}
\end{subfigure}
\caption{Result of growing phantom faces on an example boundary piece $\mathcal{P}$ of the complex $T$ from Figure \ref{fig:immersion}. All immersions of $T$ into coverings of $\Sigma$ uniquely extend to the added cells. This particular $\mathcal{P}$ would not need covering in the actual algorithm ($\dft(\mathcal{P})<0$), but serves for illustration.}
\label{fig:spackling}
\end{figure}

Let us set up some vocabulary that will be needed. To deal with phenomena on the boundary we will split it into smaller parts called slots.
\begin{definition}\label{def:site}
A \emph{slot} of $T$ is a fan of phantom corners on $\partial T$ anchored at a common vertex, separated only by hanging half-edges and delimited by exposed sides. A \emph{simplicial slot} is a fan of adjacent phantom corners of simplices, 2- or 3-orbigons, delimited by exposed sides or corners of phantom $m$-orbigons with $m\geqslant4$.

Formally, a slot $\check{v}$ is a vertex $v\in V(T)$ together with a connected component of the complement of $p\left(\tn{lk}_T(v)\right)$ inside the simplicial circle $\tn{lk}_\Sigma(p(v))$, where $\tn{lk}$ is the link of a vertex in a cell complex. A simplicial slot is formally a vertex $v$ and a component of the complement of $p\left(\tn{lk}_T(v)\right)$ in the link $\tn{lk}_{\Sigma\setminus\bigcup_{m\geqslant4}\Forb(\Sigma)}(p(v))$ of $p(v)$ in the branched $\Delta$-complex obtained from $\Sigma$ by removing $m$-orbigons with $m\geqslant4$.
\end{definition}
For illustration of slots and simplicial slots on an example complex see Figures~\ref{sfig:slot}, \ref{sfig:spx_slot}. Important quantities will admit expressions as sums over slots, including the defect, a notion of boundary curvature, and the change in boundary length after growing phantom faces. In the course of the proofs, we will make a distinction between simplex-facing and orbigon-facing exposed sides. The latter will be grouped into arcs as below.
\begin{definition}\label{def:orbiarc}
An \emph{$m$-orbiarc} is a contiguous subset of $\partial T$ consisting of adjacent exposed sides facing a phantom $m$-orbigon with $m\geqslant4$ and not separated by any hanging half-edges. An $m$-orbiarc containing at least $\lfloor\tfrac{5m}{6}\rfloor-\mathbbm{1}_{\{m=11\}}$ edges is called a \emph{long orbiarc}.
\end{definition}
Orbiarcs are marked on an example complex in Figure~\ref{sfig:orbiarc}.

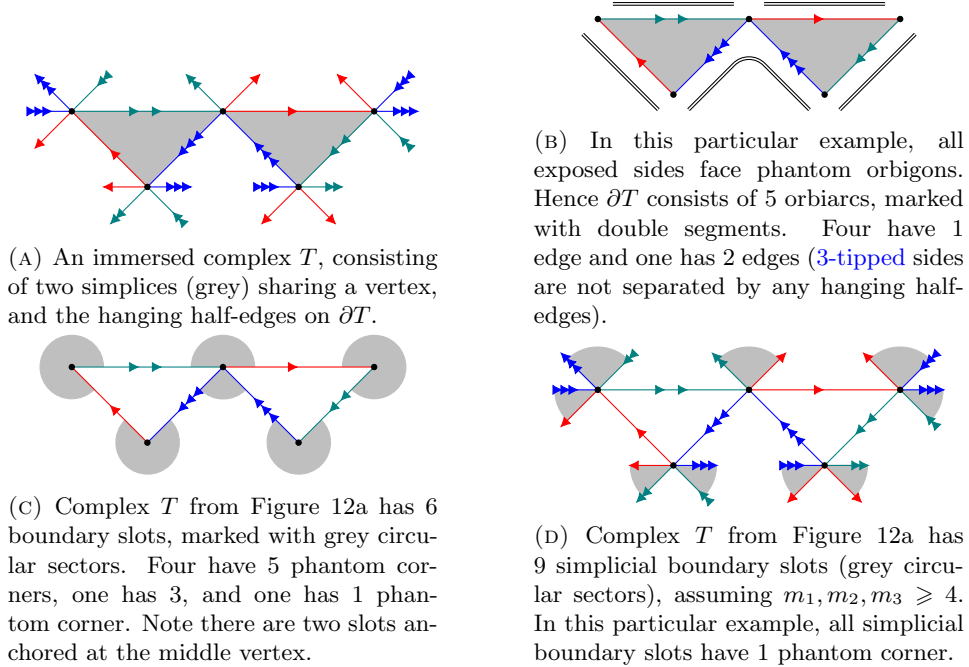
\begin{figure}[h]
\begin{subfigure}{.45\textwidth}
\centering
\begin{tikzpicture}
\filldraw[lightgray](0,0)--(-2,0)--(-1,-1)(0,0)--(2,0)--(1,-1);
\begin{scope}[red]\draw[->-](0,0)--(2,0);\draw[->-](-1,-1)--(-2,0);
\begin{scope}[-{Latex[length=4,width=4]}]\draw(-2,0)--(-2.5,-.5);\draw(2,0)--(1.5,.5);\draw(1,-1)--(.5,-1.5);\end{scope}
\begin{scope}[{Latex[length=4,width=4]}-]\draw(-1.6,-1)--(-1,-1);\draw(.5,.5)--(0,0);\draw(1.5,-1.5)--(1,-1);\end{scope}\end{scope}
\begin{scope}[teal]\draw[->>-](-2,0)--(0,0);\draw[->>-](2,0)--(1,-1);
\begin{scope}[>={Latex[length=4,width=4,sep=-1pt]Latex[length=4,width=4]}]
\draw[->](0,0)--(-.5,.5);\draw[->](-1,-1)--(-1.5,-1.5);\draw[->](1,-1)--(1.6,-1);\draw[>-](-1.5,.5)--(-2,0);\draw[>-](-.5,-1.5)--(-1,-1);\draw[>-](2.5,-.5)--(2,0);
\end{scope}\end{scope}
\begin{scope}[blue]\draw[->>>-](1,-1)--(0,0);\draw[->>>-](0,0)--(-1,-1);
\begin{scope}[>={Latex[length=4,width=4,sep=-1pt]Latex[length=4,width=4,sep=-1pt]Latex[length=4,width=4]}]
\draw[->](-2,0)--(-2.5,.5);\draw[->](-1,-1)--(-.4,-1);\draw[->](2,0)--(2.6,0);\draw[>-](-2.6,0)--(-2,0);\draw[>-](2.5,.5)--(2,0);\draw[>-](.4,-1)--(1,-1);
\end{scope}\end{scope}
\filldraw(-2,0)circle(1pt)(0,0)circle(1pt)(2,0)circle(1pt)(-1,-1)circle(1pt)(1,-1)circle(1pt);
\end{tikzpicture}
\caption{An immersed complex $T$, consisting of two simplices (grey) sharing a vertex, and the hanging half-edges on $\partial T$.}
\label{sfig:oa_slot_example}
\end{subfigure}
\hfill
\begin{subfigure}{.45\textwidth}
\centering
\begin{tikzpicture}
\filldraw[lightgray](0,0)--(-2,0)--(-1,-1)(0,0)--(2,0)--(1,-1);
\begin{scope}[red]\draw[->-](0,0)--(2,0);\draw[->-](-1,-1)--(-2,0);\end{scope}
\begin{scope}[teal]\draw[->>-](-2,0)--(0,0);\draw[->>-](2,0)--(1,-1);\end{scope}
\begin{scope}[blue]\draw[->>>-](1,-1)--(0,0);\draw[->>>-](0,0)--(-1,-1);\end{scope}
\filldraw(-2,0)circle(1pt)(0,0)circle(1pt)(2,0)circle(1pt)(-1,-1)circle(1pt)(1,-1)circle(1pt);
\draw[double](-1.8,.2)--(-.2,.2)(.2,.2)--(1.8,.2)(2.2,-.2)--(1.2,-1.2)(-2.2,-.2)--(-1.2,-1.2)(-.8,-1.2)--(-.2,-.6)to[out=45,in=135](.2,-.6)--(.8,-1.2);
\end{tikzpicture}
\caption{In this particular example, all exposed sides face phantom orbigons. Hence $\partial T$ consists of 5 orbiarcs, marked with double segments. Four have 1 edge and one has 2 edges ({\color{blue}3-tipped} sides are not separated by any hanging half-edges).}
\label{sfig:orbiarc}
\end{subfigure}
\begin{subfigure}{.45\textwidth}
\centering
\begin{tikzpicture}
\filldraw[lightgray](.42,0)arc(0:180:.42)(.3,-.3)arc(-45:-135:.42)--(0,0)(-1.58,0)arc(0:315:.42)--(-2,0)(-.7,-.7)arc(45:-224:.42)--(-1,-1)(1.58,0)arc(180:-135:.42)--(2,0)(1.3,-.7)arc(45:-224:.42)--(1,-1);
\begin{scope}[red]\draw[->-](0,0)--(2,0);\draw[->-](-1,-1)--(-2,0);\end{scope}
\begin{scope}[teal]\draw[->>-](-2,0)--(0,0);\draw[->>-](2,0)--(1,-1);\end{scope}
\begin{scope}[blue]\draw[->>>-](1,-1)--(0,0);\draw[->>>-](0,0)--(-1,-1);\end{scope}
\filldraw(-2,0)circle(1pt)(0,0)circle(1pt)(2,0)circle(1pt)(-1,-1)circle(1pt)(1,-1)circle(1pt);
\end{tikzpicture}
\caption{Complex $T$ from Figure~\ref{sfig:oa_slot_example} has 6 boundary slots, marked with grey circular sectors. Four have 5 phantom corners, one has 3, and one has 1 phantom corner. Note there are two slots anchored at the middle vertex.}
\label{sfig:slot}
\end{subfigure}
\hfill
\begin{subfigure}{.45\textwidth}
\centering
\begin{tikzpicture}
\filldraw[lightgray](.4,.4)arc(45:135:.57)--(0,0)(-1.6,.4)arc(45:135:.57)--(-2,0)(-2.57,0)arc(180:225:.57)--(-2,0)(-1.57,-1)arc(180:225:.57)--(-1,-1)(-.43,-1)arc(0:-45:.57)--(-1,-1)(2.4,.4)arc(45:135:.57)--(2,0)(2.4,-.4)arc(-45:0:.57)--(2,0)(.6,-1.4)arc(225:180:.57)--(1,-1)(1.4,-1.4)arc(-45:0:.57)--(1,-1);
\begin{scope}[red]\draw[->-](0,0)--(2,0);\draw[->-](-1,-1)--(-2,0);
\begin{scope}[-{Latex[length=4,width=4]}]\draw(-2,0)--(-2.5,-.5);\draw(2,0)--(1.5,.5);\draw(1,-1)--(.5,-1.5);\end{scope}
\begin{scope}[{Latex[length=4,width=4]}-]\draw(-1.6,-1)--(-1,-1);\draw(.5,.5)--(0,0);\draw(1.5,-1.5)--(1,-1);\end{scope}\end{scope}
\begin{scope}[teal]\draw[->>-](-2,0)--(0,0);\draw[->>-](2,0)--(1,-1);
\begin{scope}[>={Latex[length=4,width=4,sep=-1pt]Latex[length=4,width=4]}]
\draw[->](0,0)--(-.5,.5);\draw[->](-1,-1)--(-1.5,-1.5);\draw[->](1,-1)--(1.6,-1);\draw[>-](-1.5,.5)--(-2,0);\draw[>-](-.5,-1.5)--(-1,-1);\draw[>-](2.5,-.5)--(2,0);
\end{scope}\end{scope}
\begin{scope}[blue]\draw[->>>-](1,-1)--(0,0);\draw[->>>-](0,0)--(-1,-1);
\begin{scope}[>={Latex[length=4,width=4,sep=-1pt]Latex[length=4,width=4,sep=-1pt]Latex[length=4,width=4]}]
\draw[->](-2,0)--(-2.5,.5);\draw[->](-1,-1)--(-.4,-1);\draw[->](2,0)--(2.6,0);\draw[>-](-2.6,0)--(-2,0);\draw[>-](2.5,.5)--(2,0);\draw[>-](.4,-1)--(1,-1);
\end{scope}\end{scope}
\filldraw(-2,0)circle(1pt)(0,0)circle(1pt)(2,0)circle(1pt)(-1,-1)circle(1pt)(1,-1)circle(1pt);
\end{tikzpicture}
\caption{Complex $T$ from Figure~\ref{sfig:oa_slot_example} has 9 simplicial boundary slots (grey circular sectors), assuming $m_1,m_2,m_3\geqslant4$. In this particular example, all simplicial boundary slots have 1 phantom corner.}
\label{sfig:spx_slot}
\end{subfigure}
\caption{Slots and orbiarcs of an example complex immersed into $S^2$ with 3 cone points triangulated as in Figure~\ref{sfig:orbi_S2}.}
\end{figure}

\subsection{Boundary length}\label{ssec:bdry_len}

In this section we introduce the boundary length $\ell(\partial T)$, which is a measure of complexity of $\partial T$. We prove that complexes whose boundary cannot be shortened by growing faces often have negative maximal defect. This is sufficient to prove Proposition~\ref{prop:resolution} for many orbifolds.

We set the length of each exposed side facing a phantom simplex to 1. Treatment of $m$-orbiarcs (Definition~\ref{def:orbiarc}) is a bit more involved --- we measure them with a certain function $Ł$ depending on $m$ and the number of constituent exposed sides. To ensure desired niceness properties of $\ell$, the function $Ł$ must satisfy a list of conditions stated in detail in Proposition~\ref{prop:orbilen} below.
\begin{proposition}\label{prop:orbilen}
There exists a function $Ł\colon\mathbb{N}_{\geqslant2}\times\mathbb{N}\to\mathbb{R}_{\geqslant0}$ satisfying the following conditions.
\begin{enumerate}
\item For every $m$ we have $Ł(m,0)=0$.
\item\label{itm:orbi_prop_1lip} $\L(m,l+1)\leqslant\L(m,l)+1$
\item If $l\leqslant\lfloor\tfrac{5m}{6}\rfloor-\mathbbm{1}_{\{m=11\}}-1$ then
\begin{enumerate}
\item\label{itm:orbi_prop_gain_usual} $Ł(m,l+2)-Ł(m,l)\leqslant3-\tfrac{6(l+1)}{m}$, unless $m\in[2,12]$ and $l\geqslant m-3$
\item\label{itm:orbi_prop_gain_except} if $m\!\in\![4,\!12]$ and $l\!\geqslant\!m\!-\!3$, then $(m\!-\!l)\!-\!Ł(m,l)\!\leqslant\!5\!-\!\tfrac{6(l+1)}{m}\!+\!\tfrac{\mathbbm{1}_{\{m=5,l=3\}}}{5}$
\item\label{itm:orbi_prop_gain_2der} $2Ł(m,l+1)-Ł(m,l)-Ł(m,l+2)\leqslant6-\tfrac{6l}{m}$.
\end{enumerate}
\item\label{itm:orbi_prop_loa} If $l\geqslant\lfloor\tfrac{5m}{6}\rfloor-\mathbbm{1}_{\{m=11\}}$, then $\L(m,l)\geqslant m-l-\tfrac{1}{m}\left(\mathbbm{1}_{\{m\equiv\pm1\bmod(6),m\geqslant7\}}+\mathbbm{1}_{\{m=11\}}\right)$.
\item\label{itm:orbi_prop_subadd} For any $\max\{l_2,l_3\}\leqslant l_1\leqslant l_2+l_3$ we have $Ł(m,l_1)\leqslant Ł(m,l_2)+Ł(m,l_3)$.
\item\label{itm:orbi_prop_circlen_dft} $Ł(m,l)\geqslant\min\{2,\tfrac{2l}{m}\}$
\item\label{itm:orbi_prop_atleast1} If $l\geqslant1$ then $Ł(m,l)\geqslant1$.
\item\label{itm:orbi_prop_denominator} For every $l$ we have $Ł(m,l)\in\tfrac{1}{2m}\mathbb{Z}$.
\end{enumerate}
\end{proposition}
\begin{proof}
The properties we require from $Ł$ almost specify it uniquely. The exact formula turns out to be $Ł(m,l)=l$ for $m\in\{2,3\}$, $Ł(m,l)=\min\{2,\tfrac{2l}{m}\}$ for $m\in[4,12]\cup\{17\}$ and $l\geqslant m-1$; and for all other arguments
\begin{equation*}
Ł(m,l)\defeq\begin{cases}l&\text{if }l\leqslant k+1\\k+\tfrac{3(l-k)(m-l-k)}{2m}+\tfrac{3-r}{2m}\cdot\mathbbm{1}_{\{l\not\equiv k\bmod(2)\}}&\text{if }k\leqslant l\leqslant\left\lfloor\tfrac{5m}{6}\right\rfloor+1\\k-\mathbbm{1}_{\{r=0\}}+\tfrac{1}{m}(\mathbbm{1}_{\{r=1\}}-\mathbbm{1}_{\{r=5\}})&\text{if }l\geqslant\lfloor\tfrac{5m}{6}\rfloor+1,\end{cases}
\end{equation*}
where $m=6k+r$ with $0\leqslant r\leqslant5$. Knowing the formula, conditions of Proposition~\ref{prop:orbilen} can be verified directly.
\end{proof}
This allows us to formally define the boundary length.
\begin{definition}\label{def:bdry_len}
Suppose $\partial T$ consists of $\mathfrak{e}_\tn{spx}$ exposed sides facing phantom simplices, 2-orbigons or 3-orbigons, and $a+b$ orbiarcs. Suppose the first $a$ orbiarcs are segments and the last $b$ are circles, and the $i$-th one has $l_i$ edges of a phantom $m_i$-orbigon. The \emph{boundary length} of $T$ is
\begin{equation*}
\ell(\partial T)\defeq \mathfrak{e}_\tn{spx}+\sum_{i=1}^aŁ(m_i,l_i)+\sum_{i=a+1}^{a+b}\min\left\{2,\frac{2l_i}{m_i}\right\},
\end{equation*}
where $Ł\colon\mathbb{N}_{\geqslant2}\times\mathbb{N}\to\mathbb{R}_{\geqslant0}$ is the function from Proposition~\ref{prop:orbilen}.
\end{definition}

Proposition~\ref{prop:orbilen} is a bit complicated; its purpose is that each condition will activate in a certain combinatorial configuration to guarantee some desirable property of $\ell$. We can already demonstrate one of them below.
\begin{proposition}\label{prop:quot_short_bdry}
If $T\twoheadrightarrow T'$ is a quotient of immersed complexes (Definition~\ref{def:imm_cplx}), then $\ell(\partial T')\leqslant\ell(\partial T)$.
\end{proposition}
\begin{proof}
The number of exposed sides facing a phantom simplex or 2- or 3-orbigon clearly cannot increase in the quotient, so keeping the notation from Definition~\ref{def:bdry_len} we have $\mathfrak{e}_\tn{spx}'\leqslant\mathfrak{e}_\tn{spx}$. Consider an $m$-orbiarc $\alpha\subseteq\partial T'$ of $l$ exposed sides. It must come from some orbiarcs $\alpha_1,\dots,\alpha_k\subseteq\partial T$ with $l_1,\dots,l_k$ edges respectively (possibly with overlapping images). There are two possibilities.
\begin{itemize}
\item If the images of $\alpha_i$ in $T'$ are all segments, then applying condition~\ref*{itm:orbi_prop_subadd} of Proposition~\ref{prop:orbilen} gives
\begin{equation*}
Ł(m,l)\leqslant Ł(m,l_1)+\dots+Ł(m,l_k).
\end{equation*}
The contribution of $\alpha$ to $\ell(\partial T')$ is either $Ł(m,l)$ if it is a segment or $\min\{2,\tfrac{2l}{m}\}$ if it is a circle. By condition \ref*{itm:orbi_prop_circlen_dft} of Proposition~\ref{prop:orbilen}, both are upper-bounded by $Ł(m,l)$ on the LHS of the inequality above.
\item If the image of some $\alpha_i$ in $T'$ folds into a circle, then $\alpha$ must also be a circle, and $l\leqslant l_i$. The contribution of $\alpha_i$ to $\ell(\partial T)$ is either $Ł(m,l_i)$ or $\min\{2,\tfrac{2l_i}{m}\}$, which again by condition~\ref*{itm:orbi_prop_circlen_dft} of Proposition~\ref{prop:orbilen} is no less than the contribution $\min\{2,\tfrac{2l}{m}\}$ of $\alpha$ to $\ell(\partial T')$.
\end{itemize}
In either case the contribution of $\alpha$ to $\ell(\partial T')$ is no larger than the sum of contributions of $\alpha_i$ to $\ell(\partial T)$. Orbiarcs $\alpha_i$ do not land in any other orbiarc of $T'$, so summing these estimates gives $\ell(\partial T')\leqslant\ell(\partial T)$.
\end{proof}

The real utility of the boundary length as a complexity measure comes from the proposition below. It is a manifestation of the principle that positive defect boundary can be shortened.
\begin{proposition}\label{prop:spackling_defect}
Assume $\partial T$ contains no long orbiarcs, and $\dft(\mathcal{P})\geqslant0$ for some nonempty boundary piece $\mathcal{P}\subseteq\partial T$ that does not contain circular orbiarcs. Then there exists an embedding $T\hookrightarrow T'$ satisfying the following conditions.
\begin{enumerate}
\item\label{itm:spackling_extension} Every immersion $T\looparrowright\whS$ of $T$ in a covering $\whS$ extends uniquely to $T'\looparrowright\whS$.
\item\label{itm:spackling_shorten} $\ell(\partial T')\leqslant\ell(\partial T)+\tfrac{1}{5}\#\{f\in\Forb[5](T'\setminus T)\mid\mathfrak{s}_f=5\}$
\item\label{itm:spackling_p5} Every $f\in\Forb[5](T'\setminus T)$ has at least 3 vertices in $T$.
\item\label{itm:spackling_nontriv} There is at least one face in $T'\setminus T$.
\end{enumerate}
\end{proposition}

Note that at this point we can already prove Proposition \ref{prop:resolution} when $\Sigma$ has no cone points of order congruent to 1 or 5 modulo 6. Indeed, under this extra assumption, we can shorten the boundary whenever $\md\geqslant0$: if there are long orbiarcs, then adding the phantom orbigon works by condition~\ref*{itm:orbi_prop_loa} of Proposition~\ref{prop:orbilen} (it covers the $m$-orbiarc of $l$ edges and creates $m-l$ new simplex-facing exposed sides), while in the absence of long orbiarcs we apply Proposition \ref{prop:spackling_defect}. The iteration of this process and enumeration of quotients (which do not increase $\ell$ by Proposition~\ref{prop:quot_short_bdry}) must terminate, either by running out of nonnegative-defect pieces or decreasing $\cho$ arbitrarily much, at which point we obtain the desired resolution. To cover long orbiarcs and apply this argument without restrictions on cone point orders modulo 6 we would need $l\geqslant\lfloor\tfrac{5m}{6}\rfloor-\mathbbm{1}_{\{m=11\}}\Rightarrow Ł(m,l)\geqslant m-l$, but with this strengthening of condition~\ref*{itm:orbi_prop_loa} the function $Ł$ in Proposition~\ref{prop:orbilen} happens to no longer exist.

\smallskip

The rest of this section is devoted to proving Proposition~\ref{prop:spackling_defect}. The strategy is to grow phantom faces on a nonnegative-defect piece, as in Figure~\ref{fig:spackling}. To bound the new boundary length we partition $\partial T$ into simplicial slots (Definition~\ref{def:site}), correspondingly break up the defect and boundary length change as sums of local contributions, and compare the resulting two kinds of local terms. The contributions of each orbiarc and simplex-facing exposed side will be split in half and attributed to the two adjacent simplicial slots.
\begin{proof}[Proof of Proposition \ref{prop:spackling_defect}]
Consider a boundary piece $\emptyset\neq\mathcal{P}\subseteq\partial T$ with $\dft(\mathcal{P})\geqslant0$. Let $T'$ be the immersed subcomplex obtained by adding to $T$ all the phantom simplices, 2-orbigons and 3-orbigons that touch $\mathcal{P}$ at an interior vertex, and phantom $m$-orbigons with $4\leqslant m\leqslant12$ that touch at least $m-3$ exposed sides of $\mathcal{P}$. Such an inclusion $T\hookrightarrow T'$ can be realised as a simple homotopy equivalence followed by Stallings folding, so every immersion $T\looparrowright\whS$ into a covering extends uniquely to $T'$, ensuring the condition~\ref*{itm:spackling_extension} of Proposition~\ref{prop:spackling_defect} holds. Conditions~\ref*{itm:spackling_p5},\ref*{itm:spackling_nontriv} are also straightforward to confirm.

We claim that $T'$ satisfies condition~\ref*{itm:spackling_shorten} too. To prove it we express the defect $\dft(\mathcal{P})=\sum_{\check{v}\in\mathcal{P}}D_{\check{v}}$ and boundary length change $\ell(\partial T')-\ell(\partial T)=\sum_{\check{v}\in\mathcal{P}}\Delta\ell_{\check{v}}$ in terms of simplicial slots $\check{v}\in\mathcal{P}$. There are four possibilities for $\check{v}$, as shown in Figure~\ref{fig:bdry_verts}: either it separates two exposed sides facing phantom simplices, two orbiarcs, lies between a simplex-facing edge and an orbiarc, or is one of two types of ends of $\mathcal{P}$. In each case we will analyse the contribution $D_{\check{v}}$ to $\tn{dft}(\mathcal{P})$ and $\Delta\ell_{\check{v}}$ to boundary length change, and verify that
\begin{equation}\label{eq:len_dft}
\Delta\ell_{\check{v}}\leqslant-6D_{\check{v}}+\tfrac{1}{10}\cdot\#\{\text{newly added }5\text{-sided }5\tn{-orbigons touching }\check{v}\}.
\end{equation}
Summation of inequality (\ref{eq:len_dft}) over $\check{v}$ will then yield condition~\ref*{itm:spackling_shorten} of Proposition~\ref{prop:spackling_defect}.

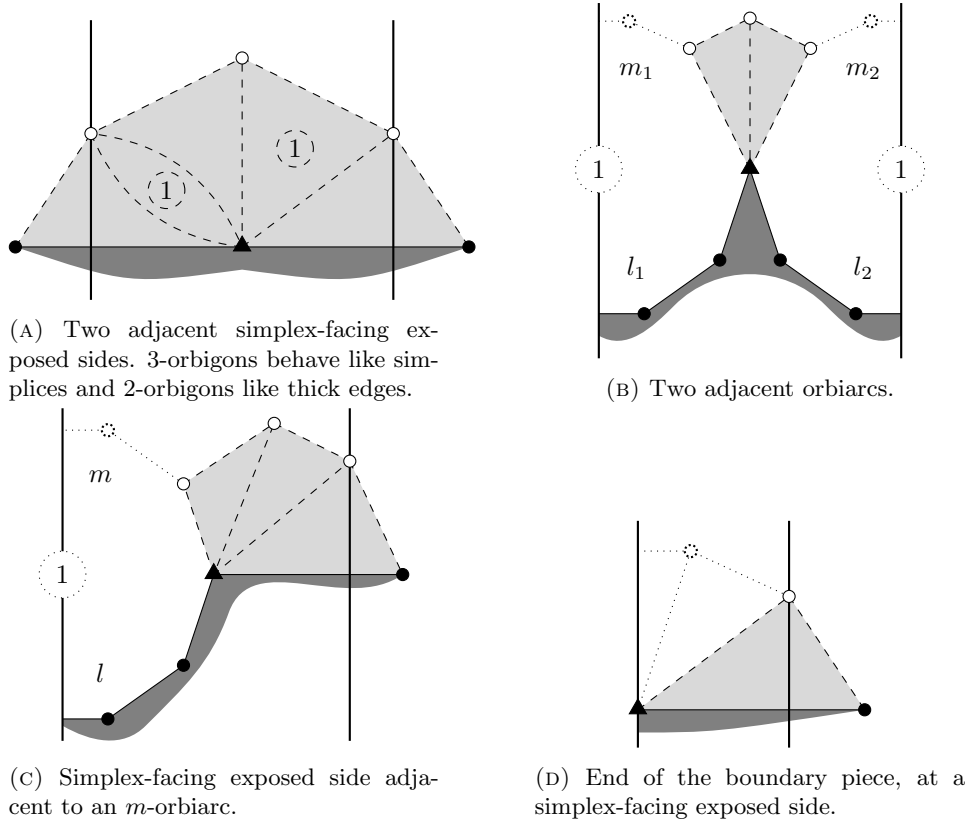
\begin{figure}[h]
\begin{subfigure}{.45\textwidth}
\centering
\begin{tikzpicture}
\fill[gray!30!white](-3,0)--(-2,1.5)--(0,2.5)--(2,1.5)--(3,0);
\fill[gray](-3,0)--(3,0)..controls(1.5,-.5)..(0,-.3)..controls(-1.5,-.5)..(-3,0);
\draw(-3,0)--(3,0);\filldraw(-3,0)circle(.08)(3,0)circle(.08);\node[draw,isosceles triangle,isosceles triangle apex angle=60,fill=black,inner sep=1.5pt,rotate=90]at(0,0){};
\draw[dashed](0,0)to[bend left](-2,1.5)to[bend left](0,0)--(0,2.5)(0,0)--(2,1.5)--(0,2.5)--(-2,1.5)(-3,0)--(-2,1.5)(2,1.5)--(3,0);
\node[circle,dashed,draw,inner sep=2pt]at(.7,1.3){1};\node[circle,dashed,draw,inner sep=2pt]at(-1,.75){1};
\draw[thick](-2,3)--(-2,-.7)(2,3)--(2,-.7);
\draw[fill=white](-2,1.5)circle(.08)(0,2.5)circle(.08)(2,1.5)circle(.08);
\end{tikzpicture}
\caption{Two adjacent simplex-facing exposed sides. 3-orbigons behave like simplices and 2-orbigons like thick edges.}
\label{sfig:bdry_V_ss}
\end{subfigure}
\hfill
\begin{subfigure}{.45\textwidth}
\centering
\begin{tikzpicture}
\fill[gray!30!white](0,0)--(-.8,1.6)--(0,2)--(.8,1.6);
\fill[gray](-2,-1.91)--(-1.4,-1.91)--(-.4,-1.2)--(0,0)--(.4,-1.2)--(1.4,-1.91)--(2,-1.91)--(2,-2.2)to[out=-150,in=-45](1,-1.8)to[out=135,in=45](-1,-1.8)to[out=-135,in=-30](-2,-2.2);
\draw(-2,-1.91)--(-1.4,-1.91)--(-.4,-1.2)--(0,0)--(.4,-1.2)--(1.4,-1.91)--(2,-1.91);
\draw[thick](-2,2.2)--(-2,-2.5)(2,2.2)--(2,-2.5);
\node[circle,draw,dotted,fill=white]at(-2,0){1};\node[circle,draw,dotted,fill=white]at(2,0){1};
\draw[dotted](-2,1.96)--(-1.6,1.96)--(-.8,1.6)(.8,1.6)--(1.6,1.96)--(2,1.96);\draw[dashed](0,0)--(-.8,1.6)--(0,2)--(0,0)--(.8,1.6)--(0,2);
\filldraw(-1.4,-1.91)circle(.08)(-.4,-1.2)circle(.08)(.4,-1.2)circle(.08)(1.4,-1.91)circle(.08)(-.8,1.6);
\node[draw,isosceles triangle,isosceles triangle apex angle=60,fill=black,inner sep=1.5pt,rotate=90]at(0,0){};
\draw[fill=white](-.8,1.6)circle(.08)(0,2)circle(.08)(.8,1.6)circle(.08);\draw[densely dotted,fill=white,thick](-1.6,1.96)circle(.08)(1.6,1.96)circle(.08);
\node at(-1.5,-1.3){$l_1$};\node at(1.5,-1.3){$l_2$};\node at(-1.5,1.3){$m_1$};\node at(1.5,1.3){$m_2$};
\end{tikzpicture}
\caption{Two adjacent orbiarcs.}
\label{sfig:bdry_V_oo}
\end{subfigure}
\begin{subfigure}{.45\textwidth}
\centering
\begin{tikzpicture}
\fill[gray!30!white](0,0)--(-.4,1.2)--(.8,2)--(1.8,1.5)--(2.5,0);
\fill[gray](-2,-2)to[out=-30,in=-135](-.9,-2)to[out=45,in=-110](.2,-.5)to[out=70,in=-145](2.5,0)--(0,0)--(-.4,-1.2)--(-1.4,-1.91)--(-2,-1.91);
\draw(-2,-1.91)--(-1.4,-1.91)--(-.4,-1.2)--(0,0)--(2.5,0);
\draw[thick](-2,2.2)--(-2,-2.2)(1.8,2.2)--(1.8,-2.2);\node[circle,dotted,draw,fill=white]at(-2,0){1};
\draw[dashed](.8,2)--(1.8,1.5)--(0,0)--(.8,2)--(-.4,1.2)--(0,0)(1.8,1.5)--(2.5,0);\draw[dotted](-.4,1.2)--(-1.4,1.91)--(-2,1.91);
\filldraw(-1.4,-1.91)circle(.08)(-.4,-1.2)circle(.08)(2.5,0)circle(.08);
\node[draw,isosceles triangle,isosceles triangle apex angle=60,fill=black,inner sep=1.5pt,rotate=90]at(0,0){};
\draw[fill=white](-.4,1.2)circle(.08)(.8,2)circle(.08)(1.8,1.5)circle(.08);\draw[densely dotted,fill=white,thick](-1.4,1.91)circle(.08);
\node at(-1.5,-1.3){$l$};\node at(-1.5,1.3){$m$};
\end{tikzpicture}
\caption{Simplex-facing exposed side adjacent to an $m$-orbiarc.}
\label{sfig:bdry_V_os}
\end{subfigure}
\hfill
\begin{subfigure}{.45\textwidth}
\centering
\begin{tikzpicture}
\fill[gray!30!white](0,0)--(3,0)--(2,1.5);
\fill[gray](0,0)--(3,0)..controls(1,-.3)..(0,-.3);
\draw(0,0)--(3,0);
\draw[thick](0,2.5)--(0,-.5)(2,2.5)--(2,-.5);
\draw[dotted](0,2.1)--(.7,2.1)--(2,1.5)(0,0)--(.7,2.1);
\draw[dashed](0,0)--(2,1.5)(2,1.5)--(3,0);
\filldraw(3,0)circle(.08);\node[draw,isosceles triangle,isosceles triangle apex angle=60,fill=black,inner sep=1.5pt,rotate=90]at(0,0){};
\draw[fill=white](2,1.5)circle(.08);\draw[densely dotted,fill=white,thick](.7,2.1)circle(.08);
\end{tikzpicture}
\caption{End of the boundary piece, at a simplex-facing exposed side.}
\label{sfig:bdry_V_end}
\end{subfigure}
\caption{Types of simplicial slots in $\mathcal{P}$. Light grey faces, dashed edges, and empty vertices will be added to $T$ in order to obtain $T'$. The dotted cells describe the further phantom neighbourhood and will not be added. Filled vertices are already present in $T$, and the dark grey represents the interior of $T$. The triangle marks the vertex at which the simplicial slot $\check{v}$ is anchored, and thick vertical solid lines demarcate the region attributed to $\check{v}$ for calculations of $D$ and $\Delta\ell$.}
\label{fig:bdry_verts}
\end{figure}

First, consider a simplicial slot $\check{v}$ separating two exposed sides that face phantom simplices, as shown in Figure~\ref{sfig:bdry_V_ss}. Suppose that, in addition to the two halves of exposed sides, $\check{v}$ touches $t$ phantom 2-orbigons and $h$ hanging half-edges that separate simplices or 3-orbigons ($t=1,h=2$ in the example in Figure~\ref{sfig:bdry_V_ss}). From Definition~\ref{def:dft} of boundary defect we see that the contribution of $\check{v}$ to $\dft(\mathcal{P})$ is
\begin{equation*}
D_{\check{v}}=\tfrac{1}{3}+2t\cdot(-\tfrac{1}{12})+h\cdot(-\tfrac{1}{6})=\frac{2-t-h}{6}
\end{equation*}
If $t=h=0$, then $\check{v}$ consists of a single corner of a phantom simplex $f$ incident to both exposed sides; adding such $f$ creates a new exposed side but covers two existing ones, so $T'\defeq T\cup f$ witnesses Proposition~\ref{prop:spackling_defect}. Otherwise, adding the adjacent phantom faces creates $t+h-1$ new exposed sides on the other side of simplices and 3-orbigons, but covers the two halves of exposed sides currently present. Each new exposed side increases the boundary length by at most 1 (condition \ref*{itm:orbi_prop_1lip} from Proposition \ref{prop:orbilen}), so the contribution of $\check{v}$ to $\ell(\partial T')-\ell(\partial T)$ is at most
\begin{equation*}
\Delta\ell_{\check{v}}\leqslant t+h-2.
\end{equation*}
Combining the above two inequalities confirms the bound (\ref{eq:len_dft}) in this case.

Second, consider a simplicial slot $\check{v}$ separating two $m_i$-orbiarcs of $l_1,l_2$ exposed sides for $i=1,2$, as in Figure \ref{sfig:bdry_V_oo} (we allow $l_i=0$, since we are working with simplicial slots). Suppose $\check{v}$ touches $h$ phantom 2-orbigons or half-edges incident to simplices or 3-orbigons ($h=1$ in the example in Figure \ref{sfig:bdry_V_oo}). Summing these half-edges, the two halves of orbiarcs, and two hanging half-edges on the boundary of phantom orbigons gives the total contribution of $\check{v}$ to $\dft(\mathcal{P})$ equal
\begin{equation*}
D_{\check{v}}=-\frac{h}{6}+\frac{l_1+1}{2m_1}+\frac{l_2+1}{2m_2}-\frac{2}{3}.
\end{equation*}
Adding to $T$ the phantom simplices, 2- and 3-orbigons adjacent to $\check{v}$ lengthens the adjacent orbiarcs by two (since they gain one edge on the other side too, unless $\mathcal{P}$ ends there which will be corrected later), and creates $h+1$ other new exposed sides. Therefore, the contribution of $\check{v}$ to $\ell(\partial T')-\ell(\partial T)$ is
\begin{equation*}
\Delta\ell_{\check{v}}=\frac{Ł(m_1,l_1+2)-Ł(m_1,l_1)}{2}+\frac{Ł(m_2,l_2+2)-Ł(m_2,l_2)}{2}+h+1.
\end{equation*}
The boundary $\partial T$ contains no long orbiarcs, so $l_i\leqslant\lfloor\tfrac{5m_i}{6}\rfloor-\mathbbm{1}_{\{m_i=11\}}-1$. Unless $m_i\leqslant12\wedge l_i\geqslant m_i-3$ for some $i$, from condition \ref*{itm:orbi_prop_gain_usual} of Proposition \ref{prop:orbilen} we obtain inequality (\ref{eq:len_dft}) in this case.

Third, consider a slot $\check{v}$ separating an exposed side facing a phantom simplex from an $m$-orbiarc of $l$ exposed sides, as in Figure \ref{sfig:bdry_V_os}. Suppose it touches $h$ missing 2-orbigons or other hanging half-edges, apart from the half-edge bordering the $m$-orbigon ($h=2$ in the example in Figure \ref{sfig:bdry_V_os}). From Definition~\ref{def:dft} we see that $\check{v}$ contributes
\begin{equation}\label{eq:dft_contrib_oa_spx}
D_{\check{v}}=-\frac{h+1}{6}+\frac{l+1}{2m}
\end{equation}
to the defect. Growing the phantom simplices lengthens the orbiarc by two edges, introduces $h$ other new exposed sides, and covers half of the current simplex-facing exposed side, so the local change in the boundary length is
\begin{equation*}
\Delta\ell_{\check{v}}=\frac{Ł(m,l+2)-Ł(m,l)}{2}+h-\frac{1}{2}.
\end{equation*}
Again, unless there is a long orbiarc or $m\leqslant12\wedge l\geqslant m-3$, we have that condition \ref*{itm:orbi_prop_gain_usual} of Proposition \ref{prop:orbilen} implies the inequality (\ref{eq:len_dft}) in this case.

Fourth, $\mathcal{P}$ may contain an end of one of two kinds: either at a simplex-facing exposed side or at an orbiarc. Consider $\check{v}$ in the former subcase, as shown in Figure \ref{sfig:bdry_V_end}. Half-edges enter the defect with negative weights, so the contribution of $\check{v}$ to $\dft(\mathcal{P})$ is no larger than that of the adjacent half of an exposed side together with the piece end, i.e. $D_{\check{v}}\leqslant\frac{1}{6}-\tfrac{1}{2}=-\tfrac{1}{3}$. Adding to $T$ the adjacent simplex creates a new exposed side but covers the adjacent half of a current exposed side, so the contribution of $\check{v}$ to $\ell(\partial T')-\ell(\partial T)$ is $\Delta\ell_{\check{v}}=\frac{1}{2}$. Combining these two confirms (\ref{eq:len_dft}).

The other subcase is the end lying on an $m$-orbiarc $\alpha$, say with $l$ exposed sides. Then to the slot $\check{v}$ we attribute half of the $\alpha$'s defect contribution $\tfrac{l}{m}$, and $-\tfrac{1}{2}$ from a piece end, for a total of $D_{\check{v}}=\tfrac{l}{2m}-\tfrac{1}{2}$. In the growing operation, $\alpha$ is unaffected near $\check{v}$ unless $m\leqslant12\wedge l\geqslant m-3$, so it lengthens from $l$ to $l+1$ exposed sides. But at the other side of $\alpha$ we added a term $\tfrac{1}{2}(Ł(m,l+2)-Ł(m,l))$ to $\Delta\ell$, so now we have to subtract it. Overall, the contribution of $\check{v}$ to boundary length change is
\begin{equation*}
\Delta\ell_{\check{v}}=Ł(m,l+1)-Ł(m,l)-\frac{Ł(m,l+2)-Ł(m,l)}{2}.
\end{equation*}
In this case inequality (\ref{eq:len_dft}) follows from condition~\ref*{itm:orbi_prop_gain_2der} of Proposition~\ref{prop:orbilen}.

Finally, let us consider the exceptional case of an $m$-orbiarc with $m\in[4,12]$ containing $l\in[m-3,\lfloor\tfrac{5m}{6}\rfloor-\mathbbm{1}_{\{m=11\}}-1]$ sides of edges. Up until now we had $\Delta\ell\leqslant-6D$ and no $m$-orbigons with $m\geqslant4$ added, but this time the corrections are necessary. Let $\check{v}$ be a slot touching such an arc, as illustrated in Figure~\ref{fig:53arc} (in the most complicated case $m=5$). There are three subcases: either $\check{v}$ separates the orbiarc from a simplex-facing side, or separates two orbiarcs, or is an end of $\mathcal{P}$.

The first subcase is drawn in Figure~\ref{sfig:53arc_int}. Instead of lengthening the orbiarc by 2 edges as before, we will close it by covering it with the phantom $m$-orbigon, thus introducing $m-2-l$ new simplex-facing exposed sides. Consequently, the term $Ł(m,l+2)$ describing a part of the ``new'' boundary length $\ell(\partial T')$ has to be replaced by $m-2-l$, leading to
\begin{equation*}
\Delta\ell_{\check{v}}=\frac{m-2-l}{2}-\frac{Ł(m,l)}{2}+h-\frac{1}{2}.
\end{equation*}
The defect contribution $D_{\check{v}}$ is the same as in the earlier non-exceptional (\ref{eq:dft_contrib_oa_spx}). Instead of condition~\ref*{itm:orbi_prop_gain_usual} of Proposition~\ref{prop:orbilen} we need to use the condition~\ref*{itm:orbi_prop_gain_except}, which leads to
\begin{equation}\label{eq:dft_len_exceptional}
\Delta\ell_{\check{v}}\leqslant-6D_{\check{v}}+\frac{\mathbbm{1}_{\{(m,l)=(5,3)\text{ adjacent to }\check{v}\}}}{10},
\end{equation}
implying inequality (\ref{eq:len_dft}). The same reasoning confirms (\ref{eq:dft_len_exceptional}) when $\check{v}$ lies between two orbiarcs, except that each orbiarc gets its own term $\mathbbm{1}_{\{(m,l)=(5,3)\}}$.

In the third subcase of $\check{v}$ lying at the end of piece $\mathcal{P}$, as in Figure~\ref{sfig:53arc_end}, we have $D_{\check{v}}=\tfrac{1}{2}\cdot\frac{l}{m}-\tfrac{1}{2}=-\tfrac{m-l}{2m}$ as with the usual piece end. Growing the phantom $m$-orbigon covers the orbiarc (half of which is attributed to $\check{v}$), but creates new simplex-facing exposed sides. Either $l=m-2$ and we create 1 new side adjacent to $\check{v}$ (this is the situation drawn in Figure~\ref{sfig:53arc_end}), or $l=m-3$ and $\tfrac{3}{2}$ new sides are attributed to $\check{v}$; in total $\Delta\ell_{\check{v}}=\tfrac{m-l}{2}-\tfrac{1}{2}Ł(m,l)$. Again condition~\ref*{itm:orbi_prop_gain_except} of Proposition~\ref{prop:orbilen} implies $\Delta\ell_{\check{v}}<-6D_{\check{v}}$, which is stronger than inequality (\ref{eq:dft_len_exceptional}). Note that in both cases, if $m=5$ then the newly added phantom 5-orbigon had at least 3 vertices already present in $T$, in accordance with condition~\ref*{itm:spackling_p5} of Proposition~\ref{prop:spackling_defect}.

\begin{figure}[h]
\begin{subfigure}{.48\textwidth}
\centering
\begin{tikzpicture}
\fill[gray!30!white](-1.5,1.2)--(-3,0)--(-2.4,-1.4)--(-.6,-1.4)--(0,0)--(2,0)--(1,1.2)--(-.1,1.8);
\fill[gray](-3,0)to[out=-80,in=150](-2.5,-1.6)to[out=-30,in=-150](-.5,-1.6)to[out=30,in=-135](.2,-.3)to[out=45,in=-150](2,0)--(0,0)--(-.6,-1.4)--(-2.4,-1.4)--(-3,0);
\draw(-3,0)--(-2.4,-1.4)--(-.6,-1.4)--(0,0)--(2,0);
\draw[dashed](0,0)--(-1.5,1.2)--(-.1,1.8)--(0,0)--(1,1.2)--(-.1,1.8)(1,1.2)--(2,0)(-3,0)--(-1.5,1.2);
\draw[thick](-1.5,1.9)--(-1.5,-2.1)(1,1.9)--(1,-2.1);
\filldraw(-3,0)circle(.08)(-2.4,-1.4)circle(.08)(-.6,-1.4)circle(.08)(2,0)circle(.08);
\draw[fill=white](-1.5,1.2)circle(.08)(-.1,1.8)circle(.08)(1,1.2)circle(.08);
\node[circle,dashed,draw,fill=gray!30!white,inner sep=3pt]at(-1.5,-.2){1};
\node[draw,isosceles triangle,isosceles triangle apex angle=60,fill=black,inner sep=1.5pt,rotate=90]at(0,0){};
\end{tikzpicture}
\caption{Inside the piece, and bordering an exposed side facing a phantom simplex.}
\label{sfig:53arc_int}
\end{subfigure}
\hfill
\begin{subfigure}{.48\textwidth}
\centering
\begin{tikzpicture}
\fill[gray!30!white](1.5,1.2)--(3,0)--(2.4,-1.4)--(.6,-1.4)--(0,0);
\fill[gray](0,0)to[out=-80,in=150](.5,-1.6)to[out=-30,in=-150](2.5,-1.6)to[out=30,in=-100](3,0)--(2.4,-1.4)--(.6,-1.4)--(0,0);
\draw[thick](0,1.7)--(0,-2.1)(1.5,1.7)--(1.5,-2.1);
\draw[dashed](0,0)--(1.5,1.2)--(3,0);
\draw(0,0)--(.6,-1.4)--(2.4,-1.4)--(3,0);
\filldraw(.6,-1.4)circle(.08)(2.4,-1.4)circle(.08)(3,0)circle(.08);
\draw[fill=white](1.5,1.2)circle(.08);
\node[circle,dashed,draw,fill=gray!30!white,inner sep=3pt]at(1.5,-.2){1};
\node[draw,isosceles triangle,isosceles triangle apex angle=60,fill=black,inner sep=1.5pt,rotate=90]at(0,0){};
\end{tikzpicture}
\caption{At the piece end.}
\label{sfig:53arc_end}
\end{subfigure}
\caption{Possibilities for the exceptional configuration $m\leqslant12,l\geqslant m-3$ in a boundary piece. The pictures show the ``most complicated'' case $m=5$. Markings of vertices, edges and faces have the same meaning as in Figure~\ref{fig:bdry_verts}.}
\label{fig:53arc}
\end{figure}
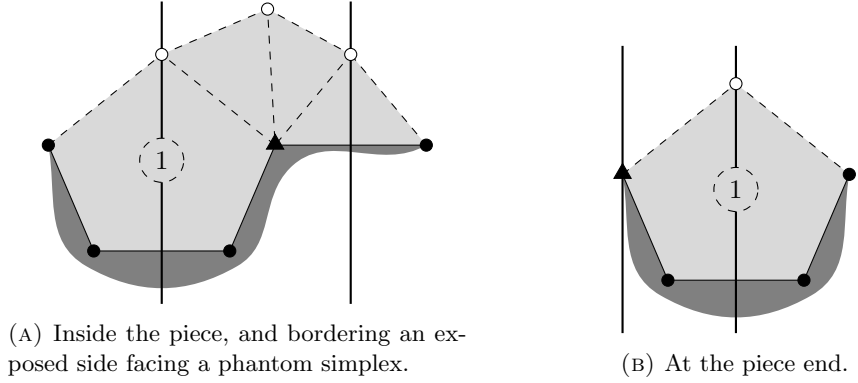

\smallskip

This exhausts all possibilities for $\check{v}$. In every case we confirmed inequality (\ref{eq:len_dft}), saying that the change in boundary length is at most $-6$ times the contribution to defect, plus the correction for 5-orbigons. Summing it over all $\check{v}\in\mathcal{P}$ gives
\begin{align*}
\ell(\partial T')-\ell(\partial T)&=\sum_{\check{v}\in\mathcal{P}}\Delta\ell_{\check{v}}\leqslant\\
&\leqslant-6\sum_{\check{v}\in\mathcal{P}}D_{\check{v}}+\tfrac{1}{10}\sum_{\check{v}\in\mathcal{P}}\#\{\tn{added }5\tn{-orbigons touching }\check{v}\}=\\
&=-6\cdot\dft(\mathcal{P})+\frac{1}{5}\#\{f\in\Forb[5](T'\setminus T)\mid f\text{ is }5\text{-sided}\}.
\end{align*}
We assumed that $\dft(\mathcal{P})\geqslant0$, which proves condition~\ref*{itm:spackling_shorten} of Proposition \ref{prop:spackling_defect}.
\end{proof}

\subsection{Combinatorial negative curvature}\label{ssec:combi_negcurv}

In this section we explore the negative-curvature-like properties of immersed complexes. We prove a discrete variant of Gauss--Bonnet formula; compared to \cite[Theorem 2.3]{sectional_curvature_cpct_cores_local_quasiconvexity_Wise04} our version is specific to surfaces but accounts for cone points. From it we deduce a bound for $\cho$ that will be important for constructing resolutions.

The key concept is an \emph{angle structure}, a function $\angle\colon C(T)\to\mathbb{R}$ assigning a number to every corner of a complex $T$. The difference between the sum of angles obtained from $\angle$ and its value expected from Euclidean geometry is a combinatorial analogue of curvature (in our normalisation the full angle has measure 1). If $f$ is a simplex, then its curvature equals
\begin{equation*}
\kappa(f)\defeq-\frac{1}{2}+\sum_{c\in C(f)}\angle c,
\end{equation*}
and if $f$ is an $m$-orbigon with $\mathfrak{s}_f$ sides and a cone point of order $o=\tfrac{m}{\mathfrak{s}_f}$, then
\begin{equation*}
\kappa(f)\defeq\frac{1}{o}-\frac{\mathfrak{s}_f}{2}+\sum_{c\in C(f)}\angle c.
\end{equation*}
If $\Sigma$ is a branched $\Delta$-complex structure on an orbifold, then the curvature of the (unique according to the convention from Section~\ref{sec:simplicial}) vertex is
\begin{equation*}
\xi\defeq1-\sum_{c\in C(\Sigma)}\angle c.
\end{equation*}
The curvature of a boundary slot $\check{v}\in\partial T$ is the deviation from the straight angle
\begin{equation}\label{eq:curv_site}
\kappa(\check{v})\defeq-\frac{1}{2}+\sum_{c\in C(\check{v})}\angle c.
\end{equation}

As with the continuous analogue, the total curvature equals the orbifold Euler characteristic, independently of the angle structure.
\begin{lemma}(Combinatorial Gauss--Bonnet)\label{lmm:combi_GB}
Let $\angle\colon C(\Sigma)\to\mathbb{R}$ be an angle structure on $\Sigma$. By abuse of notation, denote by $\angle=\angle\circ p\colon C(T)\to\mathbb{R}$ its pullback to $T$. If $T$ has no isolated vertices then
\begin{equation*}
\cho(T)=\sum_{f\in F(T)}\kappa(f)+\xi\cdot\#V(T)+\sum_{\check{v}\in\partial T}\kappa(\check{v}).
\end{equation*}
\end{lemma}
\begin{proof}
At each corner $c$ of a face of $T$, add
\begin{itemize}
\item $\angle c$ for the vertex of $c$
\item $-\tfrac{1}{4}$ for each of the two edges adjacent to $c$, for a total of $-\tfrac{1}{2}$
\item for the face containing $c$, $\tfrac{1}{3}$ if it is a simplex and $\tfrac{1}{m}$ if it is an $m$-orbigon.
\end{itemize}
At each slot $\check{v}$ add
\begin{itemize}
\item $\sum_{c\in C(\check{v})}\angle c$ for the adjacent vertex (this sum is over phantom corners)
\item $-\tfrac{1}{2}$, spread over the two adjacent exposed sides (which exist unless $\check{v}$ is anchored at an isolated vertex).
\end{itemize}
This counts every vertex with total weight $\sum_{c\in C(\Sigma)}\angle c$, edge with $-1$, simplicial face with $+1$, and each $\mathfrak{s}_f$-sided $m$-orbigon face $f$ with total weight $\tfrac{\mathfrak{s}_f}{m}$. Weights of edges and faces exactly match the coefficients with which they enter $\cho(T)$ in equation (\ref{eq:cho_combi}), so
\begin{align*}
\sum_{c\in C(T)}&\left(\angle c-\frac{1}{2}+\frac{1}{m(c)}\right)+\sum_{\check{v}\in\partial T}\left(-\frac{1}{2}+\sum_{c\in C(\check{v})}\angle c\right)=\\
=&\left(\sum_{c\in C(\Sigma)}\angle c\right)\#V(T)-\#E(T)+\#\Fspx(T)+\sum_{f\in\Forb(T)}\frac{\mathfrak{s}_f}{m}=\\
=&\cho(T)+\left(-1+\sum_{c\in C(\Sigma)}\angle c\right)\#V(T),
\end{align*}
where $m(c)=3$ if $c$ is a corner of a simplex. The second term is the boundary curvature $\Sigma_{\check{v}\in\partial T}\kappa(\check{v})$, and the last term is $-\xi\#V(T)$. The summands in the first term can be grouped by the face they belong to, and since $\sum_{c\in C(f)}\left(\angle c-\tfrac{1}{2}+\tfrac{1}{m(c)}\right)=\kappa(f)$, this recovers the sum of face curvatures.
\end{proof}

To extract useful estimates from Lemma~\ref{lmm:combi_GB} we will need a suitable angle structure. Intuitively, it makes the vertex flat and allocates the negative curvature to faces in a controlled way.
\begin{proposition}\label{prop:nice_angles}
If $\cho(\Sigma)<0$ and $\varepsilon\in\mathbb{R}_+$ is sufficiently small, then there exists an angle structure $\angle\colon C(\Sigma)\to\mathbb{R}$ such that the following hold.
\begin{enumerate}
\item For every corner $c$ we have $\angle c\in[0,\tfrac{1}{2}]$.
\item\label{itm:niceang_v} The angles $\angle c$ sum up to 1.
\item\label{itm:niceang_spx} If $c_1,c_2,c_3$ are the corners of the same simplex, then $\angle c_1\!+\!\angle c_2\!+\!\angle c_3\leqslant\tfrac{1}{2}\!-\!\varepsilon$.
\item\label{itm:niceang_orb} If $c$ is the corner of $f\in\Forb$, then $\angle c\leqslant\tfrac{1}{2}\!-\!\frac{1}{m}-\tfrac{\mathbbm{1}_{\{m\equiv\pm1\bmod(6)\}}}{49}-\varepsilon\mathbbm{1}_{\{m\neq2\}}$.
\item\label{itm:niceang_5} If $c$ is the corner of a 5-orbigon, then $\angle c\leqslant\tfrac{13}{50}-\varepsilon$.
\end{enumerate}
\end{proposition}
The role of $\varepsilon$ is to make many faces have a small negative $\kappa$, so that the algorithm for finding resolutions in Section~\ref{ssec:construct_resolution} will make nonzero progress at every step.
\begin{proof}[Proof of Proposition~\ref{prop:nice_angles}]
If $\Sigma$ is non-orientable, or orientable of genus at least 1, or has at least 4 cone points, then its triangulation contains at least two simplices. Assumption $\cho(\Sigma)<0$ excludes the nonsingular torus and Klein bottle, $S^2$ with 4 cone points of order 2, and $\mathbb{RP}^2$ with 2 cone points of order 2, so $\Sigma$ has one more corner $c$ not of a 2-orbigon. Assigning $\angle=\tfrac{1}{6}-\tfrac{\varepsilon}{3}$ to those six simplicial corners, $\angle=2\varepsilon$ to $c$ and zero everywhere else satisfies the conditions of Proposition \ref{prop:nice_angles}.

The only remaining case is that of $\Sigma$ being a sphere with three cone points. Then its triangulation consists of one simplex $f_\vartriangle$ and three orbigons $f_1,f_2,f_3$, say of orders $m_1\leqslant m_2\leqslant m_3$ respectively, as drawn in Figure~\ref{sfig:orbi_S2}. At the corners of $f_\vartriangle$ we set $\angle=\tfrac{1}{6}-\tfrac{\varepsilon}{3}$. The condition $\cho(\Sigma)<0$ translates to
\begin{equation*}
\frac{1}{m_1}+\frac{1}{m_2}+\frac{1}{m_3}<1,
\end{equation*}
which forces one of the following cases.
\begin{enumerate}
\item If $m_1\geqslant3$, then $m_3\geqslant4$. We assign $\angle=\tfrac{1}{6}-\varepsilon$ to the corners of $f_1,f_2$ and $\angle=\tfrac{1}{6}+3\varepsilon$ to $f_3$.
\item If $m_1=2,m_2\geqslant4$, then $m_3\geqslant5$. We assign $\angle=0$ to the corner of $f_1$, $\angle=\tfrac{1}{4}-\varepsilon$ to $f_2$, and $\angle=\tfrac{1}{4}+2\varepsilon$ to $f_3$.
\item If $m_1=2,m_2=3$, then $m_3\geqslant7$. We assign $\angle=0$ to the corner of $f_1$, $\angle=\tfrac{1}{6}-\varepsilon$ to $f_2$, and $\angle=\tfrac{1}{3}+2\varepsilon$ to the corner of $f_3$.
\end{enumerate}
Direct verification shows that all conditions of Proposition \ref{prop:nice_angles} hold in each case.
\end{proof}

Combining Lemma~\ref{lmm:combi_GB} with Proposition~\ref{prop:nice_angles} gives the corollary below. The intuition behind it is that a complex with many faces ``contains a lot of negative curvature'', which forces either large negative $\cho$ or a long boundary.
\begin{corollary}\label{cor:cho_bd}
If $T$ has no isolated vertices and $\varepsilon\in\mathbb{R}_+$ is sufficiently small then
\begin{equation*}
\cho(T)\leqslant\frac{\ell(\partial T)}{2}-\frac{1}{25}\sum_{f\in\Forb[5](T)}\hspace*{-10pt}\mathfrak{s}_f-\frac{1}{49}\sum_{\tiny\begin{matrix}m\equiv_6\pm1,m\geqslant7\\f\in\Forb(T)\end{matrix}}\hspace*{-10pt}\mathfrak{s}_f-\varepsilon\sum_{f\in F(T)\setminus\Forb[2](T)}\hspace*{-10pt}\mathfrak{s}_f,
\end{equation*}
where $\ell(\partial T)$ is the boundary length from Definition~\ref{def:bdry_len}.
\end{corollary}
\begin{proof}
Let $\angle\colon C(\Sigma)\to\mathbb{R}$ be the angle structure we just constructed in Proposition \ref{prop:nice_angles}. Condition~\ref*{itm:niceang_v} of Proposition~\ref{prop:nice_angles} translates to the vertex being flat, i.e. $\xi=0$. For faces we obtain from Proposition~\ref{prop:nice_angles} that
\begin{itemize}
\item $\kappa(f)\leqslant-\tfrac{1}{25}-\varepsilon$ when $f\in F(\Sigma)$ is a 5-orbigon (with one side according to the convention from Section~\ref{sec:simplicial}), by condition~\ref*{itm:niceang_5}
\item $\kappa(f)\leqslant-\tfrac{1}{49}-\varepsilon$ when $f\in\Forb(\Sigma)$ with any other $m\equiv\pm1\bmod(6)$, by condition~\ref*{itm:niceang_orb}
\item $\kappa(f)=0$ for $f\in\Forb[2](\Sigma)$, again by \ref*{itm:niceang_orb}
\item $\kappa\leqslant-\varepsilon$ for any other face, by \ref*{itm:niceang_spx} and \ref*{itm:niceang_orb}.
\end{itemize}
Therefore combinatorial Gauss--Bonnet~\ref{lmm:combi_GB} gives
\begin{align}\begin{split}\label{eq:GB_application}
\cho(T)\leqslant&-\frac{1}{25}\sum_{f\in\Forb[5](T)}\hspace*{-10pt}\mathfrak{s}_f-\frac{1}{49}\sum_{\tiny\begin{matrix}m\equiv_6\pm1,m\geqslant7\\f\in\Forb(T)\end{matrix}}\mathfrak{s}_f-\\
&-\varepsilon\sum_{f\in F(T)\setminus\Forb[2](T)}\hspace*{-10pt}\mathfrak{s}_f+\sum_{\check{v}\in\partial T}\kappa(\check{v}).
\end{split}\end{align}
The sums over faces are exactly the term we want in Corollary \ref{cor:cho_bd}, so it only remains to bound the boundary contribution.

Since the angles are non-negative and sum to 1 over all corners of $\Sigma$, we have $\kappa(\check{v})\leqslant\tfrac{1}{2}$ for every slot $\check{v}$. Furthermore, if $\check{v}$ is a single phantom corner of an orbigon, then $\kappa(\check{v})\leqslant0$ since all angles are at most $\tfrac{1}{2}$. Overall we obtain
\begin{align}\begin{split}\label{eq:bdry_curv_bd}
&\sum_{\check{v}\in\partial T}\kappa(\check{v})\leqslant\frac{1}{2}\#A\\
\text{where}\quad&A=\{\check{v}\in\partial T\mid\check{v}\text{ not in the interior of an orbiarc}\}.
\end{split}\end{align}

The elements of the set $A$ from inequality (\ref{eq:bdry_curv_bd}) sit on circular boundary components, and every consecutive two are separated by an exposed side facing a phantom simplex or a non-empty orbiarc. The former have length 1, and the latter have length at least 1 by condition~\ref*{itm:orbi_prop_atleast1} of Proposition~\ref{prop:orbilen}. Therefore the total boundary length is at least the size of $A$, i.e. $\#A\leqslant\ell(\partial T)$. Combining this with equations (\ref{eq:GB_application}, \ref{eq:bdry_curv_bd}) completes the proof.
\end{proof}

\subsection{Constructing resolutions}\label{ssec:construct_resolution}

In this section we combine the boundary length from Section~\ref{ssec:bdry_len} with combinatorial curvature estimates from Section~\ref{ssec:combi_negcurv} to produce a new invariant we call score. We show that it enjoys two properties: first (Proposition~\ref{prop:resolve_defect}), a complex with nonnegative maximal defect can be grown in a way that decreases the score and preserves immersions into coverings, and second (Proposition~\ref{prop:score_bd}), score is an upper bound for the quantity $\cho+\md$ that appears in our master Theorem~\ref{thm:E_combi_emb}. Finally we describe an algorithm for constructing resolutions by descent on the score, and thus prove Proposition~\ref{prop:resolution}.
\begin{definition}\label{def:score}
The \emph{score} of an immersed complex $T\looparrowright\Sigma$ is
\begin{align*}
\tn{score}(T)\defeq&\ \ell(\partial T)+\#\{v\in V(T)\mid v\text{ isolated}\}-\\
&-\frac{1}{25}\sum_{f\in\Forb[5](T)}\hspace*{-10pt}\mathfrak{s}_f-\frac{1}{49}\sum_{\tiny\begin{matrix}m\equiv_6\pm1,m\geqslant7\\f\in\Forb(T)\end{matrix}}\mathfrak{s}_f-\varepsilon\hspace*{-10pt}\sum_{f\in F(T)\setminus\Forb[2](T)}\hspace*{-10pt}\mathfrak{s}_f,
\end{align*}
where $\varepsilon\in\mathbb{Q}_+$ is sufficiently small for Corollary~\ref{cor:cho_bd} to hold.
\end{definition}

We will need to rule out some silly reasons that preclude $T$ from embedding into any covering of $\Sigma$.
\begin{definition}\label{def:nonpato}
Immersed complex $T$ is \emph{non-pathological} when
\begin{enumerate}
\item every $m$-orbiarc with $m$ exposed sides is a circle
\item every circular $m$-orbiarc with $l$ exposed sides satisfies $l|m$.
\end{enumerate}
\end{definition}

The first key result of this section shows that any complex with nonnegative boundary defect can be replaced with complexes of smaller score while preserving the set of its embeddings into coverings of $\Sigma$. We will use it in the proof of Proposition~\ref{prop:resolution} as the descent step to remove complexes with problematic boundaries from a resolution.
\begin{proposition}\label{prop:resolve_defect}
If $\md(T)\geqslant0$ and $T$ is non-pathological, then there exists an embedding $T\hookrightarrow T'$ satisfying the following conditions.
\begin{enumerate}
\item Every immersion $T\looparrowright\whS$ of $T$ into a covering $\whS$ extends uniquely to an immersion $T'\looparrowright\whS$ of $T'$.
\item All quotients $T'\twoheadrightarrow T''$ in which $T$ embeds satisfy $\tn{score}(T'')<\tn{score}(T)$ or $\md(T'')<0$.
\end{enumerate}
\end{proposition}
\begin{proof}
First we take care of three nuisance possibilities. First, if $\partial T$ contains a circular $m$-orbiarc of $l$ exposed sides, then $l|m$ so the phantom orbigon can be added to $T$ at this orbiarc. This operation decreases $\ell(\partial T)$ and hence the score, and preserves immersions into coverings, so we are done. Second, if $\partial T$ contains an exposed side facing a phantom 2-orbigon $f$, then $T_1\defeq T\cup f$ is also non-pathological; if $\md(T_1)<0$ then we are done (the only quotient in which $T$ embeds is $T_1$ itself), and otherwise if $\md(T_1)\geqslant0$ then we can proceed with the proof for $T_1$ and compose the resulting $T_1\hookrightarrow T'$ with $T\hookrightarrow T_1$. Third, if $v\in V(T)$ is isolated then we run the argument with $T\setminus\{v\}$. Therefore, without loss of generality we assume that $\partial T$ contains no circular orbiarcs or exposed sides facing a phantom 2-orbigon, and no isolated vertices. We can now move on to the core of the proof. There are two cases, depending on whether $\partial T$ contains a long orbiarc.

The first case is when $\partial T$ contains no long orbiarcs. Then Proposition~\ref{prop:spackling_defect} produces a non-trivial embedding $T\hookrightarrow T'$ that already satisfies the first condition. Additionally
\begin{equation*}
\ell(\partial T')\leqslant\ell(\partial T)+\frac{1}{5}\#\{f\in\Forb[5](T'\setminus T)\mid\mathfrak{s}_f=5\},
\end{equation*}
every $f\in\Forb[5](T'\setminus T)$ has at least 3 vertices in $T$, and $F(T'\setminus T)\neq\emptyset$. We claim that these properties imply the second condition too, and hence $T'$ witnesses Proposition~\ref{prop:resolve_defect}.

Proving the second condition is easier when $\Sigma$ has no cone points of order 5. Consider a quotient immersed complex $T'\twoheadrightarrow T''$ such that $T$ embeds in $T''$. We asserted in Proposition~\ref{prop:quot_short_bdry} that boundary length does not increase under quotients, so $\ell(\partial T'')\leqslant\ell(\partial T')$, which in turn is at most $\ell(\partial T)$ by the construction of $T'$. The function $F(T)\to F(T'')$ induced by $T\hookrightarrow T''$ is an injection preserving the number of sides, so the sums of $\mathfrak{s}_f$ over faces of $T''$ are no smaller than the respective sums over faces of $T$. Moreover, $F(T''\setminus T)$ is nonempty and cannot contain exclusively 2-orbigons, so the sum over $F\setminus\Forb[2]$ is strictly larger for $T''$ than for $T$. Combining these estimates gives $\tn{score}(T'')<\tn{score}(T)$.

When $\Sigma$ has cone points of order 5 we need to adjust the argument. From Proposition~\ref{prop:quot_short_bdry} we still have $\ell(\partial T'')\leqslant\ell(\partial T')$. If any two $f_1\neq f_2\in\Forb[5](T'\setminus T)$ were identified in $T''$, then by the pigeonhole principle some vertex $v_1\in f_1\cap T$ would be identified in $T''$ with some $v_2\in f_2\cap T$, which is impossible if $T$ embeds in $T''$. This means that the set $\Forb[5](T'\setminus T)$ injects into $\Forb[5](T''\setminus T)$. This cancels the correction to $\ell(\partial T')$ and gives
\begin{equation*}
\ell(\partial T'')-\frac{1}{25}\sum_{f\in\Forb[5](T''\setminus T)}\mathfrak{s}_f\leqslant\ell(\partial T).
\end{equation*}
Adding to the above inequality the sums over $\Forb[5](T)$ and other faces for $m\neq5$ establishes $\tn{score}(T'')<\tn{score}(T)$.

The second case is when $\partial T$ contains a long orbiarc, say with $l$ edges. Let $T'$ be obtained by filling in the missing $m$-orbigon $f$. Because the long orbiarc contains more than half of all vertices of $f$, the only quotient of $T'$ in which $T$ embeds is $T'$ itself. Adding $f$ covers a boundary arc of length $Ł(m,l)$ but creates $m-l$ new sides of edges, so condition \ref*{itm:orbi_prop_loa} of Proposition~\ref{prop:orbilen} implies that
\begin{equation*}
\ell(\partial T')\leqslant\ell(\partial T)+\frac{1}{m}\left(\mathbbm{1}_{\{m\equiv\pm1\bmod(6),m\geqslant7\}}+\mathbbm{1}_{\{m=11\}}\right).
\end{equation*}
Again, more than half of $V(f)$ were already present in $T$, so $f$ must have $m$ edges ($\mathfrak{s}_f=m$) and a cone point of order 1. Therefore, the sums over faces in the score of $T'$ gain a new summand equal
\begin{equation*}
-\frac{m}{49}\cdot\mathbbm{1}_{\{m\equiv\pm1\bmod(6),m\geqslant7\}}-m\varepsilon.
\end{equation*}
This compensates for the potential increase in boundary length, so $\tn{score}(T')<\tn{score}(T)$.
\end{proof}

The second key result of this section is that score upper-bounds the exponent of expected number of embeddings from Theorem~\ref{thm:E_combi_emb}. It will appear in the proof of Proposition~\ref{prop:resolution} as the stopping condition.
\begin{proposition}\label{prop:score_bd}
If $T$ is non-pathological, then
\begin{equation*}
\cho(T)+\max\{0,\md(T)\}\leqslant\tn{score}(T).
\end{equation*}
\end{proposition}
\begin{proof}
Addition of an isolated vertex increases both sides of the inequality by 1, so without loss of generality we can assume $T$ does not have any. Corollary~\ref{cor:cho_bd} gives an upper bound on $\cho(T)$, with which it is sufficient to prove
\begin{equation}\label{eq:dft_bdrylen}
\md(T)\leqslant\frac{\ell(\partial T)}{2}.
\end{equation}
This is what we do in the rest of the proof.

Consider a boundary piece $\mathcal{P}\subseteq\partial T$. Suppose $\mathcal{P}$ contains $a+b$ orbiarcs, $i$-th of which consists of $l_i$ edges of a phantom $m_i$-orbigon, such that the first $a$ are segments and the last $b$ are circles. The $i$-th segment orbiarc has at its end either a hanging half-edge (contributing $-\tfrac{1}{3}+\tfrac{1}{2m_i}$ to the defect) or an end of the piece (contributing $-\tfrac{1}{2}$). Therefore, $i$-th segment orbiarc contributes at most $\tfrac{l_i+1}{m_i}-\tfrac{2}{3}$ to $\dft(\mathcal{P})$. All other hanging half-edges enter the defect with negative weights, so ignoring them bounds the defect as
\begin{equation}\label{eq:dft_arcwise_bd}
\dft(\mathcal{P})\leqslant\frac{\mathfrak{e}_\tn{spx}(\mathcal{P})}{3}+\sum_{i=1}^a\left(\frac{l_i+1}{m_i}-\frac{2}{3}\right)+\sum_{i=a+1}^{a+b}\frac{l_i}{m_i}.
\end{equation}
By the assumption of $T$ being non-pathological, we know that $l_i\leqslant m_i-1$ for $1\leqslant i\leqslant a$. Therefore, by condition~\ref*{itm:orbi_prop_atleast1} of Proposition~\ref{prop:orbilen} for $1\leqslant i\leqslant a$ we have
\begin{equation*}
\frac{l_i+1}{m_i}-\frac{2}{3}\leqslant\frac{1}{3}\leqslant\frac{Ł(m_i,l_i)}{2}.
\end{equation*}
Again by $T$ being non-pathological, for $a<i\leqslant a+b$ we have $l_i\leqslant m_i$, so
\begin{equation*}
\frac{l_i}{m_i}\leqslant\frac{1}{2}\cdot\min\left\{2,\frac{2l_i}{m_i}\right\}.
\end{equation*}
Plugging these estimates into inequality (\ref{eq:dft_arcwise_bd}) proves $\dft(\mathcal{P})\leqslant\tfrac{\ell(\partial T)}{2}$. Maximising over $\mathcal{P}$ proves inequality (\ref{eq:dft_bdrylen}).
\end{proof}

Now, we are ready to describe an algorithm for finding resolutions and prove Proposition \ref{prop:resolution}. We will use the score from Definition \ref{def:score} as a complexity measure for immersed complexes. The algorithm is a descent on complexity, with descent step applying Proposition \ref{prop:resolve_defect} to complexes having $\md\geqslant0$. Proposition \ref{prop:score_bd} provides the stopping condition and guarantees correctness once it is reached.
\begin{proof}[Proof of Proposition \ref{prop:resolution}]
Enumerate all equivalence relations on the vertices $V(T)$ of $T$, Stallings-fold the resulting quotients, and ignore all pathological ones in the sense of Definition~\ref{def:nonpato}. Denote the resulting set of morphisms of immersed complexes by $\mathcal{R}=\{T\twoheadrightarrow W\}$. Every immersion of $T$ in a covering of $\Sigma$ is an embedding of a quotient, so it factors through a unique element of $\mathcal{R}$, and $\mathcal{R}$ is a finite resolution (but yet we have no control on boundaries of its elements).

We will now modify $\mathcal{R}$ to ensure the condition on maximal defect. While $\mathcal{R}$ contains a morphism $q\colon T\looparrowright W$ with $\md(W)\geqslant0$ and $\tn{score}(W)\geqslant K$ execute the following step. Apply Proposition~\ref{prop:resolve_defect} to obtain an embedding $\iota\colon W\hookrightarrow W'$. Let $\{r_i\colon W'\twoheadrightarrow W''_i\}_{i\in I}$ be all the (finitely many) Stallings-folded non-pathological quotients of $W'$ in which $W$ embeds. Modify the resolution $\mathcal{R}$ as
\begin{equation}\label{eq:resolution_algo_step}
\mathcal{R}\qquad\longmapsto\qquad\mathcal{R}\ \cup\ \{r_i\circ\iota\circ q\colon T\looparrowright W''_i\mid i\in I\}\ \setminus\ \{(q,W)\}
\end{equation}
by replacing the bad-boundary immersed complex $(q,W)$ with quotients of $W'$. Proposition~\ref{prop:resolve_defect} guaranteed that every embedding $W\hookrightarrow\whS$ in a covering extends to an immersion $W'\looparrowright\whS$, which in turn factors uniquely as one of the morphisms $r_i$ followed with an embedding of $W''_i$ into $\whS$. Therefore, $\mathcal{R}$ remains a resolution after the operation (\ref{eq:resolution_algo_step}).

It remains to prove that this process terminates. The resolution $\mathcal{R}$ contains only non-pathological complexes throughout the whole process. By Proposition~\ref{prop:resolve_defect}, the immersed complexes $W''_i$ added in (\ref{eq:resolution_algo_step}) have score strictly smaller than the score of removed $W$, or $\md<0$. The score takes discrete values (condition~\ref*{itm:orbi_prop_denominator} of Proposition~\ref{prop:orbilen}, and we chose $\varepsilon\in\mathbb{Q}$ in Definition~\ref{def:score}), so decreases by at least some fixed amount. Therefore, after finitely many steps all complexes in $\mathcal{R}$ have negative defect or score smaller than $K$. Then Proposition~\ref{prop:score_bd} guarantees that complexes in the latter group have $\cho+\max\{0,\md\}<K$.
\end{proof}

\subsection{Proofs of Proposition~\ref{prop:no_cap} and Lemma~\ref{lmm:no_dft_subgp}}\label{ssec:pfs_aux_combi}

Equipped with the tools of combinatorial negative curvature from Section~\ref{ssec:combi_negcurv}, we are ready to complete two proofs deferred from Section~\ref{ssec:res_main_pf} --- of Lemma~\ref{lmm:no_dft_subgp} and Proposition~\ref{prop:no_cap}. Recall that they guarantee predictable behaviour of immersed complexes under embeddings, conditional on controlled boundary defect and mimicking geodesic convexity.

To prove Proposition~\ref{prop:no_cap} we will need a special case of Gauss--Bonnet formula~\ref{lmm:combi_GB}, in the form that relates Euler characteristic to defect. We pick an angle structure that makes the faces flat and concentrates the negative curvature at the vertex (which is the opposite to the situation in Proposition~\ref{prop:nice_angles}).
\begin{corollary}[of Lemma~\ref{lmm:combi_GB}]\label{cor:cho_V_dft}
We have
\begin{equation*}
\cho(T)=\cho(\Sigma)\cdot\#V(T)-\dft(\partial T).
\end{equation*}
\end{corollary}
\begin{proof}
Choose an angle structure on $\Sigma$ by setting $\angle(c)=\tfrac{1}{6}$ when $c$ is a corner of a simplex and $\angle(c)=\tfrac{1}{2}-\tfrac{1}{m}$ when $c$ is a corner of an $m$-orbigon. This way $\kappa(f)=0$ for every $f\in F(\Sigma)$, so by Lemma~\ref{lmm:combi_GB} we must have $\xi=\cho(\Sigma)$ at the vertex.

For a slot $\check{v}$ define $\kappa'(\check{v})$ as the sum of
\begin{itemize}
\item $-\tfrac{1}{6}$ for each exposed side adjacent to $\check{v}$ and facing a phantom simplex
\item $-\tfrac{1}{2m}$ for each adjacent exposed side facing a phantom $m$-orbigon
\item $\tfrac{1}{6}$ for each adjacent hanging half-edge separating two phantom simplices
\item $\tfrac{1}{3}-\tfrac{1}{2m}$ for each adjacent hanging half-edge separating a simplex and $m$-orbigon.
\end{itemize}
Then by Definition~\ref{def:dft} we have $\dft(\partial T)=-\sum_{\check{v}\in\partial T}\kappa'(\check{v})$. Note that the contribution of each exposed side to $\dft$ was split in half and attributed to two adjacent slots, similarly as in the proof of Proposition~\ref{prop:spackling_defect}.

Corollary~\ref{cor:cho_V_dft} will follow from Lemma~\ref{lmm:combi_GB} once we establish $\kappa'(\check{v})=\kappa(\check{v})$, where $\kappa(\check{v})$ is the slot curvature induced by $\angle$ according to equation (\ref{eq:curv_site}). The first quantity $\kappa'(\check{v})$ depends on exposed sides and hanging half-edges adjacent to $\check{v}$, while the second $\kappa(\check{v})$ involves the corners of $\check{v}$. The combinatorial objects involved in one sum are interlaced with those in the other, so to prove equality we find a common refinement. Namely, consider all incidences between hanging half-edges or exposed sides and phantom faces, and add
\begin{itemize}
\item $-\tfrac{1}{6}$ for an inclusion of exposed side in a phantom simplex
\item $-\tfrac{1}{2m}$ for an inclusion of exposed side in a phantom $m$-orbigon
\item $\tfrac{1}{12}$ for an inclusion of hanging half-edge in a phantom simplex
\item $\tfrac{1}{4}-\tfrac{1}{2m}$ for inclusion of hanging half-edge in a phantom $m$-orbigon.
\end{itemize}
Grouping such summands by edges gives $\kappa'(\check{v})$, while grouping them by phantom corners gives $\kappa(\check{v})$. Therefore indeed we have $\kappa'=\kappa$.
\end{proof}
\begin{remark}\label{rmk:inner_dft}
Perhaps a more illuminating way to state Corollary~\ref{cor:cho_V_dft} would be to write it as
\begin{equation*}
\cho(T)=\cho(\Sigma)\cdot\#\left\{\begin{matrix}\text{vertices of }T\\\text{not on the boundary}\end{matrix}\right\}+\left(\begin{matrix}\text{inner-facing}\\\text{boundary defect}\end{matrix}\right).
\end{equation*}
It would also simplify the upcoming proof of Proposition~\ref{prop:no_cap}. Essentially, what will happen there is that when a component $\gamma$ of $\partial T$ is capped with a finite complex $C$, then $\dft(\gamma)$ becomes the inner-facing defect of $C$, and hence an upper bound for $\cho(C)$. However, we stick with the current phrasing of Corollary~\ref{cor:cho_V_dft} to avoid the complication of rigorously defining the ``inner-facing'' boundary defect.
\end{remark}

Recall that for an embedding $T\hookrightarrow\whS$ of a complex into a covering, Proposition~\ref{prop:no_cap} bounded the difference $\cho(\Sigma)-\cho(T)$ in terms of $\md(T)$, and provided a sufficient condition for $T$ to be a deformation retract. Now we are ready to deduce it from Corollary~\ref{cor:cho_V_dft}.
\begin{proof}[Proof of Proposition~\ref{prop:no_cap}]
Take an embedding $T\hookrightarrow\whS$. First we deal with statement~\ref*{itm:nocap_general}. Let $W$ be the union of $T$ and the finite connected components of its complement in $\whS$. Then every (open) connected component of $\whS\setminus W$ is infinite, so its orbifold Euler characteristic is nonpositive (as a topological surface it has at least one funnel or is of infinite type). Therefore
\begin{equation}\label{eq:cho_covering_subcplx}
\cho(\whS)\leqslant\cho(W).
\end{equation}
We need to relate the characteristics of $T$ and $W$ to $\partial T$. Applying Corollary~\ref{cor:cho_V_dft} to $T$ and $W$ and taking the difference yields
\begin{equation*}
\cho(W)-\cho(T)=\cho(\Sigma)\cdot\#V(W\setminus T)+\dft(\partial T)-\dft(\partial W)
\end{equation*}
The first term on the RHS is nonpositive. Each boundary component of $T$ either remains a boundary component of $W$ in its entirety, or is completely covered in $W$. Therefore
\begin{equation*}
\dft(\partial T)-\dft(\partial W)=\dft(\mathcal{P}),
\end{equation*}
where the boundary piece $\mathcal{P}$ is the union of those components of $\partial T$ that touch a finite component of $\whS\setminus T$. By definition we have
\begin{equation}\label{eq:dft_diff_dft}
\max\{0,\md(T)\}\geqslant\dft(\mathcal{P}),
\end{equation}
zero appearing because $\mathcal{P}$ may be empty. Combining the above inequalities gives $\cho(\whS)\leqslant\cho(T)+\max\{0,\md(T)\}$, which is precisely the statement~\ref*{itm:nocap_general}.

To prove statement~\ref*{itm:nocap_nodft} we need to analyse when the above become equalities, under the additional assumption $\md(T)<0$. Equality in (\ref{eq:dft_diff_dft}) tells us that the collection $\mathcal{P}$ of components of $\partial T$ covered by $W$ must have defect 0, so by assumption it is empty. This means that $W=T$, or equivalently all components of $\whS\setminus T$ are infinite. Equality in (\ref{eq:cho_covering_subcplx}) requires all of them to be topological annuli, i.e. funnels. These deformation retract onto their respective boundaries, so $\whS$ deformation retracts onto $T$.
\end{proof}

\begin{remark}
By ``complement component'' in the proof of Proposition~\ref{prop:no_cap} we meant the open component and not its closure. This is because we allow $T$ to contain dangling contractible subcomplexes attached to the boundary at a single vertex, analogously to leaves of a tree or spurs in \cite[Definition 2.4]{sectional_curvature_cpct_cores_local_quasiconvexity_Wise04}. When these are present, $\partial T$ is only a circle when treated as a ``slight push-off'' away from $T$ into $\whS\setminus T$, but not when thought of as a subcomplex of $T$. Consequently, Euler characteristic is additive if we work with open complement components, but not necessarily with their closures.
\end{remark}

We move on to the proof of Lemma~\ref{lmm:no_dft_subgp}. Recall that it promised a few convexity-like properties for complexes with $\md<0$, including being a subcomplex of some covering space of $\Sigma$ (akin to sub-cover from \cite[Definition 2.1]{local_stats_random_perms_free_prods_PuderZimhoni24}), and $\pi_1$-injectivity of the immersion.
\begin{proof}[Proof of Lemma~\ref{lmm:no_dft_subgp}]
We will prove that if $\md(T)<0$ then $T$ embeds in $\sfrac{\widetilde\Sigma}{p_*(\pio T)}$ as a deformation retract. The other statements in Lemma~\ref{lmm:no_dft_subgp} are easy consequences.

First we deal with embedding. Our master counting result, Theorem~\ref{thm:E_combi_emb}, tells us that the expected number of embeddings of $T$ in a degree-$n$ covering is $n^{\cho(T)}(1+o(1))$. It is positive for sufficiently large $n$, so there exists a finite-degree covering $\whS\to\Sigma$ in which $T$ embeds. Then the fundamental group $\pio\whS$ must contain $p_*(\pio T)$, so $\whS$ is covered by $\sfrac{\widetilde\Sigma}{p_*(\pio T)}$. This forms a diagram
\begin{equation*}\begin{tikzcd}
T\arrow[r,{Glyph[glyph math command=looparrowleft, swap]}->]&\sfrac{\widetilde\Sigma}{p_*(\pio T)}\arrow[r]&\whS
\end{tikzcd}\end{equation*}
where the first map is our immersion of interest, second is a covering, and the composition is injective on vertices. But then the first map must be vertex-injective too, i.e. be an embedding.

Second, we establish the deformation retraction. We just proved that the morphism $T\to\sfrac{\widetilde\Sigma}{p_*(\pio T)}$ is an embedding, so can be considered as an inclusion of a subcomplex. By definition, it induces a surjection on fundamental groups. This means that any connected component of the complement of $T$ can only meet $T$ at a single boundary component, and by classification of surfaces its underlying topological surface must be a disk or an annulus. In the former case it cannot contain more than 2 cone points, and in the latter case it must be non-singular. A disk with 2 cone points is impossible if $\cho(\Sigma)<0$. Therefore, a component of $\left(\sfrac{\widetilde\Sigma}{p_*(\pio T)}\right)\setminus T$ can only be a funnel or a disk with zero or one cone point. Hence
\begin{equation*}
\cho\left(\sfrac{\widetilde\Sigma}{p_*(\pio T)}\right)\geqslant\cho(T).
\end{equation*}
From statement~\ref*{itm:nocap_general} of Proposition~\ref{prop:no_cap} we obtain the inequality in the opposite direction, so both Euler characteristics must be equal. Then statement~\ref*{itm:nocap_nodft} of Proposition~\ref{prop:no_cap} implies that $T$ is a deformation retract (the last argument is equivalent to Proposition~\ref{prop:no_cap} ruling out disk and once-singular disk complement components and leaving only funnels).
\end{proof}

\section{Proofs of auxiliary results}

\subsection{Lemma \ref{lmm:change_sum_tors}}\label{ssec:pf_change_sum_tors}

A slightly more general statement is true.
\begin{lemma}\label{lmm:change_sum_irreps}
Let $G$ be a finite group, $U$ an irreducible representation of a subgroup $H\leq G$, $x\in G$ commute with $H$, and $C\subseteq G$ be a conjugacy class in $G$. Then
\begin{equation*}
\sum_{V\in\tn{Irr}(G)}\#C\cdot\chi_V(C)\cdot\tr\left(x\curvearrowright\Hom[H]{U}{V}\right)=[G:H]\sum_{y\in H\cap xC}\chi_U(y).
\end{equation*}
\end{lemma}
Here $\chi_V(C)$ is the value of $\chi_V$ on any element of $C$, which makes sense because characters are constant on conjugacy classes.
\begin{proof}[Proof of Lemma \ref{lmm:change_sum_tors}]
Sum the main equality of Lemma \ref{lmm:change_sum_irreps} over all conjugacy classes $C$ consisting of $m$-torsion elements.
\end{proof}
\begin{proof}[Proof of Lemma \ref{lmm:change_sum_irreps}]
Start by
\begin{align}\begin{split}\label{eq:double_summation_lhs}
\frac{1}{\#H}\sum_{\tiny\begin{matrix}y\in H\\z\in C\end{matrix}}&\chi_U(y)\sum_{V\in\tn{Irr}(G)}\chi_V(1)\chi_V(y^{-1}xz)=\\
&=\frac{1}{\#H}\sum_{\tiny\begin{matrix}y\in H\\z\in C\end{matrix}}\chi_U(y)\chi_\tn{reg}(y^{-1}xz)=\\
&=\frac{\#G}{\#H}\sum_{\tiny\begin{matrix}y\in H\\z\in C\\y=xz\end{matrix}}\chi_U(y)=[G:H]\sum_{y\in H\cap xC}\chi_U(y),
\end{split}\end{align}
where $\chi_\tn{reg}$ is the character of the regular representation. This yields the RHS in Lemma~\ref{lmm:change_sum_irreps}. On the other hand, changing the order of summation on the LHS of equation (\ref{eq:double_summation_lhs}) gives
\begin{align}\begin{split}\label{eq:summed_ccl}
\sum_{V\in\tn{Irr}(G)}&\chi_V\left(\left(\frac{1}{\#H}\sum_{y\in H}\chi_U(y)y^{-1}\right)x\left(\chi_V(1)\sum_{z\in C}z\right)\right)=\\
&=\sum_{V\in\tn{Irr}(G)}\#C\cdot\chi_V(C)\cdot\chi_V\left(\left(\frac{1}{\#H}\sum_{y\in H}\chi_U(y)y^{-1}\right)x\right)
\end{split}\end{align}
where to perform the summation over $C$ we used that $\chi_V(1)\sum_{z\in C}z=\#C\chi_V(C)\tn{id}_V$ as endomorphisms of $V$ (which follows from Schur's Lemma). To simplify (\ref{eq:summed_ccl}) further, recall there is a $H\cdot Z_G(H)$-equivariant decomposition
\begin{equation*}
V=\bigoplus_{U'\in\tn{Irr}(H)}U'\otimes\Hom[H]{U'}{V},
\end{equation*}
which tells us that
\begin{align}\begin{split}\label{eq:filter_Hrep}
\chi_V&\left(\left(\frac{1}{\#H}\sum_{y\in H}\chi_U(y)y^{-1}\right)x\right)=\\
&=\sum_{U'\in\tn{Irr}(H)}\chi_{U'}\left(\frac{1}{\#H}\sum_{y\in H}\chi_U(y)y^{-1}\right)\cdot\tr\left(x\curvearrowright\Hom[H]{U'}{V}\right)=\\
&=\tr\left(x\curvearrowright\Hom[H]{U}{V}\right),
\end{split}\end{align}
where in the last equality we used the orthonormality of irreducible characters. Combining equations (\ref{eq:double_summation_lhs}, \ref{eq:summed_ccl}, \ref{eq:filter_Hrep}) gives the desired equality in Lemma \ref{lmm:change_sum_irreps}.
\end{proof}

\subsection{Proposition \ref{prop:prod_sgn_compat}}\label{ssec:pf_prod_sgn_compat}

Denote our quantity of interest by $S(T,\vec\gamma)\defeq\prod_{c\in C(\Sigma)}\tn{sgn}(\gamma_c)$. We will reduce $T$ and the corner permutations into more manageable data, in a way that preserves $S$.
\begin{proof}[Proof of Proposition~\ref{prop:prod_sgn_compat}]
In the first step we get rid of the faces of $T$. Consider $f\in\Fspx(T)$. Removing $f$ from $T$ creates 3 new exposed sides, and assigning to each of them the numeric label $n-\mathfrak{f}_{p(f)}+1$ extends the boundary numbering to $\partial(T\setminus f)$. The permutations $(T\setminus f)$-matching this new numbering differ from the old $\vec\gamma$ only at the three corners $c\in C(p(f))$, where they are extensions of $\gamma_c\in S_{[n-\mathfrak{v}+1,n-\mathfrak{f}_{p(f)}]}$ by the identity on the 1-element set $\{n-\mathfrak{f}_{p(f)}+1\}$. This way all signs are the same. Similarly, removing $f\in\Forb(T)$ from $T$ creates a new circular boundary $m$-orbiarc with $\mathfrak{s}_f$ exposed sides. We arbitrarily number the new sides with $[n-\mathfrak{f}_{p(f)}+1,n-\mathfrak{f}_{p(f)}+\mathfrak{s}_f]$, and the new $(T\setminus f)$-matching permutations differ from old $\vec\gamma$ only at the unique corner $c$ of $p(f)$, where we extend $\gamma_c$ by an $\mathfrak{s}_f$-cycle. Since $\mathfrak{s}_f|m$, if $m$ is odd then this cycle is an even permutation, so again all signs are the same. By repeating these operations, without loss of generality we can assume that $T$ is a graph.

In the second step we get rid of edges. Suppose we pick a new numbering at a side of an edge $e\in E(\Sigma)$ adjacent to corners $c_\pm$. This affects the corner permutations by pre- or postmultiplying $\gamma_{c_\pm}$ by $\delta^{\pm1}$, where $\delta\in S_{[n-\mathfrak{e}_e+1,n]}$ converts between new and old numeric labels (we used a similar argument in Figure~\ref{fig:tr_not_affected} to prove equation (\ref{eq:tr_not_affected})). Therefore $\gamma_{c_\pm}$ either both keep or both change their signs, so $S$ is preserved. This means $S$ is independent of the numbering of exposed sides, so let us pick one where the numeric labels are the same on both sides of each edge. If we now cut $e\in E(T)$ into two hanging half-edges, and number them with the numeric label of sides of $e$, then the $T$-matching permutations will not change, because they will see the same numbers. By repeating this at every edge, without loss of generality we can assume that $T$ is a disjoint union of vertices.

In the third step, suppose we change the numeric labels of hanging half-edges in $H_\iota$ for some incidence $\iota\colon v\hookrightarrow e$ lying between corners $c_\pm$. Similarly as in the second step, this multiplies $\gamma_{c_\pm}$ by some permutation of $[n-\mathfrak{v}+1,n]$ or its inverse, so the product of signs of $\gamma_{c_\pm}$ and hence $S$ is preserved. This means that $S$ is the same for any numbering of hanging half-edges. Choosing it to be a function of the vertex makes all $T$-matching permutations be the identity, so $S=1$.
\end{proof}

\subsection{Proposition \ref{prop:tall_tors_bd}}\label{ssec:pf_tall_tors_bd}

In the proof below we use the standard notation $\langle U,V\rangle\defeq\dim\Hom[G]{U}{V}$ for the multiplicity of a $G$-representation $U$ in $V$.
\begin{proof}[Proof of Proposition~\ref{prop:tall_tors_bd}]
Let us relate $(\lambda^\vee[n])^\vee$ to something more computation-friendly. Consider the $S_n$-representation
\begin{equation*}
V_n\defeq\tn{Ind}^{S_n}_{S_{n-|\lambda|}\times S_{[n-|\lambda|+1,n]}}\tn{sgn}\otimes\lambda
\end{equation*}
obtained by induction from the sign representation of $S_{n-|\lambda|}$ tensored with the representation $\lambda$ of $S_{[n-|\lambda|+1,n]}\cong S_{|\lambda|}$. We claim that $V_n$ is a direct sum of $(\lambda^\vee[n])^\vee$ and a bounded number of $((\lambda')^\vee[n])^\vee$ with $|\lambda'|<|\lambda|$, i.e.
\begin{equation}\label{eq:find_in_ind}
V_n=(\lambda^\vee[n])^\vee\oplus\bigoplus_{\lambda'}((\lambda')^\vee[n])^\vee.
\end{equation}
Indeed, Frobenius reciprocity for an $S_n$-representation $\mu$ says
\begin{equation}\label{eq:Frob_reciprocity}
\left\langle V_n,\mu\right\rangle=\left\langle\tn{sgn}\otimes\lambda,\tn{Res}^{S_n}_{S_{n-|\lambda|}\times S_{[n-|\lambda|+1,n]}}\mu\right\rangle.
\end{equation}
By equation (\ref{eq:branching_rule}) from Section \ref{sec:rep_Sn}, RHS is zero unless $\mu$ consists of a column $n-|\lambda|$ boxes with some skew-diagram $\mu'$ of size $|\lambda|$ attached to it, in which case it equals $\langle\lambda,\mu'\rangle$. If $\mu'$ is entirely contained outside of the first column, then we must have $\mu'=\lambda$. Otherwise $\mu'$ has at least one box in the first column below the $n-|\lambda|$ boxes corresponding to the sign representation, so it has less than $|\lambda|$ outside the first column, and $\mu=((\lambda')^\vee[n])^\vee$ for some $\lambda'$ with strictly fewer boxes than $\lambda$. This confirms (\ref{eq:find_in_ind}).

Using the decomposition (\ref{eq:find_in_ind}), we can express the traces as
\begin{equation}\label{eq:tr_ind_sum_smaller}
\tr\left(\frac{1}{\tau_m(S_n)}\sum_{\tiny\begin{matrix}\sigma\in S_n\\\sigma^m=1\end{matrix}}\sigma\curvearrowright V_n\right)=\chi_{(\lambda^\vee[n])^\vee}^{m\tn{-tors}}+\sum_{\lambda'}\chi_{((\lambda')^\vee[n])^\vee}^{m\tn{-tors}},
\end{equation}
where the sum over $\lambda'$ contains boundedly many terms, all with $|\lambda'|<|\lambda|$. We will now use the knowledge of the LHS to apply induction on $|\lambda|$.

The base case is the empty partition $\lambda=\emptyset\vdash0$. Then $(\lambda^\vee)[n]^\vee$ is a single column of $n$ boxes, which by Section~\ref{sec:rep_Sn} corresponds to the sign representation of $S_n$. Since
\begin{equation}\label{eq:sgn_mtors}
\chi_\tn{sgn}^{m\tn{-tors}}=\frac{2\tau_m(A_{n-|\lambda|})-\tau_m(S_{n-|\lambda|})}{\tau_m(S_{n-|\lambda|})},
\end{equation}
it follows from \cite[Lemma 2.18]{fuchsian_gps_coverings_Riemann_surfaces_subgp_growth_random_quots_walks_LiebeckShalev04} that $\chi^{m\tn{-tors}}_\tn{sgn}$ decays as $O(\exp(-c(n-|\lambda|)^{\frac{1}{m}}))$ for some $c>0$.

As the induction step, we need to prove a statement analogous to Proposition \ref{prop:tall_tors_bd} but for the induced representation $V_n$ instead of $(\lambda^\vee)[n]^\vee$. This will give the decay rate of the LHS of equation (\ref{eq:tr_ind_sum_smaller}), and since the decay of each summand coming from $\lambda'$ is controlled by the induction hypothesis, the term $\chi^{m\tn{-tors}}_{(\lambda^\vee[n])^\vee}$ will also have to decay at the specified rate. The Frobenius formula for the character of an induced representation can be stated, for $x\in G$, $H\leq G$ and $H\curvearrowright U$, in the form
\begin{equation*}
\tr\left(x\curvearrowright\tn{Ind}^G_HU\right)=\frac{1}{\#H}\sum_{y\in H}\tr\left(y\curvearrowright U\right)\cdot\sum_{z\in G}\mathbbm{1}_{\{z^{-1}yz=x\}}.
\end{equation*}
Applying it to $G=S_n$, $H=S_{n-|\lambda|}\times S_{[n-|\lambda|+1,n]}$, $U=\tn{sgn}\otimes\lambda$, and summing over $m$-torsion elements gives
\begin{equation}\label{eq:factor_tr_mtors}
\tr\left(\sum_{\tiny\begin{matrix}\sigma\in S_n\\\sigma^m=1\end{matrix}}\hspace*{-5pt}\sigma\curvearrowright V_n\right)=\frac{n!}{(n-|\lambda|)!|\lambda|!}\sum_{\tiny\sigma\in S_{n-|\lambda|}\times S_{[n-|\lambda|+1,n]}}\hspace*{-20pt}\tr\left(\sigma\curvearrowright\tn{sgn}\otimes\lambda\right)\!\cdot\!\mathbbm{1}_{\{\sigma^m=1\}}.
\end{equation}
This equation can also be obtained by applying Lemma~\ref{lmm:change_sum_tors} with $H,G,U$ as above and $x=1$, and using (\ref{eq:Frob_reciprocity}). Equation (\ref{eq:factor_tr_mtors}) can be rewritten as
\begin{align*}
\chi_{V_n}^{m\tn{-tors}}=&\frac{1}{\tau_m(S_n)}\begin{pmatrix}n\\|\lambda|\end{pmatrix}\tr\left(\sum_{\tiny\begin{matrix}\sigma\in S_{n-|\lambda|}\\\sigma^m=1\end{matrix}}\sigma\curvearrowright\tn{sgn}\right)\tr\left(\sum_{\tiny\begin{matrix}\sigma\in S_{[n-|\lambda|+1,n]}\\\sigma^m=1\end{matrix}}\sigma\curvearrowright\lambda\right)=\\
=&\begin{pmatrix}n\\|\lambda|\end{pmatrix}\frac{\tau_m(S_{n-|\lambda|})}{\tau_m(S_n)}\cdot\chi_\tn{sgn}^{m\tn{-tors}}\cdot\tau_m(S_{|\lambda|})\chi_\lambda^{m\tn{-tors}}
\end{align*}
where $\chi_\tn{sgn}^{m\tn{-tors}}$ is the average sign over $m$-torsion elements in $S_{n-|\lambda|}$ (i.e. sgn in the last line is a representation of $S_{n-|\lambda|}$ not $S_n$). We argued around equation (\ref{eq:sgn_mtors}) that $\chi^{m\tn{-tors}}_\tn{sgn}$ decays like the exponential of $-cn^{\frac{1}{m}}$. The other factors are at most polynomial in $n$ (see Lemma~\ref{lmm:Sn_tors_growth} for the ratio of $\tau_m$'s), so the whole expression converges to 0 at such rate (with some smaller but still positive constant $c$).
\end{proof}

\subsection*{Disclosure of computational assistance}

Formulating and proving Proposition~\ref{prop:orbilen} were aided by a numerical brute-force search. LLMs have been used for adversarial reading of this manuscript.

\printbibliography

\end{document}